\documentclass[a4paper,12pt,oneside,reqno]{amsart}

\usepackage{hyperref}
\usepackage[headinclude,DIV13]{typearea}
\areaset{15.1cm}{25.0cm}
\usepackage{amsfonts,amssymb,amsmath,amsthm,bbm}
\usepackage[latin1] {inputenc}
\usepackage{graphicx, psfrag}

\usepackage{bold-extra}

\newtheorem{theorem}{Theorem}[section]
\newtheorem{lemma}[theorem]{Lemma}
\newtheorem{proposition}[theorem]{Proposition}
\newtheorem{corollary}[theorem]{Corollary}

\theoremstyle{definition}

\theoremstyle{remark}
\newtheorem*{remark}{Remark}

\def\paragraph#1{\noindent \textbf{#1}}

\numberwithin{equation}{section}

\def\dist{\mathop{\rm dist}\nolimits}

\def\d{\mathrm{d}}

\def\<{\langle}
\def\>{\rangle}
\def\a{\alpha}
\def\b{\beta}
\def\e{\epsilon}

\def\g{\gamma}

\def\l{\lambda}

\def\s{\sigma}
\def\t{\tau}

\def\o{\omega}

\def\L{\Lambda}
\def\G{\Gamma}
\def\O{\Omega}

\def\del{\partial}
\def\R{{\Bbb R}}  %%
\def\N{{\Bbb N}}  %%
\def\P{{\Bbb P}}  %% carateri piu belle per campi di nombre
\def\Q{{\Bbb Q}}  %%
\def\E{{\Bbb E}}  %%

\let\cal=\mathcal
\def\AA{{\cal A}}

\def\CC{{\cal C}}

\def\EE{{\cal E}}
\def\FF{{\cal F}}
\def\GG{{\cal G}}

\def\II{{\cal I}}

\def\LL{{\cal L}}

\def\OO{{\cal O}}
\def\PP{{\cal P}}
\def\QQ{{\cal Q}}
\def\RR{{\cal R}}
\def\SS{{\cal S}}
\def\TT{{\cal T}}
\def\VV{{\cal V}}

\def\VV{{\cal V}}
\def\WW{{\cal W}}

 \def \G {{\Gamma}}
 \def \L {{\Lambda}}
 
 \def \b {{\beta}}
\def \e {{\epsilon}}
 \def \s {{\sigma}}

 \def \t {{\tau}}

 \def \g {{\gamma}}
 \def \l {{\lambda}}
 \def \d {{\delta}}
 \def \a {{\alpha}}
 \def \o {{\omega}}
 \def \O {{\Omega}}

 \def \del {{\partial}}

 \def \ba {\begin{array}}
 \def \ea {\end{array}}

 \newcommand{\be}{\begin{equation}}
 \newcommand{\ee}{\end{equation}}

\newcommand{\bea}{\begin{eqnarray}}
 \newcommand{\eea}{\end{eqnarray}}
\def\TH(#1){\label{#1}}\def\thv(#1){\ref{#1}}
\def\Eq(#1){\label{#1}}\def\eqv(#1){(\ref{#1})}

\def\sfrac#1#2{{\textstyle{#1\over #2}}}

 \def \1{\mathbbm{1}}
\def\wt {\widetilde}
\def\wh{\widehat}

\begin{document}

 \title[Aging in Metropolis dynamics]
{Convergence of clock processes %of random dynamics\\
 and aging in Metropolis dynamics of a truncated REM}
\author[V. Gayrard]{V\'eronique Gayrard}
 \address{V. Gayrard\\I2M, Aix-Marseille Universit\'e\\
39, rue F. Joliot Curie\\13453 Marseille cedex 13, France}
\email{veronique.gayrard@math.cnrs.fr, veronique@gayrard.net}
%\email{veronique.gayrard@math.cnrs.fr}

\date{\today}

 \begin{abstract}
 
 We study the aging behavior of a truncated version of the Random Energy Model evolving under Metropolis dynamics. 
 We prove that the natural time-time correlation function defined through the overlap function converges
% as the size of the system diverges 
 to an arcsine law distribution function, almost surely in the random environment and in the full range of 
 time scales and temperatures for which such a result can be expected to hold.
%This confirms the predictive power of Bouchaud's REM-like trap model.
%This extends the universality class of Bouchaud's REM-like trap model.
%Thus Metropolis dynamics ages in the same way as Bouchaud's REM-like trap model, confirming the latter as a universal aging mechanism
This establishes that the dynamics ages in the same way as Bouchaud's REM-like trap model, thus extending the universality class of the latter model.   The proof relies on a clock process convergence result of a new type where the number of summands is itself 
  %viewed as 
  a clock process.
 %, that generalizes . 
 This 
 reflects 
 %arises from 
 the fact that the exploration process of Metropolis dynamics is itself an aging process, 
 governed by its own clock.
 %an inner clock
 Both clock processes are shown to converge to stable subordinators below certain critical lines in their time-scale and temperature domains, almost surely in the random environment.
 
% Thus for the first time Bouchaud's REM-like trap model is shown to correctly predict the
 % the aging behavior of a physically realistic microscopic spin glass dynamics.

 %on the basis of 
% This proves for the first time  that the predictions based on Bouchaud's REM-like trap models are correct
%  for a physically realistic microscopic spin glass dynamics.
%
% This proves for the first time  that Bouchaud's REM-like trap models correctly predicts the aging behavior of
% a physically realistic microscopic spin glass dynamics.
%
%Thus for the first time Bouchaud's predictions based on the REM-like trap model are shown to be correct for a physically realistic microscopic spin glass dynamics.
 %
 %This confirms the predictive power of Bouchaud's REM-like trap model in a physically realistic microscopic dynamics
% contributes to proving the univer

 %It plays for the chain $J_n$ the role that $\wt S_n$ plays for $X_n$
%Il joue pour la chaine J le meme role que l'horloge
%$K_n$ itself can be view as a clock process associated to $J_n$
%it can be though of as the clock process of $J_n$.

%Our next theorem shows that this analogy is very deep.
%much deeper
%goes beyond .  

 %can be interpreted as aging of $J_n$/reflects the fact that

 %we prove a clock process convergence theorem of a new kind

 %{\bf ...but this has no impact on physical time}

 %{ $\Rightarrow$ 
 % phenomenological trap models 
 % % % %make the right predictions for aging of Metropolis dynamics of the REM.
 % {\bf correctly predict} aging of Metropolis dynamics of the truncated REM.

 \end{abstract}

\thanks{
I would like to thank the Hausdorff Research Institute for Mathematics and the Institut Henri Poincar\'e where part of this work was carried out.}
%for their kind hospitality.}

 \maketitle

%%%%%%%%%%%%%%%%%%%%%%%%%%%%%%%%%%%%%%%%%%%%%%%%%%%%%%%%%%%%

%%%%%%%%%%%%%%%%%%%%%%%%%%%%%%%%%%%%%%%%%%%%%%%%%%%%%%%%%%%%

 %%%%%%%%%%%% SECTION 1 %%%%%%%%%%%%%%%%%%%%%%%%%%%%%%%%%%%%%%%%%%%%%%

 \section{Introduction and main results}
 \label{1}
 
% \cite{G12} (EJP), \cite{G10b} (REM), \cite{BG13} ($p$-spin) extremal aging \cite{BAGun11} \cite{BGS12}
%\cite{De1} \cite{De2}
%\cite{Vin07} \cite{BCKM98} \cite{BD95} \cite{Bou92}

% truncation does not affect equilibrium except at very low temperatures and all known aging behaviors for RHT eg extremal aging

As evidenced by an extensive body of experiments, glassy systems
%have an anomalously slow relaxation pattern 
are never in equilibrium on laboratory time scales \cite{BCKM98}, \cite{Vin07}; instead, their dynamics become increasingly slower as time elapses. Termed \emph{aging}, this pattern of behavior was most successfully accounted for, at a theoretical level, by 
Bouchaud's 
phenomenological trap models \cite{Bou92}, \cite{BD95}.
%(see references to Section 2.4 in \cite{BCKM98}). 
%These are effective dynamics on reduced state space  whose time-time correlation function successfuly reproduce the power law decays observed experimentally and that characterise aging.
%some of the features
These are effective dynamics that, reviving ideas of Goldstein {\it{et al.}} \cite{Gold69},  model the long time behavior of spin glass dynamics in terms of thermally activated
% (energy) 
barrier crossing in a state space reduced to  the configurations of lowest energy  (see \cite{BCKM98} for a review). 
Main examples of  microscopic systems that trap models aim to describe are Glauber dynamics on state spaces 
$\{-1,1\}^n$ reversible with respect to the
Gibbs measures associated to random Hamiltonians of mean-field spin glasses, such as the Random Energy Model (REM) and $p$-spin SK models \cite{De1}, 
%\cite{De2}, 
\cite{D85}. The link between such dynamics and their associated trap models is, however, simply postulated. 
%and awaits a theoretical and mathematical justification

When trying to establish this link rigorously, a main question that arises is what Glauber dynamics to choose.
While classical
%most desirable
choices are {Metropolis} \cite{Metro} or {Heath-Bath} dynamics \cite{Glaub}, most of the focus so far was on the so-called \emph{Random Hopping} dynamics whose transition rates do not depend on the variation of energy along a given transition but only on the energy of its starting point \cite{BBG03a}, \cite{BBG03b}, \cite{BC06b}, \cite{BBC08}, 
\cite{G10b},  \cite{BG13}, \cite{BGS13}, \cite{BGu12}, \cite{FL09}. 
Although physically unrealistic, the relative simplicity of this choice
%allowed to gain important insights into the causes of aging:
%allowed important insights  into the causes of aging to be gained:
% allowed to gain important insights into what causes activated aging. 
%led to important progress / allowed important progress to be made.
allowed important insights to be gained:
a rigorous justification of the connection between the REM dynamics and trap models 
was given, first on times scales close to equilibrium \cite{BBG02,BBG03a, BBG03b},  later also on shorter 
(but still exponential in $n$) 
 time scales \cite{BC06b}, and these results were partially extended to the $p$-spin SK models \cite{BBC08} on a sub-domain
% limited range
  of times scales, albeit only in law with respect to the random environment and for $p\geq 3$.  
The SK model itself ($p=2$) could be dealt with on time scales that are sub-exponential in $n$ and
again in law with respect to the random environment \cite{BGu12}. 
%In \cite{BGu12} the SK model itself ($p=2$) was dealt with, on time scales that are sub-exponential in $n$ and
%again in law with respect to the random environment.
%A variant of the Random Hopping dynamics of the REM
%(a version of the so-called Bouchaud's asymmetric dynamics in which the asymmetry parameter decays to zero as $n$ diverges)
%%with asymmetry parameter going to zero like $\sqrt{\log n/n}$)
A variant of the so-called Bouchaud's asymmetric dynamics in which the asymmetry parameter tends to zero as $n\uparrow\infty$  is considered in \cite{MM13} for the REM.

%Going 
Beyond model-based analysis, a general aging mechanism was isolated that linked
%links/explain/ relates
aging to the arcsine law for subordinators through the asymptotic behavior
 of a partial sum process called \emph{clock process}. 
First implemented 
%spelt out /explicitly 
in \cite{BC06b} in the setting of Random Hopping dynamics
this mechanism was
% taken further ahead in \cite{G12}
revisited in \cite{G12} and \cite{BG13} where,
using a method developed by Durrett and Resnick \cite{DR78} to prove functional limit theorems for dependent random variables, simple and robust criteria for convergence of clock processes to subordinators were given, suited for dealing
%capable of dealing 
with general Glauber dynamics.
%with a special view to dealing with general Glauber dynamics.
% designed for dealing with
%this scheme was extented to general Glauber dynamics in \cite{G12} 
%a new approach was proposed in \cite{G12} that allowed to extend this scheme to general Glauber dynamics.
%a fresh view on this scheme was taken in \cite{G12} with the aim of dealing with general Glauber dynamics.
Applied to the Random Hopping dynamics of the
REM \cite{G10b} and $p$-spin SK models \cite{BG13},  \cite{BGS13},  these criteria allowed to improve all earlier results,
turning statements previously obtained in law into almost sure statements in the random environment.
In the present paper the approach of \cite{G12}  is applied to Metropolis dynamics of the REM for which it was primarily intended, although only for a truncated version of the REM Hamiltonian. While the ultimate goal is of course to deal with the full REM, 
% the activated dynamics of 
the truncated model  
%already 
does captures a number of features that are present in
%that affect
%/are present in
 the activated dynamics of the full model,
 % (i.e. on time scales that are exponentially large in $n$), 
 and enables us to clarify a number of issues 
%(both technical and phenomenological )
 on a problem for which nothing is 
 %otherwise 
 known 
 %(or predicted) 
 at a theoretical level and no computer simulations are available.
 % based on computer simulations.
%offers/captures a number of features that affect the activated dynamics (ie on time scales of the form...) of the full REM/ clearly remain present in the full model/affects the `long time scales'/ in line with the modelling of the past decades through RHT where ingredients that favour trap models are present./ does not affect equilibrium properties of the model (but at very high temperature/diverging with $n$) and its spectral gap/ this question is beset with a number of hard problems 
%difficulties
%that still await a solution.
% but sheds light on new phenomena/brings important new insights in.../brings us one step further down on the path to the understanding of the full REM's dynamics
%This further step on the road to a full treatment of Metropolis  dynamics of the REM claryfies a number of issues (as {\bf nothing was known} on this problem neither theoretically nor from computater simulations).
%
%In the present paper the approach of \cite{G12} is applied to 
%Metropolis dynamics, for which it was primarily intended, 
%when the latter models
%%describes
% the evolution of a  truncated version of the REM.
%
%when it describes the evolution of a  truncated version of the REM.
%as it describes
%when  describing
% in a truncated REM lanscape.
%albeit in a truncated REM lanscape.
%for a truncated version of the REM Hamiltonian.
%

\subsection{The setting} 
%Model and key objects
\TH(1.1)
Before entering into the details,
%Before entering into more details on the issues raised
%Before going into details 
let us specify the model. Denote by $\VV_n=\{-1,1\}^n$ the n-dimensional discrete cube
%, and by $\EE_n$ its edges set.
and let $(g(x), x\in\VV_n)$ be a collection of independent standard Gaussian random variables, defined on a common probability space $(\O, \FF, \P)$. We will refer to these Gaussians as to the \emph{random environment.}
% In what follows we assume that the sequences $(g(x), x\in\VV_n)$, $n\geq 1$, are defined on a common probability space, $(\O, \FF, \P)$.
 %is defined on a common probability space, denoted $(\O, \FF, \P)$.
%defined on some abstract probability space, $(\O,\FF, \P)$.
%We are interested in dynamic properties of this model that are true (for almost
%all realizations of the random Hamiltonian $H_n$) when $n\uparrow\infty$.
%We therefore assume that the sequence $(H_n(x), x\in\VV_n)$, $n\geq 1$,
%is defined on a common probability space, denoted $(\O, \FF, \P)$.
The Hamiltonian or energy function of the standard REM simply is the random funtion
\be
H_n^{\scriptscriptstyle{\textsf{REM}}}(x)\equiv\sqrt n g(x)\,,\quad x\in\VV_n.
\Eq(1.1.1)
\ee
Given a sequence $u_n>0$ (our \emph{truncation level}) the  \emph{truncated} REM Hamiltonian then is
%consider the truncated REM
\be
H_n(x)\equiv
\begin{cases}
\sqrt n g(x),&\hbox{\rm if}\,\,\,\, 
%w_n(x)\geq r_n(\rho^{\star}_n),\\
g(x)\leq 
%-\b^{-1} \log r_n(\rho^{\star}_n)
-u_n,\\
0,&\hbox{\rm else};
\end{cases}\,,\quad x\in\VV_n.
\Eq(10.1.2)
\ee
Here we follow the physical convention that the configurations of minimal
%negative 
energy are the most stable ones,
that is to say,  Gibbs measure at inverse temperature $\b>0$ is defined as
\be
\textstyle
G_{\b,n}(x)
=e^{-\b H_n(x)}/(\sum_{x\in\VV_n}e^{-\b H_n(x)})\,,\quad x\in \VV_n\,.
\Eq(1.1.3)
\ee
%$G_{\b,n}(x)$ describes the model at equilibrium. 
%The function $u_n$ will be specified later.

We are interested in the single spin-flip continuous time Metropolis dynamics for this model. This is a
Markov jump process $(X_{n}(t), t>0)$ 
%that moves along the edges of 
on $\VV_n$
%and is/that is 
that is usually defined through its jump rates, given by 
%with initial distribution $\mu_n$ and
%(or infinitesimal generator matrix elements)
\be
\l_n(x,y)=
\begin{cases}
\frac{1}{n}
e^{-\b\left[H_n(y)-H_n(x)\right]^+},
%\exp\{-\b\left[H_n(y)-H_n(x)\right]^+\}
%%&\hbox{\rm if $(x,y)\in\EE_n$} , \,\,\,\cr
%%0 , &\hbox{if $(x,y)\notin\EE_n$}\,\,\,\cr
&\hbox{\rm if}\, (x,x')\in\EE_n,\\
%\dist(x,x')=1,\\
0,&\hbox{\rm else};
\end{cases}
\Eq(1.1.10)
\ee
where $a+=\max\{a,0\}$,
$
\EE_n=\left\{(x,y)\in\VV_n\times\VV_n\,:\, \dist(x,y)=1
%\sum_{i=1}^n|x_i-y_i|=2
\right\}
$
is the set of edges of $\VV_n$, and
$\dist(x,x')\equiv\frac 12 \sum_{i=1}^n |x_i-x'_i|$ is  the graph distance on $\VV_n$.
%and $\dist(\cdot,\cdot)$ is the graph distance on $\VV_n$,
%$
%%\be\Eq(graph.dist)
%\dist(x,x')\equiv\frac 12 \sum_{i=1}^n |x_i-x'_i|
%%\ee
%$.
%%%%%%%%%%%%%%%%%%%%%%%%%%%%%%%%%%%%%%%%%%%%%%%%%%%%%%%%%%%
%The clock process-based aging mechanism centers on the observation that, 

Equivalently, $X_n$ can be defined as a time change of its  \emph{jump chain}, 
%which is
namely, 
%defined as
%that is to say, of 
the discrete time chain, $J_n$,
that describes the trajectories of  $X_n$, through the relation
%The clock process-based aging mechanism centers on the observation that, equivalently, $X_n$ can be defined using its  \emph{jump chain} and \emph{clock process},  $J_n$ and $\wt S_n$, through the relation
\be
X(t)=J(\wt S_n^{\leftarrow}(t)), \quad t\geq 0,
\Eq(1.1.14)
\ee
where 
$\wt S_n^{\leftarrow}$ denotes the generalized right continuous inverse of $\wt S_n$, and
%$J_n$ is the \emph{jump chain} of $X_n$, 
$\wt S_n$, the  so-called 
\emph{clock process},
%by the inverse of the clock process, $\wt S_n$,
%The {\it clock process} is then defined as
 is the partial sum process that records the total time spent by $X_n$ along the trajectories of $J_n$. 
%Thus  
%Making these objects explicit,
% in explicit form
Spelling out these objects explicitly,
%Explicitly,
the jump chain is the 
%discrete time 
Markov chain  $(J_n(i), i\in\N)$ on $\VV_n$
%having
with 
one-step transition probabilities
\be
p_n(x,y)=
\frac
{
e^{-\b\left[H_n(y)-H_n(x)\right]^+}
%\exp\left\{-\b\left[H_n(y)-H_n(x)\right]^+\right\}
}{
\sum_{y:(x,y)\in\EE_n}
e^{-\b\left[H_n(y)-H_n(x)\right]^+}
%\exp\left\{-\b\left[H_n(y)-H_n(x)\right]^+\right\}
}\,,\quad\hbox{\rm if}\, (x,y)\in\EE_n\,,
\Eq(1.1.12)
\ee
and $p_n(x,y)=0$ otherwise, and the clock process is given by
%defined through
\be
\wt S_n(k)=\sum_{i=0}^{k-1}\l^{-1}_n(J_n(i))e_{n,i}\,,\quad k\geq 1,
%k\in \N\,,
\Eq(1.1.13)
\ee
where $(e_{n,i}\,,n\in\N, i\in\N)$ is a collection of independent
mean one exponential random variables, independent of $J_n$, and
%where 
the $\l_n(\cdot)$'s are the classical holding time parameters 
%at $x$,
\be
\l_n(x)\equiv
\frac{1}{n}
\sum_{y:(x,y)\in\EE_n}
e^{-\b\left[H_n(y)-H_n(x)\right]^+}
%\exp\left\{-\b\left[H_n(y)-H_n(x)\right]^+\right\}
\,,\quad\forall x\in\VV_n.
\Eq(1.1.11)
\ee

In the clock process-based aging mechanism, one aims to infer knowledge of the aging behavior of $X_n$
 as $n\uparrow\infty$ from the asymptotic behavior of the properly rescaled clock process, using relation \eqv(1.1.14).
%of $c_n^{-1}\wt S_n$. 
%More precisely,
%To make this precise
%In order 
To formulate this more precisely
%To give this a precise formulation
%The main question of interest can then be formulated as follows.
let $K_n(t)$ be a nondecreasing right continuous function with range $\{0,1,2,\dots\}$ and let $c_n$ be a nondecreasing sequence. Both $K_n(t)$ and $c_n$ are time scales. 
Consider the re-scaled clock process
\be
S_n(t)=c_n^{-1}\wt S_n(K_n(t))\,,\quad t\geq 0.
\Eq(2.1.0)
%anciennement \Eq(2.1.1) mais double definition => transmorme en \Eq(2.1.0)
\ee
%Note the doubly stochastic nature of $S_n$
This is a doubly stochastic object: one the one hand,  for each fixed realization of the random environment  (that is, of the random Hamiltonian $H_n$), $S_n$ is a partial sum process with increasing paths that increase only by jumps
% in Skorokhod's space $D((0,\infty])$
and whose  increments depend
%and increments that depend
 on the $J_n(i)$'s and the $e_{n,i}$'s; on the other hand,  both the $\l_n(\cdot)$'s and the law of $J_n$ depend on the random environment.
 %$H_n$.
%
% which, for each fixed realization of the random Hamiltonian $H_n$ (or \emph{random environment}), is a partial sum process with paths in Skorokhod's space $D((0,\infty])$ and increments that depend
% on the $J_n(i)$'s and the $e_{n,i}$'s, whereas both the $\l_n(\cdot)$'s and the law of $J_n$ depend on $H_n$.
%
% the random Hamiltonian.
% the process $S_n$ depends on the random Hamiltonian (both through the $\l_n(\cdot)$'s and the law of $J_n$).
%is a random variable on probability space $(\O, \FF, \P)$ of the random Hamiltonain. 
%One then seeks time scales 
One then asks whether there exist time scales
$K_n(t)$ and $c_n$ that make $S_n$ converge weakly, as $n\uparrow\infty$, as a sequence of random elements
%the space 
 in Skorokhod's space $D((0,\infty])$, $\P$-almost surely in the random environment.
%with respect to law $\P$ of the random Hamiltonian
%
% (or, failing this,  in the strongest possible convergence mode with respect to  $\P$).
%
% in the strongest possible convergence mode with respect to the law $\P$ of the random Hamiltonian ($\P$-almost surely being the convergence mode that yields the most useful information from a physics view point). 
%Once this is
Such a result will be useful for deriving aging information
% in the sense that/provided that
if it enables one to control the behavior of 
%certain 
the two-time correlation functions 
%that serve to quantify this phenomenon,
that are used in theoretical physics to quantify this phenomenon,
the natural choice in mean-field models being the two-time overlap function
%a typical/classical choice being the two-time overlap function
%classical choice being the two-time overlap function
%
%such function being the overlap function
%aging in physics, 
%In mean-field models, a typical such function is the overlap function
%a typical choice in mean-field models being the overlap function
\be
\CC_n(t,s)=n^{-1}\big(X_n(c_n t), X_n(c_n (t+s)\big)
%\CC_n(t,s)= \PP\left(n^{-1}\big(X_n(c_n t), X_n(c_n (t+s)\big)>1-\rho\right),\quad \rho \in (0,1),
\Eq(intro.0)
\ee
where $(\cdot,\cdot)$ denotes the inner product in $\R^n$. 
Clearly, how successful this can be strongly depends on the topology on the space $D((0,\infty])$ in which weak convergence of $S_n$ is obtained.
%(The success of this second step will strongly depend on the topology on the space $D((0,\infty])$ in which weak convergence of $S_n$ is obtained.)
%Finally,  \emph{normal} aging occurs if 
\emph{Normal} aging is then said to occur if, 
%in some sense
for some convergence mode,
%in a sense to be specified
%for some convergence mode with respect to the process and the random environment,
\be
\lim_{n\rightarrow\infty}\CC_n(t,s)=\CC_\infty(t/s)
\Eq(intro.1)
\ee
 for some non trivial function $\CC_\infty$ (see \cite{G12}
 %, \cite{GS14} 
 for more general aging behaviors).
% this behavior characterizes \emph{normal aging}, see \cite{G12}, \cite{GS14} for more general ones).
%; other limits can be taken, and other limiting functions obtained 
%The key idea of \cite{BC06b} is that if
The key idea put forward in \cite{BC06b} is that if
% will result from the arcsine law for stable subordinators whenever 
$S_n$ converges to an $\a$-stable subordinator with $\a\in (0,1)$ then \eqv(intro.1) is nothing but a manifestation of the self-similarity of such subordinators, as  captured by
%echoed through?
%expressed through 
the Dynkin-Lamperti arcsine law Theorem. 
%The key idea of \cite{G12}, taken up in \cite{BG13}, \cite{BGS13},  \cite{BGS13}, is 

%If this can be done one then asks whether the convergence of the clock process in the
%form obtained is useful for deriving ageing information in the sense that
%one can control the behaviour of 
%%certain 
%the correlation functions used to quantify aging in physics, the typical choice in mean-field models being the overlap %function
%%which is typically chosen to be the overlap function
%\be
%%\CC_n(t,s)= n^{-1}\big(X_n(c_n t), X_n(c_n (t+s)\big),
%\ee
%where $(\cdot,\cdot)$ denotes the inner product.

%The  following notation are used throughout the paper. 
For future reference, we denote $\FF^J$ and $\FF^X$, respectively, the $\s$-algebra generated by the
variables $J_n$ and $X_n$. We  write $P$ for the law of the process $J_n$, conditional on
the $\s$-algebra $\FF$, i.e. for fixed realizations of the random environment.
Likewise we  call $\PP$ the law of $X_n$ conditional on $\FF$.
If the initial distribution, say $\mu_n$, has to be specified we write $\PP_{\mu_n}$
and $P_{\mu_n}$. Expectation with respect to $\P$, $P$, and $\PP$ are denoted  by $\E$, $E$, and $\EE$, respectively.

\subsection{Results.}
\TH(1.2)
We must now specify the truncation level in \eqv(10.1.2).
%In order to state our results we need to specify the truncation level, $u_n$ , and the time scales, $c_n$ and $K_n$.
%We now come to the choice of the truncation $u_n$.
%Let us now  define the truncation level in \eqv(10.1.2).
%To begin with, we must define the truncation level in \eqv(10.1.2).
Given $c_{\star}>0$, we let $u_n\equiv u_n(c_{\star})$ be the sequence defined through
\be
\Eq(intro.3)
\P(
%\exp\{-\b H_n^{\scriptscriptstyle{\textsf{REM}}}(x))\}
g(x)\leq -u_n(c_{\star}))=n^{-c_{\star}}.
\ee
Viewing the vertices of $\VV_n$ as independently occupied with probability \eqv(intro.3), one sees that
%$n^{-c_{\star}}$
this probability  increases from $0$ to $1$ as $c_{\star}$ decreases from $+\infty$ to $0$, and so,
the set of occupied vertices evolves from the empty set to the entire $\VV_n$.
%Let us view the vertex $x$ as occupied if $g(x)\leq -u_n(c_{\star})$ and vacant else. 
%By \eqv(intro.3), the vertices of $\VV_n$ are independently occupied with probability $n^{-c_{\star}}$.
%As  $c_{\star}$ decreases from $+\infty$ to $0$, $n^{-c_{\star}}$ varies from $0$ to $1$, and
%the set of occupied vertices evolves from the empty set to the entire set $\VV_n$.
Set 
\be
\VV_n^{\star}\equiv \left\{x\in\VV_n\mid \text{$x$ is occupied}\right\}\setminus I^{\star}_n,
%\equiv \cup_{l=1}^{L^{\star}} C^{\star}_{n,l},\quad {L^{\star}}\equiv L(\rho^{\star}_n),
\Eq(intro.4)
\ee
where $I^{\star}_n$ is the set of isolated occupied vertices, namely, $x\in I^{\star}_n$ if it is occupied but none of its $n$ neighbors is. 
Our results
%techniques of proof 
are closely tied to the graph properties of this set. 
%%%%%%%%%%%%%%%%%%%%%%%%%%%%%
%%Formal 
%(The needed results
%%definitions and statements 
%on the evolution of this graph
%%these properties 
%as a function of $c_{\star}$ 
%are given in Section \thv(3).) 
%%%%%%%%%%%%%%%%%%%%%%%%%%%%%%%
Let us only mention here
%Let us emphasize that our
that $c_{\star}$ must be chosen such that $c_{\star}>c^{\scriptstyle{\textsf{crit}}}\geq 2$. This precludes the emergence of a giant connected component and guarantees that, $\P$- almost surely,
the graph of $\VV_n^{\star}$ is made of an exponentially large number ($\approx \OO(2^n/n^{2c_{\star}-1})$) of small, disjoint connected components of size smaller than $n$.
% the size of the largest connected component of the graph remains small compared to $n$ 
%Formal definitions and statements are given in Section \thv(3).
% which means
%%to briefly put this into context, 
%before the emergence of a giant component, and in a regime where the size of the largest component remains small %compared to $n$ (see Section \thv(3) for formal definitions and statements).
%%provide some context
%%put this in perspective
%
%It is known that the graph properties of $\VV_n^{\star}$ undergo a complete change near the critical value $c=1$: if $c_{\star}=1+\epsilon$, $\epsilon>0$, then $\VV_n^{\star}$ decomposes into an exponentially large number, $L^{\star}_n\approx 2^n/n^{2c_{\star}-1}$ of small disjoint, connected components 
%\be
%\VV_n^{\star}=\cup_{l=1}^{L^{\star}} C^{\star}_{n,l}
%\ee
%of size $|C^{\star}_{n,l}|\ll n$, whereas if $c_{\star}<1-\epsilon/\log n$, $\epsilon>0$, a giant connected component emerges that engulfs $\OO(2^{n}/n)$ vertices of $\VV_n^{\star}$
%Also note that 
In explicit form, 
the sequence
 $u_n$ obeys
\be
\textstyle
u_n(c_{\star})
=
\sqrt{2c_{\star}\log n}
-
\left(\frac{\log\log n+\log 4\pi}{2\sqrt{2 c_{\star}\log n}}
+
\OO\Bigl(\frac{1}{\sqrt{\log n}}\Bigr)\right).
\Eq(10.1.3)
\ee
%We assume throughout that $c_{\star}>1$.
%When $c_{\star}$ is an integer \eqv(10.1.3)
%which  is the size of the maximum of $n^{c_{\star}}$ i.i.d.~standard Gaussians. 
%One sees that 
Hence, the truncation only prunes energies
% (or rather small Boltzmann weights),
%values of  $H_n^{\scriptscriptstyle{\textsf{REM}}}$,
 such that $-H_n^{\scriptscriptstyle{\textsf{REM}}}(x)\lesssim\sqrt{2c_{\star}n\log n}$, while activated
aging typically involves energies of size $-H_n^{\scriptscriptstyle{\textsf{REM}}}(x)\geq \gamma n$, $\gamma>0$, that is to say, of the order of $\max_{x\in\VV_n}( -H_n^{\scriptscriptstyle{\textsf{REM}}}(x))$.
%One also sees that there is a strictly positive probability to find a vertex $x$ with $H_n(x)>0$ 
% within Hamming distance $\dist(x,y)\leq c_{\star}$ of any given vertex $y$. 

We are  concerned with finding 
%We now seek 
sequences $c_n$ and $K_n$
 for which the rescaled clock process \eqv(2.1.0) converges
%such that
%that make the rescaled clock process \eqv(2.1.0) converge 
for some (ideally, the smallest possible) $c_{\star}>c^{\scriptstyle{\textsf{crit}}}$.
%We now come to the choice of the time scales $c_n$ and $K_n$.
Note that in physical terms, $c_n$ is the time scale on which the continuous time process $X_n$ is observed, while  $K_n(t)$ is the total number of steps the 
 %underlying 
discrete time chain $J_n$ takes during the period of observation.
 % which is nothing but the clock itself
 %
%%%%%%%%%%%%%%%%%%%%%%%%%%%% dire ca ailleurs
%Since the chain takes one step per time unit, $K_n(t)$ is similar to a time.
%Thoughout this paper the word time
%%is applied 
%is used either in reference to $X_n$ (physical time), or to $J_n$ (number of steps), but the meaning should always be %clear from the context.
In all previously mentioned works on mean-field spin glasses
(that is, the REM and $p$-spin SK models with $p\geq 3$)
where convergence of \eqv(2.1.0) could be proved, this was 
on time scales of the form $c_n\sim \exp(\b\gamma n)$, $\gamma>0$.
% It is worth recalling here that in all 
%%% previously mentioned
% works on aging of Random Hopping dynamics of mean-field spin glasses
%% %,previously mentioned,
%%%mentioned earlier, 
%% % evolving under the Random Hopping dynamics
%where activated aging could be proved (that is, the REM and $p$-spin SK models with $p\geq 3$), this was 
%on time scales of the form $c_n\sim \exp(\b\gamma n)$, $\gamma>0$.
%
%for different domains of the parameters $\gamma,\b>0$.
%(in the REM the optimal domain is known
%sometimes optimal.
Furthermore,
% in order for the associated clock process
% (sometimes blocked version of the clock) 
%to converge to a stable subordinator, 
$K_n$  invariably had to be chosen
%must be chosen 
of the form $K_n(t)=\lfloor a_n t\rfloor$, where $a_n$ is defined through 
%the relation
%satisfies 
$
a_n\P(w_n(x)\geq c_n)\sim 1
$, 
and where $w_n(x)$ denotes the Boltzmann weight of the 
considered model; 
%model under study.
%and this, irrespective of the model.
in the standard REM, this is
\be
w_n(x)\equiv
\exp\{-\b H_n^{\scriptscriptstyle{\textsf{REM}}}(x))\}\,,\quad  x\in\VV_n.
%\exp\{-\b \sqrt n g(x)\}\,,\quad  x\in\VV_n\,,
\Eq(1.1.2)
\ee
Finally, a common $\a$-stable subordinator emerged as the limit of the clock processes.

%Turning to Metropolis dynamics of the truncated REM, 
As might reasonably be expected, the physical time scale, $c_n$, on which activated aging occurs  in Metropolis dynamics 
% of the truncated REM 
is the same as in the Random Hopping dynamics.  
% what is completely different, however
%In contrast, the choice of $K_n$ is strongly affected
%What is strongly affected, however, is the choice of $K_n$
%What does differ is $K_n$ which, given a sequence $a_n$, is now defined through
What does differ, however, is the choice of $K_n$. Given  a sequence $a_n$, we now set
%here let $K_n(t)$ be defined by
%, which we choose as follows: given a sequence $a_n$, let $K_n(t)$ be defined by
\be
K_n(t)\equiv\min\left\{k\geq 1 \,\Big|\, \textstyle\sum_{i=0}^{k-1}
{
\displaystyle
\1_{\{J_n(i)\in \VV_n\setminus \VV_n^{\star}\}}
}
=\lfloor a_n t\rfloor\right\}
\,,\quad t\geq 0.
\Eq(1.3.1)
\ee
%(this is a nondecreasing right continuous random function with range $\{0,1,2,\dots\}$).
This is the number of steps $J_n$ must take in order to take $\lfloor a_n t\rfloor$ steps
% in $\VV_n\setminus \VV_n^{\star}$.
 outside $\VV_n^{\star}$.
%%the set  $\VV_n^{\star}$ defined in \eqv(intro.4).
%We now seek sequences $a_n$ and $c_n$ that
Our first theorem states that the resulting 
rescaled 
clock process \eqv(2.1.0)
%, $S_n$,
%with this choice,  $S_n$ 
converges to the same limiting subordinator as in the Random Hopping dynamics, for  the very same 
%and this for exactly the same choice of the 
sequences $a_n$ and $c_n$, and in the same $\b$ range.
%the same values of 
% domain of $\b$.
For  $0\leq\varepsilon\leq 1$ and  $0<\beta<\infty$, set
\bea
\Eq(1.3.3)
%\Eq(1.2.4)
&\b_c(\varepsilon)=\sqrt{\varepsilon2\log 2},
\\
\Eq(1.3.4)
%\Eq(1.2.5)
&\a(\varepsilon)=\b_c(\varepsilon)/\b.
\eea
%In this paper, 
Throughout this paper the 
 initial distribution is  the uniform distribution on $\VV_n\setminus\VV_n^{\star}$.
%(this choice is made for simplicity only).
%; more general choices can be made leaving the results unchanged).
%(See ... for a discussion on  initial distributions).

\begin{theorem}
     \label{1.theo1.Main} 

Assume that $c_{\star}>3$. Given  $0<\varepsilon<1$ let $a_n$ and $c_n$ be defined through
\be
\lim_{n\rightarrow\infty}\frac{\log a_n}{n\log 2}=\varepsilon,\quad a_n\P(w_n(x)\geq c_n)\sim 1.
\Eq(1.theo1.M1)
\ee
%
% ici je demande \varepsilon<1 pour etre tres en dedans du domaine des echelles intermediares et ne pas avoir a
% distinguer convergence ps et en proba, et \varepsilon>0 pour ne pas avoir un alpha qui tends vers zero, auquel cas
% il faut etudier le processus extremal
%
Then, for all $0<\varepsilon<1$ and all $\b>\b_c(\varepsilon)$, $\P$-almost surely,
\be
 S_n\Rightarrow_{J_1}  S_{\infty}
 \Eq(1.theo1.M2)
\ee
where $S_{\infty}$ is  a stable subordinator with zero drift and L\'evy measure
$\nu$ defined through
\be
\nu(u,\infty)= u^{-\a(\varepsilon)}\a(\varepsilon)\G(\a(\varepsilon)),\quad u>0,
\Eq(1.theo1.M3)
\ee 
and where
$\Rightarrow_{J_1}$ denotes weak convergence in the space
$D([0,\infty))$ of c\`adl\`ag functions equipped with the Skorokhod $J_1$-topology. 
\end{theorem}

In the rest of the paper the symbol $\Rightarrow_{J_1}$ (sometimes only $\Rightarrow$) has the same meaning as in Theorem \thv(1.theo1.Main).

Let us now elucidate the meaning of $K_n$.
There is a clear parallel between the definitions \eqv(1.3.1) and  \eqv(1.1.13) of $K_n$ and $\wt S_n$.
%The parallel between the definitions \eqv(1.3.1) and  \eqv(1.1.13) of $K_n$ and $\wt S_n$ is clear.
%Just as 
Like $\wt S_n$, $K_n$ is similar to a time, each step of the chain $J_n$ lasting one time unit.
%Since the chain $J_n$ takes one step per time unit, $K_n(t)$ is similar to a time, just as $\wt S_n$. 
%Since the chain $J_n$ takes one step per time unit, $K_n(t)$, like $\wt S_n$, is similar to a time.
%Just as 
Just like $\wt S_n$  also, it is a function of an underlying 'faster chain', 
%immersed chain
namely, the chain $J_n$ observed only at 
%the times of
 its visits to $\VV_n\setminus\VV_n^{\star}$.
Thus $K_n$ can be viewed as  the total time spent by the chain $J_n$ along the first $\lfloor a_n t\rfloor$ steps of that  fast chain -- in other words, as a clock process for $J_n$.
% Thus $K_n$ is
%% nothing but 
%the total time spent by the chain $J_n$ along the first $\lfloor a_n t\rfloor$ steps of that  'faster chain', and so,  it plays %the role of a clock process for $J_n$.
One may probe this parallel further by asking if there exist sequences $b_n$
% (or \emph{auxiliary time scales}) 
for which the rescaled process  $b_n^{-1}K_n$ converges.  As the next theorem shows, the nature of the limit
undergoes a transition at the critical value $\b=2\b_c(\varepsilon/2)$.
%% depends  on whether $\b$ is larger or smaller than the critical value $\b=2\b_c(\varepsilon/2)$.
%% depends  on how $\b$ compares to the critical value $\b=2\b_c(\varepsilon/2)$.

\begin{theorem}
  \label{1.theo2.Main} 
Assume that  $c_{\star}>3$ and, given  $0<\varepsilon<1$, let  $a_n$ be as in Theorem \thv(1.theo1.Main).
\item{(i)} If $\b> 2\b_c(\varepsilon/2)$, let $b_n$ be  defined through
%\be
$
\sqrt{na_n}\P(w_n(x)\geq (n-1) b_n)\sim 1.
$
%\ee
Then, for all $0<\varepsilon<1$ and all $\b> 2\b_c(\varepsilon/2)$, $\P$-almost surely,
\be
b_n^{-1}K_n \Rightarrow_{J_1}  S^\dagger_{\infty},
\Eq(1.theo2.M1)
\ee
where $S^\dagger_{\infty}$ is a stable subordinator with zero drift and L\'evy measure
$\nu^\dagger$ defined through
\be
\nu^\dagger(u,\infty)=
u^{-2\a(\varepsilon/2)}2\a(\varepsilon/2)\G(2\a(\varepsilon/2)),\quad u>0.
\Eq(1.theo2.M2)
\ee 
%and where
%$\Rightarrow_{J_1}$ has the same meaning as in Theorem \thv(1.theo1.Main).
%denotes weak convergence in the space $D([0,\infty))$ of c\`adl\`ag functions equipped with the Skorokhod $J_1$ topology. 
%\smallskip

\item{(ii)} If $0<\b< 2\b_c(\varepsilon/2)$, set
$
b_n=a_n\exp(n(\b/2)^2)/(\b\sqrt{\pi n})
$.
%$
%b_n=a_n\exp(n(\b/2)^2)/((\b/\sqrt{2})\sqrt{2\pi n})
%$.
Then, for all $0<\varepsilon<1$ and all $\b< 2\b_c(\varepsilon/2)$, $\P$-almost surely,
\be
(b_n^{-1}K_n(t), t\geq 0)  \underset{n\rightarrow\infty}{\overset{\PP - a.s.}{\longrightarrow}} (t,  t\geq 0),
%(b_n^{-1}K_n(t), t\geq 0)  \stackrel{\PP - a.s.}{\longrightarrow}  (t,  t\geq 0),
\Eq(1.theo2.M3)
\ee
where convergence holds in the space $C([0,\infty))$ of continuous functions equipped 
with the topology of the uniform convergence on compact sets.
%
% alternativement
%
%Then, for all $0<\varepsilon<1$, all $\b< 2\b_c(\varepsilon/2)$,  and  all $0<T<\infty$, $\P$-almost surely,
%\be
%(b_n^{-1}K_n(t), 0\leq t\leq T)  \stackrel{\PP - a.s.}{\longrightarrow}  (t, 0\leq t\leq T),
%\Eq(1.theo2.M3)
%\ee
%where convergence holds in the space $C([0,T])$ of continuous functions equipped 
%with the topology of the uniform convergence. 
%
% est-ce que ca veut dire ceci: 
%where convergence holds in the space $C([0,\infty))$ of continuous functions equipped 
%with the topology of the uniform convergence on compact sets
% ????
% apparemment oui : http://en.wikipedia.org/wiki/Uniform_norm
%
% compacts  http://www.cmap.polytechnique.fr/~lefebvre/SEMESTRE_EV2/Cours2.pdf
%
\end{theorem}

\begin{remark} A transition similar to that of Theorem \thv(1.theo2.Main)
is present in $S_n$ at the critical value $\b= \b_c(\varepsilon)$.
%A transition similar to that which is present in $b_n^{-1}K_n$ at the critical value 
%$\b= 2\b_c(\varepsilon/2)$, occurs  in $S_n$ at the critical value $\b= \b_c(\varepsilon)$.
%More precisely, one can prove (along the lines of the proof of Theorem \thv(1.theo2.Main)) 
%that for $\b< \b_c(\varepsilon)$, $0<\varepsilon<1$, and suitably chosen $c_n$, 
%\eqv(1.theo2.M3) holds with $b_n^{-1}K_n$ replaced by $S_n$.
Since in the region $\b< \b_c(\varepsilon)$ activated aging is interrupted
(and in order not to make the paper longer) we leave out the explicit statement.
%to restrain the size of this paper

%This dichotomy is the main signature of a dynamical phase transition in that model.
% would mean that we have a continuum of dynamical transitions that depend on 
%the choice of the observation time scale????????

\end{remark}

The occurence
%emergence
 of stable subordinators as limits of both $S_n$ and $b_n^{-1}K_n$ 
above the critical lines $\b= \b_c(\varepsilon)$ and $\b= 2\b_c(\varepsilon/2)$, $0<\varepsilon<1$, respectively, 
%(see \eqv(1.theo1.M2) and \eqv(1.theo2.M1))
can be explained through
 % is driven by 
 a single, universal mechanism 
%The occurence of stable subordinators as limits of both $S_n$ and $b_n^{-1}K_n$ 
% in \eqv(1.theo1.M2) and \eqv(1.theo2.M1) can be explained through a common, universal mechanism that 
which is best described as
%that can be described as 
 an exploration 
 mechanism of a set of \emph{extreme accessible states}
 whose effective waiting times are \emph{heavy tailed}. 
%What triggers this mechanism, however, is completely different
%The reasons behind this mechanism, however, are completely different
%depending on whether one considers $S_n$ or $b_n^{-1}K_n$.
What gives rise to this mechanism, however, is very different
depending on whether one considers $S_n$ or $b_n^{-1}K_n$.
%The reasons that make this occur, however, are completely different.
%depending on whether one considers $S_n$ or $b_n^{-1}K_n$.
%The reasons why such a mechanism emerges, however, are completely different 
%The reasons why such a mechanism should arise, however, are completely different 
Let us briefly explain this.
%Let us explain this in a little more detail. 

When dealing with $S_n$, the processes at work are analogous to those 
already present in the Random Hopping dynamic of the REM: the set of extreme accessible states identifies with the vertices such that $w_n(x)\sim c_n$,
%and
%their effective holding times have exponential tails of mean value $w_n(x)$
%which, scaled down by $c_n$,  are asymptotically heavy tailed with parameter $\a(\varepsilon)$. 
and most such vertices belong to the set $\II_n^{\star}$ of isolated occupied vertices
of \eqv(intro.4), but $J_n$ typically does not revisit  the elements of $\II_n^{\star}$  twice so that 
the associated effective waiting times typically coincide with the exponential holding times 
$\l^{-1}_n(x)e_{n,i}=w_n(x)e_{n,i}$ (see \eqv(1.1.13)) and these, scaled down by $c_n$,  are asymptotically 
heavy tailed with parameter $\a(\varepsilon)$. 

This is in sharp contrast with the mechanisms that govern the behavior of $b_n^{-1}K_n$.
%The situation is very different for $b_n^{-1}K_n$.
%This is in sharp contrast with the mechanisms behind the convergence of $b_n^{-1}K_n$.
%This is in sharp contrast with the mechanisms that makes $b_n^{-1}K_n$ converge. 
Viewing the set $\VV_n^{\star}\cup \II_n^{\star}$
as the level set of the REM's landscape, and  its disjoint components as separated valleys, $K_n$ can be interpreted as the sum of the sojourn times in
%successive 
the valleys of size $\geq 2$ that $J_n$  visits along its path.
% (to the exclusion of valleys made of a single vertex). 
Thus holding times now arise dynamically from metastable trapping times in local valleys.
%It turns out 
%What we prove is 
The analysis of these times reveals 
that the set of extreme accessible states is the set of pairs  
$(x,y)\in\EE_n$ such that $\min(w_n(x),w_n(y))\sim b_n$, that
%the associated
their effective waiting times have exponential tails of mean value $\min(w_n(x),w_n(y))$, and that,
scaled down by $b_n$, these waiting times are asymptotically heavy tailed with parameter $2\a(\varepsilon/2)$.

Below the critical line $\b= 2\b_c(\varepsilon/2)$, $0<\varepsilon<1$, this picture breaks down. The leading contributions to $b_n^{-1}K_n$ no longer come from 
%rare, 
%extreme events, but now come from typical events that consist of visits to valleys 
%whose effective mean waiting times have finite mean values.
%
extreme events but from \emph{typical} events
%and
that
 consist of visits 
to valleys whose effective mean waiting times have finite mean values.
%
%the rare visits to the valleys with largest waiting times (that is, from extreme events) but from the very large number of visits to valleys that are "typical". 
Note that even here,
%in this domain of the $\b,\varepsilon$ parameters, 
%%%% of Metropolis dynamics 
the jump chain  does not resemble the symmetric random walk.
%%%%  of the Random hopping dynamics. 
%%%% ndeed Theorem \thv(1.theo1.Main) and Theorem \thv(1.theo2.Main) 
%%%% This shows that 
%%%% n fact, 
%%%% Therefore 
%%%% This shows that
In fact, our results show that on the time scales of activated aging,
 %the behavior of 
 Metropolis dynamics never can be reduced to 
%or approximated by 
%that of
 the Random Hopping dynamics, just as the latter
% Random Hopping dynamics itself 
cannot be reduced to Bouchaud's phenomenological trap model.
%
%%%%%%%%%%%%%%%%%%%%%%%%%%%%%%%%%%%%%%%%%%%%%%%%%%%%%%%%%%
% pour ce choix d'echelle de temps $b_n$ la jump chain a le temps d'aller a l'equilibre (i.e. b_n>trou spectral de J_n)
% mais la mesure invariante de J_n n'est jamais la mesure uniforme
%
%+ la troncation infuence ceci a \beta qui decroissent vers zero avec n, donc seulement a tres tres haute temperature
%
%%%%%%%%%%%%%%%%%%%%%%%%%%%%%%%%%%%%%%%%%%%%%%%%%%%%%%%%%%%
%
%It is quite remarkable that despite this, Bouchaud's  trap model does correctly predict the aging behavior of both %dynamics:
%%Metropolis 
%both dynamics:
%Notwithstanding this
%despite this,
%
%It is very remarkable that despite this, all three dynamics have the same limiting correlation function:
%However the behavior of the jump chain has no impact on the physical time
%However 
%
%%%%%%%%%%%%%%%%%%%%%%%%%%%%%%%%%%%%%%%%%%%%%%%%%%%%%%%%%%
%Despite this, all three dynamics have the same limiting correlation function:
%both dynamics:
Despite this Bouchaud's  trap model does correctly predict the aging behavior of 
%%Metropolis dynamics:
both dynamics:
% and both have the same limiting correlation function:

%Despite this Bouchaud's  trap model does correctly predict that both dynamics have the same
%% arcsine law as 
%limiting correlation function:

\begin{theorem}[Correlation function]
  \label{1.theo3.Aging} 
 
Let $\CC_n(t,s)$ be defined in \eqv(intro.0). Under the hypothesis of Theorem \thv(1.theo1.Main), for all $\rho \in (0,1)$, $t>0$ and $s>0$, $\P$-almost surely,
\be
%\textstyle
\lim_{n\to \infty }\PP\bigl(\CC_n(t,s)\ge 1-\rho\bigr)
= \frac{\sin\alpha \pi }{\pi }\int_0^{t/(t+s )} u^{\a(\varepsilon) -1}(1-u)^{-\a(\varepsilon) }\,d u.
\Eq(1.theo3.Aging.1)
\ee
\end{theorem}

%Let $A_n^\rho (t,s)$ be the event defined  by $A_n^\rho (t,s)= \{ \CC_n(t,s)\ge 1-\rho\}$.\lfloor  b_n t\rfloor

\begin{remark} The convergence statement  of Theorem \thv(1.theo2.Main), (i), is of course
is a manifestation of the fact that above the critical line the jump chain is itself an aging process. This can be quantified using e.g.~the function
$\CC'_n(t,s)=n^{-1}\big(J_n(\lfloor  b_n t\rfloor), J_n(\lfloor  b_n (t+s)\rfloor\big)$ for which a 
 statement similar to \eqv(1.theo3.Aging.1) can be proved
%under the assumptions of Theorem \thv(1.theo2.Main), (i), 
with $2\a(\varepsilon/2)$ subsituted for $\a(\varepsilon)$.
\end{remark} 

Let us 
%now briefly describe 
highlight
the content of the next two sections. What we need to know about the random graph induced by the truncation \eqv(intro.3) is collected in Section \thv(3). In Section \thv(2) we 
%introduce
isolate 
two processes, called the \emph{front end} and \emph{back end clock processes} (hereafter  \textsc{fecp} and \textsc{becp}), that are central to the proofs of 
Theorem \thv(1.theo1.Main) and Theorem \thv(1.theo2.Main). 
%Subsections \thv(2.1) and  \thv(2.3) contains their definition
We show that the processes $S_n$, respectively $K_n$, can be written as the sum of
\textsc{fecp}, respectively \textsc{becp},  and remainders. Based on this we decompose the proofs of 
Theorem \thv(1.theo1.Main) and Theorem \thv(1.theo2.Main)
into proving  on the one hand that \textsc{fecp} and \textsc{becp} converge, and showing on the other hand that the remainders are asymptotically negligeable. 
This strategy strongly relies on two abstract theorems (Theorem \thv(6.theo2) in Section  \ref{6} and Theorem \thv(7.theo1)  in Section  \ref{7}) that give sufficient conditions for  \textsc{fecp} and \textsc{becp} to converge to L\'evy subordinators.
The organisation of the rest of the paper is detailed at the end of Section \thv(2).

%can be viewed as the sum of ... and remainders

%To prove Thoerem ... we first isolate two important processes, called the \emph{front end} and \emph{back end clock processes} (hereafter \textsc{becp} and \textsc{fecp}).

%To prove Theorems ... we introduce two processes called the \emph{front end} and \emph{back end clock processes} (hereafter \textsc{becp} and \textsc{fecp}). These are the central object of this paper.

%Technically, $b_n^{-1}K_n$ is much harder to control than $S_n$. Notice in particular that its effective effective holding times are correlated in the random environment. 

% Convergence of the back end clock to a subordinator is a manifestation of the fact that {\it the jump chain is itself aging}!!! 

%%%%%%%%%%%%%%%%%%%%%%%%%%%%%%%%%%%%%%%%%%%%%%%%%%%%%%%%%%%%

%%%%%%%%%%%%%%%%%%%%%%%%%%%%%%%%%%%%%%%%%%%%%%%%%%%%%%%%%%%%

\section{Random Graph properties of the REM's landscape}%Properties of the REM's landscape
 \label{3}

% Clearly $G(\VV_n)$ is a bipartite graph, and so are any of its subgraphs.

Given $V\subseteq\VV_n$ we denote by $G\equiv G(V)$ the undirected graph which has vertex set $V$ and edge set consisting of pairs of vertices
$\{x,y\}$ in $V$ with $\dist(x,y)=1$.
%Denote by $w_n(x)$,
%\be
%w_n(x)=
%\exp\{-\b H_n^{\scriptscriptstyle{\textsf{REM}}}(x))\}\,,\quad  x\in\VV_n\,,
%%\exp\{-\b \sqrt n g(x)\}\,,\quad  x\in\VV_n\,,
%\Eq(1.1.2)
%\ee
%the Bolzman weights of the non truncated REM Hamiltonian \eqv(1.1.1).
This short section is concerned with the graph properties of the level sets
\be
V_n(\rho)=\left\{x\in\VV_n\mid w_n(x)\geq r_n(\rho)\right\},
\Eq(3.0.1)
\ee
where, given  $\rho>0$,
% $0<\rho\leq 1$, 
the truncation level $r_n(\rho)$ is the sequence defined through
\be
\Eq(3.0.2)
2^{\rho n}\P(w_n(x)\geq r_n(\rho))=1.
\ee
This is a 
convenient 
reparametrization of \eqv(intro.3), that is,  \eqv(intro.3) follows from \eqv(3.0.2) by taking
\be
\Eq(10.1.1)%et aussi \Eq(10trunc.1.2) de l'ancien fichier
\rho=\rho_n^{\star}
%\rho_n(c_{\star})
\equiv\frac{c_{\star}\log n}{n\log 2},\quad  r_n(\rho_n^{\star})\equiv \exp(\b \sqrt n u_n(c_{\star})).
\ee
Viewing the vertices of $\VV_n$ as independently occupied with probability  $ 2^{-\rho n}$,
questions on $G(V_n(\rho))$ reduce to
%translate into 
questions on random subgraphs of the hypercube graph 
$\QQ_n\equiv G(\VV_n)$.

\subsection{Component structure of $V_n(\rho)$.}
%Characterization of the 
%{"Evolution" of the component structure of $V_n(\rho)$.}
 \label{3.1}
%Largest component --- Emergence of clusters of size $m$ (the cluster process)
%sequence $0\leq\varepsilon_n\leq 1$
% the cluster process
%
The set $V_n(\rho)$ of occupied vertices can be decomposed into components that we classify according to their connectedness and size.
We call $C\subseteq V_n(\rho)$ a connected component of size $|C|$
%of $V_n(\rho)$
if the subgraph $G(C)\subseteq G(V_n(\rho))$ is connected.
All connected components have size $\geq 2$.
We call isolated occupied vertices of $V_n(\rho)$ components of size 1.
%\subsection{Characterization of the component structure.}
%%\subsection{Consequences/(Elementary) Properties of the component structure.}
% \label{3.2}
Given $V_n(\rho)$, $\VV_n$ can uniquely be decomposed into
\be
\VV_n= N_n(\rho)\cup I_n(\rho) \cup\bigl(\cup_{l=1}^L C_{n,l}(\rho)\bigr),\quad L\equiv L_n(\rho),
\Eq(3.0.4)
\ee
where $ N_n(\rho)$ is the set of all non occupied vertices,
% sous entendu dans $V_n(\rho)$
$I_n(\rho)$ is the set of all isolated occupied ones, and $C_{n,l}(\rho)$, $1\leq l\leq L$,
%enumerates the
is a collection of disjoint connected components satisfying
\be
\Eq(3.0.3)
G(V_n(\rho))=\cup_{l=1}^L G(C_{n,l}(\rho)),\quad C_{n,l}(\rho)\cap C_{n,k}(\rho) \,\,\,\forall  l\neq k.
\ee
%$ N_n\equiv V_n(\rho)\setminus\bigl(\cup_{l=1}^L C_{n,l}(\rho)\bigr)$
%$ N_n\equiv \VV_n\setminus V_n(\rho)$
As $\rho$ decreases, the set  $V_n(\rho)$ grows and the graph $G(V_n(\rho))$ potentially acquires new edges. Little is known about such graphs compared to those obtained by selecting edges independently.
%, rather than vertices. 
It is chiefly known
% \cite{R09}, 
\cite{BKL91} that the size of the largest $C_{n,l}(\rho)$
%connected component 
undergoes a transion near the value
$
\rho\approx\frac{\log n}{n\log 2}
$,
with a unique ``giant'' componant of size $\OO(n^{-1}2^n)$ emerging slightly below this value.
We are interested here in choosing $\rho$ in such a way as to garantee that  
 the size of the largest $C_{n,l}(\rho)$ remains small compared to $n$.
 %, $\P$-a.s..
%The next lemma guarantees that for large enough $\rho$, the largest component remain small compared to $n$

\begin{lemma}
  \TH(3.lem2)
Let $0<\eta_n\downarrow 0$ be
%a decreasing sequence
such that $\eta_n\log n\uparrow\infty$ as $n\uparrow\infty$.
Set
\be
\Eq(3.lem2.1)
\rho_n^{\scriptstyle{\textsf{crit}}}\equiv c^{\scriptstyle{\textsf{crit}}}\sfrac{\log n}{n\log 2},
\quad
c^{\scriptstyle{\textsf{crit}}}\equiv
c^{\scriptstyle{\textsf{crit}}}_n=
2\left[1+\sfrac{\log 2}{2\eta_n\log n}\left(1+\sfrac{3\log n}{n{\log 2}}\right)\right].
\ee
There exists $\O^{\scriptscriptstyle{\textsf{CRIT}}}\subset \O$ with $\P\left(\O^{\scriptscriptstyle{\textsf{CRIT}}}\right)=1$
such that on $\O^{\scriptscriptstyle{\textsf{CRIT}}}$, for all but a finite number of indices $n$,
for all $\rho\geq \rho_n^{\scriptstyle{\textsf{crit}}}$, the connected components in \eqv(3.0.4) satisfy
\be
\Eq(3.lem2.2)
2\leq\max_{1\leq l\leq L}\left|C_{n,l}(\rho)\right|\leq n\eta_n\ll n.
\ee
%
% soit calculer with $\rho_n^{\scriptstyle{\textsf{crit}}}$ or with $\rho\geq \rho_n^{\scriptstyle{\textsf{crit}}}$ arbitrary.
%
\end{lemma}

%As an immediate corollary to Lemma \thv(3.lem2) and Lemma \thv(A1.lem1)  we have:
%
%\begin{corollary}
%  \TH(3.cor1)
%On $\O^{\scriptscriptstyle{\textsf{CRIT}}}$, for all but a finite number of indices $n$, 
%the graph $\QQ_n\setminus G(V_n(\rho_n^{\scriptstyle{\textsf{crit}}}))$  is totally connected:
%%%% forms a totally connected giant component
%there exists a path in $\QQ_n\setminus G(V_n(\rho_n^{\scriptstyle{\textsf{crit}}}))$ connecting
%any pair of vertices $\{x,y\}$ in $V_n(\rho_n^{\scriptstyle{\textsf{crit}}})$ such that, at each vertex $z$ of the path
%except at its extremities, 
%$
%w_n(z)< r_n\left(\rho_n^{\scriptstyle{\textsf{crit}}}\right)
%\equiv \exp(\b \sqrt n u_n(c^{\scriptstyle{\textsf{crit}}}_n))
%$,
%where $u_n$ obeys \eqv(10.1.3).
%\end{corollary}

\begin{remark}
Although $c^{\scriptstyle{\textsf{crit}}}$ in \eqv(3.lem2.1) is very likely not optimal, we clearly must have
$c^{\scriptstyle{\textsf{crit}}}>1$.
%its optimal value must be larger than 1. 
%We will see however that improving $c^{\scriptstyle{\textsf{crit}}}$ is not enough to improve the results of Theorem \thv(1.theo1.Main) and Theorem \thv(1.theo2.Main), and that several other obstacles must be overcome.
%and that there are several other obstacles to this.
\end{remark}

\begin{proof}[Proof of Lemma \thv(3.lem2)] 
The proof relies on a lemma that we first state and prove.
% Write $\CC_{n}^{max}(\rho)$ for the largest component in $V_n(\rho)$ and
Define
% the sets
\be
\Eq(3.1.3)
\wt\O_n(m)=\left\{\o\in\O\, \big|\textstyle \max_{1\leq l\leq L}\left|C_{n,l}(\rho)\right|<m
%|\CC_{n}^{max}(\rho)|<m
\right\}\,,\quad m=2,3,\dots
\ee
In what follows $\rho\equiv \rho_n>0$ and $m\equiv m_n>1$ are, respectively, positive and integer valued sequences. To keep the notation simple we do not make this dependence explicit.
\begin{lemma}
  \TH(3.lem1)
 % Let $\rho\equiv \rho_n>0$ and let $m\equiv m_n>1$ be an integer valued sequence.
If
$\rho\geq \rho_n^+(m)\equiv \frac{1}{m}\bigl(1+\frac{(m+2)\log n+\log m!}{n\log 2}\bigr)$
then 
$\P\bigl(\liminf_{n\rightarrow\infty}\wt\O_n(m)\bigr)=1$.
%$
%\d^+_n(m)=\frac{(m+2)\log n+\log m!}{n\log 2}
%%\left(\frac{m+2}{\log 2}\right)
%%\Eq(3.1.1)
%$, 
\end{lemma}

%Clearly the lemma is meaningful as long as that $m$ does not diverge faster than ${n}/{\log n}$.

%%%%%%%%%%%%%%%%%%%%%%%%%%%%%%%%%%% Beguin Proof of Lemma \thv(3.lem1) %%%%%%%%%%%%%%%%%%%%%%%%%%%%%%%%%%%
%\begin{proof}
\begin{proof}[Proof of Lemma \thv(3.lem1)]
Call $\left(\chi_n(x), x\in\VV_n\right)$, $\chi_n(x)\equiv\1_{\left\{w_n(x)\geq r_n(\rho)\right\}}$,
the occupancy variables. These are i.i.d.~Bernoulli r.v.'s with
$
\P\left(\chi_n(x)=1\right)=1-\P\left(\chi_n(x)=0\right) =2^{-\rho n}
$.
Set 
$
\P\bigl(\wt\O^c_n(m)\bigr)
=1-\P\bigl(\wt\O_n(m)\bigr)
=\P\bigl(\exists_{\CC_n\subseteq V_n(\rho) : |\CC_n|=m} \,G(\CC_n)\, \text{is connected}\, \bigr)
$.
%Write $\P\bigl(\wt\O_n(m)\bigr)=1-\P\bigl(\wt\O^c_n(m)\bigr)$,
%$
%%\be
%%\Eq(3.lem1.1)
%\P\bigl(\wt\O^c_n(m)\bigr)=\P\bigl(\exists_{\CC_n\subseteq V_n(\rho) : |\CC_n|=m} \,G(\CC_n)\, \text{is connected}\, 
%\bigr).
%%\ee
%$
%%%%%%%%%%%%%%%%%%%%%%%%%%%%%%%%%%%%%%%%%%%%%%%%%%%%%%%%%%%%%%%%%%%%%%%%%%%%%%%%%%%%%%%%%%%%%%%%%%%%%%%%%%%
%% this is because all connected components $G$ of size larger than $m$ contain at least one subcomponent ,
%% note that here \CC_n is any subset of size $m$, not ncesserarily an element of the components set
%%%%%%%%%%%%%%%%%%%%%%%%%%%%%%%%%%%%%%%%%%%%%%%%%%%%%%%%%%%%%%%%%%%%%%%%%%%%%%%%%%%%%%%%%%%%%%%%%%%%%%%%%%%
The number of connected components of size $m$ is at most $m!n^m2^n$, and by independence, if $|\CC_n|=m$ then
$
\P\bigl(G(\CC_n)\, \text{is connected}\, \bigr)
=\P\bigl(\prod_{x\in\CC_n} \chi_n(x)=1\bigr)
%=q_n^{-|\CC_n|}(\rho)
=(2^{-\rho n})^{m}
$.
%if $|\CC_n|=m$.
Thus, for $\rho\geq \rho_n^+(m)$,
$
%\be
%\Eq(3.lem1.2)
\P\bigl(\wt\O^c_n(m)\bigr)
\leq m!n^m2^{(1-m\rho)n}
\leq m!n^m2^{(1-m\rho_n^+(m))n}
\leq n^{-2},
%\ee
$
so that $\sum_{n\geq 1}\P\bigl(\wt\O^c_n(m)\bigr)<\infty$. The lemma
now follows from the first Borel-Cantelli Lemma.\end{proof}

%%%%%%%%%%%%%%%%%%%%%%%%%%%%%%%%%%%%% End Proof of Lemma \thv(3.lem1) %%%%%%%%%%%%%%%%%%%%%%%%%%%%%%%%%%%%%

%%%%%%%%%%%%%%%%%%%%%%%%%%%%%%%%%%% Beguin Proof of Lemma \thv(3.lem2) %%%%%%%%%%%%%%%%%%%%%%%%%%%%%%%%%%%
%\begin{proof}
 Lemma \thv(3.lem2) follows from Lemma \thv(3.lem1) by taking
$
m=n\eta_n
%m=\frac{n\log 2}{\eta_n\log n}
$,
observing that for this choice
$
\rho_n^{\scriptstyle{\textsf{crit}}}>\rho^+_n(m)
$,
and setting
$
\O^{\scriptscriptstyle{\textsf{CRIT}}}\equiv\liminf_{n\rightarrow \infty}\wt\O_n(m)
$.\end{proof}

\subsection{Truncation and related quantities.}
%elementary properties
%component structure of $V_n(\rho_n^{\star})$
%structure below criticality.}
%{Truncation and graph related properties.}
 \label{3.3}

%We formally define the truncation level in \eqv(10.1.2) as
%%The truncation level $u_n$ in (10.1.2) is formally defined as follows:
%\be
%u_n\equiv (\sqrt n \b)^{-1}\log r_n\left(\rho^{\star}_n\right),
%\Eq(10trunc.1.2)
%\ee
%where, given an integer $c_{\star}>1$. 
%\be
%\Eq(10.1.1)
%\rho^{\star}_n\equiv\frac{c_{\star}\log n}{n{\log 2}}.
%\ee
%(We assume throughout this paper that $c_{\star}>1$; when no assumption on $c_{\star}$ is explicitly made, that one %prevails.)
%%let $r_n(\rho^{\star}_n)$ be the sequence defined through
%%\be
%%\Eq(3.0.2.trunc)
%%2^{n\rho^{\star}_n}\P(w_n(x)\geq r_n(\rho^{\star}_n))=1.
%%\ee
%Eq.~\eqv(10.1.3)  then  follows from Lemma \thv(A1.lem1) and, in view of the remark below Corollary \thv(3.cor1), 
%$u_n$ is nothing but the size of  the maximum of $n^{c_{\star}}$ i.i.d.~standard Gaussians.

%%%%%%%%%%%%%%%%%%%%%%%%%%%%%%%

Throughout the rest of the paper we assume that $c_{\star}>2$ in \eqv(intro.3)
so that $c_{\star}\geq c^{\scriptstyle{\textsf{crit}}}_n$  for  large enough $n$.
 %; when no assumption on $c_{\star}$ is explicitly made, that one prevails.
According to \eqv(3.0.4)-\eqv(3.0.3), for $\rho=\rho_n^{\star}$ as in \eqv(10.1.1), $\VV_n$ be decomposed in a unique way into
\be
\VV_n= N^{\star}_n\cup I^{\star}_n\cup\bigl(\cup_{l=1}^{L^{\star}} C^{\star}_{n,l}\bigr),\quad {L^{\star}}\equiv L(\rho^{\star}_n),
\Eq(10.1.4)
\ee
where $ N^{\star}_n\equiv N_n(\rho^{\star}_n)$,
$I^{\star}_n\equiv I_n(\rho^{\star}_n)$, 
and $C^{\star}_{n,l}\equiv C_{n,l}(\rho^{\star}_n)$, $1\leq l\leq L^{\star}$. 
By construction $H_n(x)=0$ if and only if $x\in  N^{\star}_n$ (see \eqv(10.1.2) and \eqv(intro.3)). Furthermore  $\VV_n^{\star}$ in \eqv(intro.4) becomes
\be
\VV_n^{\star}=\cup_{l=1}^{L^{\star}} C^{\star}_{n,l}.
\Eq(10.1.4bis)
\ee
%is the collection of all disjoint connected components of $G(V_n(\rho^{\star}_n))$.

\begin{lemma}
  \TH(10.lem1)
Assume that $c_{\star}>2$. 
There exists $\O^{\star}\subset \O$ with $\P\left(\O^{\star}\right)=1$
such that on $\O^{\star}$, for all but a finite number of indices $n$ the following holds:
% the graph $\QQ_n\setminus G(V_n(\rho^{\star}_n))$  is connected
%% Thus, there exists a path in $\QQ_n\setminus G(V_n(\rho^{\star}_n))$ connecting
%% any pair of vertices $\{x,y\}$ in $V_n(\rho^{\star}_n)$ such that, at each vertex $z$ visited by the path
%% except at its extremities, $H^{\star}_n(z)=0$.
%%(hence $w_n(z)=1$)
% and
 \be
%2\leq\left|C^{\star}_{n,l}\right|\leq \frac{n\log 2}{(c_{\star}-1-\OO(\log n/n))\log n}\ll n,
%2\leq\left|C^{\star}_{n,l}\right|\leq\frac{n\log 2}{c_{\star}\log n}(1-c_{\star}^{-1}(1+\OO(\log n/n)))^{-1}\ll n,
%2\leq\left|C^{\star}_{n,l}\right|\leq\frac{1}{\rho^{\star}_n}(1-c_{\star}^{-1}(1+\OO(\log n/n)))^{-1}\ll n,
2\leq\left|C^{\star}_{n,l}\right|
\leq\{\rho^{\star}_n[1-2c_{\star}^{-1}(1+\OO(\log n/n))]\}^{-1},\quad1\leq l\leq {L^{\star}}.
\Eq(10.lem1.4)
\ee
% on peut preciser combien on taille 2, 3, etc...
Furthermore,
\bea
\Eq(10.lem1.2)
\left|I^{\star}_n\right|&=&2^n n^{-c_{\star}}(1-n^{-(c_{\star}-1)})(1+\OO(n^{-2(c_{\star}-1)})+o(n^{-c_{\star}})),
% seul o(n^{-c_{\star}}) survit si c_{\star}>2
\\
\Eq(10.lem1.1)
\left|V_n(\rho^{\star}_n)\right|
&=&\bigl|\VV_n\setminus  N^{\star}_n\bigr|
=2^n n^{-c_{\star}}(1-n^{-c_{\star}})(1+o(n^{-c_{\star}})),
%=2^n n^{-c_{\star}}(1+o(n^{-c_{\star}})),
\\
\Eq(10.lem1.3)
\textstyle
\sum_{l=1}^{{L^{\star}}} \left|C^{\star}_{n,l}\right|
%= \left|\cup_{l=1}^{{L^{\star}}}C^{\star}_{n,l}\right|
&=&\left|V_n(\rho^{\star}_n)\setminus I^{\star}_n\right|
=2^n n^{-2c_{\star}+1}(1+\OO(n^{-(c_{\star}-1)})),
\eea
and, denoting by
$
\del A\equiv\{y\in \VV_n\setminus A : \dist(y,A)=1\}
$
the boundary of the set $A\subset\VV_n$,
%ICI
\bea
\Eq(10.lem1.6)
\textstyle
 n \left|C^{\star}_{n,l}\right|(1-\OO(\frac{1}{\log n}))
 \leq \left|\del C^{\star}_{n,l}\right| \leq n \left|C^{\star}_{n,l}\right|,
 \quad \quad \quad \quad
\\
\Eq(10.lem1.7)
\textstyle
|\del C^{\star}_{n,l}\cap\del x|
\geq n (1-\OO(\frac{1}{\log n}))
\,\,\,\text{for all}\,\,\,  x\in  C^{\star}_{n,l}, \quad \quad \quad\,\,
\\
\Eq(10.lem1.5)
\textstyle
 n \left|C^{\star}_{n,l}\right|(1-\OO(\frac{1}{\log n}))
 \leq
\textstyle
\sum_{x\in C^{\star}_{n,l}}\sum_{y\in\del C^{\star}_{n,l}:\{x,y\}\in\EE_n}1
\leq n \left|C^{\star}_{n,l}\right|. 
\eea
\end{lemma}

\begin{proof} 
%The fact that the graph $\QQ_n\setminus G(V_n(\rho^{\star}_n))$   is
%connected follows from Corollary \thv(3.cor1) since $\rho^{\star}_n>\rho_n^{\scriptstyle{\textsf{crit}}}$.
The claim of \eqv(10.lem1.4)  follows from Lemma \thv(3.lem1). Next,
\be
\textstyle
\left|V_n(\rho^{\star}_n)\right|=\sum_{x\in\VV_n}\chi_n(x), \,\,\,
\left|I^{\star}_n\right|=\sum_{x\in\VV_n}\chi_n(x)\prod_{y\in\VV_n:(x,y)\in\EE_n}(1-\chi_n(y)),
\ee
and
$
\sum_{l=1}^{{L^{\star}}} \left|C^{\star}_{n,l}\right|
=\sum_{x\in\VV_n}\chi_n(x)[1-\prod_{y\in\VV_n:(x,y)\in\EE_n}(1-\chi_n(y))]
$,
where, as in the proof of Lemma \thv(3.lem1), $\chi_n(x)\equiv\1_{\left\{w_n(x)\geq r_n(\rho^{\star}_n)\right\}}$
are i.i.d.~Bernoulli r.v.~with 
$
\P\left(\chi_n(x)=1\right)
%= q_n^{-1}(\rho_n)=2^{-\rho^{\star}_n n}
=n^{-c_{\star}}
$.
From these expressions 
\eqv(10.lem1.1), \eqv(10.lem1.2), and \eqv(10.lem1.3) are easily obtained.
% from concentration estimates.
% readily follow from large deviation estimates.
Turning to \eqv(10.lem1.5) note that the sum therein can be written as
$
%\sum_{x\in C^{\star}_{n,l}}\sum_{y\in\del C^{\star}_{n,l}:\{x,y\}\in\EE_n}
\sum_{x\in C^{\star}_{n,l}}(n-d_n(x))
$
where $d_n(x)$ denotes the connectivity of the vertex $x$ in the graph $G(V_n(\rho^{\star}_n))$. 
This, the bound $1\leq d_n(x)\leq \left|C^{\star}_{n,l}\right|$,
and \eqv(10.lem1.4) yield the desired result. 
To prove the lower bound of  \eqv(10.lem1.6) reason that 
each vertex $x$ in $C^{\star}_{n,l}$ has at least $n-d_n(x)$ nearest neighbors vertices in $\del C^{\star}_{n,l}$,
and that no two vertices in $C^{\star}_{n,l}$ can have more than one common nearest neighbor vertex 
in $\del C^{\star}_{n,l}$. 
% si $x$ et $y$ sont dans $C^{\star}_{n,l}$ on ne peut pas trouver 2 points qui sont a la fois dans 
%$\del x$ et dans $\del y$; 
Hence
$
\left|\del C^{\star}_{n,l}\right| 
\geq 
\sum_{x\in C^{\star}_{n,l}}(n-d_n(x)-( \left|C^{\star}_{n,l}\right|-1))
\geq 
\sum_{x\in C^{\star}_{n,l}}(n-2\left|C^{\star}_{n,l}\right|)
$
and the lower bound in  \eqv(10.lem1.6) follows from  \eqv(10.lem1.4).
Eq.~\eqv(10.lem1.7) is proved in the same way since 
$|\del C^{\star}_{n,l}\cap\del x|=n-d(x)$ for $x\in C^{\star}_{n,l}$.
Finally, the upper bound of \eqv(10.lem1.6)  is immediate.
\end{proof}

%%%%%%%%%%%%%%%%%%%%%%%%%%%%%%%%%%%%%%%%%%%%%%%%%%%

We conclude this section with two elementary lemmata that are repeatedly needed.
%two lemmata that will repeatedly be needed in the sequel.
%bounds that will often be needed in the sequel. 
The first expresses the function $r_n(\rho)$ defined through \eqv(3.0.2).

\begin{lemma}
  \TH(9.lem4')
  For all $\rho>0$, possibly depending on $n$, such that $\rho n\uparrow \infty$ as $n\uparrow \infty$,
%\be
%\Eq(A1.lem1.0)
%\frac{\log r_n(\rho)}{\b\sqrt n}=
%\sqrt{\rho n(2\log 2)}-
%\frac{\log(\rho n\log 2)+\log 4\pi}
%{2\sqrt{\rho n(2\log 2)}}
%+o\left((\rho n(2\log 2))^{-1/2}\right)
%\,.
%\ee
%Thus, with the notation of \eqv(1.3.3), \eqv(1.3.4),
\be
\Eq(A1.lem1.1)
r_n(\rho)=\exp\left\{n\b\b_c(\rho)-(\b/2\b_c(\rho))\left[\log(\b^2_c(\rho)n/2)+\log 4\pi\right]+o(\b/\b_c(\rho))\right\}
\,.
\ee
In particular, for $\rho^{\star}_n$ as in \eqv(10.1.1) and $c_{\star}>2$,
\be
\textstyle
r_n\bigl(\rho^{\star}_n\bigr)
=\exp\Bigl\{\b\Bigl(
\sqrt{2c_{\star}n\log n}
-
\sqrt{\frac{ n}{\log n}}
\bigl(\frac{\log\log n}{2\sqrt{2 c_{\star}}}
+
\OO(1)\bigr)\bigr)\Bigr\}.
\Eq(9.lem4'.2)
\ee
\end{lemma}
%\begin{proof}
%%[Proof of Lemma \thv(9.lem4')]
%This follows from \eqv(10.1.1) and
%Lemma \thv(A1.lem1). 
%\end{proof}
\begin{proof}
 %[Proof of Lemma \thv(A1.lem1)]
 Denote by  $\Phi$ and $\phi$ the standard Gaussian distribution function and density, respectively.
Setting  $b_n=2^{\rho n}$ and $\overline B_n={\log r_n(\rho)}/{\b\sqrt n}$, \eqv(3.0.2) becomes
$
b_n\bigl(1-\Phi(\overline B_n)\bigr)=1
$.
It is shown in \cite{G10b} (see  paragraph below (2.20)) that
$
(\overline B_n-B_n)B_n=o(1)
$
where $B_n$ is defined through
$
b_n\frac{\phi(B_n)}{B_n}=1
$.
Eq.~\eqv(A1.lem1.1) then readily follows from the well known fact that
% expression
% for $B_n$ 
(see \cite{Cr}, p. 374)
$
B_n=(2\log b_n)^{\frac{1}{2}}-\sfrac{1}{2}(\log\log b_n +\log 4\pi)/(2\log b_n)^{\frac{1}{2}}+\OO(1/\log b_n)
$.
%Since $\Phi$ is monotone and increasing,
\end{proof}

%Futhermore we have the following almost sure bound:
\begin{lemma}
  \TH(9.lem7)
%There exists a subset $\O_0\subseteq\O$ with $\P\left(\O_0\right)=1$ such that on $\O_0$,
There exists a subset $\O_0\subseteq\O$ with $\P\bigl(\O_0\bigr)=1$ such that on $\O_0$,
for all but a finite number of indices $n$ the following holds: for all $1\leq l\leq L^{\star}$
\bea
\Eq(9.lem7.1)
&e^{-\b\min\left\{\max\left(H_n(y),H_n(x)\right)\,|\, \{x,y\}\in G(C^{\star}_{n,l})\right\}}
\leq 
e^{\b n\sqrt{\log 2}(1+2\log n/n\log 2)},&\\
%e^{\b n\sqrt{\log 2}(1+4\log n/n)},&\\
\Eq(9.lem7.2)
&e^{-\b\min\left\{H_n(x)\,|\, x\in C^{\star}_{n,l}\right\}}\leq e^{\b n\sqrt{2\log 2}(1+2\log n/n)}.&
\eea
%Thus $\bar\varrho_{n,l}(0)\leq\leq e^{\b n\sqrt{2\log 2}(1+o(1))} .
%and also
%$
%e^{-\b\min\left\{H_n(x)\,|\, x\in C^{\star}_{n,l}\right\}}\leq e^{2\b n\log n}\leq e^{2\b n\log n}
%$
\end{lemma}

\begin{proof}
%[Proof of Lemma \thv(9.lem7)] 
Set $\rho(1)\equiv1+2\log n/n\log 2$,
$
\O_n(1)\equiv\{\o\in\O \mid \max_{ x\in \VV_n}w_n(x)\leq r_n(\rho(1))\}
$,
and $\O_\infty(1)\equiv\liminf_{n\rightarrow\infty}\O_n(1)$. Further set
$\rho(2)\equiv \frac{1}{2}(1+3\log n/n\log 2)$,
$
\O_n(2)\equiv\{\o\in\O \mid \max_{ (x,y)\in\EE_n}\min(w_n(x),w_n(y))\leq r_n(\rho(2))\}
$,
and $\O_\infty(2)\equiv\liminf_{n\rightarrow\infty}\O_n(2)$. 
%Let us prove that $\P(\O_\infty(1)\cap \O_\infty(2))=1$.
By independence and \eqv(3.0.2),
$
\P(\O^c_n(1))=2^n 2^{-n} n^{-2}= n^{-2}
$
which is summable, hence 
$
\P(\O_\infty(1))=1
$.
Next,
$
\P(\O^c_n(2))\leq n2^{n-1} 2^{-n} (r_n(\rho(2)))^2\leq  n^{-2}
$
which is also summable, and so 
$
\P(\O_\infty(2))=1
$.
Taking
$\O_0\equiv\O_\infty(1)\cap\O_\infty(2)$ and
using \eqv(A1.lem1.1) to bound $r_n(\rho(1))$ and $r_n(\rho(2))$ yields the claim of the lemma.
%, respectively, \eqv(9.lem7.1)and \eqv(9.lem7.2).
\end{proof}

%%%%%%%%%%%%%%%%%%%%%%%%%%%%%%%%%%%%%%%%%%%%%%%%%%%%%%%%%%%%

%%%%%%%%%%%%%%%%%%%%%%%%%%%%%%%%%%%%%%%%%%%%%%%%%%%%%%%%%%%%

\section{
Front end and back end clock processes, 
%and their relations to the clock processes $S_n$ and $K_n$.
and proofs of the theorems of Section \thv(1).
%and how to prove
%Front end, back end, and effective clock processes
}
\label{2}

%We are now ready to 
%introduce 
In this section we formally define
the \emph{front end} and \emph{back end} clock processes, and show how they 
relate to the clock processes $S_n$ and $K_n$. These relations are then used to decompose the proofs of 
Theorem \thv(1.theo1.Main) and Theorem \thv(1.theo2.Main) into five main steps.
%parts
%basic steps
Let $C^{\star}_{n,l}$, $1\leq l\leq L^{\star}$, be the collection of connected components
defined through \eqv(10.1.4) and set
\be
\VV^\circ_n\equiv\VV_n\setminus\left(\cup_{1\leq l\leq L^{\star}}C^{\star}_{n,l}\right).
\Eq(2.1.1)
\ee

\subsection{Front end clock process.}
\label{2.1}
%/ Foreground / First order.
We call \emph{front end clock process} the process defined through
\be
\wt S_n^\circ (k^\circ)=\sum_{i=0}^{k^\circ-1}\l^{-1}_n(J^\circ_n(i))e^\circ_{n,i},\quad k^\circ\in \N,
\Eq(2.1.2)
\ee
where $(e^\circ_{n,i}\,,n\in\N, i\in\N)$ are independent
mean one exponential random variables and where, introducing the times of  consecutive visits of $J_n$ to $\VV^\circ_n$,
%denoting by $T_{n,i}^\circ$ the times of the successive visits of $J_n$ to $\VV^\circ_n$,
\bea
\Eq(2.1.3)
T_{n,0}^\circ
&=&
\inf\{i\geq 0 \mid J_n(i)\in\VV^\circ_n\}\,,
\\
\Eq(2.1.4)
T_{n,j+1}^\circ
&=&
\inf\{i> T_{n,j}^\circ \mid J_n(i)\in\VV^\circ_n\},\, j=0,1,2,\dots,
\eea
$(J_n^\circ(i), i\in\N)$ is the reversible Markov chain on $\VV^\circ_n$ obtained by setting 
$
J_n^\circ(i)\equiv J_n(T_{n,i}^\circ)
$.
%Note that $J_n^\circ$ is a (reversible) Markov chain on $\VV^\circ_n$ with transition matrix elements
Note that $J_n^\circ$ has  transition matrix elements
\be
p^\circ_n(x,y)=P_x\left(J_n(T_{n,1}^\circ)=y\right),\, x,y\in\VV^\circ_n,
\Eq(2.1.5)
\ee
and invariant measure  
%$\pi^\circ_n$,
\be
\pi^\circ_n(x)=\frac{\pi_n(x)}{\sum_{x'\in \VV^\circ_n}\pi_n(x')}, \quad x\in\VV^\circ_n,
\Eq(2.1.6)
\ee
where $\pi_n$ denotes the invariant measure of $J_n$ (see \eqv(1.1.15) for its expression). We call $J_n^\circ$ the front chain and denote by  $(\O^{J^\circ}, \FF^{J^\circ}, P^\circ)$ its probability space.
%on which $J_n^\circ$ is defined
The associated graph, $G^\circ(\VV^\circ_{n})$,  is described in \eqv(4.0.1). 

\subsection{Back end clock process.}
\label{2.2}
%/ Background / Second order
The description of this process involves three time sequences.
The first two are the intertwined sequences of  consecutive hitting times of 
 $\VV_n\setminus \VV^\circ_n$ and their ensuing exit times.
%of times of first visits of  $J_n$ to $\VV_n\setminus \VV^\circ_n$, and ensuing exit times. Namely,
Namely, set 
%\be
%\overline T_{n,0}=0, 
%\Eq(2.2.1)
%\ee
\be
\overline T_{n,0}=0, \quad
\overline T'_{n,0}=
\begin{cases}
\inf\{i>  0 \mid J_n(i)\notin\VV^\circ_n\},
&\hbox{\rm if}\, J_n(0)\in\VV^\circ_n,\\
0,&\hbox{\rm if}\, J_n(0)\notin\VV^\circ_n,
\end{cases}
\Eq(2.2.2)
\ee
and, for $ j=0,1,2,\dots,$ 
\bea
\Eq(2.2.3)
\overline T_{n,j+1}
&=&
\inf\{i> \overline T'_{n,j}\mid J_n(i)\in\VV^\circ_n\}\,,
\\
\Eq(2.2.3')
\overline T'_{n,j+1}
&=&
\inf\{i> \overline  T_{n,j+1} \mid J_n(i)\notin\VV^\circ_n\}.
\eea
Clearly,
$
%\be
0
= \overline T_{n,0}\leq \overline T'_{n,0}
<\overline T_{n,1}\leq \overline T'_{n,1}
<\dots
<\overline T_{n,j}\leq \overline T'_{n,j}
<\dots
%\Eq(2.2.4)
%\ee
$.
Clearly also, to each $j$ there corresponds an $i$ such that 
%for each $j$ there exists an $i$ such that 
$
%\be
T_{n,i-1}^\circ< \overline T'_{n,j}=T_{n,i-1}^\circ+1<T_{n,i}^\circ
%\Eq(2.2.5)
%\ee
$.
Merging 
$
\bigl(
T_{n,i}^\circ
\bigr)_{ i\geq 0}
$
and
$
\bigl(
\overline T'_{n,j}
\bigr)_{ j\geq 0}
$
into a single sequence, %denoted by 
$
\bigl(
T^\dagger_{n,j}
\bigr)_{ j\geq 0} 
$,
and arranging its elements in increasing order of magnitude,
\be
0\leq T_{n,0}^\dagger<T_{n,1}^\dagger<\dots<T_{n,j}^\dagger<\dots.
\Eq(2.2.6)
\ee
we define the \emph{back end clock process} through
\be
\wt S_n^\dagger (k^\dagger)=\sum_{i=0}^{k^\dagger-1}\L_n^\dagger(J_n^\dagger(i)),\quad k^\dagger\in \N,
\Eq(2.2.7)
\ee
where $(J_n^\dagger(i), i\in\N)$ is the chain on $\VV_n$ obtained by setting 
$
J_n^\dagger(i)\equiv J_n(T_{n,i}^\dagger)
$,
and where
\be
\L_n^\dagger(J_n^\dagger(i)) 
=
\begin{cases}
\overline T_{n,j+1}- \overline T'_{n,j},
&
\hbox{\rm if}\, J_n^\dagger(i)\notin\VV^\circ_n\,\hbox{\rm and}\, 
\sum_{k=0}^{i}\1_{\{ J_n^\dagger(k)\notin\VV^\circ_n\}}=j,\\
0,&\hbox{\rm if}\, J_n^\dagger(i)\in\VV^\circ_n.
\end{cases}
\Eq(2.2.8)
\ee
%To see how this clock behaves let us take a closer look at the chain $J_n^\dagger$.
%Let us take a closer look at the chain $J_n^\dagger$.
Clearly,  $J_n^\dagger$ is Markovian with one-step transitions probabilities, $p_n^\dagger(x,y)$, as follows:
% $p_n^\dagger(x,y)\equiv P(J_n^\dagger(i+1)=y\mid J_n^\dagger(i)=x)$. 
when it is at
$
x\in\VV^\circ_n
%\equiv\VV_n\setminus\left(\cup_{1\leq l\leq L^{\star}}C^{\star}_{n,l}\right)
$,
$J_n^\dagger$
chooses its next step according to the transition probabilities of $J_n$,
\be
p_n^\dagger(x,y)=p_n(x,y), \quad x\in\VV^\circ_n, y\in\VV_n,
\Eq(2.2.9)
\ee
and when it enters 
$
\cup_{1\leq l\leq L^{\star}}C^{\star}_{n,l}
$,  
say at a vertex of $C^{\star}_{n,l}$, it exits in  just one step
through one of the boundary points $\del C^{\star}_{n,l}$; that is, for all
$x\in C^{\star}_{n,l}$, $y\in\del C^{\star}_{n,l}$, and $1\leq l\leq L^{\star}$, 
\be
p_n^\dagger(x,y)=P_x(J(T^{\star}_{n,l})=y),
\quad 
\Eq(2.2.10)
\ee 
where
$
 T^{\star}_{n,l}=\inf\{i>0 \mid J_n(i)\in\del C^{\star}_{n,l}\}
$.
%Note that  $J_n^\dagger$ is not reversible.
Clearly also, the increments $\L_n^\dagger(J_n^\dagger(i))$
of the clock at the times of the visits of $J_n^\dagger(i)$ to $\cup_{1\leq l\leq L^{\star}}C^{\star}_{n,l}$
are the sojourn times of $J_n$ in the sets $C^{\star}_{n,l}$ being visited.
%that are visited
 In other words,
 $\L_n^\dagger(J_n^\dagger(i))$ is
equal in distribution to some $T^{\star}_{n,l}$. 
%This relation will be made precise in Subsection \thv(5.3).

% \textsc{fecp} and \textsc{becp}

\smallskip
%In summery,
Summarizing our definitions,
%the front and back end clock process (hereafter FECP and BECP) can be described as follows:
%We may summarize the definition of this section as follows:
%the front end clock process \eqv(2.1.2) 
\textsc{fecp} \eqv(2.1.2)
records the total time spent by the proces $X_n$ in $\VV^\circ_n$ along the first $k^\circ$ 
steps of $J_n^\circ$ whereas
%the back end clock process  \eqv(2.2.7) 
\textsc{becp} \eqv(2.2.7) 
records the total time spent by the chain $J_n$ in   
%$\VV_n\setminus\VV^\circ_n$
$\cup_{1\leq l\leq L^{\star}}C^{\star}_{n,l}$
along the first   $k^\dagger$ steps of $J_n^\dagger$. The
chains $J_n^\dagger$ and $J_n^\circ$ differ in that $J_n^\dagger$ does visit the sets $C^{\star}_{n,l}$, and 
steps out of these sets
%of a given $C^{\star}_{n,l}$ 
right after stepping in,  while  $J_n^\circ$ straddles over the $C^{\star}_{n,l}$'s, never entering them. 
 Technically, this makes the two chains
%They however are, technically, 
very different objects. In particular, 
%Most notably, 
%In particular,
$J_n^\circ$  is reversible but $J_n^\dagger$ isn't.

\subsection{Rewriting the clock process.}
\label{2.3}
Our aim is to express the processes $K_n$ and $S_n$ defined in \eqv(1.3.1) and \eqv(2.1.0), respectively, 
using \textsc{fecp} and \textsc{becp}.
% the front and back end clock processes. 
%defined in \eqv(2.1.2)\eqv(2.2.7).
We first deal with $K_n$. For $a_n$  as in \eqv(1.3.1) let $k^\dagger_n(t)$ be defined through
\be
k^\dagger_n(t)=\min\left\{k\geq 1 \,\Big|\, \textstyle\sum_{i=0}^{k-1}\1_{\{J_n^\dagger(i)\in \VV^\circ_n\}}=\lfloor a_n t\rfloor\right\}
%\1_{\left\{J_n(i)\in \VV_n\setminus\left(\cup_{1\leq l\leq L^{\star}}C^{\star}_{n,l}\right)\right\}}=\lfloor a_n t\rfloor\right\}
\,,\quad t\geq 0,
\Eq(2.3.0)
\ee
%(this is the total number of steps the chain $J_n^\dagger$ must take in order to take 
%$\lfloor a_n t\rfloor =k^\circ_n(t)$ (see below) steps in $\VV^\circ_n$) 
and, taking
$
k^\dagger=k^\dagger_n(t)
%\equiv\lfloor a_n t\rfloor
$
in \eqv(2.2.7), set
\be
S_n^\dagger(t) = b_n^{-1}\wt S_n^\dagger (k^\dagger_n(t)),\quad t\geq 0.
\Eq(2.3.1)
\ee
Then $K_n(t)$ can be writen as
\be
K_n(t)= \lfloor a_n t\rfloor+ b_n S_n^\dagger (t),\quad t\geq 0.
%in fact
%K_n(t)=  k^\circ_n(t)+ b_n S_n^\dagger (t),\quad t\geq 0.
% mais je remplace k^\circ_n(t) \lfloor a_n t\rfloor pour ne pas avoir a definir k^\circ_n(t) a ce stade.
\Eq(2.3.2)
\ee
To see this  write
$
K_n(t)=\sum_{i=0}^{K_n(t)-1}\1_{\{J_n(i)\in \VV^\circ_n\}}+\sum_{i=0}^{K_n(t)-1}\1_{\{J_n(i)\notin \VV^\circ_n\}}
$
and note that
\bea
\Eq(2.3.4)
&&\textstyle
\sum_{i=0}^{K_n(t)-1}\1_{\{J_n(i)\notin \VV^\circ_n\}}
=\sum_{i=0}^{k^\dagger_n(t)-1}\L_n^\dagger(J_n^\dagger(i))
=b_n S_n^\dagger (t),
\\
\Eq(2.3.3)
&&\textstyle
\sum_{i=0}^{K_n(t)-1}\1_{\{J_n(i)\in \VV^\circ_n\}}
=\sum_{i=0}^{k^\dagger_n(t)-1}\1_{\{J_n^\dagger(i)\in \VV^\circ_n\}}
=\lfloor a_n t\rfloor \equiv k^\circ_n(t),
%\neq \sum_{i=0}^{\lfloor a_n t\rfloor-1}\1_{\{J_n^\circ(i)\in \VV^\circ_n\}}
\eea
where we introduced the notation $k^\circ_n(t)$ for later convenience.
%Further define
%$
%k^\circ_n(t)\equiv\textstyle\sum_{i=0}^{k^\dagger_n(t)-1}\1_{\{J_n^\dagger(i)\in \VV^\circ_n\}}
%$.
%Of course
%$
%k^\circ_n(t)=\lfloor a_n t\rfloor
%$.
%
In  words,
%That is, 
%Thus
when $J_n$ takes $K_n(t)$ steps, $J_n^\dagger$ takes $k^\dagger_n(t)$ steps, of which $k^\circ_n(t)$  
are visits of $J_n^\dagger$ to $\VV^\circ_n$.
%(equivalently, when $J_n^\dagger$ takes $k^\dagger_n(t)$ steps,  $J_n$ takes $K_n(t)$ steps, of which $k^\circ_n(t)$  
%are visits of $J_n$ to $\VV^\circ_n$. )

To deal with the clock process $S_n$ we likewise split the sum in \eqv(2.1.0)
in two terms according to whether $J_n(i)\in \VV^\circ_n$ or $J_n(i)\notin \VV^\circ_n$.
From the above 
%observations and 
definitions and those of $J_n^\dagger$ and $J_n^\circ$ we have that on the one hand, writing $\stackrel{d}{=}$ for equality in distribution,
\bea
\Eq(2.3.5)
\textstyle\sum_{i=0}^{K_n(t)-1}\l^{-1}_n(J_n(i))e_{n,i}\1_{\{J_n(i)\in \VV^\circ_n\}}
\Eq(2.3.5')
&\stackrel{d}{=}&
\textstyle\sum_{j=0}^{k^\dagger_n(t)-1}\l^{-1}_n(J^\dagger_n(j))e^\dagger_{n,j}\1_{\{J^\dagger_n(i)\in \VV^\circ_n\}}\quad
\\
\Eq(2.3.6)
&\stackrel{d}{=}&
\textstyle\sum_{j=0}^{k^\circ_n(t)-1}\l^{-1}_n(J^\circ_n(j))e^\circ_{n,j}\1_{\{J^\circ_n(i)\in \VV^\circ_n\}}\quad
\\
\Eq(2.3.7)
&=&
\wt S_n^\circ (\lfloor a_n t\rfloor),\quad
\eea
where $(e^\dagger_{n,j})$ and $(e^\circ_{n,j})$  are families of independent mean one exponential random variables, and
$\wt S_n^\circ $ is the front end clock process \eqv(2.1.2).
On the other hand,
\bea
\Eq(2.3.8)
&&
\textstyle\sum_{i=0}^{K_n(t)}\l^{-1}_n(J_n(i))e_{n,i}\1_{\{J_n(i)\notin \VV^\circ_n\}}
\\
\Eq(2.3.9)
&=&
\textstyle\sum_{j=0}^{k^\dagger_n(t)-1}
\left(\sum_{i=0}^{\L_n^\dagger(J_n^\dagger(j))-1}\l^{-1}_n(J_n(\overline T'_{n,j}+i))e_{n,i}\right)
\1_{\{J^\dagger_n(j)\notin \VV^\circ_n\}}
\\
\Eq(2.3.10)
&\equiv&
\textstyle\sum_{j=0}^{k^\dagger_n(t)-1}\wh\L_n^\dagger(J_n^\dagger(j))
\eea
where the last line defines $\wh\L_n^\dagger(J_n^\dagger(j))$. If we now set, for $t\geq 0$,
\bea
\Eq(2.3.11)
S_n^\circ(t)&\equiv&c_n^{-1}\wt S_n^\circ (\lfloor a_n t\rfloor),
\\
\Eq(2.3.12)
\wh S_n(t)&\equiv&c_n^{-1}\textstyle\sum_{j=0}^{k^\dagger_n(t)-1}\wh\L_n^\dagger(J_n^\dagger(j)),
%\\
%\Eq(2.3.12')
%\Delta_n (t)&\equiv&
%c_n^{-1}\textstyle\wt \Delta_n (t),
\eea
the rescaled clock process \eqv(2.1.0) can be rewritten as
\be
S_n(t)\stackrel{d}{=} S_n^\circ (t)+\wh S_n(t).
\Eq(2.3.13)
\ee
Here the rescaled front end  clock process, $S_n^\circ (t)$, records the time spent \emph{by the process $X_n$} during its visits 
%of $J_n$ 
to the set $\VV^\circ_n$, while the remainder term, $\wh S_n(t)$,
records the time spent in its complement. The back end  clock process 
$b_n S_n^\dagger (t)$ is the time needed to actually be able to observe a transition \emph{ of the chain $J_n$} 
from one vertex of $\VV^\circ_n$ to the next.
%from one "deep trap" of the lanscape to another "deep trap", the set of these deep traps
%being determined by the choice of the scaling $c_n$.
%
%extreme point of the landscape to another extreme point, the set of these extremes being 
%determined by the choice of the scaling $c_n$.

\subsection{Proofs of the theorems of Section \thv(1).}
%Proofs of the Theorem \thv(1.theo1.Main) and  Theorem \thv(1.theo2.Main)
%Satement of results.}
\label{2.4}

The proofs of the theorems of Section \thv(1.2)  rely on five theorems stated below.
%each of them controlling one of the
% each one of them controlling one of the
Each of them controls one of the 
%five 
processes $k^\dagger_n(t)$, $S_n^\circ (t)$, $\wh S_n(t)$,  and $S_n^\dagger(t)$ above, respectively below, the critical line $\b= 2\b_c(\varepsilon/2)$, $0<\varepsilon<1$. 
As in Section \thv(1.2) the initial distribution of  $J_n$ is the
 uniform distribution on $\VV_n^{\circ}$. By \eqv(4.prop2.1), this is nothing but
the invariant measure, $\pi_n^\circ$, of $J^\circ_n$.
Hence $J^\circ_n$ and $J^\dagger_n$ also start in $\pi_n^\circ$. 

The first theorem  shows that $k^\dagger_n(t)$ behaves like $k^\circ_n(t)=\lfloor a_n t\rfloor$ for large $n$.
%and the last allows to substitute $\lfloor a_n t\rfloor$ for $k^\circ_n(t)$ in $S_n^\circ$ for large $n$

\begin{theorem}
  \label{2.theo4}
  Assume that $c_{\star}>2$. 
%Let $c_\circ>0$ be any constant.  
For all $0< t<\infty$, any constant $c_\circ>0$,  and any sequence $a_n>0$ we have that
%, the following holds:
on $\O^{\star}$, for all but a finite number of indices $n$,
\be
P_{\pi^\circ_n}\left(
1
\leq
k^\dagger_n(t)/k^\circ_n(t)
%k^\circ_n(t)/\lfloor a_n t\rfloor
\leq 1+n^{-c_\circ}\right) 
\geq 1- n^{-2(c_{\star}-1)+c_\circ}(1+\OO(n^{-(c_{\star}-1)})).
\Eq(2.theo4.1)
\ee
%where $c_\circ>0$ is a constant.
\end{theorem}

The next two theorems are the building blocks of the proof of Theorem \thv(1.theo1.Main).
The first establishes convergence of the front end clock process, $S_n^\circ$. The second implies, in particular, that 
the contribution of $\wh S_n$ to \eqv(2.3.13)
%converges to zero.
vanishes as $n$ diverges.

\begin{theorem}[Front end clock process]
  \label{2.theo1} 
Assume that $c_{\star}>3$.
%$n^{c_{\star}}>n^3\log n$. 
Let the sequences $a_n$ and $c_n$ be as in Theorem \thv(1.theo1.Main).
%Let $\rho^{\star}_n$  be given by \eqv(10.1.1) for some $c_{\star}$ such that $n^{c_{\star}}\gg n\log n$. 
Then, for all $0<\varepsilon<1$ and $\b> \b_c(\varepsilon)$,  $\P$-almost surely,
\be
S^\circ_n\Rightarrow_{J_1}  S^\circ_{\infty},
\Eq(2.theo1.1)
\ee
 where $S^\circ_{\infty}$ is a subordinator with zero drift and L\'evy measure
$\nu^\circ=\nu$ defined in \eqv(1.theo1.M3).
%(2.theo1.1)
\end{theorem}

%Finally we show that  the distance between $S_n$ and $S_n^\circ(\cdot)$ goes to zero as $n\uparrow\infty$.
%The next theorem controls the distance between $S_n$ and $S_n^\circ$ as $n\uparrow\infty$.

\begin{theorem}
% [Tightness]
[Remainder]
  \label{2.theo3}
Assume that $c_{\star}>2$. 
Let the sequences $a_n$ and $c_n$ be as in Theorem \thv(1.theo1.Main).
Then, for all $0<\varepsilon<1$ and $\b> \b_c(\varepsilon)$,
$\P$-almost surely, 
%for all $\e>0$,
\be
\limsup_{n\rightarrow\infty}\PP_{\pi_n^\circ}\left(\rho_\infty\bigl(S_n(\cdot),S_n^\circ(\cdot)\bigr)>\e\right)=0,
\quad\forall\e>0,
\Eq(2.theo3.1)
\ee
where $\rho_\infty$ is Skorohod metric on $D([0,\infty))$.
\end{theorem}

We now turn to the back end clock process. The next result parallels Theorem \thv(2.theo1).
%in the low temperature domain.

\begin{theorem}[Back end clock process above the critical line]
  \label{2.theo2}
%Given  $0<\varepsilon<1$ let $a_n$ and $c_n$ be as in Theorem \thv(1.theo1.Main).
Assume 
that $c_{\star}>3$.
%is such that $n^{c_{\star}}>n^3\log n$. 
Let  the sequence $a_n$ and $b_n$ be as in Theorem \thv(1.theo1.Main) 
and Theorem \thv(1.theo2.Main), (i), respectively.
Then, for all $0<\varepsilon<1$ and $\b>2\b_c(\varepsilon/2)$,  $\P$-almost surely,
\be
S_n^\dagger \Rightarrow_{J_1}  S^\dagger_{\infty},
\Eq(2.theo2.1)
\ee
 where $S^\dagger_{\infty}$ is a stable subordinator with zero drift and L\'evy measure
$\nu^\dagger$ defined in \eqv(1.theo2.M2).
\end{theorem}

Finally, the last result covers the high temperature domain.

\begin{theorem}[Back end clock process below the critical line]
  \label{2.theo5}
Assume that $c_{\star}>3$.  
Let  the sequence $a_n$ and $b_n$ be as in Theorem \thv(1.theo1.Main) 
and Theorem \thv(1.theo2.Main), (ii), respectively.
%Let the sequence $a_n$ be as in Theorem \thv(1.theo1.Main) 
%and set
%$
%b_n=a_n\exp(n(\b/2)^2)/(\b\sqrt{\pi n})
%$
%%%\be
%%%b_n=\frac{a_ne^{n(\b/2)^2}}{\b\sqrt{\pi n}}
%%%\ee.
Then, for all $0<\varepsilon<1$, all $0<\b< 2\b_c(\varepsilon/2)$, and  all $0<T<\infty$, $\P$-almost surely,
\be
\PP_{\pi_n^\circ}\Bigl(\lim_{n\rightarrow\infty}\sup_{t\in[0,T]}\left|b_n^{-1}S_n^\dagger(t) -t \right|=0\Bigr)=1.
\ee
\end{theorem}

Assuming these theorems we may prove Theorem \thv(1.theo1.Main) and Theorem \thv(1.theo2.Main). The proof of Theorem \thv(1.theo3.Aging) also uses properties of the chain $J_n^\circ$ that are taken from Section \thv(4).

\begin{proof}[Proof of Theorem \thv(1.theo1.Main)] In view of \eqv(2.3.13) Theorem \thv(1.theo1.Main) is an 
immediate consequence of Theorem \thv(2.theo1) and Theorem \thv(2.theo3)
\end{proof}

\begin{proof}[Proof of Theorem \thv(1.theo2.Main)] Recall the expression \eqv(2.3.2) of $K_n$ and 
notice that $a_n/b_n\downarrow 0$ both under the assumptions
 on $a_n$ and $b_n$ 
of Theorem \thv(2.theo5) and Theorem \thv(2.theo2) 
(use \eqv(A1.lem1.1) to check this in the latter case). 
%By this and Theorem \thv(2.theo4), 
%$
%\PP_{\pi_n^\circ}(\lim_{n\rightarrow\infty}\sup_{t\in[0,T]}b_n^{-1}k^\circ_n(t)=0)=1
%$
%on $\O^{\star}$.
 Thus the first assertion of Theorem \thv(1.theo2.Main) is a an immediate consequence of
 Theorem \thv(2.theo2) while the second assertion follows from Theorem \thv(2.theo5).
\end{proof}

\begin{proof}[Proof of Theorem \thv(1.theo3.Aging)] The proof follows classical arguments.
% that we only sketch
%(see \cite{BBC08}, \cite{BG13}).
Let $A_n^\rho (t,s)$ be the event $A_n^\rho (t,s)\equiv \{ \CC_n(t,s)\ge 1-\rho\}$. Denote respectively by $\RR_n$ and $\RR^\circ_n$ 
the ranges of the processes $c_n^{-1}\wt S_n$ and $c_n^{-1}\wt S_n^\circ$.
% (see \eqv(1.1.13) and \eqv(2.1.2)). 
Clearly, for all $\rho \in (0,1)$, 
$
A_n^\rho (t,s)\supseteq\{\RR_n\cap(t,t+s)=\emptyset\}
$.
Now
$
\lim_{n\rightarrow \infty}\PP_{\pi_n^\circ}(\RR_n\cap(t,t+s)=\emptyset)
=\lim_{n\rightarrow \infty}\PP_{\pi_n^\circ}(\RR^\circ_n\cap(t,t+s)=\emptyset)
$
as follows from Theorem \thv(2.theo3), and by Theorem \thv(2.theo1), proceeding as in the proof of Theorem 1.6 in \cite{G12}, 
$
\lim_{n\rightarrow \infty}\PP_{\pi_n^\circ}(\RR^\circ_n\cap(t,t+s)=\emptyset)
=\PP(\{S^\circ_{\infty}(u), u>0\}\cap(t,t+s)=\emptyset)
$,
but by the arcsine law for stable subordinators (see e.g.~Theorem 1.8 of \cite{G12}) the last probability is equal to the right hand side of \eqv(1.theo3.Aging.1).

It thus remains to show that
$
\lim_{n\rightarrow \infty}\PP_{\pi_n^\circ}\bigl(A_n^\rho (t,s)\cap\{\RR_n\cap(t,t+s)\neq\emptyset\}\bigr)=0
$.
Invoking as before Theorem \thv(2.theo3), we can substitute $\RR^\circ_n$ for $\RR_n$ in the probability.
%In the course of the proof of Theorem \thv(2.theo1) we show that 
%$
%\lim_{\e\downarrow 0}\limsup_{n\uparrow \infty}
%\PP_{\pi_n^\circ}(c_n^{-1}\sum_{i=0}^{\lfloor a_n t\rfloor-1}\l^{-1}_n(J^\circ_n(i))e^\circ_{n,i}\1_{J^\circ_n(i)\in \TT_n(\e)})%=0
%$
%(see \eqv(6.1.7) in Condition  (C3) of Theorem \thv(6.theo1)).
%Thus each jump of the limiting process $S^\circ_{\infty}$ comes from a visit to a vertex of the set $\TT_n(\e)$.
%Furthermore, with a probability that tends to one as $n\uparrow \infty$ and 
Consider the set 
$
\TT_n(\e)\equiv\{x\in I^{\star}_n\mid w_n(x)\geq \e c_n\}
$, 
$\e>0$.
By Theorem \thv(2.theo1), if $\RR^\circ_n\cap(t,t+s)\neq\emptyset$ then with a probability that tends to one 
as $n\uparrow \infty$ and $\e\downarrow 0$ there exists  $u_-\leq u_+$ such that on the one hand
$c_n^{-1}\wt S_n^\circ(\lfloor a_nu_-\rfloor-1)<t<c_n^{-1}\wt S_n^\circ(\lfloor a_nu_-\rfloor)$
while
$c_n^{-1}\wt S_n^\circ(\lfloor a_nu_+\rfloor)<t+s<c_n^{-1}\wt S_n^\circ(\lfloor a_nu_+\rfloor+1)$
%$c_n^{-1}\wt S_n^\circ(\lfloor a_nu_-\rfloor-1)<t$ and $c_n^{-1}\wt S_n^\circ(\lfloor a_nu_-\rfloor)\in(t,t+s)$ on the one %hand, and  
%$c_n^{-1}\wt S_n^\circ(\lfloor a_nu_+\rfloor)\in(t,t+s)$ and $c_n^{-1}\wt S_n^\circ(\lfloor a_nu_+\rfloor+1)>t+s$
on the other,
and  these two increments correspond to visits to vertices $z_-$ and $z_+$ in $\TT_n(\e)$
(that is to say, with probability one, the points $t$ and $t+s$ lie in constancy intervals of the process, and such intervals are produced, asymptotically, by visits to  $\TT_n(\e)$).
Let us now argue that, firstly,  that starting from a given vertex $z_-\in\TT_n(\e)$, the chain $J_n^\circ$ quickly moves at a distance greater than $n\rho/2 $ from it, and secondly, that it does not visit any vertex in
$
%\BB_n(\e, \rho)\equiv
\{z\in\TT_n(\e)\mid \dist(z_-,z)\leq n\rho/2\}
$
in the ensuing
%during the /remaining
 $\lfloor C a_n \rfloor$ steps, for any $0<C<\infty$, $0<\rho<1$, and small $\e>0$.
For this we use three results of Section \thv(4). By Proposition \thv(4.prop1), the chain $J_n^\circ$ started in $z_-$ reaches stationarity in $\ell^\circ_n\sim n^3$ steps, and by Proposition \thv(4.prop2), 
$
\pi^\circ_n(\{z\in\VV_n^\circ \mid \dist(z_-,z)> n\rho/2\})\geq 1-\exp\{-n\II(\rho)\}
$, 
where $\II(\rho)>0$ if $0<\rho<1$. This proves the first claim. The second claim is an immediate consequence of
Proposition \thv(4.prop5).
%
%%%%%%%%%%%%%%%%%%%%%%%%%%%%%%%%%%%%%%%%%%%%%%%%%%%%%%%%%%%
%
%with 
%$$
%|A\cap I^{\star}_n|=|\{z\in\TT_n(\e)\mid \dist(z_-,z)\leq n\rho/2\}|\sim \exp\{n\II(\rho)\}2^{-\varepsilon n},
%$$
%$$
%t=Ca_n\sim C 2^{\varepsilon n},
%$$
%and 
%$$
%|\VV^\circ_{n}|\sim 2^n,
%$$
%we get that
%$$
%t|A\cap I^{\star}_n|/|\VV^\circ_{n}|
%\sim C 2^{\varepsilon n}\exp\{n\II(\rho)\}2^{-\varepsilon n} 2^{-n}
%=C\exp\{n\II(\rho)\}2^{-n}
%$$
%and this tends to zero for all $0<\rho<1$. The "dilution" of $A$ by $I^{\star}_n$ is NECESSARY to gain the 
%$2^{-\varepsilon n} $ that will kill the term $t=Ca_n$. Without this we get a condition that links $\varepsilon$ and $\rho$,
%
%%%%%%%%%%%%%%%%%%%%%%%%%%%%%%%%%%%%%%%%%%%%%%%%%%%%%%%%%%%
%
 The proof of Theorem \thv(1.theo3.Aging) is done.
\end{proof}

%Theorem \thv(2.theo1) and Theorem \thv(2.theo2) clearly have a similar shape. 

%\subsection{Strategy and structure of the rest of paper.}
%\label{2.5}

%Theorem \thv(6.theo2) in Section  \ref{6} and Theorem \thv(7.theo1)  in Section  \ref{7} give sufficient conditions for the %processes $S^\circ_n$ and $S^\dagger_n$ to converge to L\'evy subordinators. 
%These conditions are expressed in terms of certain functionals of the chain $J_n^\circ$ that, of course, depend on the 
%the increments of the process under consideration and, thus, on the random environment...... explain strategy.....
%%(along the trajectories of $J_n^\circ$).
%%mean holding time  (of the clock process under consideration) along trajectories of $J_n^\circ$.
%% They state that certain (random in the random environment) functionals of the chain $J_n^\circ$ converge in law 
 
The rest of this paper is organized as follows. 
 In  Section \ref{9} we focus on the increments of the process $\wh S_n$ and prove an upper bound on their tail distribution.
%--{Effective holding times in the sets $C^{\star}_{n,l}$.}
% we establish upper bounds on the tail distribution of  ... 
%focuses likewise on the 
A similar analysis is carried out in Section  \ref{5}
%, focusing this time on
for the increments of the back end clock process  $\wt S_n^\dagger$; an explicit expression is also obtained for the distribution of the sojourn times of $J_n$ in sets $C^{\star}_{n,l}$ of size 2.
% an explicit expression is obtained in the case of sets $C^{\star}_{n,l}$ of size 2.
%
% we likewise establish upper bounds on the tail distribution of the increments of the back end clock process  
% $\wt S_n^\dagger$ (an explicit expression is obtained in the case .
The properties of $J_n^\circ$ (invariant measure, mixing time through spectral gap, mean local times) are studied in 
%we will see in 
Section  \ref{4}, where it is shown that $J_n^\circ$ has several of the attributes of the symmetric random walk.
%: its invariant measure is the uniform measure
Using these preparations, the proofs of Theorem \thv(2.theo4) and Theorem \thv(2.theo3) are carried out in Section  \ref{8}.
%, where some elementary properties of the chains $J_n^\circ$ and  $J_n^\dagger$ are also collected.
Those of Theorem \thv(2.theo1), Theorem \thv(2.theo2), and Theorem \thv(2.theo5) are carried out in Section \ref{6}, Section  \ref{7}, and Section \ref{12}, respectively.
%Finally, the back end clock process below the critical temperature is dealt with in Section \ref{12}, where 
%Theorem \thv(2.theo5) is proved. 
%Two appendices  appendices gather 
Finally, Appendix \ref{B}  gathers needed results on the speed of convergence to Perron projector for nonnegative and primitive matrices (Subsection \ref{B.1}), and for irreducible and periodic matrices with period 2  (Subsection \ref{B.2}).

%%%%%%%%%%%%%%%%%%%%%%%%%%%%%%%%%%%%%%%%%%%%%%%%%%%%%%%%%%%%

%%%%%%%%%%%%%%%%%%%%%%%%%%%%%%%%%%%%%%%%%%%%%%%%%%%%%%%%%%%%

\section{Distribution of the increments of the process $\wh S_n$.}
%--{Effective holding times in the sets $C^{\star}_{n,l}$.}
 \label{9}

In this section we focus on the increments of the process $\wh S_n$,
% of \eqv(2.3.12), namely
that is to say, on the quantities defined through  \eqv(2.3.9)-\eqv(2.3.10) by
%In this section we give bounds on the
% so-called effective holding times, introduced in \eqv(2.3.9)-\eqv(2.3.10) and given 
\be
\wh\L_n^\dagger(J_n^\dagger(j))
\equiv
\sum_{i=0}^{\L_n^\dagger(J_n^\dagger(j))-1}\l^{-1}_n(J_n(\overline T'_{n,j}+i))e_{n,i}
\Eq(9.prop1.00)
\ee
if $J^\dagger_n(j)\in\cup_{1\leq l\leq L^{\star}}C^{\star}_{n,l}$, 
and $\wh\L_n^\dagger(J_n^\dagger(j))=0$  otherwise. 
These are the sojourn times of the process  $X_n$ in the sets $C^{\star}_{n,l}$
(we may think of them as ``effective holding times'' in those sets).
%which is why we will think if them as ``effective holding times''.
%Because the $\wh\L_n^\dagger$'s are the sojourn time of the process 
%$X_n$ in the sets $C^{\star}_{n,l}$ we call them ``effective holding times''.
%For this reason we will think if them as ``effective holding times''.
As expected, these times have exponential tails.
For $1\leq l\leq L^{\star}$, set
\bea
\Eq(9.lem1.0)
&\bar\varrho_{n,l}(0)=
e^{-\b\min\left\{H_n(x)\,|\,x\in C^{\star}_{n,l}\right\}},
\\
\Eq(9.lem1.0')
&\bar\varrho_{n,l}(1)=
e^{
-\b\min\left\{H_n(x)\,|\,x\in C^{\star}_{n,l}\right\}
+\b\max\left\{H_n(y)\,|\,y\in  C^{\star}_{n,l}\right\}},
\eea
and
\be
\Eq(9.prop1.0)
\bar\theta^{\star}_{n,l}\equiv3\b n^2|C^{\star}_{n,l}|^3\bar\varrho_{n,l}(1).
\ee

\begin{proposition}
  \TH(9.prop1)
Let $\O^{\star}$ and $\O_0$ be as in Lemma \thv(10.lem1) and Lemma \thv(9.lem7), respectively.
On $\O^{\star}\cap\O_0$, for all but a finite number of indices $n$, the following holds for all $1\leq l\leq L^{\star}$.
For all $t\geq \bar\theta^{\star}_{n,l}$ and all $x$ in $C^{\star}_{n,l}$,
\be
\Eq(9.prop1.1)
%\PP\left(\wh\L_n^\dagger(J_n^\dagger(j))>t\t_n(J_n^\dagger(j))\mid J^\dagger_n(j)=x\right)=e^{-t}(1+\epsilon^{\star}_n),
\PP\left(\wh\L_n^\dagger(J_n^\dagger(j))>t\mid J^\dagger_n(j)=x\right)
\leq e^{-t/\t^{\star}_{n,l}}(1+\bar\epsilon^{\star}_n),
\ee
where
$
0\leq \bar\epsilon^{\star}_n\leq \sfrac{\bar\theta^{\star}_{n,l}|C^{\star}_{n,l}|}{\bar\varrho_{n,l}(0)}(1+o(1))
$
and
\be
\Eq(9.prop1.14)
\sfrac{1}{|C^{\star}_{n,l}|}\bar\varrho_{n,l}(0)
\leq
\t^{\star}_{n,l}
\leq 
%\frac{1}{1-\OO(\frac{1}{\log n})}
(1+\OO(\sfrac{1}{\log n}))
\bar\varrho_{n,l}(0).
\ee
%\be
%n\bar\varrho_{n,l}(0)\leq\t^{\star}_{n,l}\leq{\bar\varrho_{n,l}(0)}/{n|C^{\star}_{n,l}|}.
%\ee
Moreover $(\t^{\star}_{n,l}, 1\leq l\leq L^{\star})$ are independent random variables on 
%the probability space 
$(\O,\FF, \P)$.
% of the random environment
%\ee
\end{proposition}

%Thus the effective holding times resemble exponentially distributed random 
%variables whose mean values depend only  on the label of the set  $C^{\star}_{n,l}$
%entered at time $j$ (that is, the set $C^{\star}_{n,l}\ni J_n^\dagger(j)$)
%%%%that the chain $J_n^\dagger$ enters at time $j$, 
%and not on the particular entrance point in that set.
%%%%As the reader might expect we will see in the course of the proof that 
%%%%$\varsigma_{n,l}(0)$ is the largest eigenvalue of the process $X_n$ restricted to $C^{\star}_{n,l}$ 
%%%%and killed on the boundary $\del C^{\star}_{n,l}$. 

The next corollary is a crucial ingredient of the proof of Theorem \thv(2.theo3).

\begin{corollary}
  \TH(9.cor1)
  Assume that $a_n\leq 2^n$.
On $\O_0\cap\O^{\star}$, for all but a finite number of indices $n$,
%Let $\o\in\O_0\cap\O^{\star}$ be fixed.
\be
%\nonumber
\Eq(9.cor1.1)
\hspace{-5pt}\PP_{\pi_n^\circ}\left(\exists_{0\leq j\leq k^\dagger_n(t)-1}\exists_{1\leq l\leq L^{\star}}
\wh\L_n^\dagger(J_n^\dagger(j))\1_{\{J_n^\dagger(j)\in C^{\star}_{n,l}\}}
>n\bar\varrho_{n,l}(0)
%, J^\dagger_n(j)\in C^{\star}_{n,l}
\right)
\leq t e^{-n}+n^{-2(c_{\star}-1)+c_\circ}
%2n^{-2c_{\star}+3}.
\ee
where $c_\circ>0$ is a constant that can be chosen arbitrarily small.
\end{corollary}

Let us sketch the proof of Proposition \thv(9.prop1). Suppose that $J^\dagger_n(j)\in C^{\star}_{n,l}$. Then 
$\wh\L_n^\dagger(J_n^\dagger(j))$ is equal in distribution to the absorption time of the restriction
%, $X^{\star}_{n,l}$,
of the process $X_n$ to $C^{\star}_{n,l}$, killed on the boundary $\del C^{\star}_{n,l}$ and started in $J^\dagger_n(j)$. 
It is well known that its law 
%of such absorption time 
is governed by the spectral characteristics of  the associated infinitesimal generator (more specifically by the corresponding Dirichlet eigenvalues) and is, on suitable time scales, approximated to very good precision by an exponential law.
This section is thus organized as follows. In Subsection \thv(9.1) we introduce the absorbing processes of interest, study their low lying Dirichlet eigenvalues in Subsection \thv(9.2) and, in Subsection \thv(9.3), deduce from this the distribution of the absorption times, using a result from  Appendix \thv(B).  
%prove Proposition \thv(9.prop1) using the bounds of Appendix in Subsection \thv(9.3). 
The proof of Corollary \thv(9.cor1) is done in Subsection \thv(9.4). 
%This of course is strongly reminiscent of metastability \cite{????????}.
%but since we only seek a simple upper bound we keep this simple.

%and deduce from this the distribution of the absorption times by reducing the problem to a discrete time problem, 
%handled 
%dealt with 
%using Appendix \thv(B) ------ .

%is similar in spirit to the prof of ???. We introduce The proof of 

%Suppose that $J^\dagger_n(j)\in C^{\star}_{n,l}$. Then, with the notation of Section \thv(5), 
%$\wh\L_n^\dagger(J_n^\dagger(j))$,  has is equal in distribution
%\be
%\Eq(9.prop1.2)
%\L^{\star}_{n,l}\equiv\sum_{i=0}^{T^{\star}_{n,l}-1}\l^{-1}_n(J^{\star}_{n,l}(i))e_{n,i},
%\ee
%for $J^{\star}_{n,l}$ started in $J^{\star}_{n,l}(0)=J^\dagger_n(j)$. 
%$\L^{\star}_{n,l}$ 
%in turn can be viewed as the absorption time the process $X_n$ restricted to $C^{\star}_{n,l}$ 
%and killed on the boundary $\del C^{\star}_{n,l}$

%To prove Proposition \thv(9.prop1) one could start from.... Alternatively, one may observe that $\L^{\star}_{n,l}$ 
%Of course this is nothing but 
%is the sojourn time of the (continuous time) process 
%$X_n$ in the set $C^{\star}_{n,l}$, given that it enters this set at time zero in 
%$X_n(0)=J^\dagger_n(j)$. Adopting this last view point we establish bounds on the tail distribution of $\L^{\star}_{n,l}$.

%To each visit of $J^\dagger_n(j)\notin \VV^\circ_n$ 

\subsection{The `starred' absorbing processes}
 \label{9.1}

Let $C^{\star}_{n,l}$, $1\leq l\leq L^{\star}$, be the collection of connected components defined through \eqv(10.1.4).
To each component $C^{\star}_{n,l}$ we associate an absorbing Markov process $X^{\star}_{n,l}$ with state space $C^{\star}_{n,l}\cup\Delta$, where the absorbing point, $\Delta$, represents the boundary $\del C^{\star}_{n,l}$;
its infinitesimal generator $\overline\LL^{\star}_{n,l}=\left(\bar\l^{\star}_{n,l}(x,y)\right)$ has entries
$\bar\l^{\star}_{n,l}: \{C^{\star}_{n,l}\cup\Delta\}\times \{C^{\star}_{n,l}\cup\Delta\} \rightarrow \R$,
given by
\be
\bar\l^{\star}_{n,l}(x,y)=
\begin{cases}
\l_n(x,y) &\hbox{\rm if}\quad (x,y)\in G(C^{\star}_{n,l}),\\
\sum_{y'\notin C^{\star}_{n,l}}\l_n(x,y')
%\sum_{y\in \del C^{\star}_{n,l}}p_n(x,y)
&\hbox{\rm if}\quad x\in C^{\star}_{n,l},y=\Delta,\\
-\sum_{y'\in \VV_n}\l_n(x,y')&\hbox{\rm if}\quad x=y\in C^{\star}_{n,l} ,\\
%\quad x=y,\neq \Delta
0&\hbox{\rm else}.\\
% in particylar the row $\l^{\star}_{n,l}(\Delta,\cdot)$ is identically zero
\end{cases}
\Eq(9.prop1.3)
\ee
%Thus the absorbing point $\Delta$ stands for the boundary $\del C^{\star}_{n,l}$,
Thus $X^{\star}_{n,l}$ can be viewed as the restriction of $X_n$
to $C^{\star}_{n,l}$, killed on the boundary $\del C^{\star}_{n,l}$.
We also call $\LL^{\star}_{n,l}=\left(\l^{\star}_{n,l}(x,y)\right)$
%$Q_{n,l}=\left(q_{n,l}(x,y)\right)$
the sub-Markovian restriction of
$\overline\LL^{\star}_{n,l}$ to $C^{\star}_{n,l}$, namely
$\l_{n,l} : C^{\star}_{n,l}\times C^{\star}_{n,l} \rightarrow \R$,
\be
\l^{\star}_{n,l}(x,y)=
\begin{cases}
\l_n(x,y) &\hbox{\rm if}\quad (x,y)\in G(C^{\star}_{n,l})\\
-\sum_{y'\in \VV_n}\l_n(x,y')&\hbox{\rm if}\quad x=y\in C^{\star}_{n,l}.\\
\end{cases}
\Eq(9.prop1.4)
\ee
With this notation $\wh\L_n^\dagger(J_n^\dagger(j))$ in \eqv(9.prop1.00) is nothing but the absorption time 
\be
\Eq(9.prop1.5)
\L^{\star}_{n,l}\equiv\inf\{t>0 \mid X^{\star}_{n,l}(t)=\Delta\}
\ee
of the process $X^{\star}_{n,l}$ started in $X^{\star}_{n,l}(0)=J^\dagger_n(j)$.
% It is well known that the distribution 
%of $\L^{\star}_{n,l}$ is a mixture of exponentials laws whose parameters are governed by the spectral characteristics of 
%$\LL^{\star}_{n,l}$ \cite{?????}. 
  
\subsection{Spectrum of  $-\LL^{\star}_{n,l}$ and other properties.}
 \label{9.2}
 It follows from the properties of the sets $C^{\star}_{n,l}$ that $\{\Delta\}$ is the unique recurrence class 
 of $\overline\LL^{\star}_{n,l}$ and that $\LL^{\star}_{n,l}$ is irreducible (in fact primitive) and non periodic. 
 Also note that $\LL^{\star}_{n,l}$ is reversible with respect to the restriction of Gibbs measure (see \eqv(1.1.3)) to 
$C^{\star}_{n,l}$,
\be
\Eq(9.prop1.6)
\mu^{\star}_{\b,n,l}(x)
=\frac{G_{\b,n}(x)}{\sum_{x'\in C^{\star}_{n,l}}G_{\b,n}(x')}, \quad x\in C^{\star}_{n,l}.
%=\frac{\exp\{-\b H_n(x)\}}{\sum_{x'\in C^{\star}_{n,l}}\exp\{-\b H_n(x')\}}, \quad x\in C^{\star}_{n,l}.
%\frac{w_n(x)}{\sum_{x'\in C^{\star}_{n,l}}w_n(x')}, \quad x\in C^{\star}_{n,l}.
\ee
Let $I_d$ denote the identity matrix in $\R^d$ and set $N\equiv N_{n,l}=|C^{\star}_{n,l}|$. In view of \eqv(1.1.10) the matrix
$
\overline R_{n,l}\equiv\overline\LL^{\star}_{n,l}+I_{N+1}
$
is stochastic (its entries are in $[0,1]$ and its rows sum up to one) and $\{\Delta\}$ is its unique recurrence class. 
From this and Perron's theorem we deduce  (see \eqv(B1.lem1.1))  that the eigenvalues of $-\LL^{\star}_{n,l}$ satisfy
\be
\Eq(9.prop1.7)
0<\varsigma_{n,l}(0)<\varsigma_{n,l}(1)\leq \varsigma_{n,l}(2)\dots\leq \varsigma_{n,l}(N-1).
\ee

The next two Lemmata state bounds on the eigenvalues
$\varsigma_{n,l}(0)$ and $\varsigma_{n,l}(1)$. 
Let  $\bar\varrho_{n,l}(0)$ and $\bar\varrho_{n,l}(1)$ 
be given by
%defined in
\eqv(9.lem1.0) and \eqv(9.lem1.0') respectively.
%ICI

\begin{lemma}
  \TH(9.lem1)
  On $\O^{\star}$, for all but a finite number of indices $n$,
$
%\displaystyle
%(1-\OO({1}/{\log n}))
\frac{1-\OO({1}/{\log n})}{\bar\varrho_{n,l}(0)}
\leq
\varsigma_{n,l}(0)
\leq
\frac{|C^{\star}_{n,l}|}{\bar\varrho_{n,l}(0)}.
$
%(Clearly $\bar\varrho_{n,l}(0)\leq \bar\varrho_{n,l}(0)$.)
\end{lemma}

\begin{lemma}
  \TH(9.lem2)
$
\displaystyle
\varsigma_{n,l}(1)\geq
1/\left(n|C^{\star}_{n,l}|^2\bar\varrho_{n,l}(1)\right).
$
\end{lemma}

\begin{proof}[Proof of Lemma \thv(9.lem1)] 
Rather than considering the matrix $-\LL^{\star}_{n,l}$, it is convenient to consider
% the nonnegative and substochastic restriction
%there are technical advantages in considering
\be
\Eq(9.prop1.0')
R_{n,l}\equiv\LL^{\star}_{n,l}+I_{N}
\ee
%of  $\overline R_{n,l}$ to $C^{\star}_{n,l}$, 
that is to say, the nonnegative and substochastic matrix $R_{n,l}=\left(r_{n,l}(x,y)\right)$ with entries
%the matrix $R_{n,l}=\left(r_{n,l}(x,y)\right)$ with entries
% $r_{n,l} : C^{\star}_{n,l}\times C^{\star}_{n,l} \rightarrow \R$, 
\be
r^{\star}_{n,l}(x,y)=
\begin{cases}
\l_n(x,y), &\hbox{\rm if}\quad (x,y)\in G(C^{\star}_{n,l}),\\
1-\sum_{y'\in \VV_n}\l_n(x,y'),&\hbox{\rm if}\quad x=y\in C^{\star}_{n,l} \\
\end{cases}
\Eq(9.lem1.1)
\ee
and eigenvalues
$
1>\vartheta_{n,l}(0)>\vartheta_{n,l}(1)\geq \dots\geq \vartheta_{n,l}(N-1)>-1
$,
\be
\Eq(9.lem1.1')
\vartheta_{n,l}(i)\equiv 1-\varsigma_{n,l}(i).
\ee
%In view of \eqv(1.1.10), $R_{n,l}$ is a nonnegative, substochastic matrix and is reversible with respect to 
To bound $\vartheta_{n,l}(0)$ from below note that since $R_{n,l}$ is reversible with respect to $\mu^{\star}_{\b,n,l}$, it is similar to the matrix
with entries
$
r_{n,l}(x,y)\sqrt{{\mu^{\star}_{n,l}(x)}/{\mu^{\star}_{n,l}(y)}}
$,
%its eigenvalues coincide with those of $R_{n,l}$ (see e.g. \cite{HJ}, p. 44-45)
and the latter being symmetric, 
the minimax characterization of its largest eigenvalue (see e.g. \cite{HJ}, p. 176) gives
\be
\Eq(9.lem1.2)
\vartheta_{n,l}(0)=\sup_{u\neq 0}\frac{
\sum_{\{x,y\}\in C^{\star}_{n,l}\times C^{\star}_{n,l}}u(x)u(y)\mu^{\star}_{n,l}(x)r_{n,l}(x,y)
}
{
\sum_{x\in C^{\star}_{n,l}}u^2(x)\mu^{\star}_{n,l}(x)
},
\ee
where supremum is taken over all non zero functions
real valued function $u$ on $C^{\star}_{n,l}$.
%$u=(u(x), x\in C^{\star}_{n,l})\in\R^N$.
Thus $\vartheta_{n,l}(0)$ is bounded below by the ratio in the r.h.s. of \eqv(9.lem1.2) evaluated at
constant functions,
\be
\Eq(9.lem1.3)
\vartheta_{n,l}(0)
\geq
\textstyle
\sum_{\{x,y\}\in C^{\star}_{n,l}\times C^{\star}_{n,l}}
\mu^{\star}_{n,l}(x)r_{n,l}(x,y)
=1-A/B
%\frac{A}{nB},
\ee
where, by  \eqv(9.prop1.6), \eqv(9.lem1.1), and the definition \eqv(1.1.10) of $\l_n(x,y)$,
\bea
\Eq(9.lem1.4)
&A\equiv\sum_{x\in C^{\star}_{n,l}}\sum_{y\in\del C^{\star}_{n,l}:\{x,y\}\in\EE_n}
e^{-\b\max\left(H_n(y),H_n(x)\right)},
\\
\Eq(9.lem1.5)
&B\equiv n\sum_{x\in C^{\star}_{n,l}}
e^{-\b H_n(x)}.
\eea
%as follows from \eqv(9.prop1.6), \eqv(9.lem1.1), and the definition \eqv(1.1.10) of $\l_n(x,y)$.
Now by construction (see \eqv(10.1.4) in Subsection \thv(3.3)) the maximum in \eqv(9.lem1.4) is equal to zero. Hence
$A=\sum_{x\in C^{\star}_{n,l}}\sum_{y\in\del C^{\star}_{n,l}:\{x,y\}\in\EE_n}1\leq n|C^{\star}_{n,l}|$.
This and the bound
\be
\Eq(9.lem1.6)
n\leq B e^{\b\min\left\{H_n(x)\,|\, x\in C^{\star}_{n,l}\right\}}\leq n|C^{\star}_{n,l}|
\ee
yields 
$
\vartheta_{n,l}(0)=1-\varsigma(0)
\geq
1-\frac{|C^{\star}_{n,l}|}{\bar\varrho_{n,l}(0)}
$,
which is the desired lower bound.
%proving the upper bound of Lemma \thv(9.lem1).

%%%%%%%%%%%%%%%%%%%%%%%%%%%%%%%%%%%%%%%%%%%%%%%%%%%
We now turn to the upper bound. Denote by $R^{\star}_{n,l}$ the stochastic matrix
$R^{\star}_{n,l}\equiv R_{n,l}+D_{n,l}$ where
$D_{n,l}=\left(d_{n,l}(x,y)\right)$ is the diagonal matrix with entries
$d_{n,l} : C^{\star}_{n,l}\times C^{\star}_{n,l} \rightarrow [0,1]$,
\be
d_{n,l}(x,y)=
\begin{cases}
\sum_{y'\notin C^{\star}_{n,l}}\l_n(x,y'), &\hbox{\rm if}\quad x=y, x\in C^{\star}_{n,l},\\
0,&\hbox{\rm else}.\\
\end{cases}
\Eq(9.lem1.20)
\ee
For $0\leq k\leq N-1$ denote by $\vartheta^{\star}_{n,l}(k)$ and $\vartheta^{d}_{n,l}(k)$
the eigenvalues of $R^{\star}_{n,l}$ and $D_{n,l}$, respectively, arranged in decreasing order of magnitude.
Then, since $R_{n,l}=R^{\star}_{n,l}-D_{n,l}$, by Weyl's theorem on eigenvalues (see e.g. \cite{HJ} p. 181),
\be
\Eq(9.lem1.21)
\vartheta_{n,l}(k)\leq \vartheta^{\star}_{n,l}(k)-\vartheta^{d}_{n,l}(N-1),\quad 0\leq k\leq N-1.
\ee
Here $\vartheta^{\star}_{n,l}(0)=1$ and
$
\vartheta^{d}_{n,l}(N-1)=\min_{x\in C^{\star}_{n,l}}\sum_{y'\in\del C^{\star}_{n,l}:\{x,y'\}\in\EE_n}\l_n(x,y')
%1-\max_{x\in C^{\star}_{n,l}}\sum_{y\in C^{\star}_{n,l}:\{x,y\}\in\EE_n}q_{n,l}(x,y)
$.
Thus,
%applying \eqv(9.lem1.21) with $k=0$ gives
\be
\Eq(9.lem1.9)
%\textstyle
\vartheta_{n,l}(0)
\leq 1-\min_{x\in C^{\star}_{n,l}}\sum_{y'\in\del C^{\star}_{n,l}:\{x,y'\}\in\EE_n}\l_n(x,y')
=1-\frac{|\del C^{\star}_{n,l}\cap\del x|}
{n\bar\varrho_{n,l}(0)},
%e^{-\b\min\left\{H_n(x)\,|\,x\in C^{\star}_{n,l}\right\}}},
\ee
where we used that
$
0=H_n(y')\geq H_n(x)
$
for all $\{x,y'\}\in\EE_n$ with $x\in C^{\star}_{n,l}$ and $y'\in\del C^{\star}_{n,l}$.
Inserting the bound \eqv(10.lem1.7) of  Lemma \thv(10.lem1) we obtain that, on $\O^{\star}$, for all but a finite number of indices $n$,
$
\vartheta_{n,l}(0)=1-\varsigma(0)\leq 1- (1-\OO(\frac{1}{\log n}))\frac{1}{\bar\varrho_{n,l}(0)}
$ 
for all $x\in C^{\star}_{n,l}$. The proof of Lemma \thv(9.lem1) is done.
\end{proof}

\begin{proof}[Proof of Lemma \thv(9.lem2)] Applying the bound \eqv(9.lem1.21) with $k=1$ gives
$\vartheta_{n,l}(1)\leq \vartheta^{\star}_{n,l}(1)$ since, clearly, $\vartheta^{d}_{n,l}(N-1)\geq 0$.
%The point of doing this is that
$R^{\star}_{n,l}$ being a stochastic matrix, questions on $\vartheta^{\star}_{n,l}(1)$ then
reduce to classical questions on the spectral gap $\t_{n,l}^{-1}\equiv 1-\vartheta^{\star}_{n,l}(1)$.
It is in particular well known that
\be
\Eq(9.lem2.2)
\t_{n,l}=\sup\left\{\frac{\EE(u,u)}{Var(u)}\,:\,u \,\,\hbox{\rm is nonconstant}\right\}\,,
\ee
where, denoting by $u$ a real valued function on $C^{\star}_{n,l}$,
\bea
\Eq(9.lem2.3)
&Var(u)=\frac{1}{2}\sum_{x,y}\left(u(x)-u(y)\right)^2\mu^{\star}_{n,l}(x)\mu^{\star}_{n,l}(y),
\\
\Eq(9.lem2.4)
&\EE(u,u)=\frac{1}{2}\sum_{x,y}\left(u(x)-u(y)\right)^2\mu^{\star}_{n,l}(x)r_{n,l}(x,y).
\eea
It is also well known  that based on \eqv(9.lem2.2), one may derive bounds on $\t_{n,l}$
in terms of ``canonical paths'' (see e.g. \cite{SJ}).
The bound stated below is taken from \cite{DS} (see Proposition 1' p. 38).
For each pair of distinct vertices $x,y\in C^{\star}_{n,l}$, choose a path $\g_{x,y}$ going from $x$ to $y$ in the graph $G(C^{\star}_{n,l})$. Paths may have repeated vertices but a given edge appears at most once in a  given path. Let $\G_{n,l}$ denote a collection of paths (one for each pair $\{x,y\}$).
%Let $\G_{n,l}$ be a complete set of self-avoiding directed paths in the graph $G(C^{\star}_{n,l})$. 
Then
\be
\Eq(9.lem2.5)
\t_{n,l}\leq \max_{e}\rho^{-1}_n(e)\sum_{\g_{x,y}\ni e}\left|\g_{x,y}\right|\mu^{\star}_{n,l}(x)\mu^{\star}_{n,l}(y),
\ee
where the max is over all edges $e=\{x',y'\}$ of $G(C^{\star}_{n,l})$,
$\rho_n(e)\equiv\mu^{\star}_{n,l}(x')r_{n,l}(x',y')$, and the summation is over all paths in $\G_{n,l}$
%$\g_{x,y}\in\G_{n,l}$ going from $x$ to $y$ 
that pass through $e$.
The quality of this bound usually  depends on the choice of $\G_{n,l}$. Here however we simply
bound the length of the longest path by the total number of edges in the graph,
that is  $\left|\g_{x,y}\right|\leq |C^{\star}_{n,l}|$. Eq. \eqv(9.lem2.5) then immediately gives
$
\t_{n,l}
\leq
\left|C^{\star}_{n,l}\right|\max_{e}\rho^{-1}_n(e)
$,
and so
\be
\Eq(9.lem2.7)
\vartheta_{n,l}(1)\leq 1-\left|C^{\star}_{n,l}\right|^{-1}\min_{\{x,y\}\in G(C^{\star}_{n,l})}\mu^{\star}_{n,l}(x)r_{n,l}(x,y).
\ee
Proceeding as in \eqv(9.lem1.3) to evaluate the r.h.s. above, we obtain
\be
\Eq(9.lem2.8)
\vartheta_{n,l}(1)\leq 1-\left|C^{\star}_{n,l}\right|^{-1}
\min_{\{x,y\}\in G(C^{\star}_{n,l})}B^{-1}
e^{-\b\max\left(H_n(y),H_n(x)\right)},
\ee
%%%%%%%%%%% Notice that this is an equality !!!!!!!!!!!!! %%%%%%%%%%%%%%%%%%
%
%and where we used that
%$
%\min_{\{x,y\}\in G(C^{\star}_{n,l})}e^{-\b\max\left(H_n(y),H_n(x)\right)}
%= e^{-\b\max\left\{H_n(y)\,|\,y\in C^{\star}_{n,l}\right\}
%$
%%%%%%%%%%%%%%%%%%%%%%%%%%%%%%%%%%%%%%%%%%%%%%%%%%%%%%%%%%%%%%%%%%%%%%%%%%%
where  $B$ is defined in  \eqv(9.lem1.5). Finally, plugging in the upper bound of \eqv(9.lem1.6)
we arrive at
$
\vartheta_{n,l}(1)=1-\varsigma_{n,l}(1)\leq 1-(n|C^{\star}_{n,l}|^2\bar\varrho_{n,l}(1))^{-1}
$.
The proof of  Lemma \thv(9.lem2) is complete.
\end{proof}

%%%%%%%%%%%%%%%%%%%%%%%%%%%%%%%%%%%%%%%%%%%%%%%%%%%%%%%%
We close this subsection with bounds 
 that will be needed in several places.
%that will repeatedly be needed in the sequel.
%two lemmata that will repeatedly be needed in the sequel.
%bounds that will often be needed in the sequel. 
\begin{lemma}
  \TH(9.lem4)
For
$
r_n\bigl(\rho^{\star}_n\bigr)
$
%defined in \eqv(10.1.1) and bounded 
as in \eqv(9.lem4'.2),
%With the notation of Lemma \thv(3.lem3), for all $n$ large enough,
\be
\nonumber
%\Eq(9.lem4.1)
\textstyle
e^{\b n\sqrt{2\log 2}(1+2\log n/n)}
\geq
\bar\varrho_{n,l}(0)
\geq 
\varrho_{n,l}(1)
\frac{\bar\varrho_{n,l}(0)}{\bar\varrho_{n,l}(1)}
=e^{-\b\max\left\{H_n(x)\,|\,x\in C^{\star}_{n,l}\right\}}
\geq  r_n\left(\rho^{\star}_n\right),
\ee
where
%, using the notation of Lemma  \thv(9.lem7), 
the first inequality holds on $\O_0$,
for all but a finite number of indices $n$.
\end{lemma}

\begin{proof}
%[Proof of Lemma \thv(9.lem4)]
The first inequality
%in \eqv(9.lem4.1)
is \eqv(9.lem7.2).
The second is immediate. The third  follows
by construction from the truncation  (see \eqv(10.1.2), \eqv(intro.3), and \eqv(10.1.4))
 and relation \thv(10.1.1). \end{proof}

%%%%%%%%%%%%%%%%%%%%%%%%%%%%%%%%%%%%%%%%%%%%%%%%%%%%%%%%%

\subsection{
%Tail distribution of $\L^{\star}_{n,l}$, 
Absorption times: proof  of Proposition \thv(9.prop1).}
 \label{9.3}

%\be
%\Eq(9.prop1.16)
%e^{t\overline\LL^{\star}_{n,l}}
%=e^{-tI_{N+1}+ t(\overline\LL^{\star}_{n,l}+I_{N+1})}
%=e^{-t}e^{t\overline R_{n,l}}
%%R_{n,l}\equiv\LL^{\star}_{n,l}+I_{N}
%\ee

\begin{proof}[Proof of Proposition \thv(9.prop1)] Assume that \eqv(9.prop1.1)
holds for integer values of $t$  in $[\bar\theta^{\star}_{n,l}-1,\infty)$. Then, writing $t-1<\lfloor t\rfloor \leq t$ 
where $\lfloor t\rfloor$ is the largest integer smaller than or equal to $t$, we have, for all $t\geq \bar\theta^{\star}_{n,l}$,
%Then, writing
%$
%
%$
%where $\lfloor t\rfloor$ denotes the largest integer smaller than of equal to $t$, we have
\be
\Eq(9.prop1.???)
\textstyle
\PP_x\left(\L^{\star}_{n,l}>t\right)
\leq 
\PP_x\left(\L^{\star}_{n,l}>\lfloor t\rfloor \right)
\leq 
e^{-\lfloor t\rfloor/\t^{\star}_{n,l}}(1+\bar\epsilon^{\star}_n)
\leq 
e^{-t/\t^{\star}_{n,l}}e^{1/\t^{\star}_{n,l}}(1+\bar\epsilon^{\star}_n),
\ee
where by \eqv(9.prop1.14),
$
e^{1/\t^{\star}_{n,l}}(1+\bar\epsilon^{\star}_n)
\leq  \sfrac{\bar\theta^{\star}_{n,l}|C^{\star}_{n,l}|}{\bar\varrho_{n,l}(0)}(1+o(1))
$.

It thus suffices to prove  \eqv(9.prop1.1) for integer $t$'s  in $[\bar\theta^{\star}_{n,l}-1,\infty)$.
Assume from now on that
$t=m>1$ is an integer and set $A\equiv e^{\LL^{\star}_{n,l}}$. 
Writing $A^m_{n,l}=(a^{(m)}_{n,l}(x,y))$, it follows from \eqv(9.prop1.3)-\eqv(9.prop1.5) and the semigroup property that
\be
\Eq(9.prop1.17)
\textstyle
\PP_x\left(\L^{\star}_{n,l}>m\right)
%=\sum_{y\in C^{\star}_{n,l}}\PP_x\left(X^{\star}_{n,l}(t)=y, \L^{\star}_{n,l}>t\right)
=\sum_{y\in C^{\star}_{n,l}}(\d_x,e^{m\LL^{\star}_{n,l}}\d_y)
%=\sum_{y\in C^{\star}_{n,l}}(\1_x,A^m\1_y)
=\sum_{y\in C^{\star}_{n,l}}a^{(m)}_{n,l}(x,y),
\ee
where $(\cdot,\cdot)$ denotes the inner product in $\R^N$, $N\equiv |C^{\star}_{n,l}|$, and $\d_x$ is the vector with components $\d_x(x')=1$ if $x'=x$ and  zero otherwise.
%,  are the canonical unit vectors in the Cartesian coordinate system.
% (namely,  $\1_x(x')=1$ if $x=x'$ and  zero otherwise). 
%The last equality is the semi-group property.
Clearly, $A_{n,l}$ is a nonnegative and primitive matrix,  and is reversible with respect to $\mu^{\star}_{\b,n,l}$.
We can therefore use Lemma  \thv(B1.lem1) to evaluate the right hand side of \eqv(9.prop1.17). 
%The idea now is to use Lemma  \thv(B1.lem1) to evaluate the last line. 
For this note that $A_{n,l}$ has eigenvalues $\exp\{-\varsigma_{n,l}(i)\}$, $0\leq i\leq N-1$, 
where the $\varsigma_{n,l}(i)$'s obey \eqv(9.prop1.7),
and denote by $u^{\star}_{n,l}$ and $v^{\star}_{n,l}$ the left and right Perron eigenvectors of $A_{n,l}$ 
(see \eqv(B1.lem1.1)-\eqv(B1.lem1.2)), 
\bea
\Eq(9.prop1.21)
&&u^{\star}_{n,l}A_{n,l}=e^{-\varsigma_{n,l}(0)}u^{\star}_{n,l},\quad  u^{\star}_{n,l}>0,
\\
&&\Eq(9.prop1.22)
A_{n,l}v^{\star}_{n,l}=e^{-\varsigma_{n,l}(0)}v^{\star}_{n,l},\quad v^{\star}_{n,l}>0,
\eea
normalized to make  
\be
\Eq(9.prop1.18)
\textstyle
\sum_{x\in C^{\star}_{n,l}}u^{\star}_{n,l}(x)=1,\quad \sum_{x\in C^{\star}_{n,l}}u^{\star}_{n,l}(x)v^{\star}_{n,l}(x)=1.
\ee
Then, by \eqv(B1.lem1.5) of Lemma \thv(B1.lem1),
\be
\Eq(9.prop1.9)
\PP_x\left(\L^{\star}_{n,l}>m\right)=v^{\star}_{n,l}(x)e^{-m\varsigma_{n,l}(0)}\left\{1+\RR_ne^{-m(\varsigma_{n,l}(1)-\varsigma_{n,l}(0))}\right\},
\ee
where
\be
\Eq(9.prop1.19)
|\RR_n|
\leq 
\Bigl(\min_{x\in C^{\star}_{n,l}}\bigl\{v^{\star}_{n,l}(x)(\mu^{\star}_{n,l}(x))^{1/2}\bigr\}\Bigr)^{-1}.
\ee
%It remains to estimate the term in brackets in \eqv(9.prop1.9).

In order to control $|\RR_n|$ we need a lower bound on $v^{\star}_{n,l}(x)$.

\begin{lemma}
  \TH(9.lem3)
For all $x\in C^{\star}_{n,l}$, 
$
v^{\star}_{n,l}(x)\geq (1/\bar\varrho_{n,l}(0))^{|C^{\star}_{n,l}|}
$.
%\be
%\Eq(9.lem3)
%v^{\star}_{n,l}(x)\geq 1/(2n\bar\varrho_{n,l}(0))^{|C^{\star}_{n,l}|}
%\ee
\end{lemma}

\begin{proof}[Proof of Lemma \thv(9.lem3)] 
%Proceeding as in the proof of Lemma \thv(9.lem5) one  shows that
Write
%\be
%\Eq(9.prop1.16)
$
A_{n,l}=e^{\LL^{\star}_{n,l}}
=e^{-I_{N}+ (\LL^{\star}_{n,l}+I_{N})}
=e^{-1}e^{R_{n,l}}
$
and recall from \eqv(9.prop1.0') that $R_{n,l}\equiv\LL^{\star}_{n,l}+I_{N}$ is a nonnegative and primitive matrix.
%\ee
By the spectral decomposition the left and right Perron eigenvectors  
of $A_{n,l}$ and $R_{n,l}$ coincide, and thus
\bea
\Eq(9.prop1.23)
&&u^{\star}_{n,l}R_{n,l}=(1-\varsigma_{n,l}(0))u^{\star}_{n,l},\quad  u^{\star}_{n,l}>0,
\\
&&\Eq(9.prop1.24)
R_{n,l}v^{\star}_{n,l}=(1-\varsigma_{n,l}(0))v^{\star}_{n,l},\quad v^{\star}_{n,l}>0,
\eea
where $u^{\star}_{n,l}$ and $v^{\star}_{n,l}$ obey \eqv(9.prop1.18).
This implies in particular that 
\be
\Eq(9.prop1.25)
\textstyle
\sum_{x\in C^{\star}_{n,l}}v^{\star}_{n,l}(x)\geq 1.
\ee
Indeed, by \eqv(9.prop1.18), $u^{\star}_{n,l}(x)\leq 1$ for all  $x\in C^{\star}_{n,l}$ and so
%\be 
%\Eq(9.lem5.7)
%\textstyle
$
1=\sum_{x\in C^{\star}_{n,l}}u^{\star}_{n,l}(x)v^{\star}_{n,l}(x)\leq \sum_{x\in C^{\star}_{n,l}}v^{\star}_{n,l}(x)
$.
%\ee
Equipped with \eqv(9.prop1.24) and \eqv(9.prop1.25) we prove the lemma by contradiction.
Assume that $v^{\star}_{n,l}(x)<\varepsilon_0\equiv (1/\bar\varrho_{n,l}(0))^{|C^{\star}_{n,l}|}$ for some  $x\in C^{\star}_{n,l}$.
Then, by \eqv(9.prop1.24),
\be 
\Eq(9.lem5.3)
\textstyle
\sum_{x\in C^{\star}_{n,l}}r_{n,l}(y,x)v^{\star}_{n,l}(x)=(1-\varsigma_{n,l}(0))v^{\star}_{n,l}(y)\leq \varepsilon_0,
\ee
and since
$
\min\left\{r_{n,l}(y,x)\,|\, x,y\in C^{\star}_{n,l}\right\}
\geq 
\frac{ r_n\left(\rho^{\star}_n\right)}{n\bar\varrho_{n,l}(0)}
$,
%so that
%$
%\sum_{y\in C^{\pm}_{n,l}}u^{\star}_{n,l}(y)\leq \varepsilon_0\frac{ n\varrho_{n,l}(0)}{r_n\left(\rho^{\star}_n\right)}
%$
%and 
\be 
\Eq(9.lem5.4)
\textstyle
v^{\star}_{n,l}(x)\leq \varepsilon_1\equiv\varepsilon_0\frac{ n\bar\varrho_{n,l}(0)}{r_n\left(\rho^{\star}_n\right)}\,\,\,\text{for each}\,\,\, x\in C^{\star}_{n,l}\cap\del y.
\ee
Repeating this reasoning for each $x\in C^{\star}_{n,l}\cap\del y$, then for each $x\in C^{\star}_{n,l}\cap\del_2 y$, and so on and so forth, we arrive at
\be 
\Eq(9.lem5.6)
\textstyle
\max\left\{v^{\star}_{n,l}(x)\,|\, x\in C^{\star}_{n,l}\right\}
\leq \varepsilon_0 
\bigl(
\frac{n\bar\varrho_{n,l}(0)}{r_n\left(\rho^{\star}_n\right)}
\bigr)^{|C^{\star}_{n,l}|}
<
\frac{n}{r_n\left(\rho^{\star}_n\right)}.
%\ll |C^{\star}_{n,l}|^{-1},
\ee
In view of \eqv(9.lem4'.2) this implies that
$\sum_{x\in C^{\star}_{n,l}}v^{\star}_{n,l}(x)\ll1$,
contradicting \eqv(9.prop1.25). Therefore 
$v^{\star}_{n,l}(x)\geq\varepsilon_0$ for all  $x\in C^{\star}_{n,l}$. 
\end{proof}

Clearly, $\mu^{\star}_{n,l}(x)\geq 1/(\bar\varrho_{n,l}(1)|C^{\star}_{n,l}|)$ for all $x\in C^{\star}_{n,l}$. 
Using this bound and Lemma \thv(9.lem3) in \eqv(9.prop1.19), 
it follows from   \eqv(10.lem1.4) and the first and last inequality 
%in \eqv(9.lem4.1) 
of Lemma \thv(9.lem4)  
% \eqv(9.lem7.2) of Lemma \thv(9.lem7) 
%
% GARDER \eqv(9.lem7.2) of Lemma \thv(9.lem7) 
%
that on $\O_0\cap \O^{\star}$, for all but a finite number of indices $n$,
\be
|\RR_n|\leq \sqrt{n/r_n\left(\rho^{\star}_n\right)}
e^{\b n(|C^{\star}_{n,l}|+1/2)\sqrt{2\log 2}(1+2\log n/n)}.
%e^{\b 2n^2/\log n}.
\Eq(9.prop1.11)
\ee
Next, by  Lemma \thv(9.lem1),  Lemma \thv(9.lem2), and 
%\eqv(9.lem4.1) of 
Lemma \thv(9.lem4) 
%
% GARDER \eqv(9.lem4.1) of Lemma \thv(9.lem4) 
%
we have that on $\O^{\star}$, for all but a finite number of indices $n$,
\be
\varsigma_{n,l}(1)-\varsigma_{n,l}(0)
\geq
%\frac{1}{n|C^{\star}_{n,l}|^3\bar\varrho_{n,l}(1)}
%\left(
%1-|C^{\star}_{n,l}|^3
%\frac{r_n\left(\sfrac{1}{\lfloor\kappa m^{\star}\rfloor}\right)}{r_n\left(\sfrac{1}{\lfloor\kappa m^{\star}\rfloor -1}\right)}
%\right)
%\geq
(1-o(1))/n|C^{\star}_{n,l}|^2\bar\varrho_{n,l}(1)
\equiv 3\b n|C^{\star}_{n,l}|(1-o(1))/\bar\theta^{\star}_{n,l},
\Eq(9.prop1.10)
\ee
where $\bar\theta^{\star}_{n,l}$ is defined in \eqv(9.prop1.0).
%and the fact that, by \eqv(9.lem1.0)-\eqv(9.lem1.0'), $\bar\varrho_{n,l}(1)=\frac{\varrho_{n,l}(0)}{\varrho_{n,l}(1)}$.
Inserting \eqv(9.prop1.11) and \eqv(9.prop1.10) in \eqv(9.prop1.9) we obtain that  on $\O_0\cap\O^{\star}$,  
for all 
$
m\geq
%2\bar\varrho_{n,l}(1)n^3|C^{\star}_{n,l}|^3=
\bar\theta^{\star}_{n,l}-1
$
and all but a finite number of indices $n$.
\be
\Eq(9.prop1.12)
\PP_x\left(\L^{\star}_{n,l}>m\right)=v^{\star}_{n,l}(x)e^{-m\varsigma_{n,l}(0)}(1+\bar\epsilon'_n),
\ee
where $\bar\epsilon'_n=\OO(\exp(-\b n))$.
%We now seek an upper bound on $v^{\star}_{n,l}(x)$. Writing that
$
 \PP_x\left(\L^{\star}_{n,l}>\bar\theta^{\star}_{n,l}\right)\leq 1
%=v^{\star}_{n,l}(x)e^{-\bar\theta^{\star}_{n,l}\varsigma_{n,l}(0)}(1+\bar\epsilon^{\star}_n)
$
yields 
$(1+\bar\epsilon'_n)v^{\star}_{n,l}(x)\leq e^{\bar\theta^{\star}_{n,l}\varsigma_{n,l}(0)}
%\leq 1+({\bar\theta^{\star}_{n,l}|C^{\star}_{n,l}|}/{\bar\varrho_{n,l}(0)})(1+o(1))
$. 
Plugging in the upper bound on $\varsigma_{n,l}(0)$ of Lemma \thv(9.lem1) and inserting the result in \eqv(9.prop1.12) 
we obtain that
\be
\Eq(9.prop1.20)
\PP_x\left(\L^{\star}_{n,l}>m\right)=e^{-m\varsigma_{n,l}(0)}(1+\bar\epsilon^{\star}_n),
\ee
where $0\leq \bar\epsilon^{\star}_n\leq \sfrac{\bar\theta^{\star}_{n,l}|C^{\star}_{n,l}|}{\bar\varrho_{n,l}(0)}(1+o(1))$.
Taking
$
\t^{\star}_{n,l}=1/\varsigma_{n,l}(0)
$
in \eqv(9.prop1.12) now yields the statement of  \eqv(9.prop1.1) 
%with $t=m\in\N$
for all integer values of $t$  in $[\bar\theta^{\star}_{n,l}-1,\infty)$. 
The claimed independance of the $\t^{\star}_{n,l}$'s is a direct consequence of the definition of the $\varsigma_{n,l}(0)$'s. The proof of Proposition \thv(9.prop1) is complete.
\end{proof}

\subsection{Proof  of Corollary  \thv(9.cor1).}
 \label{9.4}

First note that the choice  
$
t=n\bar\varrho_{n,l}(0)
$
%
% Remark : choose $t=2n\bar\varrho_{n,l}(0)$ or $t=n^2\bar\varrho_{n,l}(0)$ to improve decay in \eqv(9.cor1.4)
% from the current $2\lfloor a_n t\rfloor e^{-2n}$ to $2\lfloor a_n t\rfloor e^{-4n}$ or $2\lfloor a_n t\rfloor e^{-2n^2}$
%
in Proposition \thv(9.prop1) is admissible since by  \eqv(9.lem1.0), \eqv(9.lem1.0'), \eqv(9.prop1.0),  \eqv(10.lem1.4),
%\eqv(9.lem4.1) of 
Lemma \thv(9.lem4), and \eqv(9.lem4'.2) of Lemma \thv(9.lem4'), on $\O^{\star}$, for all $n$ large enough,
\be
\sfrac{n\bar\varrho_{n,l}(0)}{\bar\theta^{\star}_{n,l}}
=
[3\b n|C^{\star}_{n,l}|^3]^{-1} \sfrac{\bar\varrho_{n,l}(0)}{\bar\varrho_{n,l}(1)}
\geq 
[3\b n|C^{\star}_{n,l}|^3]^{-1} \varrho_{n,l}(1)
\geq  
n^{-4} r_n\left(\rho^{\star}_n\right)
\gg 1.
\Eq(9.cor1.2)
\ee
Moreover by \eqv(9.prop1.14), for this choice, 
$
t/\t^{\star}_{n,l}
=n\bar\varrho_{n,l}(0)/\t^{\star}_{n,l}\geq n|C^{\star}_{n,l}|
\geq 2n
$.
Hence  Proposition \thv(9.prop1) applies so that on $\O_0\cap\O^{\star}$, for all but a finite number of indices $n$,
%the following holds: for all $1\leq l\leq L^{\star}$ and all $x\in C^{\star}_{n,l}$,
\be
\Eq(9.cor1.3)
\PP\bigl(\wh\L_n^\dagger(J_n^\dagger(j))>n\bar\varrho_{n,l}(0)\mid J^\dagger_n(j)=x\bigr)
\leq e^{-2n}(1+o(1))
\ee
for all $1\leq l\leq L^{\star}$ and all $x\in C^{\star}_{n,l}$. 
Let $\AA$ be the event  in the left hand side of \eqv(9.cor1.1).
By Theorem  \thv(2.theo4),
$
\PP_{\pi_n^\circ}(\AA)\leq \PP_{\pi_n^\circ}(\AA, \{k^\dagger_n(t)\leq \lfloor a_n t \rfloor (1+n^{-c_\circ})\})
+2n^{-2(c_{\star}-1)+c_\circ}
$,
and by \eqv(9.cor1.3), on $\O_0\cap\O^{\star}$,  
$
%\Eq(9.cor1.4)
\PP_{\pi_n^\circ}(\AA, \{k^\dagger_n(t)\leq \lfloor a_n t \rfloor (1+n^{-c_\circ})\})
\leq 
2 \lfloor a_n t \rfloor (1+n^{-c_\circ})e^{-2n}
%\PP_{\mu}\left(\exists_{0\leq j\leq  k^\dagger_n(t)-1}\exists_{1\leq l\leq L^{\star}}
%\wh\L_n^\dagger(J_n^\dagger(j))
%>n\bar\varrho_{n,l}(0), J^\dagger_n(j)\in C^{\star}_{n,l}\right)
%\leq 
%2 \lfloor a_n t \rfloor (1+n^{-1})e^{-2n},
$
for all but a finite number of indices $n$.
Since  $a_n< 2^n$,
% + see remark above if $a_n$ would be larger
\eqv(9.cor1.1) follows.

%%%%%%%%%%%%%%%%%%%%%%%%%%%%%%%%%%%%%%%%%%%%%%%%%%%%%%%%%%%%

%%%%%%%%%%%%%%%%%%%%%%%%%%%%%%%%%%%%%%%%%%%%%%%%%%%%%%%%%%%%

%\section{The clusters' chains }
%\section{The $\star$-components' chains}
%\section{The (starred) absorbing chains}
\section{Distribution of the increments of the back end clock process  $\wt S_n^\dagger$.}
 \label{5}

 This section parallels Section \thv(9), focusing this time on the increments, $\L_n^\dagger$, 
 of the process $\wt S_n^\dagger$.
Just as the $\wh\L_n^\dagger$'s are the sojourn times of the process  $X_n$  in the sets $C^{\star}_{n,l}$,
the $\L_n^\dagger$'s are the sojourn times of the chain $J_n$ in those sets.   
Indeed, by \eqv(2.2.8), 
%setting
%\be
% T^{\star}_{n,l}=\inf\{i>j \mid J_n(i)\in\del C^{\star}_{n,l}\},
% \Eq(2.2.10'bis)
%\ee
\be
\L_n^\dagger(J_n^\dagger(j))
\equiv
\inf\{i>j \mid J_n(i)\in\del C^{\star}_{n,l}\}-j
%=_d T^{\star}_{n,l}
\Eq(5.prop1.00)
\ee
if $J^\dagger_n(j)\in C^{\star}_{n,l}$ for some ${1\leq l\leq L^{\star}}$,
%if $J^\dagger_n(j)\notin \VV^\circ_n$, 
and $\wh\L_n^\dagger(J_n^\dagger(j))=0$  otherwise. 
For $1\leq l\leq L^{\star}$, set
\bea
\Eq(5.lem1.0)
&\varrho_{n,l}(0)=
e^{-\b\min\left\{\max\left(H_n(y),H_n(x)\right)\,|\, \{x,y\}\in G(C^{\star}_{n,l})\right\}},
\\
\Eq(5.lem1.0')
&\varrho_{n,l}(1)=
e^{-\b\max\left\{H_n(x)\,|\,x\in C^{\star}_{n,l}\right\}},
\eea
and
\be
\Eq(5.prop1.1)
%\theta^{\star}_{n,l}=2n^2|C^{\star}_{n,l}|^4\left({\varrho_{n,l}(0)}/{\varrho_{n,l}(1)}\right).
%
% je change le \theta^{\star}_{n,l} : le changer aussi pour les chaines a temps continu?
%
\theta^{\star}_{n,l}=2\b n|C^{\star}_{n,l}|^5\left({\varrho_{n,l}(0)}/{\varrho_{n,l}(1)}\right).
\ee

\begin{proposition}
 \TH(5.prop1)
%For all  $1\leq l\leq L^{\star}$ such that $|C^{\star}_{n,l}|=2$ we have,
%for all $ i>0$ and all $x$ in $C^{\star}_{n,l}$,
{(i)} For each 
%$C^{\star}_{n,l}$ such that
 $1\leq l\leq L^{\star}$ such that 
 $|C^{\star}_{n,l}|=2$ we have, for all $ i>0$ and all $x$ in $C^{\star}_{n,l}$,
 \be
\Eq(5.prop1.4)
\PP\left(\L_n^\dagger(J_n^\dagger(j))>i\mid J^\dagger_n(j)=x\right)
%=P_{x}\left(T^{\star}_{n,l}>i\right)=\vartheta_{n,l}^{i}(0)
=\left(1-\sfrac{1}{1+\varrho_{n,l}(0)/(n-1)}\right)^i.
\ee
{(ii)} Furthermore,
%Let $\O_0$ and $\O^{\star}$ be as in Lemma \thv(10.lem1) and Lemma \thv(9.lem7), respectively.
on $\O^{\star}\cap\O_0$ (for $\O^{\star}$ and $\O_0$ as in Lemma \thv(10.lem1) and Lemma \thv(9.lem7),  respectively), for all but a finite number of indices $n$, 
the following holds for all $1\leq l\leq L^{\star}$: for all $ i\geq \theta^{\star}_{n,l}$ 
and all $x$ in $C^{\star}_{n,l}$,
\be
\Eq(5.prop1.0')
\PP\left(\L_n^\dagger(J_n^\dagger(j))>i\mid J^\dagger_n(j)=x\right)
\leq e^{-i(n/\varrho_{n,l}(0)|C^{\star}_{n,l}|)(1-o(1))}(1+\hat\epsilon_n),
\ee
where
$
0\leq \hat\epsilon_n\leq \sfrac{\theta^{\star}_{n,l}n}{\varrho_{n,l}(0)|C^{\star}_{n,l}|}(1+o(1))
$.
\end{proposition}

This section is organized as Section \thv(9). 
In Subsection \thv(5.3) we introduce an absorbing Markov chain $J^{\star}_{n,l}$, defined as the restriction of $J_n$
to $C^{\star}_{n,l}$, killed on the boundary $\del C^{\star}_{n,l}$.
In Subsection \thv(5.1) we establish properties of the spectrum 
of an associated sub-Markovian transition matrix, $Q_{n,l}$,
needed in Subsection \thv(5.2) to prove bounds on the time to  absorption.
%All these results are established for a fixed realization of the random landscape, that is for fixed $\o\in\O$. 

\subsection{The `starred' absorbing chains}
 \label{5.3}
Let $C^{\star}_{n,l}$, $1\leq l\leq L^{\star}$, be the collection of connected components defined through \eqv(10.1.4).
To each component $C^{\star}_{n,l}$ we associate an absorbing Markov chain $J^{\star}_{n,l}$ with state space
$C^{\star}_{n,l}\cup\Delta$, where the absorbing point $\Delta$ represents the boundary $\del C^{\star}_{n,l}$;
its transition matrix $P^{\star}_{n,l}=\left(p^{\star}_{n,l}(x,y)\right)$ has entries
$p^{\star}_{n,l}: \{C^{\star}_{n,l}\cup\Delta\}\times \{C^{\star}_{n,l}\cup\Delta\} \rightarrow [0,1]$,
\be
p^{\star}_{n,l}(x,y)=
\begin{cases}
p_n(x,y) &\hbox{\rm if}\quad (x,y)\in G(C^{\star}_{n,l}),\\
1-\sum_{y'\in C^{\star}_{n,l}}p_n(x,y')
%\sum_{y\in \del C^{\star}_{n,l}}p_n(x,y)
,&\hbox{\rm if}\quad x\in C^{\star}_{n,l},y=\Delta,\\
1,&\hbox{\rm if}\quad x=y=\Delta,\\
0,&\hbox{\rm else}.\\
\end{cases}
\Eq(5.1.1)
\ee
%Thus the absorbing point $\Delta$ stands for the boundary $\del C^{\star}_{n,l}$,
Thus $J^{\star}_{n,l}$ can be viewed as the restriction of $J_n$
to $C^{\star}_{n,l}$, killed on the boundary $\del C^{\star}_{n,l}$.
We also call $Q_{n,l}=\left(q_{n,l}(x,y)\right)$ the sub-Markovian restriction of
$P^{\star}_{n,l}$ to $C^{\star}_{n,l}$, namely
$q_{n,l} : C^{\star}_{n,l}\times C^{\star}_{n,l} \rightarrow [0,1]$,
\be
q_{n,l}(x,y)=
\begin{cases}
p_n(x,y) &\hbox{\rm if}\quad (x,y)\in G(C^{\star}_{n,l}),\\
0,&\hbox{\rm else}.\\
\end{cases}
\Eq(5.1.2)
\ee
Then, $\L_n^\dagger(J_n^\dagger(j))$ in \eqv(5.prop1.00) is equal in distribution
to the absorption time
\be
\Eq(5.2.1)
%\L_n^\dagger(J_n^\dagger(j))\equiv
T^{\star}_{n,l}=\inf\{i\in\N \mid J^{\star}_{n,l}(i)=\Delta\}
\ee
 of the process $J^{\star}_{n,l}$ started in $J^{\star}_{n,l}(0)=J^\dagger_n(j)$.
%Note that
%this notation also is consistant with  \eqv(2.2.10), namely $T^{\star}_{n,l}$  
%in \eqv(2.2.10) is equal in distribution to $T^{\star}_{n,l}$ in \eqv(5.2.1).

\subsection{Spectrum of $Q_{n,l}$.}
 \label{5.1}

Clearly, 
%the sub-Markovian chain of transition matrix
$Q_{n,l}$ is reversible with respect to the 
%probability
 measure
$\pi^{\star}_{n,l}$,
\be
\Eq(5.1.3)
%m_{n,l}(x)
\pi^{\star}_{n,l}(x)
=\frac{\pi_n(x)}{\sum_{x'\in C^{\star}_{n,l}}\pi_n(x')}, \quad x\in C^{\star}_{n,l},
\ee
where $\pi_n$ denotes the invariant measure of $J_n$ (see \eqv(1.1.15)).
%where $\pi_n$ is defined in \eqv(1.1.15). 
From the properties of the sets $C^{\star}_{n,l}$ it follows that $\{\Delta\}$ is the unique recurrence class 
of $P^{\star}_{n,l}$ and that $Q_{n,l}$ is irreducible and periodic with period 2 (indeed
%as results from the fact that
$G(C^{\star}_{n,l})\subseteq\QQ_n$ is a connected sub-graph of the bipartite graph $\QQ_n$).
%$\QQ_n$ is bipartite and that the sub-graph $G(C^{\star}_{n,l})\subseteq\QQ_n$ is connected).
%,  any connected subgrap of a bipartite graph is bipartite, and ).
%(that is to say, in the matrix analysis terminology of Appendix \thv(B), $Q_{n,l}$ 
%Irreducible periodic matrices with period 2 
%setting of Appendix \thv(B.2)
Therefore, by \eqv(B2.1) and \eqv(B2.2), the eigenvalues of $Q_{n,l}$ satisfy
\be
\Eq(5.1.4)
1>\vartheta_{n,l}(0)> \vartheta_{n,l}(1)\geq \vartheta_{n,l}(2)\dots\geq \vartheta_{n,l}(N-1)>-1,
\ee
where $N\equiv N_{n,l}=|C^{\star}_{n,l}|$.
%and
%\be
%\Eq(5.1.5)
%\vartheta_{n,l}(k) = -\vartheta_{n,l}(N-k-1),\quad 0\leq k\leq N-1,
%\ee
%as follows from the fact that the graph  $G\left(C^{\star}_{n,l}\right)$ is bipartite
%\footnote{
%Note that $\QQ_n$ is bipartite, and that the sub-graph $G(C^{\star}_{n,l})\subseteq\QQ_n$ is connected.
%}.
%
%cf message from 28/03/2011
%
The next two lemmata give bounds on the first and second eigenvalues, respectively, in terms of the quantities
\eqv(5.lem1.0) and \eqv(5.lem1.0').

\begin{lemma}
  \TH(5.lem1)  
  If $|C^{\star}_{n,l}|=2$ then
\be
\vartheta_{n,l}(0)=1-[1+\varrho_{n,l}(0)/(n-1)]^{-1}.
\Eq(5.lem1bis.1)
\ee
Furthermore on $\O^{\star}$, for all but a finite number of indices $n$, for all $1\leq l\leq L^{\star}$,
\be
\Eq(5.lem1.1)
\displaystyle
1-\frac{ 1+\OO({1}/{\log n})}{1+2\varrho_{n,l}(0)/(n |C^{\star}_{n,l}|)}
\leq
\vartheta_{n,l}(0)
\leq
1-\frac{1-\OO({1}/{\log n})}{1+\varrho_{n,l}(0)|C^{\star}_{n,l}|/n}.
\ee
\end{lemma}

\begin{remark} Note that the bounds of Lemma
\thv(5.lem1) are rather sharp, i.e.~the prefactors of $\varrho_{n,l}(0)$  only differ through the terms
$|C^{\star}_{n,l}|^{\pm1}$, and $2\leq |C^{\star}_{n,l}|\leq cnst. n/\log n$. 
%When $|C^{\star}_{n,l}|=cnst$, these bounds  are very sharp. 
%The case  $|C^{\star}_{n,l}|=2$, that  plays a special role in the sequel, can be worked out by hand and yields:
%(of course the cases $|C^{\star}_{n,l}|=2$ and $3$ can be worked out by hand).
\end{remark}

\begin{lemma}
  \TH(5.lem2)
 On $\O^{\star}$, for all but a finite number of indices $n$, for all $1\leq l\leq L^{\star}$,
\be
\Eq(5.lem2.1)
\displaystyle
\vartheta_{n,l}(1)\leq 1
-\varrho_{n,l}(1)\left[n|C^{\star}_{n,l}|^3(1+\varrho_{n,l}(0)|C^{\star}_{n,l}|/n)\right]^{-1}.
%-{(n|C^{\star}_{n,l}|^3)}^{-1}\frac{\varrho_{n,l}(1)}{1+\varrho_{n,l}(0)|C^{\star}_{n,l}|/n}.
\ee
\end{lemma}

\begin{proof}[Proof of Lemma \thv(5.lem1) and Lemma \thv(5.lem2)] Bounds on the first two eigenvalues of $Q_{n,l}$
are established in exactly the same way as in the case of the matrix $R_{n,l}$ defined in \eqv(9.prop1.0')
(see the proof of Lemma \thv(9.lem1) and Lemma \thv(9.lem2)). The case  $|C^{\star}_{n,l}|=2$ is easily
%(that  plays a special role in the sequel) 
worked out by hand. We skip the details.
\end{proof}

%As an immediate consequence, we have:

%
% est-ce que je me sers de ce lemme quelque part???
%
%\begin{lemma}
%  \TH(5.lem3)
%   On $\O^{\star}$, for all but a finite number of indices $n$,
%$
%%\be
%%\Eq(5.lem3.1)
%\displaystyle
%\frac{1-\vartheta_{n,l}(0)}{1-\vartheta_{n,l}(1)}\leq \frac{n|C^{\star}_{n,l}|^4}{\varrho_{n,l}(1)}.
%%\ee
%$
%\end{lemma}

%\begin{proof}[Proof of Lemma \thv(5.lem3)] 
%This bound follows from the bound \eqv(5.lem1.3) on $\vartheta_{n,l}(0)$,
%the bound \eqv(5.lem2.8) on $\vartheta_{n,l}(1)$, the expression \eqv(5.lem1.6) 
%of $A$ and \eqv(10.lem1.5) of Lemma \thv(10.lem1).
%\end{proof}

\begin{corollary}
  \TH(5.lem4)
On $\O^{\star}$, for all but a finite number of indices $n$,
\be
\Eq(5.lem4.2)
{\vartheta_{n,l}(1)}/{\vartheta_{n,l}(0)}
%utiliser peut etre la borne ci-dessous quand $\varrho_{n,l}(1)=\varrho_{n,l}(0)$
% 1-|C^{\star}_{n,l}|^{-4}\frac{\varrho_{n,l}(1)}{\varrho_{n,l}(0)}(1+o(1)).
\leq
\exp\left\{-|C^{\star}_{n,l}|^{-4}({\varrho_{n,l}(1)}/{\varrho_{n,l}(0)})(1+o(1))\right\}.
\ee
\end{corollary}

%\begin{remark}
%  \label{rem5.4}
%When equality holds in $\varrho_{n,l}(0)\geq\varrho_{n,l}(1)$
%the bound of Lemma \thv(5.lem2) becomes $1-\vartheta_{n,l}(1)\geq 1/2|C^{\star}_{n,l}|^4$.
%Note that when $|C^{\star}_{n,l}|=2$ equality is achieved.
%%when the energy landscape $H_n$ on $C^{\star}_{n,l}$ takes at most two values and thus 
% It also is achieved when the graph
%$G(C^{\star}_{n,l})$ is star shaped with a global energy maximum at its center.
%%  Thus equality holds for all set $C^{\star}_{n,l}$ with  $|C^{\star}_{n,l}|=2$.
%\end{remark}

\begin{proof}
%[Proof of Corolloray \thv(5.lem4)]
By Lemmata \thv(5.lem1) and \thv(5.lem2), 
$
\vartheta_{n,l}(1)/\vartheta_{n,l}(0)\leq |C^{\star}_{n,l}|^{-4}({\varrho_{n,l}(1)}/{\varrho_{n,l}(0)})(1+o(1))
$,
and since $1-x\leq e^{-x}$ for $0<x<1$, \eqv(5.lem4.2) follows.
\end{proof}

%We close this subsection with a lemma that plays the same role as  Lemma \thv(9.lem4).

% an almost sure bound that will often be needed in the sequel. 

Finally, the lemma below collects useful bounds.

\begin{lemma}
  \TH(5.lem7)
For each $1\leq l\leq L^{\star}$ and for
$
r_n\bigl(\rho^{\star}_n\bigr)
$
%defined in \eqv(10.1.1) and bounded 
as in \eqv(9.lem4'.2),
\be
e^{\b n\sqrt{\log 2}(1+o(1))}\geq 
\varrho_{n,l}(0)
\geq \varrho_{n,l}(1)
\geq  r_n\left(\rho^{\star}_n\right),
\Eq(5.lem7.0)
\ee
where
%, using the notation of Lemma  \thv(9.lem7), 
the first inequality holds on $\O_0$,
for all but a finite number of indices $n$.
\end{lemma}

\begin{proof}
The first inequality is \eqv(9.lem7.1). The remaining ones are immediate.
\end{proof}

%%%%%%%%%%%%%%%%%%%%%%%%%%%%%%%%%%%%%%%%%%%%%%%%%%%%%%%%%%%%
\subsection{Absorption time: proof of Proposition \thv(5.prop1).}
%Time to quasi-stationarity and absorption time.}
% and quasi-stationary distribution.}
 \label{5.2}

%
% Idea : time to absorption in most realizations will be long compared to the time to quasi-stationarity
% cf Seneta p. 93
%
%(that is to say, in the matrix analysis terminology of Appendix \thv(B), $Q_{n,l}$ 
%Irreducible periodic matrices with period 2 
%setting of Appendix \thv(B.2)
%Our task is slightly complicated by the fact that $Q_{n,l}$ is periodic. 
This section and the next draw on Appendix \thv(B) where needed results on irreducible and periodic matrices are collected. 
%these matrices are gathered in Appendix \thv(B), to which we also refer for terminology.
%We will use the following notation.
We begin  with a little notation.
%Let us fix our notation.
Recall that ${\bf 1}$ denotes the vertex of $\VV_n$ whose coordinates are identically 1. Write
$\VV_n\equiv \VV_n^{-}\cup\VV_n^{+}$ where $\VV_n^{-}$ and $\VV_n^{+}$ are, respectively, the subsets 
of vertices that are at odd and even distance of the vertex ${\bf 1}$.
Every edge of the hypercube graph $\QQ_n$ connects  a vertex in $\VV_n^{-}$ to one in $\VV_n^{+}$
and so, every edge of the 
%(connected) 
graph $G(C^{\star}_{n,l})$ connects
a vertex in $C^{\star}_{n,l}\cap\VV_n^{-}$ to one in  $C^{\star}_{n,l}\cap\VV_n^{+}$.
%(these graphs are bipartite).
The matrices $Q_{n,l}$  can thus be written in the canonical form  \eqv(B2.0).
%Set $C^{\pm}_{n,l}\equiv C^{\star}_{n,l}\cap\VV_n^{\pm}$.
Denote by $u^{\star}_{n,l}$ and $v^{\star}_{n,l}$ the left and right Perron eigenvectors of $Q_{n,l}$ (see \eqv(B2.3)-\eqv(B2.4')), 
\bea
\Eq(5.2. PerrLeft)
&&u^{\star}_{n,l}Q_{n,l}=\vartheta_{n,l}(0)u^{\star}_{n,l},\quad  u^{\star}_{n,l}>0,
\\
&&\Eq(5.2. PerrRight)
Q_{n,l}v^{\star}_{n,l}=\vartheta_{n,l}(0)v^{\star}_{n,l},\quad v^{\star}_{n,l}>0,
\eea
normalized to make  
$
%\be
%\Eq(5.2.PerrNorm)
%\textstyle
\sum_{x\in C^{\pm}_{n,l}}2u^{\star}_{n,l}(x)=1
$
and
$
%\,\,\,\text{and} \,\,\,
\sum_{x\in C^{\pm}_{n,l}}2u^{\star}_{n,l}(x)v^{\star}_{n,l}(x)=1
%\ee
$,
where $C^{\pm}_{n,l}\equiv C^{\star}_{n,l}\cap\VV_n^{\pm}$.
%let us state the main results of this subsection. 
The next lemma is the key ingredient to the proof Proposition \thv(5.prop1)

\begin{lemma}
  \TH(5.lem3)
 %
% ici $|C^{\star}_{n,l}|^2$ suffit mais j'ai besoin du $|C^{\star}_{n,l}|^2$ pour le temps continu
%
%Let $\o\in\O_0\cap \O^{\star}$ be fixed.
%There exists a subset $\O_0\subseteq\O$ with $\P\left(\O_0\right)=1$ such that on $\O_0$,
On $\O_0\cap \O^{\star}$, for all but a finite number of indices $n$
we have, for each $x$ in $C^{\star}_{n,l}$ and  all $i\geq \theta^{\star}_{n,l}$,
\be
\Eq(5.lem3.3)
P_{x}\left(T^{\star}_{n,l}>i\right)
=v^{\star}_{n,l}(x)\vartheta_{n,l}^{i}(0)(1+\epsilon^{\star}_n),
%P_{x}\left(T^{\star}_{n,l}=i\right)
%=(1+\epsilon^{\star}_n)v^{\star}_{n,l}(x)(1-\vartheta_{n,l}(0))\vartheta_{n,l}^{i-1}(0).
\ee
where $|\epsilon^{\star}_n|\leq \exp(-\b n/4)$. 
%
%avec l'ancien \theta^{\star}_{n,l} on avait: $|\epsilon^{\star}_n|\leq \exp(-n^2)$. 
%
%
%Consequently,
%\be
%\Eq(5.lem3.2)
%P_{x}\left(T^{\star}_{n,l}=i\right)
%=(1+\epsilon^{\star}_n)v^{\star}_{n,l}(x)(1-\vartheta_{n,l}(0))\vartheta_{n,l}^{i-1}(0).
%\ee
\end{lemma}

%\begin{remark} The techniques introduced in \cite{BEGK} for the study of metastability allow to obtain precise estimates  on the  the left and right eigenvectors and on the distribution of $T^{\star}_{n,l}$. However this would require a detailed COMPLETER CETTE REMARQUE
%\end{remark}

%Before proving these results we state and prove rough bounds on $u^{\star}_{n,l}$ and $v^{\star}_{n,l}$ that will be %needed in this section and the next.
%It is not hard to see that  $u^{\star}_{n,l}$ and $v^{\star}_{n,l}$ obey the following a priori bounds.

The proof of this result relies on the following a priori lower bound on $v^{\star}_{n,l}$. 

\begin{lemma}
 \TH(5.lem5)
 
For all $x\in C^{\star}_{n,l}$, $v^{\star}_{n,l}(x)\geq (1/\varrho_{n,l}(0))^{|C^{\star}_{n,l}|}$. 
%On $\O_0\cap \O^{\star}$, for all but a finite number of indices $n$, for all $x\in C^{\star}_{n,l}$,
%\bea
%\Eq(5.lem5.2)
%v^{\star}_{n,l}(x)\geq (1/\varrho_{n,l}(0))^{|C^{\star}_{n,l}|}\geq e^{-\b n^2/\log n}.
%\eea
\end{lemma}

\begin{proof}  The proof is a simple adaptation of the proof of Lemma \thv(9.lem3). 
%The first inequality is a simple adaptation 
%of the proof of Lemma \thv(9.lem3). The deduce the second use that by Lemma \thv(5.lem7) and  \eqv(10.lem1.4),
%on $\O_0\cap \O^{\star}$, $\varrho_{n,l}(0)^{|C^{\star}_{n,l}|}\leq e^{\b n^2/\log n}$  
%for all but a finite number of indices $n$.
\end{proof}

\begin{proof}[Proof of Lemma  \thv(5.lem3)] 
By \eqv(B2.lem1.2) and \eqv(B2.lem1.4) of Lemma \thv(B2.lem1)  applied to $Q_{n,l}$ we have, 
for all $x\in C^{\star}_{n,l}$ and all $i\in\N$,
\be
\Eq(5.lem3.10)
P_{x}\left(T^{\star}_{n,l}>i\right)
=v^{\star}_{n,l}(x)\vartheta_{n,l}^{i}(0)\Bigl\{
1+\left(\sfrac{\vartheta_{n,l}(1)}{\vartheta_{n,l}(0)}\right)^{i}\wh R_n\Bigr\},
\ee
where
$
|\wh  R_n|
\leq 
\Bigl(\min_{x\in C^{\star}_{n,l}}\bigl\{v^{\star}_{n,l}(x)(\pi^{\star}_{n,l}(x))^{1/2}\bigr\}\Bigr)^{-1}
$.
The lemma will be proved if we can show that on $\O_0\cap \O^{\star}$, 
for all but a finite number of indices $n$,  
\be
\Eq(5.lem3.14)
%avec l'ancien \theta^{\star}_{n,l}} on avait :
%
%\left(\sfrac{\vartheta_{n,l}(1)}{\vartheta_{n,l}(0)}\right)^{\theta^{\star}_{n,l}}|\wh  R_n|\leq\exp(-n^2).
%
\left(\sfrac{\vartheta_{n,l}(1)}{\vartheta_{n,l}(0)}\right)^{\theta^{\star}_{n,l}}|\wh  R_n|\leq\exp(-\b n).
\ee
To this end note that, by \eqv(5.1.3), \eqv(1.1.15), and Lemma \thv(5.lem7), for all $x\in C^{\star}_{n,l}$,
$
\pi^{\star}_{n,l}(x)\geq r_n\left(\rho^{\star}_n\right)/n|C^{\star}_{n,l}|\varrho_{n,l}(0)
$.
Inserting this bound and that of Lemma \thv(5.lem5) in our bound on $|\wh  R_n|$, and using \eqv(10.lem1.4), Lemma \thv(5.lem7), and \eqv(9.lem4'.2)  we get that on $\O_0\cap \O^{\star}$, for all but a finite number of indices $n$, 
%\be
%\Eq(5.prop1.12)
%\textstyle
$
| \wh R_n|
%\leq (n/ \sqrt{r_n\left(\rho^{\star}_n\right)}) 
\leq e^{\b n(|C^{\star}_{n,l}|+1/2)\sqrt{\log 2}(1+o(1))}
$.
%\ee
Using \eqv(5.lem4.2) of Corollary \thv(5.lem4) to bound the ratio ${\vartheta_{n,l}(1)}/{\vartheta_{n,l}(0)}$
and choosing $\theta^{\star}_{n,l}$ as in \eqv(5.prop1.1) then yield \eqv(5.lem3.14).
The proof of the lemma is done.
\end{proof}

\begin{proof}[Proof of Proposition \thv(5.prop1)] When $|C^{\star}_{n,l}|=2$, $Q_{n,l}=\vartheta_{n,l}(0) I_2$
where $I_2$ denotes the identity matrix in $\R^2$. Thus
$
P_{x}\left(T^{\star}_{n,l}>i\right)
=\vartheta_{n,l}^{i}(0)
$
which by \eqv(5.lem1bis.1) leads to  \eqv(5.prop1.4).
%leading to
To prove \eqv(5.prop1.0') use that by \eqv(5.lem3.3) with $i=\theta^{\star}_{n,l}$,
\be
v^{\star}_{n,l}(x)
=P_{x}\left(T^{\star}_{n,l}>\theta^{\star}_{n,l}\right)\vartheta_{n,l}^{-\theta^{\star}_{n,l}}(0)(1+\epsilon^{\star}_n)^{-1}
\leq\vartheta_{n,l}^{-\theta^{\star}_{n,l}}(0)(1+\epsilon^{\star}_n)^{-1}.
\ee
Hence
$
P_{x}\left(T^{\star}_{n,l}> i\right)
\leq 
\vartheta_{n,l}^{-\theta^{\star}_{n,l}}(0)\vartheta_{n,l}^{i}(0)
%\leq  \vartheta_{n,l}^{i}(0)\left[1+\sfrac{\theta^{\star}_{n,l}n}{\varrho_{n,l}(0)|C^{\star}_{n,l}|}(1+o(1))\right]
$
where, by \eqv(5.lem1.1),
$
\vartheta_{n,l}^{i}(0)
%\leq \exp\{-i(n/\varrho_{n,l}(0)|C^{\star}_{n,l}|)(1-o(1))\}
%\leq e^{-i\frac{n}{\varrho_{n,l}(0)|C^{\star}_{n,l}|}(1-o(1))}
\leq e^{-i(n/\varrho_{n,l}(0)|C^{\star}_{n,l}|)(1-o(1))}
$
and
$
\vartheta_{n,l}^{-\theta^{\star}_{n,l}}(0)
\leq 1+\sfrac{\theta^{\star}_{n,l}n}{\varrho_{n,l}(0)|C^{\star}_{n,l}|}(1+o(1))
$.
\end{proof}

%%%%%%%%%%%%%%%%%%%%%%%%%%%%%%%%%%%%%%%%%%%%%%%%%%%%%%%%%%%%

%%%%%%%%%%%%%%%%%%%%%%%%%%%%%%%%%%%%%%%%%%%%%%%%%%%%%%%%%%%%

\section{Properties of the effective jump chain $J^\circ_n$}
 \label{4}

 Thi section gathers needed results on  the chain $J^\circ_n$. 
 The first proposition, which is central to the strategy of  Sections  \ref{6},  \ref{7}, and \ref{12},
 states that $J^\circ_n$ is fast mixing.
Define
\be
\ell^\circ_n=
\left\lceil
n^3
%\exp\left(2n\b\b_c\left(\sfrac{1}{m^{\star}}\right)\right)
\right\rceil.
\Eq(4.prop1.0)
\ee

\begin{proposition}
  \TH(4.prop1)
Assume that $c_{\star}>1+\log 4$. For all $\b>0$, there exists a subset $\O_1\subset\O$
with  $\P\left(\O_1\right)=1$ such that, on $\O_1$, for all but a finite number of
indices $n$, 
%the following holds
for all pairs $x\in\VV^\circ_n, y\in\VV^\circ_n$, and all $i\geq 0$,
\be
\left| 
P^\circ_{\pi^\circ_n}\left(J^\circ_n(i+\ell^\circ_n)=y,  J^\circ_n(i)=x \right)-\pi^\circ_n(x)\pi^\circ_n(y)
\right|
\leq \delta_n\pi^\circ_n(x)\pi^\circ_n(y)\,,
\Eq(4.prop1.1)
\ee
where
$
0\leq\delta_n\leq   2^{-n}
%e^{-n^2}
%2\exp\left\{-\frac{1}{2}n^2\left(1-\frac{3}{2n}\left(\b^2+\b_c^2(1)\right)\right)\right\}
$.
\end{proposition}

%Thus, observed every $\ell^\circ_n$ steps, the chain ressembles a collection of independent random variables drawn from the invariant measure $\pi^\circ_n$. 
%
Thus the random variables $J^\circ_n(\ell^\circ_ni)$, $i\in \N$, are close to independent and distributed according to the invariant distribution $\pi^\circ_n$.
%
%Thus the random variables $J^\circ_n(\ell^\circ_ni)$, $i\in \N$, ressemble a collection of independent random variables drawn from the invariant measure $\pi^\circ_n$. 
%%%%%%%%%%%%%%%%%%%%% end invariant measure %%%%%%%%%%%%%%%%%%%%%%%
%%Besides,
%%Proposition \thv(4.prop1),
%One still 
%% also 
%has to control the chain along stretches of trajectories of length $\ell^\circ_n$. The next proposition bounds certain mean local times that are needed to do 
%%precisely 
% this.
The next proposition provides bounds on certain mean local times that are needed to control stretches 
%%sections segments
 of trajectories of length $\ell^\circ_n$.
%%In the next propostion we bound various mean local times along stretches  of trajectories of length $\ell^\circ_n$.
Recall that $I^{\star}_n$ is the set of isolated vertices in the partition \eqv(10.1.4). 
%The following mean local times will be crucial ingredients of the proof of.... crucial in our strategy of proof
% we bound the mean local time in isolated vertices in the time interval $\{3,\dots, \ell^\circ_n+1\}$.
\begin{proposition}
  \TH(4.prop6)
Assume that $c_{\star}>3$.
% is such that $n^{c_{\star}}>\ell^\circ_n\log n$.
Then, 
there exists a subset $\O^{\scriptscriptstyle{\textsf{SRW}}}\subset\O$ with  
$\P\left(\O^{\scriptscriptstyle{\textsf{SRW}}}\right)=1$ such that, 
on $\O^{\scriptscriptstyle{\textsf{SRW}}}\cap \O^{\star}$, for all but a finite number of indices $n$, the following holds:
there exist constants $0<C_\circ, C'_\circ<\infty$ such that,
\item{(i)}  for all 
$
z\in I^{\star}_n
%=V_n(\rho^{\star}_n)\setminus\VV^\circ_{n}
$,
%$z\in\VV^\circ_{n}$
\be
\Eq(4.prop6.1)
\sum_{l=1}^{\ell^\circ_n-1}
P^\circ\left(J^\circ_n(l+2)=z \mid  J^\circ_n(0)=z \right)\leq \frac{C_\circ}{n\log n},
\ee
\item{(ii)} for all $1\leq l\leq L^{\star}$ 
and all $z,z'\in \del C^{\star}_{n,l}$,
\be
\Eq(4.prop6.1')
\sum_{l=1}^{\ell^\circ_n-1}
P^\circ\left(J^\circ_n(l)=z \mid  J^\circ_n(0)=z' \right)\leq
\frac{C'_\circ}{n}.
\ee
\end{proposition}

SRW in  $\O^{\scriptscriptstyle{\textsf{SRW}}}$ above stands for Symmetric Random Walk. The reason for this will become clear from the proof (see Lemma \thv(4.lem1)).
%The reason for the name of the subset $\O^{\scriptscriptstyle{\textsf{SRW}}}$, where SRW stands for symmetric Random Walk, will become clear from the proof (see Lemma \thv(4.lem1)).
%This estimate is not optimal (the constant $3/2$ should probably be 1). 
%Compared to the known almost sure bounds  \cite{VG} on the spectral gap $\t_n$ of the original jump chain $J_n$,
%%With ovious notation:
%\be
%\textstyle
%\leq \frac{1}{n}\log\t_n \leq \b\b_c\left({1}/{2}\right) + c\b\sqrt{\frac{\log n}{n}}+c'\frac{\log n}{n}\,,
%\Eq(4.prop3.1original)
%\ee
%we see however that $\frac{1}{n}\log\t^\circ_n\ll\frac{1}{n}\log\t_n$. We thus reached our goal of "making the fast chain %faster"
%% effective jumps chain mixing much faster".
%% This implies that the mixing time of the effectice jump chain is much shorter that that of the original jump chain.
%
% In view of
One may however already observe that the behavior of  $J^\circ_n$ in Proposition \thv(4.prop1) and Proposition \thv(4.prop6) is reminiscent of SRW (see e.g.~Section 3 of \cite{G10b})
and
as the next proposition shows, so is  that of its invariant measure.
% as the next proposition states,  $\pi^\circ_n$ is nothing but the uniform distribution on $\VV^\circ_{n}$.

\begin{proposition}
  \TH(4.prop2)
Assume that $c_{\star}>2$. For all $\b>0$,
\be
\pi^\circ_n(x)={1}/{|\VV^\circ_{n}|},\quad x\in\VV^\circ_{n},
%\pi^\circ_n(x)=
%\begin{cases}
%{1}/{|\VV^\circ_{n}|},
%&\hbox{\rm if}\, x\in\VV^\circ_{n},\\
%0,&\hbox{\rm else};
%\end{cases}
\Eq(4.prop2.1)
\ee
where, on $\O^{\star}$, for all but a finite number of indices $n$,
\be
\textstyle
|\VV^\circ_{n}|=2^n\left[1- n^{-2c_{\star}+1}(1+\OO(n^{-(c_{\star}-1)}))\right].
\Eq(4.prop2.0)
\ee
\end{proposition}

Let us immediately give the short proof of Proposition \thv(4.prop2).

\begin{proof}
%[Proof of Proposition \thv(4.prop2)] 
%It is well known (and easily checked) that $X_n$ has a unique reversible invariant measure given by the Gibbs %measure $G_{\b,n}$.
Because the process $X_n$ has a unique reversible invariant measure, $G_{\b,n}$,  
the jump chain also has unique reversible invariant measure, which is the measure
defined on $\VV_n$ by
\be
\pi_n(x)
=\frac{\l_n(x)G_{\b,n}(x)}{\sum_{x\in\VV_n}\l_n(x)G_{\b,n}(x)}
=
\frac
{
\sum_{y:(x,y)\in\EE_n}e^{-\b\max\left(H_n(y),H_n(x)\right)}
}{
\sum_{x\in\VV_n}\sum_{y:(x,y)\in\EE_n}e^{-\b\max\left(H_n(y),H_n(x)\right)}
}.
\Eq(1.1.15)
\ee
%where we used that
%$
%H_n(x)+\left[H_n(y)-H_n(x)\right]^+=\max\left(H_n(y),H_n(x)\right)
%$.
%Large Gibbs weights correspond to large negative $H_n(x)$'s.
By this and \eqv(2.1.6)
%and \eqv(1.1.15) 
$
%\be
%\Eq(4.3.1)
\pi^\circ_n(x)
%=\sfrac{\pi_n(x)}{\sum_{x'\in \VV^\circ_n}\pi_n(x')}
=
(nW^{\circ}_{\b,n})^{-1}
\textstyle 
\sum_{y:(x,y)\in\EE_n}e^{-\b\max\left(H_n(y),H_n(x)\right)}
%\ee
$, 
$x\in\VV^\circ_n$,
where
$
%\be
W^{\circ}_{\b,n}=n^{-1}\sum_{x\in\VV^{\circ}_n}\sum_{y:(x,y)\in\EE_n}e^{-\b\max\left(H_n(y),H_n(x)\right)}
%\Eq(4.2.1)
%\ee
$.
%But by \eqv(10.1.4) and  \eqv(2.1.1)
But by  \eqv(10.1.4) and the definition \eqv(2.1.1) of $\VV^\circ_n$,
$
\max\left(H_n(y),H_n(x)\right)=0
$
%for all $x\in\VV^\circ_n$.
whenever one of the two vertices
% of the edge 
$\{x,y\}$ lies in $\VV^\circ_n$. Hence
$
W^{\circ}_{\b,n}=|\VV^\circ_{n}|
$,
yielding \eqv(4.prop2.1). Since
$
|\VV^\circ_{n}|=2^n-\sum_{l=1}^{{L^{\star}}} \left|C^{\star}_{n,l}\right|
$,
\eqv(4.prop2.0) follows from \eqv(10.lem1.3).
\end{proof}

%Thus, $J^\circ_n$ gets close to stationarity in $\ell^\circ_n$ steps.  
Our last proposition contains a rough lower bound on hitting times at stationarity that is needed in the proof 
of Theorem \thv(1.theo3.Aging). 
%For $A\subseteq \VV^\circ_{n}$ write $T^\circ(A)\equiv\inf\{i\in\N \mid J^\circ_n(i)\in A\}$.
Write
\be
T^\circ(A)\equiv\inf\{i\in\N \mid J^\circ_n(i)\in A\},\quad A\subseteq \VV^\circ_{n}.
\ee
\begin{proposition}
  \TH(4.prop5)
%Under the assumptions and with the notation of Proposition \thv(4.prop1), 
Assume that $c_{\star}>1+\log 4$.
On $\O_1\cap \O^{\star}$, for all but a finite number of
indices $n$, we have that for all $A\subseteq \VV^\circ_{n}$ and  for $I^{\star}_n$ as in \eqv(10.1.4),
\be
P^\circ_{\pi^\circ_n}\left(T^\circ(A\cap I^{\star}_n)>t\right)
\geq
(1+o(1))
%\exp\bigl(-2t\sfrac{|A\cap I^{\star}_n|}{|\VV^\circ_{n}|}\bigr)
\exp\left({-2t|A\cap I^{\star}_n|/|\VV^\circ_{n}|}\right)
-\OO(\sfrac{1}{\log n}), \quad t>0.
\Eq(4.prop5.1)
\ee

\end{proposition}

The rest of this section is organized as follows.  The proof of Proposition \thv(4.prop1), 
given in Subsection \thv(4.5),  relies on a bound on the spectral gap of $J^\circ_n$ established 
in Subsection \thv(4.4), which itself relies on estimates on the transition probabilities 
of  $J^\circ_n$ that are established in Subsection \thv(4.0).
The proof of  Proposition \thv(4.prop6) and Proposition \thv(4.prop5)  are carried out in Subsection \thv(4.6)
and Subsection \thv(4.2), respectively

%The proof of Proposition \thv(4.prop2), which is elementary, is done in Subsection \thv(4.3). 
%In Subsection \thv(4.0) we work out estimates on the transition probabilities of  $J^\circ_n$.
%They are needed in Subsection \thv(4.4) to prove a bound on the spectral gap of $J^\circ_n$,
%which in turn is the key ingredient of the proof, carried out in 
%Subsection \thv(4.5), of Proposition \thv(4.prop1). 
%The proof of  Proposition \thv(4.prop6) is done in Subsection \thv(4.6).
%Before proving the more involved Proposition \thv(4.prop1) (Subsection \thv(4.5)) 
% we establish estimates on the transition probabilities of  $J^\circ_n$ (Subsection \thv(4.0)), we
%state and prove a bound on the spectral gap of $J^\circ_n$ (Subsection \thv(4.4)).

%We first establish estimates on the transition probabilities of  $J^\circ_n$
%(Subsection \thv(4.0)) and on the partition function of $\pi^\circ_n$ (Subsection \thv(4.2)).
%We then prove Proposition \thv(4.prop2)  (Subsection \thv(4.3)). 
%Before proving Proposition \thv(4.prop1) (Subsection \thv(4.5)) we
%state and prove a bound on the spectral gap of $J^\circ_n$ (Subsection \thv(4.4)).
%%%%Proposition \thv(4.prop2)  is proved in Subsection \thv(4.3).
%%%%The proof of Proposition \thv(4.prop1), given in Subsection \thv(4.5),  relies on a bound on the spectral gap 
%%%%of $J^\circ_n$, established in Subsection \thv(4.4).

%%%%%%%%%%%%%%%%%%%%% debut  Subsection 4.0 %%%%%%%%%%%%%%%%%%%%%%%%%%
\subsection{Estimates on the transition probabilities}
 \label{4.0}
%In this section we determine the first step transition probabilities 

We now examine the transition probabilities \eqv(2.1.5) of  $J^\circ_n$. 
%Recall that 
%$
%\VV^\circ_n=\VV_n\setminus\left(\cup_{1\leq l\leq L^{\star}}C^{\star}_{n,l}\right)
%$.
In what follows we denote by $G^\star (A)$  the complete graph on $A$.
% (with self-loops.)
Let $G^\circ(\VV^\circ_{n})$ be the graph with vertex set $\VV^\circ_{n}$ such that 
$(x,y)$ is an egde of the graph if and only if $p^\circ_n(x,y)>0$. In view of \eqv(2.1.4)-\eqv(2.1.5),
%$G^\circ(\VV^\circ_{n})$ can be decomposed as 
\be
G^\circ(\VV^\circ_{n})=G(\VV^\circ_{n})\bigcup_{1\leq l\leq L^{\star}}  G^\star(\del C^{\star}_{n,l}).
\Eq(4.0.1)
\ee

\begin{proposition}
  \TH(4.prop4)
For all $(x,y)\in G(\VV^\circ_{n})$,
\be
p^\circ_n(x,y)
%=p_n(x,y)
=1/n,
\Eq(4.prop4.2)
\ee
and, for all $1\leq l\leq L^{\star}$ and all $(x,y)$ in $G^\star(\del C^{\star}_{n,l})$, 
\be
p^\circ_n(x,y) 
=
\frac{m^{\star}_{n,l}(x)m^{\star}_{n,l}(y)}
{\sum_{z\in \del C^{\star}_{n,l}}m^{\star}_{n,l}(z)}(1+o(1)),
\Eq(4.prop4.1)
\ee
where $nm^{\star}_{n,l}(x)$ is the number of vertices of $C^{\star}_{n,l}$ that are are distance one from $x$,
\be
m^{\star}_{n,l}(x)\equiv
%\frac{1}{n|\del x\cap C^{\star}_{n,l}|=
n^{-1}|\left\{y\in C^{\star}_{n,l}\mid\dist(y,x)=1\right\}|,\,\,\, x\in\del C^{\star}_{n,l}.
\Eq(4.prop4.0)
\ee
\end{proposition}

%\begin{remark} One checks that for $(x,y)$ in $G^\star(\del C^{\star}_{n,l})$,
%\be
%\frac{1}{n}(n-m^{\star}_{n,l}(x))+\sum_{y\in \del C^{\star}_{n,l}}
%\frac{m^{\star}_{n,l}(x)m^{\star}_{n,l}(y)}
%{n\sum_{z\in \del C^{\star}_{n,l}}m^{\star}_{n,l}(z)}
%=1
%\ee
%$
%\sum_{y\in \del C^{\star}_{n,l}}p^\circ_n(x,y)
%=
%\frac{1}{n}|\del x\cap C^{\star}_{n,l}|(1+o(1))
%$
%and that
%$
%\frac{1}{n}|\del x\cap C^{\star}_{n,l}|
%=(1-\sum_{y\notin \del C^{\star}_{n,l}}p^\circ_n(x,y))
%$.
%\end{remark}

\begin{proof} 
Clearly,  if $(x,y)\in G(\VV^\circ_{n})$,
$
%\be
p^\circ_n(x,y)=p_n(x,y)=1/n
%n^{-1}
%\Eq(4.prop4.3)
%\ee
$,
yielding \eqv(4.prop4.2). We now turn to \eqv(4.prop4.1). 
Let us first state  two useful a piori relations
\bea
\Eq(4.prop4.4)
p^\circ_n(x,y)&=&p^\circ_n(y,x)\quad \forall (x,y)\in G^\circ(\VV^\circ_{n}),\\
\Eq(4.prop4.5)
m^{\star}_{n,l}(y)
&=&
\textstyle
%\sum_{y\in\del C^{\star}_{n,l} }p^\circ_n(x,y)=
\sum_{x\in\del C^{\star}_{n,l} }p^\circ_n(x,y)\quad \forall y\in\del C^{\star}_{n,l}.
\eea
Eq.~\eqv(4.prop4.4) is reversibility. Eq.~\eqv(4.prop4.5) follows from the relation 
$
%\sum_{y\in\del x} 
\sum_{y} 
p^\circ_n(x,y)=1
$, 
 \eqv(4.prop4.2),  \eqv(4.prop4.4), and the definition \eqv(4.prop4.0).
 
Given $A\subseteq \VV_n$ write
%denote  by $T(A)$ the hitting time
$
T(A)\equiv\inf\{i\in\N \mid J_n(i)\in A\}
$.
Also recall that for $1\leq l\leq L^{\star}$,
$
%\be
%\Eq(4.prop4.6)
T^{\star}_{n,l}\equiv\inf\{i\in\N \mid J_n(i)\in\del C^{\star}_{n,l}\}
%\ee
$.
Then, for all $(x,y)\in G^\star(\del C^{\star}_{n,l})$, 
\be
\textstyle
p^\circ_n(x,y)=\sum_{z\in C^{\star}_{n,l}}p_n(x,z) P_z\left(J_n(T^{\star}_{n,l})=y\right).
\Eq(4.prop4.7)
\ee
%The proof relies on Lemma \eqv(5.lem9).
The next lemma establishes that the exit distribution from $C^{\star}_{n,l}$ is independent from the entrance point, provided that the exit probability is not too small.

\begin{lemma}
 \TH(5.lem8)
For any two distinct vertices $z$ and $\bar z$ in $C^{\star}_{n,l}$ and any $y\in\del C^{\star}_{n,l}$,
 \be
\Eq(5.lem8.1)
P_z\left(J_n(T^{\star}_{n,l})=y\right)
=(1-\tilde\epsilon_n)P_{\bar z}\left(J_n(T^{\star}_{n,l})=y\right)+\tilde\epsilon_n,
\ee
where $\tilde\epsilon_n \leq {|\del C^{\star}_{n,l}|}/{\varrho_{n,l}(1)}$.
\end{lemma}

\begin{proof}[Proof of Lemma \thv(5.lem8)] 
Note that for any two vertices $z$ and $\bar z$ in $C^{\star}_{n,l}$,
\bea
\Eq(5.lem8.2)
P_z\left(T^{\star}_{n,l}\leq T(\bar z)\right)
%=P_z\left(T(\del C^{\star}_{n,l})<T(\bar z)\right)
&=&\textstyle\sum_{y\in \del C^{\star}_{n,l}}P_z\left(T(y)\leq T(\bar z\cup(\del C^{\star}_{n,l}))\right)\\
\Eq(5.lem8.2')
&=&\textstyle\sum_{y\in \del C^{\star}_{n,l}}\frac{\pi_n(y)}{\pi_n(z)}P_y\left(T(z)\leq T(\bar z\cup(\del C^{\star}_{n,l}))\right)\\
\Eq(5.lem8.2'')
&\leq&
{|\del C^{\star}_{n,l}|}{\varrho^{-1}_{n,l}(1)},
\eea
where the second equality is reversibility.
%, and the final bound 
Next decompose the event $\{J_n(T^{\star}_{n,l})=y\}$ according to whether 
$\{T(\bar z)\geq T^{\star}_{n,l}\}$ or $\{T(\bar z)< T^{\star}_{n,l}\}$: by the strong Markov property,
\be
P_z\left(T(\bar z)< T^{\star}_{n,l}, J_n(T^{\star}_{n,l})=y\right)
=P_z\left(T(\bar z)< T^{\star}_{n,l}\right)P_{\bar z}\left(J_n(T^{\star}_{n,l})=y\right),
\Eq(5.lem8.3)
\ee
whereas
$
P_z\left(T^{\star}_{n,l}\leq T(\bar z), J_n(T^{\star}_{n,l})=y\right)
\leq 
P_z\left(T^{\star}_{n,l}\leq T(\bar z)\right)
$.
%Thus \eqv(5.lem8.1) obtains.
Eq.~\eqv(5.lem8.1) now follows.
\end{proof}

%Let us now fix any
Now pick an arbitrary vertex $z^{\star}_{n,l}\in C^{\star}_{n,l}$ and denote by $\LL^{\star}_{n,l}$ the exit distribution
\be
\Eq(5.lem9.0)
\LL^{\star}_{n,l}(y)=
P_{z^{\star}_{n,l}}\left(J_n(T^{\star}_{n,l})=y\right),
\quad  y\in \del C^{\star}_{n,l}.
\ee

\begin{lemma}
 \TH(5.lem9)
For all $z\in C^{\star}_{n,l}$ and $y\in \del C^{\star}_{n,l}$
\be
P_z\left(J_n(T^{\star}_{n,l})=y\right)=(1+o(1))\LL^{\star}_{n,l}(y).
\Eq(5.lem9.1)
\ee
\end{lemma}

\begin{proof}[Proof of Lemma \thv(5.lem9)]  We readily deduce from Lemma \thv(5.lem8) that
%For all $z\in C^{\star}_{n,l}$, and $y\in \del C^{\star}_{n,l}$
if $y\in\del C^{\star}_{n,l}$  is such that
$
\LL^{\star}_{n,l}(y)\geq n\tilde\epsilon_n
$,
then
$
%\be
P_z\left(J_n(T^{\star}_{n,l})=y\right)=(1+o(1))\LL^{\star}_{n,l}(y)
%\Eq(5.lem9.2)
%\ee
$,
otherwise
\be
\textstyle
P_z\left(J_n(T^{\star}_{n,l})=y\right)<(n+1)\tilde\epsilon_n.
\Eq(5.lem9.3)
\ee
Let us prove by contradiction that 
$
\LL^{\star}_{n,l}(y)\geq n\tilde\epsilon_n
$
for all $y\in\del C^{\star}_{n,l}$.
Assume that there exists $y\in\del C^{\star}_{n,l}$ such that
$
\LL^{\star}_{n,l}(y)< n\tilde\epsilon_n
$.
Then, by \eqv(5.lem9.3) and \eqv(4.prop4.7),
\be
\textstyle
p^\circ_n(x,y)
\leq (n+1)\tilde\epsilon_n
\sum_{z\in C^{\star}_{n,l}}p_n(x,z)
= (n+1)\tilde\epsilon_n m^{\star}_{n,l}(x).
\Eq(5.lem9.4)
\ee
Summing both sides over $x\in\del C^{\star}_{n,l}$,
\be
\textstyle
\sum_{x\in\del C^{\star}_{n,l} }p^\circ_n(x,y)
\leq
(n+1)\tilde\epsilon_n \sum_{x\in\del C^{\star}_{n,l} }m^{\star}_{n,l}(x)
\leq
%(n+1)\tilde\epsilon_n|C^{\star}_{n,l}|^2
n^5\varrho^{-1}_{n,l}(1)
\ll n^{-1}.
\Eq(5.lem9.5)
\ee
However, by \eqv(4.prop4.5),
$
\sum_{x\in\del C^{\star}_{n,l} }p^\circ_n(x,y)=m^{\star}_{n,l}(x)\geq n^{-1},
$
which is a contradiction.
\end{proof}

We are now 
ready
to conclude the proof of \thv(5.lem8). By  \eqv(4.prop4.7) and \eqv(5.lem9.1),
\be
p^\circ_n(x,y)=m^{\star}_{n,l}(x)\LL^{\star}_{n,l}(y)(1+o(1)).
\Eq(4.prop4.8)
\ee
Inserting this in \eqv(4.prop4.4) and summing both sides over $x\in\del C^{\star}_{n,l}$ we get
\be
\textstyle
\LL^{\star}_{n,l}(y)=\frac{m^{\star}_{n,l}(y)}{\sum_{x\in\del C^{\star}_{n,l}}m^{\star}_{n,l}(x)}(1+o(1)),
\Eq(4.prop4.9)
\ee
and inserting this in turn in \eqv(4.prop4.8) yields \eqv(4.prop4.1). The proof of the proposition is done.
\end{proof}
%%%%%%%%%%%%%%%%%%%%% end  Subsection 4.0 %%%%%%%%%%%%%%%%%%%%%%%%%%

%%%%%%%%%%%%%%%%%%%%% beginning of  Subsection 4.4 %%%%%%%%%%%%%%%%%%%%%

\subsection
%{Bounds on the range of the spectrum of $J^\circ_n$.}
{Bound on the spectral gap of $J^\circ_n$.}
% proof of Proposition \thv(4.prop3).}
 \label{4.4}
 
Let
\be
\Eq(4.prop3.0)
1=\vartheta^\circ_n(0)> \vartheta^\circ_n(1)\geq \vartheta^\circ_n\dots\geq \vartheta_{n,l}(\left|\VV^\circ_n\right|-1)>-1
\ee
denote the eigenvalues of the matrix with entries \eqv(2.1.5) and set $\t^\circ_n=1/(1-\vartheta^\circ_n(1))$.
%(because $J^\circ_n$ is aperiodic, the rightmost inequality is strict).
%Let $\t^\circ_n$ and $\iota^\circ_n$ be defines respectively throught
%$
%\vartheta^\circ_n(1)=1-1/\t^\circ_n
%$
%and
%$
%\vartheta_{n,l}(\left|\VV^\circ_n\right|-1)=-1+1/\iota^\circ_n
%$.
The proof of Proposition \thv(4.prop1) relies on the following lower bound on the spectral gap of $J^\circ_n$.

\begin{proposition}
  \TH(4.prop3)
  Assume that $c_{\star}>1+\log 4$.
For all $\b>0$, there exists a subset $\O_2\subset\O$ with  $\P\left(\O_2\right)=1$ such that, 
on $\O_2$, for all but a finite number of indices $n$, 
%For all $\b>0$, on $\O_1$, for all but a finite number of indices $n$, 
\be
%\textstyle
%\max(\t^\circ_n, \iota^\circ_n)
\t^\circ_n\leq  \sfrac{1}{2}n^2(1+o(1)).
\Eq(4.prop3.1)
\ee
\end{proposition}

% shorter alternative
%If $c_{\star}>1+\log 4$ then for all $\b>0$, there exists a subset $\O_2\subset\O$ 
%with  $\P\left(\O_2\right)=1$ such that on $\O_2$, for all but a finite number of indices $n$, 
%$\t^\circ_n\leq  \frac{1}{2}n^2(1+o(1))$.

\begin{proof} 
The proof of Proposition \thv(4.prop3)  is based on a bound on spectral gaps already stated in  Subsection  \thv(9.1) which now reads as follows:
%, in the present notation,
 if $\G_n^\circ=\{\g^\circ_{x,y}\}$ is a set of paths in the graph 
$G^\circ(\VV^\circ_{n})$ such as described in the paragraph above \eqv(9.lem2.5),
then
\be
\Eq(4.5.1)
\textstyle
\t^\circ_n\leq \max_{e}\rho^{-1}_n(e)\sum_{\g^\circ_{x,y}\ni e}\left|\g^\circ_{x,y}\right|\pi^\circ_n(x)\pi^\circ_n(y),
\ee
where the max is over all edges $e=\{x',y'\}$ of $G^\circ(\VV^\circ_{n})$,
$\rho_n(e)\equiv\pi^\circ_{n,l}(x')p^\circ_n(x',y')$, and the summation is over all paths $\g^\circ_{x,y}$ in $\G_n^\circ$
that pass through $e$. 
The structure \eqv(4.0.1) of the graph $G^\circ(\VV^\circ_{n})$ naturally prompts us to write
\be
\G_n^\circ=\G_n\cup\left(\cup_{1\leq l\leq L^{\star}}  \G^\star_{n,l}\right),
\Eq(4.5.2)
\ee
where $\G_n$ is a set of paths in $G(\VV^\circ_{n})$ and,  for each $1\leq l\leq L^{\star}$,
\be
\G^\star_{n,l}=\left\{\g_{x,y}\equiv(x,y),\,  x,y \in\del C^{\star}_{n,l}\right\}.
\Eq(4.5.3)
\ee
That is, the paths of $\G^\star_{n,l}$ simply are the edges of $ G^\star(\del C^{\star}_{n,l})$.
%the paths $\g_{x,y}$ of $\G_n^\star_{n,l}$ simply are the edges $(x,y)$ of $ G^\star(\del C^{\star}_{n,l})$; specifically
Eq. \eqv(4.5.1) then becomes
\be
\Eq(4.5.4)
\textstyle
\t^\circ_n\leq \max\left\{\t_n, \max_{1\leq l\leq L^{\star}}\t^{\star}_{n,l}\right\},
\ee
where
%\be
%\Eq(4.lem4.5)
$
\t^{\star}_{n,l}\equiv \max_{e=(x,y), x,y \in\del C^{\star}_{n,l}}\pi^\circ_n(y)/p^\circ_n(x,y)
$,
and
%\ee
\be
\Eq(4.5.6)
\textstyle
\t_n\equiv \max_{e}
\rho^{-1}_n(e)\sum_{\g_{x,y}\ni e}\left|\g_{x,y}\right|\pi^\circ_n(x)\pi^\circ_n(y),
\ee
the max now being over all edges $e$ of $G(\VV^\circ_{n})$, and the summation over all paths $\g_{x,y}$ in $\G_n$
that pass through $e$. 
%The terms $\t^{\star}_{n,l}$ are readily bounded above using Proposition \thv(4.prop0) (see Lemma \thv(...) below).
Using Proposition \thv(4.prop2) and the fact that 
$
\max\left(H_n(z),H_n(z')\right)=0
$
whenever $(z,z')\in\EE_n$ is an edge with at least one endpoint in $\VV^\circ_n$,
\eqv(4.5.6) reduces to
\be
\textstyle
\t_n=({n}/{|\VV^\circ_{n}|}) \max_{e=(x,x')\in G(\VV^\circ_{n})}
\sum_{\g_{y,y'}\ni e}\left|\g_{y,y'}\right|.
\Eq(4'.prop2.2'')
\ee

The quality of the bound \eqv(4'.prop2.2'') now depends on making a judicious choice of the set of paths $\G_n$.
We will adopt a very clever choice made in \cite{FIKP}.
%briefly repeat its construction.

\smallskip
\noindent{\textbf{\emph{$\bullet$ A choice of $\G_n$.}}} The set
$\G_n$ is defined as 
\be
\G_n=\bigl\{\g_{x,y}\in\G'_n,\, x,y\in\VV^\circ_n\bigr\},
\Eq(4.5.7)
\ee
%\g_{x,y}^{int}\in\VV^\circ_n
where $\G'_n$ is a subset of paths in $G(\VV_{n})$ constructed as follows. 
%We begin by constructing $n$ families of paths,
%$\G_n^i$, $1\leq i\leq n$, as follows. 
Given $i\in\{1,\dots\,n\}$, and given two vertices
$x$ and $x'\in\VV_n$ such that $x_i\neq x'_i$, let $\g_{x,x'}^i$ be the path
%(going from $x$ to $x'$)
obtained by going left to right cyclically from $x$ to $x'$,
successively flipping the disagreeing coordinates, starting from the $i$-th coordinate.
Set
%\be
$
\G_n^i=\left\{\g_{x,x'}^i, x,x'\in\VV_n\right\}
$,
$1\leq i\leq n$.
%\Eq(4'.prop2.3)
%\ee
These paths are ordered in an obvious way.
Given $x,x'$ and $\g_{x,x'}$, let $\overline\g_{x,x'}$ be the set of vertices visited
by the path $\g_{x,x'}$, and let $\g_{x,x'}^{int}=\overline\g_{x,x'}\setminus\{x,x'\}$
be the subset of ``interior'' vertices.
% We say that a family of paths $\g_1,\dots,\g_k$
%is {\it interior-disjoint} if $\g_i^{int}\cap\g_j^{int}=\emptyset$ for all $1\leq i\neq j\leq k$.
We next split the set of vertices $\VV_n$ into {\it good} ones and {\it bad} ones. Recalling \eqv(10.1.4),
we  say that a vertex  is good if it belongs to $N^{\star}_n$; otherwise it is bad.
%Namely the set of good vertices is the set of non occupied vertices of the graph $G(V_n(\rho^{\star\star}_n))$.)
We say that a path $\g$ is good if all its interior points $\g^{int}$ are good,
and that a set of paths is good if all its elements are good.

The (random) set of path $\G'_n$ is then constructed as follows:

\item{(i)} Consider pairs $x$ and $x'$ such that $\dist(x,x')\geq n/\log n$.
If $\{\g_{x,x'}^i,  1\leq i\leq n\}$ 
%$\left\{\g_{x,x'}^i,  1\leq i\leq n\right\}$ 
contains a good path,
choose the first such for $\G'_n$; otherwise choose $\g_{x,x'}^1$.

\item{(ii)} Consider pairs $x$ and $x'$ such that $\dist(x,x')< n/\log n$. If there is a good vertex
$x''\in\VV_n$ such that  $\dist(x,x'')\geq n/\log n$ and $\dist(x'',x')\geq n/\log n$, and
if there are good paths, one in $\left\{\g_{x,x''}^i,  1\leq i\leq n\right\}$ and
one in $\left\{\g_{x'',x'}^i,  1\leq i\leq n\right\}$, such that the union of these two good paths is
a self avoiding path of length less than $n$, select this union as the path connecting $x$ to $x'$
in $\G'_n$ (notice that this is a good path); otherwise choose $\g_{x,x'}^1$.

The key point of this construction is that $\G'_n$ is almost surely good.
%\remark The possible choices of $c_e$ are determined by arguments made in the proof of Proposition \thv(4.prop3).
More precisely,  set
$
%\O^{\scriptstyle{\textsf{good}}}_n
%\O^{\scriptscriptstyle{\textsf{GOOD}}}_n
%\O^{\scriptscriptstyle{\textsf{GOOD}}}_n
\O^{\scriptscriptstyle{\textsf{GOOD}}}_n=\{\o\in\O \,\big|\,\G'_n\equiv \G'_n(\o) \,\hbox{\rm is good}\,\}
%\Eq(4'.prop2.5)
$,
$n\geq 1$,
and 
$
\O^{\scriptscriptstyle{\textsf{GOOD}}}=\liminf_{n\rightarrow \infty}\O^{\scriptscriptstyle{\textsf{GOOD}}}_n
$.
%\be
%\O^{\scriptscriptstyle{\textsf{GOOD}}}=\left\{\o\in\O \,\big|\,\exists n_0 : \forall n>n_0 \,,\,\o\in
%\O^{\scriptscriptstyle{\textsf{GOOD}}}_n
%\right\}.
%\Eq(4'.prop2.6)
%\ee

\begin{proposition}[Proposition 4.1 of \cite{FIKP}]
  \TH(4'.prop3)
 If $c_{\star}>1+\log 4$ then $\P\bigl(\O^{\scriptscriptstyle{\textsf{GOOD}}}\bigr)=1$.
\end{proposition}

%Consider now the set $\G_n\subseteq\G'_n$ defined in 

%Let us show that on $\O^{\scriptscriptstyle{\textsf{GOOD}}}$ the paths of $\G_n$ only visit vertices in
%$\VV^\circ_n\equiv\VV_n\setminus\left(\cup_{1\leq l\leq L^{\star}}C^{\star}_{n,l}\right)$.
%First observe that bad vertices in a good path can appear only at the ends of the path.
%Namely, if $e$ contains a bad vertex, say $x$, and a path $\g\in\G_n$ contains $e$,
%then $x\in\overline\g\setminus\g^{int}$. 
%This implies that good paths do not traverse edges of the connected components 
%$C^{\star\star}_{n,l}$, $1\leq l\leq L^{\star\star}$, of the graph $G(V_n(\rho^{\star\star}_n))$.
%Now  by assumption $c_{\star}>c_{\star\star}$. Hence
%$
%\cup_{l=1}^{L^{\star}} C^{\star}_{n,l}\subseteq \cup_{l=1}^{L^{\star\star}} C^{\star\star}_{n,l}
%$,
%implying in turn that the paths of $\G_n$ do not traverse edges 
%of the connected components $C^{\star}_{n,l}$, $1\leq l\leq L^{\star}$, 
%of the graph $G(V_n(\rho^{\star}_n))$. Thus, in a good path, 
%vertices that belong to $\cup_{l=1}^{L^{\star}} C^{\star}_{n,l}$
%can appear only at the ends of the path.
%%do not belong to $\VV^\circ_n$ 
%But by definition such paths are excluded from $\G_n$. This proves our claim.

Going back to \eqv(4.5.7), we see that
% it only remains to see that 
the set
$
\G_n
%\subseteq\G'_n
$ is obtained from 
$\G'_n$ by removing  the paths whose endpoints lie in $\cup_{1\leq l\leq L^{\star}}C^{\star}_{n,l}$.
Thus  on $\O^{\scriptscriptstyle{\textsf{GOOD}}}$ the paths of $\G_n$ only visit vertices in
$
\VV^\circ_n
%\equiv\VV_n\setminus\left(\cup_{1\leq l\leq L^{\star}}C^{\star}_{n,l}\right)
$. 
This finishes our construction.
Note that the paths constructed in this way have length smaller than $n$. Thus \eqv(4'.prop2.2'') yields
%Note that the paths constructed in this way have length smaller than $n$. Thus 
\be
\textstyle
\t_n
%\leq 
%\frac{n^2}{|\VV^\circ_{n}|} \max_{e=(x,x')\in G(\VV^\circ_{n})}\sum_{\g_{y,y'}\ni e}1
%=
\leq
({n^2}/{|\VV^\circ_{n}|}) \max_{e\in G(\VV^\circ_{n})}|\{\g\in\G \mid  e\in \g\}|
%\sum_{\g_{y,y'}\ni e}1
\,.
\Eq(4'.prop2.4)
\ee
%\smallskip
\noindent{\textbf{\emph{$\bullet$ Bound on $\t_n$.}}} 
From now on we assume that  $\o\in \O^{\scriptscriptstyle{\textsf{GOOD}}}$ so that, for all large enough $n$, $\G_n\equiv \G_n(\o)$ is good.
In that case a bad vertex can appear only at the ends of any path.
%As in (4.10) of \cite{FIKP}
Let us write
\be
\t_{n}=({n^2}/{|\VV^\circ_{n}|})(\t_n^1+\t_n^2)\,,
\Eq(4'.prop2.11)
\ee
where $\t_n^1$, repectively $\t_n^2$, is obtained by restricting
the sum in \eqv(4'.prop2.4) to paths connecting vertices at distance $n/\log n$ or more apart, repectively, 
less than $n/\log n$ apart.

On the one hand it is well known that (see e.g. Example 2.2, p.~45 in \cite{DS})
\be
\t_n^1\leq 2^{n-1}.
\Eq(4.5.8)
\ee
On the other hand, arguing as in \cite{FIKP} (see Subsection 4.2.2, page 934) that the sum in $\t_n^2$
is over a set of paths that connect vertices in a hypercube of dimension at most $n/\log n$
around $e$, we have
\be
\t_n^2\leq 2^{2n/\log n}.
\Eq(4'.lem5.10)
\ee
Plugging \eqv(4.5.8) and \eqv(4'.lem5.10) in \eqv(4'.prop2.11), and using \eqv(4.prop2.0) 
of Proposition \thv(4.prop2) to bound $|\VV^\circ_{n}|$, we 
%finally 
we get that on $\O^{\scriptscriptstyle{\textsf{GOOD}}}\cap \O^{\star}$, 
for large enough $n$,
\be
\t_n\leq  n^22^{-n}\left[1- n^{-2c_{\star}+1}(1+o(1))\right]^{-1}
\left(
2^{n-1}+ 2^{2n/\log n}
\right)\leq ({n^2}/{2})(1+o(1)).
\Eq(4.5.10)
\ee

\smallskip
\noindent{\textbf{\emph{$\bullet$ Bound on $\t^{\star}_{n,l}, 1\leq l\leq L^{\star}$.}}} Consider the
terms $\t^{\star}_{n,l}$ from \eqv(4.5.4).
By Proposition \thv(4.prop2)  and \eqv(4.prop4.1) of Proposition \thv(4.prop4), on $\O_0\cap \O^{\star}$, for large enough $n$
we have, for all $e=(x,y), x,y \in\del C^{\star}_{n,l}$ and all $1\leq l\leq L^{\star}$,
\be
\Eq(4.5.12)
\frac{\pi^\circ_n(y)}{p^\circ_n(x,y)}
\leq
\frac{1}{|\VV^\circ_{n}|}
\frac
{n\sum_{z\in \del C^{\star}_{n,l}}m^{\star}_{n,l}(z)}
{m^{\star}_{n,l}(x)m^{\star}_{n,l}(y)}(1+o(1))
%\frac{n\sum_{z\in \del C^{\star}_{n,l}}|\del z\cap C^{\star}_{n,l}|}
%{|\del x\cap C^{\star}_{n,l}||\del y\cap C^{\star}_{n,l}|}
\leq
\frac{n^2|C^{\star}_{n,l}|^2}{|\VV^\circ_{n}|}
\leq 2n^42^{-n}.
\ee
Thus, on $\O_0\cap \O^{\star}$, for large enough $n$,
\be
\Eq(4.5.13)
\textstyle
\max_{1\leq l\leq L^{\star}}\t^{\star}_{n,l} \leq 2n^42^{-n}.
\ee

Setting $\O_2=\O^{\scriptscriptstyle{\textsf{GOOD}}}\cap\O_0\cap \O^{\star}$ and plugging 
\eqv(4.5.13) and \eqv(4.5.10) in \eqv(4.5.4) finally 
yields the upper bound \eqv(4.prop3.1) on $\t^\circ_n$. The proof of Proposition \thv(4.prop3) is done.
\end{proof}

%%%%%%%%%%%%%%%%%%%%% end of  Subsection 4.4 %%%%%%%%%%%%%%%%%%%%%

\subsection{Proof of Proposition \thv(4.prop1).}
 \label{4.5}
% Note that the chain $J^\circ_n$ no longer is periodic, and is irreducible.
Consider
%Consider the continuous time Markov chain $(J^{\circ}_{n}(t), t>0)$ obtained by attaching
%an exponential waiting time with rate one to each step of the discrete chain $(J^{\circ}_n(k), k\in\N)$.
%
the continuous time Markov chain $(J^{*}_n(t), t>0)$ with
% $(J^{*}_n(t), t>0)$ with
jump chain $(J^{\circ}_n(k), k\in\N)$ 
%and rate one  exponential waiting times.
and rate one exponential waiting times. 
That is,
given a family $(e^{*}_{n,i}\,, i\in\N)$ of independent mean one exponential r.v.'s,
independent of $J^{\circ}_n$, 
%let $J^{*}_n$ be defined through
\be
J^{*}_n(t)=J^{\circ}_n(i)\,\,\, \hbox{\rm if}\,\,\, s_n(i)\leq t<s_n(i+1) \,\,\,
\hbox{\rm for some}\,\, i,
\Eq(4.4.2)%\Eq(4.3)
\ee
where
$
%\be
\textstyle
s_n(k)\equiv\sum_{i=0}^{k-1}e^{*}_{n,i}
%\,,\quad k\in \N\,.
%\Eq(4.4.3)%\Eq(4.4)
%\ee
$,
$k\in \N$.
%Its transition kernel is denoted by
%\be
%p_{n,t}^{*}(x,y)=P^{*}_{x}\left(J^{*}_n(t)=y\right)\,,\quad x,y\in\VV^\circ_n\,.
%\Eq(4.4.4)
%\ee
Write $P^{*}_{x}$ for the law of $J^{*}_n$ started in $x$. 
We know that:
%By Proposition 3 of  \cite{DS} 
%Diaconis and Stroock
% \cite{DS}  and Proposition \thv(4.prop3) we have:

%will allow us to bound
%the total variation distance to stationarity of $p_{n,t}^{*}(x,\cdot)$.
%Recall that the total variation distance between two probabilities $\mu$ and $\mu'$ on $\VV^\circ_n$ is defined as
%$$
%\|\mu-\mu'\|_{TV}=\max_{A\subset\VV^\circ_n} |\mu(A)-\mu'(A)|=\frac{1}{2}\sum_{x\in\VV^\circ_n}|\mu(x)-\mu'(x)|\,.
%$$

\begin{lemma}[Proposition 3 of \cite{DS}] 
  \TH(4.lem4)
%\lemma{\TH(4.lem1)}{\it  
% Assume that $c_{\star}>1+\log 4$. Then on $\O_2$, for all but a finite number of indices $n$, the following holds.
For all $x\in\VV^\circ_n$ and all $t\in(0,\infty)$,
%
%(for each fixed $\o$ and each $n$)
%
\be
\bigl\|
P^{*}_{x}\left(J^{*}_n(t)=\cdot\right)
%p_{n,t}^{*}(x,\cdot)
-\pi^\circ_n(\cdot)\bigr\|_{TV}\leq
%\sqrt{\sfrac{1-\pi^\circ_n(x)}{4\pi^\circ_n(x)}}
\sqrt{(1-\pi^\circ_n(x))/(4\pi^\circ_n(x))}
e^{-t/\t^{\circ}_n}\,.
\Eq(4.lem4.1)%\Eq(4.lem1.1)
\ee
\end{lemma}
%A similar statement holds for the discrete time chain that, however, bounds the rate of convergence in terms not only 
%of $\t^{\circ}_n$ but of the quantity 
%$\max(\vartheta^\circ_n(1),| \vartheta^\circ_{n,l}(\left|\VV^\circ_n\right|-1)|)$
%(see \eqv(4.prop3.0) for the notation). Rather than seeking an estimate on the smallest eigenvalue 
%(which we could do using e.g.~the bound from Proposition 2 of \cite{DS})
%%%%%%but looks difficult and hard to improve the bound I actuallay have
% our proof relies on 
%%%%%%we will bound the rate of convergence of $J^{\circ}_n$ by comparing the
%comparing the discrete and continuous time transition kernels. Set
We now aim to compare the discrete and continuous time transition kernels. Set
%The proof consists in comparing the discrete and continuous time transition kernels. Set
\be
\II(\zeta)=\zeta-\log(1+\zeta),\quad \zeta>0.
\Eq(4.lem5.0)
\ee

\begin{lemma}
  \TH(4.lem5)
%\lemma{\TH(4.lem2)}{\it
Let $0<m_n\uparrow\infty$ be an integer valued sequence.
For all $\zeta>0$ we have
% and all  $x,y\in\VV^\circ_n$, we have
\be
\nonumber
\textstyle
\left|
%\sum_{k=0}^1P^\circ_{x}\left(J^\circ_n(m_n+k)=y\right)
P^\circ_{x}\left(J^\circ_n(m_n)=y\right)
-\pi^\circ_n(y)
\right|
\leq 
%\sqrt{\sfrac{1-\pi^\circ_n(x)}{4\pi^\circ_n(x)}}
(4\pi^\circ_n(x))^{-1/2}
e^{-m_n(1-\zeta)/\t^{\circ}_n}
+2e^{-m_n\II(\zeta)}, \quad\forall x,y\in\VV^\circ_n.
%\Eq(4.lem5.1)%\Eq(4.lem2.1)
\ee
\end{lemma}

\proof By \eqv(4.4.2),
$
P^\circ_{x}\left(J^{\circ}_n(m_n)=y\right)
=
P^{*}_{x}\left(J^{*}_n(s_n(m_n))=y\right)
$. 
Now consider the event $A_{\zeta,n}=\left\{\left|s_n(m_n)-m_n\right|\geq \zeta m_n\right\}$, $\zeta>0$.
A classical large deviation estimates yields
\be
P^{*}_{x}\left(A_{\zeta,n}\right) \leq 2 e^{-m_n\II(\zeta)}.
\Eq(4.lem5.6+)
%\leq 2e^{-\frac{1}{2}\d^2(1+\OO(\d))}\,.
\ee
Hence
\be
0\leq P^{*}_{x}\left(J^{*}_n(s_n(m_n))=y, A_{\zeta,n}\right)\leq 2 e^{-m_n\II(\zeta)}.
\Eq(4.lem5.6)
\ee
Writing $A^c_{\zeta,n}=\left\{\left|s_n(m_n)-m_n\right|< \zeta m_n\right\}$,
\bea
\nonumber
P^{*}_{x}\left(J^{*}_n(s_n(m_n))=y, A^c_{\zeta,n}\right)
&\leq&
\sup_{|t-m_n|< \zeta m_n}P^{*}_{x}\left(J^{*}_n(t)=y, A^c_{\zeta,n}\right)
\\ \nonumber
&\leq&
\sup_{|t-m_n|< \zeta m_n}P^{*}_{x}\left(J^{*}_n(t)=y\right)
\\  
&\leq&
\pi^\circ_n(y)
+
%\sqrt{\sfrac{1-\pi^\circ_n(x)}{4\pi^\circ_n(x)}}
(4\pi^\circ_n(x))^{-1/2}
e^{-m_n(1-\zeta)/\t^{\circ}_n},
\Eq(4.lem5.2)
\eea
where the last inequality follows from Lemma \thv(4.lem4).
Similarly,
\bea
P^{*}_{x}\left(J^{*}_n(s_n(m_n))=y, A^c_{\zeta,n}\right)
\geq
\inf_{|t-m_n|< \zeta m_n}P^{*}_{x}\left(J^{*}_n(t)=y, A^c_{\zeta,n}\right).
\Eq(4.lem5.3)
\eea
But
$
%\be
P^{*}_{x}\left(J^{*}_n(t)=y, A^c_{\zeta,n}\right)
%&=&
%P^{*}_{x}\left(J^{*}_n(t)=y\right)-P^{*}_{x}\left(J^{*}_n(t)=y, A_{\d,n}\right).
%\\
\geq
P^{*}_{x}\left(J^{*}_n(t)=y\right)- 2 e^{-m_n\II( \zeta m_n)}
%\Eq(4.lem5.4)
%\ee
$
so that, by Lemma \thv(4.lem4)  again,
\bea
P^{*}_{x}\left(J^{*}_n(s_n(m_n))=y, A^c_{\zeta,n}\right)
\geq
\pi^\circ_n(y)
-
%\sqrt{\sfrac{1-\pi^\circ_n(x)}{4\pi^\circ_n(x)}}
(4\pi^\circ_n(x))^{-1/2}
e^{-m_n(1+\zeta)/\t^{\circ}_n}
- 2 e^{-m_n\II( \zeta m_n)}.\quad
\Eq(4.lem5.5)
\eea
Combining \eqv(4.lem5.6), \eqv(4.lem5.2), and \eqv(4.lem5.5)
proves the claim of the lemma.\endproof
%The claim of the lemma now immediately follows.\endproof

We can now proceed with the proof of Proposition \thv(4.prop1).
Since $\pi^\circ_n$ is the invariant measure of $J^{\circ}_n$, we have
$
P^\circ_{\pi^\circ_n}\left(J^\circ_n(i+\ell^\circ_n)=y,  J^\circ_n(0)=x \right)
=P^\circ_{x}\left(J^{\circ}_n(i+\ell^\circ_n)=y\right)\pi^\circ_n(x)\,.
$
It thus suffices to show that there exists a subset $\O_1\subset\O$
with  $\P\left(\O_1\right)=1$ such that, on $\O_1$, for all but a finite number of
indices $n$, 
%the following holds
for all pairs $x\in\VV^\circ_n, y\in\VV^\circ_n$, and all $i\geq 0$,
\be
\textstyle
\left| 
P^\circ_{x}\left(J^{\circ}_n(i+\ell^\circ_n)=y \right)-\pi^\circ_n(y)
\right|
\leq \delta_n\pi^\circ_n(y)\,,
\Eq(4.4.6)
\ee
for some 
$
0\leq\delta_n\leq  2^{-n}
%e^{-n^2}
%2\exp\left\{-\frac{1}{2}n^2\left(1-\frac{3}{2n}\left(\b^2+\b_c^2(1)\right)\right)\right\}
$.
To do this choose $\zeta=1/2$ and $m_n=\ell^\circ_n$ in Lemma \thv(4.lem5). 
By \eqv(4.prop1.0) and Proposition \thv(4.prop3), on $\O_2$,
$
%\be
{\ell^\circ_n}/{\tau^\circ_n}
\geq 2n(1-o(1))
%2n^2.
%\Eq(4.4.4)
%\ee
$
for all $n$ large enough.
% for all but a finite number of indices $n$.
Futhermore, by Proposition \thv(4.prop2), on $\O^{\star}$, for all $n$ large enough,
%for all but a finite number of indices $n$,
$
%\be
%%\sfrac{1}{\pi^\circ_n(y)}\sqrt{\sfrac{1-\pi^\circ_n(x)}{4\pi^\circ_n(x)}}
(\pi^\circ_n(y)\sqrt{4\pi^\circ_n(x)})^{-1}
%%\leq
%%\sfrac{1}{2\pi^\circ_n(y)\sqrt{\pi^\circ_n(x)}}
\leq 2^{3n/2}.
%\Eq(4.4.5)
%\ee
$
Thus, on $\O_1\equiv\O^{\star}\cap \O_2$,  
for all $n$ large enough,
%for all but a finite number of indices $n$,
\be
\d_n\equiv
%\sfrac{1}{\pi^\circ_n(y)}\sqrt{\sfrac{1-\pi^\circ_n(x)}{4\pi^\circ_n(x)}}
(\pi^\circ_n(y)\sqrt{4\pi^\circ_n(x)})^{-1}
e^{-\ell^\circ_n(1-\zeta)/\t^{\circ}_n}
+
%\sfrac{1}{\pi^\circ_n(y)}
(\pi^\circ_n(y))^{-1}
2e^{-\ell^\circ_n\II( \zeta m_n)}
\leq 2^{-n}
%e^{-n^2}
\,.
\Eq(4.4.7)
\ee
Inserting \eqv(4.4.7) in  \eqv(4.lem4.1) yields the claim of  \thv(4.4.6).
The proof of Proposition \thv(4.prop1) is done.

\subsection{Mean local times: proof of Proposition \thv(4.prop6).}
 \label{4.6}

Let $J^{\scriptscriptstyle{\textsf{SRW}}}_n$ and $P^{\scriptscriptstyle{\textsf{SRW}}}$ 
denote, respectively, the symmetric random walk on $\VV_n$ (hereafter SRW) and its law. More precisely,
\be
p^{\scriptscriptstyle{\textsf{SRW}}}_n(x,y)
\equiv
P^{\scriptscriptstyle{\textsf{SRW}}}\left(
J^{\scriptscriptstyle{\textsf{SRW}}}_n(1)=y\mid J^{\scriptscriptstyle{\textsf{SRW}}}_n(1)=x \right)
=
\begin{cases}
\frac{1}{n}
&\hbox{\rm if}\, \dist(x,y)=1,\\
0,&\hbox{\rm else}.
\end{cases}
\Eq(4.prop6.3)
\ee
We also write $P^{\scriptscriptstyle{\textsf{SRW}}}_x$ for the law of 
$J^{\scriptscriptstyle{\textsf{SRW}}}_n$ started in $x$.
The proof of  Proposition \thv(4.prop6) relies on three key properties
of $J^{\scriptscriptstyle{\textsf{SRW}}}_n$  that we state below in the form of three lemmata.

Our first lemma is a  reformulation of Theorem 1.1 of \cite{CG08} on the hitting time of so-called percolation clouds, namely, sets
of the form 
$V_n(\rho)$ (see \eqv(3.0.1)). 
For $A\subset\VV_n$ let
$T^{\scriptscriptstyle{\textsf{SRW}}}(A)$ be  the hitting time 
\be
\Eq(4.prop6.2)
T^{\scriptscriptstyle{\textsf{SRW}}}(A)=\inf\left\{k\in\N : J^{\scriptscriptstyle{\textsf{SRW}}}_n(k)\in A\right\}.
\ee

\begin{lemma}
%[Theorem 1.1 of \cite{CG08}]
  \TH(4.lem1)
Let $\rho^{\star}_n$  be as in \eqv(10.1.1) for some $c_{\star}$ such that $n^{c_{\star}}\gg n\log n$. 
There exists a subset $\O^{\scriptscriptstyle{\textsf{SRW}}}\subset\O$ with  
$\P\left(\O^{\scriptscriptstyle{\textsf{SRW}}}\right)=1$ such that, 
on $\O^{\scriptscriptstyle{\textsf{SRW}}}$, for all but a finite number of indices $n$ the following holds: 
for all sequences $l_n>0$ such that $l_n/n^{c_{\star}}\leq C$ for some constant $0<C<\infty$,
  \be
\Eq(4.lem1.1)
\max_{x\in\VV_n}\left|
P^{\scriptscriptstyle{\textsf{SRW}}}_x
\left(T^{\scriptscriptstyle{\textsf{SRW}}}(V_n(\rho^{\star}_n)\setminus x)\geq l_n \right) 
-e^{-l_n/n^{c_{\star}}} 
\right|
\leq
%\textstyle
C'\left[
\frac{1}{n}
+\frac{1}{\log n^{c_{\star}}}
+\frac{n\log n}{n^{c_{\star}}}\right],
\ee
where $0<C'<\infty$ is a numerical constant.
\end{lemma}

%\begin{proof}[Proof of  Lemma  \thv(4.lem1)] This is a reformulation of Theorem 1.1 of \cite{CG08}.\end{proof}

The next two lemmata bound the mean number of returns to a given vertex, $z$, respectively the mean local time in $z$, in a time interval of the form $\{3,\dots, m\}$, $m\leq  \left\lceil n^3\right\rceil$.

\begin{lemma}
  \TH(4.lem2)
For all 
$
m
%\leq 
%\ell^\circ_n
\leq  \left\lceil n^3\right\rceil
$,
all 
$
z\in\VV_n
$,
and $a\in\{0,1,2,3\}$,
\be
\Eq(4.lem2.1)
\sum_{l=1}^mP^{\scriptscriptstyle{\textsf{SRW}}}_z\left(J^{\scriptscriptstyle{\textsf{SRW}}}_n(l+a)=z \right)
\leq \frac{c}{n^b},
\quad
b=
\begin{cases}
1,
&\hbox{\rm if}\,\,\, a\in\{0,1\}\\
2,
&\hbox{\rm if}\,\,\, a\in\{2,3\}
\end{cases},
\ee
where $0<c<\infty$ is a numerical constant. \end{lemma}

\begin{proof}
%[Proof of  Lemma  \thv(4.lem2)] 
The lemma is proved in exactly the same way as Proposition 3.2 of \cite{G10b}.
  \end{proof}

\begin{lemma}
  \TH(4.lem3)
For all 
$
m
%\leq 
%\ell^\circ_n
\leq  \left\lceil n^3\right\rceil
$
and all $y,z$ such that $\dist(y,z)\geq 1$, 
\be
\Eq(4.lem3.1)
\sum_{l=1}^mP^{\scriptscriptstyle{\textsf{SRW}}}_y\left(J^{\scriptscriptstyle{\textsf{SRW}}}_n(l)=z \right)
\leq 
\frac{c'}{n}
%,\quad\forall z\in\VV_n.
\ee
where $0<c'<\infty$ is a numerical constant.
\end{lemma}

\begin{proof}[Proof of  Lemma  \thv(4.lem3)] 
%The proof closely follows the proof of Proposition 3.3 of \cite{G10b}.
%The lemma is proved using a $d$-dimensional version of the Ehrenfest scheme, called lumping, 
%that was studied in \cite{BG08}.
%In what follows we draw on the results of \cite{BG08}.
%we mostly stick to the notations of \cite{BG08} hoping that this will create no confusion. 
The proof draws on the results of \cite{BG08} where a $d$-dimensional version of 
the Ehrenfest scheme, called lumping, was introduced and analyzed 
(hereafter and whenever possible we use the notations of \cite{BG08}).
Without loss of generality we may take $y\equiv 1$
to be the vertex all of whose coordinates take the value 1. Let $\g^{\L}$ be the map
(1.7) of \cite{BG08} derived from the partition of $\L\equiv\{1,\dots,n\}$ into
$d=2$ classes, $\L=\L_1\cup\L_2$, defined through the relation:
$i\in\L_1$ if the $i^{th}$ coordinate of $z$ is 1, and $i\in\L_2$ otherwise.
 The resulting lumped chain, $X^{\L}_n\equiv\g^{\L}(J^{\scriptscriptstyle{\textsf{SRW}}}_n)$,
% of the symmetric random walk
has range $\G_{n,2}=\g^{\L}(\VV_n)\subset[-1,1]^2$. Note that the vertices
y and $z$ of $\VV_n$ are mapped, respectively, onto the corners $1\equiv(1,1)$ and $x\equiv(1,-1)$ of $[-1,1]^2$.
Denoting by $\P^{\L}$ the law of $X^{\L}_n$, we have,
\be
\Eq(4.lem3.2)
P^{\scriptscriptstyle{\textsf{SRW}}}_y\left(J^{\scriptscriptstyle{\textsf{SRW}}}_n(l)=z \right)
=\P^{\L}(X^{\L}_n(l)=x \mid X^{\L}_n(0)=1).
\ee
Write
$
\t^{x'}_x=\inf\{ k>0 \mid X^{\L}_n(0)=x', X^{\L}_n(k)=x\}
$.
Without loss of generality we may assume that $0\in\G_{n,2}$ (namely, both $\L_1$ and $\L_2$ have even cardinality).
Then, decomposing \eqv(4.lem3.2) according to whether,
starting from 1, $X^{\L}_n$ visits 0 before it visits $x$ or not, we get:
$
\P^{\L}(X^{\L}_n(l)=x \mid X^{\L}_n(0)=1)=A+B
$,
\bea
\Eq(4.lem3.3)
&&A=\P^{\L}(X^{\L}_n(l)=x, \t^{1}_0< \t^{1}_x),\\
\Eq(4.lem3.4)
&&B=\P^{\L}(X^{\L}_n(l)=x, \t^{1}_0\geq \t^{1}_x).
\eea
By Theorem 3.2 of \cite{BG08}, for all 
%$x=\g^{\L}(z)$ 
$y,z$ such that 
$
%\dist(1,x)\equiv
\dist(z,y)\geq 5
$,
\be
\Eq(4.lem3.5)
%\P^{\L}(X^{\L}_n(l+2)=x, \t^{1}_0\geq \t^{1}_x)
B
\leq \P^{\L}(\t^{1}_x\leq \t^{1}_0)
\leq F_{n,2}(\dist(z,y))
\leq
c_1n^{-4}
\ee
for some constant $0<c_1<\infty$.
Of course $A=0$ for all $l$ such that $l<n/2$ since  the chain $X^{\L}_n$ needs at least $n/2$ steps to travel from the vertex $1$ to $0$. To bound $A$ when $l\geq n/2$ we condition on the time of the last visit to 0 before time $l$,
and bound the probability of the latter event by 1. This readily yields
\be
\Eq(4.lem3.6)
\textstyle
A
\leq l\P^{\L}(\t^{0}_x< \t^{0}_0)
=l\frac{\Q_n(x)}{\Q_n(0)}\P^{\L}(\t^{x}_0< \t^{x}_x)
\leq l\frac{\Q_n(x)}{\Q_n(0)},
\ee
where the equality in the middle is reversibility, and where $\Q_n$, defined in
Lemma 2.2 of \cite{BG08}, denotes the invariant measure of $X^{\L}_n$. We are thus left to estimate the ratio
of invariant masses in \eqv(4.lem3.6). By (2.4) of \cite{BG08} we get that
$
\frac{\Q_n(x)}{\Q_n(0)}
\leq |\{x'\in\VV_n \mid \g^{\L}(x')=0\}|^{-1}
\leq
e^{-c_2n}
$
for some constant $0<c_2<\infty$. Gathering our bounds we get that  for all $y,z$ such that $\dist(y,z)\geq 5$,
\be
\Eq(4.lem3.7)
P^{\scriptscriptstyle{\textsf{SRW}}}_y\left(J^{\scriptscriptstyle{\textsf{SRW}}}_n(l)=z \right)
=A+B
\leq c_1n^{-4}+le^{-c_2n}
\leq c_3n^{-4}
\ee
for some constant $0<c_3<\infty$, so that for all $m\leq  \left\lceil n^3\right\rceil$, 
\be
\Eq(4.lem3.8)
\textstyle
\sum_{l=1}^mP^{\scriptscriptstyle{\textsf{SRW}}}_y\left(J^{\scriptscriptstyle{\textsf{SRW}}}_n(l)=z \right)
\leq c_3n^{-1}.
\ee

%Since by assumption $\dist(y,z)\geq 1$ we 
It remains to deal with the cases $1\leq \dist(y,z)\leq 4$. 
Thus assume from now on that $1\leq \dist(y,z)\leq 4$.  
Consider the event $\AA_z\equiv\{\forall i\leq l \dist(J^{\scriptscriptstyle{\textsf{SRW}}}_n(i), z)<5\}$
and denote by $\AA^c_z$ its complement. Decomposing $\AA^c_z$ on the place and time of the first visit of the chain to the ball of radius $5$, we readily get by the Markov property and \eqv(4.lem3.7) that
\be
\Eq(4.lem3.9)
\textstyle
P^{\scriptscriptstyle{\textsf{SRW}}}_{y}\left(J^{\scriptscriptstyle{\textsf{SRW}}}_n(l)=z , \AA^c_z\right)
%\leq P^{\scriptscriptstyle{\textsf{SRW}}}_{y'}\left(J^{\scriptscriptstyle{\textsf{SRW}}}_n( l+2)=z , \AA^c\right)
\leq c_3n^{-4}.
\ee
Next, by reversibility (the invariant measure of $J^{\scriptscriptstyle{\textsf{SRW}}}_n$ being the uniform measure),
\be
\Eq(4.lem3.11)
\textstyle
P^{\scriptscriptstyle{\textsf{SRW}}}_{y}\left(J^{\scriptscriptstyle{\textsf{SRW}}}_n(l)=z , \AA_z\right)
=
P^{\scriptscriptstyle{\textsf{SRW}}}_{z}\left(J^{\scriptscriptstyle{\textsf{SRW}}}_n(l)=y , \AA_z\right)
\leq
P^{\scriptscriptstyle{\textsf{SRW}}}_{z}\left(\AA_z\right).
\ee
Let us thus estimate the probability
$
P^{\scriptscriptstyle{\textsf{SRW}}}_{z}\left(\AA_z\right)
$
that starting in $z$, the chain did not exit a ball of radius $4$ centered at $z$ by time $l$.
This means that at every step it takes
 %at each step, 
 the chain flips a coordinate of $z$ in such a way that the 
total number of coordinates of $z$ and
$J^{\scriptscriptstyle{\textsf{SRW}}}_n(i)$ that disagree 
is at most $4$ for each $i\leq l$.
If $l\geq 4$, this implies that 
$
(l-4)/2
$
of its $l$ steps 
(respectively,
$
(l-4+1)/2
$
of them)
consist in flipping back a coordinate to its initial position if $l-4$ is even
(respectively, if $l-4$ is odd). Each time such a backward flip occurs the chain chooses one in at most 4 flipped coordinates.
Thus, for all $l\geq 4$,
\be
\Eq(4.lem3.10)
%P^{\scriptscriptstyle{\textsf{SRW}}}_{y}\left(J^{\scriptscriptstyle{\textsf{SRW}}}_n(l+2)=z , \AA\right)
%\leq 
P^{\scriptscriptstyle{\textsf{SRW}}}_{y}\left(\AA_z\right)
\leq \left({4}/{n}\right)^{\frac{l-4}{2}}\1_{l\,\text{even}}
+\left({4}/{n}\right)^{\frac{l-3}{2}}\1_{l\,\text{odd}}.
\ee
Plugging \eqv(4.lem3.10) in \eqv(4.lem3.11) yields that
for all $y,z$ such that $1\leq \dist(y,z)\leq 4$,
\be
\Eq(4.lem3.12)
\textstyle
\sum_{l=5}^mP^{\scriptscriptstyle{\textsf{SRW}}}_{y}\left(J^{\scriptscriptstyle{\textsf{SRW}}}_n(l)=z , \AA_z\right)
%\leq \sum_{l=7}^mP^{\scriptscriptstyle{\textsf{SRW}}}_{y}\left(\AA_z\right)
\leq c_4n^{-1},
\ee
for  all $m\leq  \left\lceil n^3\right\rceil$ and some constant $0<c_4<\infty$, while 
%We are left to bound the sum over $1\leq l<5$. 
by simple combinatorics,
% It is not hard to show that
\be
\Eq(4.lem3.13)
\textstyle
\sum_{l=1}^4P^{\scriptscriptstyle{\textsf{SRW}}}_{y}\left(J^{\scriptscriptstyle{\textsf{SRW}}}_n(l)=z , \AA_z\right)
\leq \sum_{l=1}^4P^{\scriptscriptstyle{\textsf{SRW}}}_{y}\left(J^{\scriptscriptstyle{\textsf{SRW}}}_n(l)=z \right)
\leq c_5n^{-1},
\ee
%for all $y,z$ such that $1\leq \dist(y,z)\leq 4$ and 
for some $0<c_5<\infty$.
Combining \eqv(4.lem3.8), \eqv(4.lem3.12) and \eqv(4.lem3.13)  finishes the proof. \end{proof}

%yields the claim of the lemma. \end{proof}

%Putting \eqv(4.lem3.8), \eqv(4.lem3.12) and \eqv(4.lem3.13) together finally yields the claim of the lemma. \end{proof}
 
 We are now ready to give the proof of Proposition \thv(4.prop6).

 \begin{proof}[Proof of Proposition \thv(4.prop6), (i)] 
 Given $y\in\VV_n$ denote respectively by $P^\circ_y$, $P_y$, and $P^{\scriptscriptstyle{\textsf{SRW}}}_y$ the laws of $J^\circ_n$, $J_n$, and $J^{\scriptscriptstyle{\textsf{SRW}}}_n$  started in $y$. 
 The idea behing the proof is to decompose the paths of $J_n$ at visits to the set
$\VV_n^{\star}\equiv\cup_{1\leq l\leq L^{\star}}C^{\star}_{n,l}$,
%\Eq(10.1.4bis)
%$C^{\star}_{n,l}$, $1\leq l\leq L^{\star}$, 
and use that, away from this set, $J_n$ reduces to  SRW.
To this end define
$
T^{{\scriptscriptstyle{\textsf{SRW}}}}(A)\equiv\inf\{i\in\N \mid J^{\scriptscriptstyle{\textsf{SRW}}}_n(i)\in A\}
$,
$A\subseteq \VV_n$,
and set
\be
\Eq(4.prop6.2')
T^{{\scriptscriptstyle{\textsf{SRW}}},\star}_n\equiv\inf\left\{k\in\N : 
J^{\scriptscriptstyle{\textsf{SRW}}}_n(k)\in\VV_n^{\star}\right\}.
\ee
Similarly set
\be
\Eq(4.prop6.4)
T^{\star}_n\equiv\inf\left\{i\in\N \mid J_n(i)\in\VV_n^{\star}\right\}.
\ee
%Then, given $z\in I^{\star}_n$, and since by definition 
Let $z\in I^{\star}_n$ be fixed. Using that by definition
$
J_n^\circ(i)\equiv J_n(T_{n,i}^\circ)
$, 
we may write
%we then have
%Let $z\in V_n(\rho^{\star}_n)\setminus\VV^\circ_{n}$ be fixed. By \eqv(2.1.5),
\be
\Eq(4.prop6.5)
\textstyle
\sum_{k=1}^{\ell^\circ_n-1}
P^\circ_z\left(J^\circ_n(k+2)=z\right)
=
\sum_{k=1}^{\ell^\circ_n-1}
P_z\left(J_n(T_{n,k+2}^\circ)=z \right)
=I_1+I_2
\ee
where
%, setting $  l+2\equiv l+2$,
\bea
\Eq(4.prop6.6)
I_1&\equiv& \textstyle
\sum_{k=1}^{\ell^\circ_n-1}
P_z\left(J_n(T_{n, k+2}^\circ)=z,  T^{\star}_n>  k+2\right),
\\
\Eq(4.prop6.7)
I_2&\equiv& \textstyle
\sum_{k=1}^{\ell^\circ_n-1}
P_z\left(J_n(T_{n, k+2}^\circ)=z,  T^{\star}_n\leq  k+2\right).
\eea
In view of \eqv(2.1.3)-\eqv(2.1.4), $T_{n,i}^\circ=i$ for all $i\in\{0,\dots,T^{\star}_n-1\}$. Hence
\be
\Eq(4.prop6.8)
I_1= \textstyle
\sum_{k=1}^{\ell^\circ_n-1}
P_z\left(J_n( k+2)=z,  T^{\star}_n>  k+2\right),
\ee
and since up to time $T^{\star}_n$ the transition probabilities of $J_n$
% in the time interval $[0,T^{\star}_n]$
 %$\{0,\dots,T^{\star}_n\}$ 
 are those of  SRW,
\be
\Eq(4.prop6.9)
\textstyle
I_1\leq 
\sum_{k=1}^{\ell^\circ_n-1}
P^{\scriptscriptstyle{\textsf{SRW}}}_z\left(J^{\scriptscriptstyle{\textsf{SRW}}}_n(k+2)=z \right)
\leq cn^{-2},
\ee
where the last inequality is \eqv(4.lem2.1). 

To Bound $I_2$ note that the event $\{T^{\star}_n\leq  k+2\}$
%according to the place and time of the first visit of
%$J_n$ to $\VV_n^{\star}=\VV_n^{\star}$,
can be written as the disjoint union
\be
\Eq(4.prop6.10)
\{T^{\star}_n\leq  k+2\}
=\cup_{i\leq   k+2}\cup_{y\in\VV_n^{\star}}
\{ T^{\star}_n=i, J_n(T^{\star}_n)=y\}.
\ee
Thus
\be
\Eq(4.prop6.11)
\textstyle
I_2= 
\sum_{k=1}^{\ell^\circ_n-1}
\sum_{i=1}^{k+2}\sum_{y\in\VV_n^{\star}}
P_z\left(J_n(T_{n,k+2}^\circ)=z,  T^{\star}_n=i, J_n(T^{\star}_n)=y\right).
\ee
As above note that $T_{n,i}^\circ=i$ for all $i\in\{0,\dots,T^{\star}_n-1=i-1\}$,
that $T^{\star}_n=T_{n,i-1}^\circ+1$, and that
in the time interval $\{0,\dots,T^{\star}_n\}$, $J_n$ has the same transition probabilities as SRW.
By this and 
the Markov property,  
%using the notation \eqv(4.prop6.2), 
the probability in \eqv(4.prop6.11) is equal to
\be
\Eq(4.prop6.12)
P^{\scriptscriptstyle{\textsf{SRW}}}_z\left(
T^{{\scriptscriptstyle{\textsf{SRW}}},\star}_n=i,
%T^{{\scriptscriptstyle{\textsf{SRW}}},\star}_n=i, 
J^{\scriptscriptstyle{\textsf{SRW}}}_n(i)=y
\right)
P_y\left(J_n(T_{n,k+2-i}^\circ)=z\right).
\ee
Consider now the last factor in \eqv(4.prop6.12). By construction 
%$y\in\VV_n^{\star}$.
$y\in \VV_n^{\star}$.
Hence, by \eqv(2.1.3),
\bea
\nonumber
P_y\left(J_n(T_{n,k+2-i}^\circ)=z\right)
\hspace{-6pt}&=&\hspace{-6pt}
P_y\left(J_n(T_{n,0}^\circ)=z\right)\1_{\{k+2-i=0\}}
\\
\Eq(4.prop6.15)
\hspace{-6pt}&+&\hspace{-6pt}
\textstyle
\sum_{x}P_y\left(J_n(T_{n,0}^\circ)=x\right)
P_x\left(J_n(T_{n,k+2-i}^\circ)=z\right)\1_{\{k+2-i>0\}}\quad\quad\,\,
\eea
where the sum is over $x$ in $\del\VV_n^{\star}=\cup_{1\leq l\leq L^{\star}}\del C^{\star}_{n,l}$. 
%where the exit distribution $\LL^{\star}_{n,l}$ is defined in \eqv(5.lem9.0) and estimated
Note that $P_y\left(J_n(T_{n,0}^\circ)=x\right)$ is nothing but the exit distribution from the set $C^{\star}_{n,l}$ containing $y$,
%for all $y\in C^{\star}_{n,l}$, $1\leq l\leq L^{\star}$, ,
%\eqv(5.lem9.1),
so by Lemma \thv(5.lem9), for all $y\in C^{\star}_{n,l}$, 
%$1\leq l\leq L^{\star}$,
\bea
\nonumber
P_y\left(J_n(T_{n,k+2-i}^\circ)=z\right)
\hspace{-6pt}&=&\hspace{-6pt}
(1+o(1))\bigl\{
\LL^{\star}_{n,l}(z)\1_{\{z\in\del C^{\star}_{n,l}\}}
%\sum_{1\leq l\leq L^{\star}}\1_{\{z\in\del C^{\star}_{n,l}, y\in C^{\star}_{n,l}\}}
\1_{\{k+2-i=0\}}
\\
\Eq(4.prop6.16)
\hspace{-6pt}&+&\hspace{-6pt}
\textstyle\sum_{x}\LL^{\star}_{n,l}(x)
P_x\left(J_n(T_{n,k+2-i}^\circ)=z\right)\1_{\{k+2-i>0\}}\bigr\}.
\eea
Observe that since by assumption $z$ belongs to
$
 I^{\star}_n
$
 then, by \eqv(10.1.4), 
$
z\notin \del\VV_n^{\star}
%\cup_{1\leq l\leq L^{\star}}\del C^{\star}_{n,l}
$.
Thus
%is non zero if and only if $k+2-i>0$
the first term in the r.h.s.~of   \eqv(4.prop6.16) is zero.
Now for indices $i,k$ such that $k+2-i>0$, 
let us rewrite the  probability  in the remaining term as
%decompose the probability  therein into
%appearing in \eqv(4.prop6.16) 
%To bound the last probability in \eqv(4.prop6.16) we write 
$P_x\left(J_n(T_{n,k+2-i}^\circ)=z\right)=I_3(x)+I_4(x)$,
\bea
\Eq(4.prop6.13)
I_3(x)&\equiv&P_x\left(J_n(T_{n,k+2-i}^\circ)=z, T^{\star}_n>  k+2-i\right),
\\
\Eq(4.prop6.14)
I_4(x)\equiv
A_x(k+2-i)
&\equiv&P_x\left(J_n(T_{n,k+2-i}^\circ)=z, T^{\star}_n\leq  k+2-i\right).
\eea
Proceeding as we did for $I_1$ to express $I_3(x)$ yields 
\be
\Eq(4.prop6.17)
I_3(x)
=
P^{\scriptscriptstyle{\textsf{SRW}}}_x\left(J^{\scriptscriptstyle{\textsf{SRW}}}_n(k+2-i)=z \right).
\ee
Thus, up to a prefactor $1+o(1)$, the contribution to $I_2$ coming from the terms $I_3(x)$ is 
%given by
\bea
\nonumber
K_1\equiv
\textstyle
\sum_{k=1}^{\ell^\circ_n-1}
\sum_{i=1}^{k+1}
\sum_{1\leq l\leq L^{\star}}
\hspace{-10pt}&&\hspace{-10pt}
\textstyle
\sum_{y\in C^{\star}_{n,l}}
\bigl\{
P^{\scriptscriptstyle{\textsf{SRW}}}_z\left(
T^{{\scriptscriptstyle{\textsf{SRW}}},\star}_n=i, 
J^{\scriptscriptstyle{\textsf{SRW}}}_n(i)=y
\right)
\\
\Eq(4.prop6.18)
&&
\textstyle
\sum_{x\in\del C^{\star}_{n,l}}\LL^{\star}_{n,l}(x)
P^{\scriptscriptstyle{\textsf{SRW}}}_x\left(J^{\scriptscriptstyle{\textsf{SRW}}}_n(k+2-i)=z \right)
\bigr\}.
\quad\quad
\eea
By a change of indices,
\be
\Eq(4.prop6.19)
K_1=
\textstyle
\sum_{1\leq l\leq L^{\star}}
\sum_{j=1}^{\ell^\circ_n}
\sum_{x\in\del C^{\star}_{n,l}}\LL^{\star}_{n,l}(x)
P^{\scriptscriptstyle{\textsf{SRW}}}_x\left(J^{\scriptscriptstyle{\textsf{SRW}}}_n(j)=z \right)R_1(j),
\ee
where
\bea
\nonumber
R_1(j)
\equiv\hspace{-6pt}&&\hspace{-6pt}
\textstyle
\sum_{m=1}^{\ell^\circ_n+1-j}\sum_{y\in C^{\star}_{n,l}}
P^{\scriptscriptstyle{\textsf{SRW}}}_z\left(
T^{{\scriptscriptstyle{\textsf{SRW}}},\star}_n=m, 
J^{\scriptscriptstyle{\textsf{SRW}}}_n(m)=y
\right)
\\
\nonumber
\leq\hspace{-6pt}&&\hspace{-6pt}
P^{\scriptscriptstyle{\textsf{SRW}}}_z\left(
T^{{\scriptscriptstyle{\textsf{SRW}}},\star}_n
=T^{\scriptscriptstyle{\textsf{SRW}}}(C^{\star}_{n,l})
\leq \ell^\circ_n+1-j
\right)
\\
\Eq(4.prop6.20)
\leq\hspace{-6pt}&&\hspace{-6pt}
P^{\scriptscriptstyle{\textsf{SRW}}}_z\left(
T^{{\scriptscriptstyle{\textsf{SRW}}},\star}_n
=T^{\scriptscriptstyle{\textsf{SRW}}}(C^{\star}_{n,l})
\leq \ell^\circ_n
\right).
\eea
%Again, remember that  since $x$  in \eqv(4.prop6.19)  belongs to
%$
%\del C^{\star}_{n,l}
%$
%and since $z$ belongs to
%$
% I^{\star}_n
%$
% then, by \eqv(10.1.4), 
 As observed earlier, $\dist(x,z)\geq 1$.
Thus, by  Lemma \thv(4.lem3), and since $\sum_{x}\LL^{\star}_{n,l}(x)=1$, 
% \eqv(5.lem9.0),
\be
\Eq(4.prop6.21)
\textstyle
\sum_{j=1}^{\ell^\circ_n}
\sum_{x\in\del C^{\star}_{n,l}}\LL^{\star}_{n,l}(x)
P^{\scriptscriptstyle{\textsf{SRW}}}_x\left(J^{\scriptscriptstyle{\textsf{SRW}}}_n(j)=z \right)
\leq c'/n.
\ee
%since $\sum_{x\in\del C^{\star}_{n,l}}\LL^{\star}_{n,l}(x)=1$. 
Plugging \eqv(4.prop6.20) in \eqv(4.prop6.19) and using  \eqv(4.prop6.21), 
%arrive at
\bea
\nonumber
K_1\leq\hspace{-6pt}&&\hspace{-6pt}
\textstyle
(c'/n)
\sum_{1\leq l\leq L^{\star}}
P^{\scriptscriptstyle{\textsf{SRW}}}_z\left(
T^{{\scriptscriptstyle{\textsf{SRW}}},\star}_n
=T^{\scriptscriptstyle{\textsf{SRW}}}(C^{\star}_{n,l})
\leq \ell^\circ_n
\right)
\\
\nonumber
\leq\hspace{-6pt}&&\hspace{-6pt}
(c'/n)
P^{\scriptscriptstyle{\textsf{SRW}}}_z\left(
T^{{\scriptscriptstyle{\textsf{SRW}}},\star}_n
\leq \ell^\circ_n
\right)
\\
\Eq(4.prop6.22)
\leq\hspace{-6pt}&&\hspace{-6pt}
(c'/n)
P^{\scriptscriptstyle{\textsf{SRW}}}_z
\left(T^{\scriptscriptstyle{\textsf{SRW}}}(V_n(\rho^{\star}_n)\setminus z)\leq\ell^\circ_n+1 \right).
\eea
Note that by \eqv(4.prop1.0) and the assumption that $c_{\star}>3$, $n^{c_{\star}}> \ell^\circ_n\log n$ 
for all  $n$ large enough. It thus follows from Lemma \thv(4.lem1) that on $\O^{\scriptscriptstyle{\textsf{SRW}}}$, 
for all but a finite number of indices $n$,
%Now, by \eqv(4.prop1.0) and the assumption that $c_{\star}>3$, $n^{c_{\star}}> \ell^\circ_n\log n$ 
%for all  $n$ large enough. It thus follows from Lemma \thv(4.lem1) that on $\O^{\scriptscriptstyle{\textsf{SRW}}}$, 
%for all but a finite number of indices $n$,
\be
\Eq(4.prop6.23)
K_1\leq\Bigl(\frac{c'}{n}\Bigr)\Bigl(\frac{2+C'/c_{\star}}{\log n}\Bigr).
\ee
We assume from now on that $\o\in\O^{\scriptscriptstyle{\textsf{SRW}}}$.
It remains to deal with the contribution to $I_2$ coming from the terms $I_4(x)$:
in view of \eqv(4.prop6.14) this
%this  contribution 
is given, up to a  prefactor $1+o(1)$, by
\bea
\nonumber
K_2\equiv
\textstyle
\sum_{k=1}^{\ell^\circ_n-1}
\sum_{i=1}^{k+1}
\sum_{1\leq l\leq L^{\star}}\sum_{y\in C^{\star}_{n,l}}
\hspace{-10pt}&&\hspace{-10pt}
\textstyle
\bigl\{
P^{\scriptscriptstyle{\textsf{SRW}}}_z\left(
T^{{\scriptscriptstyle{\textsf{SRW}}},\star}_n=i, 
J^{\scriptscriptstyle{\textsf{SRW}}}_n(i)=y
\right)\quad\quad
\\
\Eq(4.prop6.24)
&&
\textstyle
\sum_{x\in\del C^{\star}_{n,l}}\LL^{\star}_{n,l}(x)
A_x(k+2-i)\bigr\}.
\quad\quad
\eea
Proceeding as in  $K_1$ to rearrange the indices
%\be
%\Eq(4.prop6.25)
%K_2=
%\textstyle
%\sum_{1\leq l\leq L^{\star}}
%\sum_{j=0}^{\ell^\circ_n}
%\sum_{x\in\del C^{\star}_{n,l}}\LL^{\star}_{n,l}(x)
%A_x(j)R_1(j),
%\ee
and using \eqv(4.prop6.20),
\bea
\nonumber
K_2\leq
\textstyle
\sum_{1\leq l\leq L^{\star}}P^{\scriptscriptstyle{\textsf{SRW}}}_z\left(
T^{{\scriptscriptstyle{\textsf{SRW}}},\star}_n
=T^{\scriptscriptstyle{\textsf{SRW}}}(C^{\star}_{n,l})
\leq \ell^\circ_n+1
\right)\hspace{-10pt}&&\hspace{-10pt}\quad\quad\quad
\\\Eq(4.prop6.26)
\textstyle
\sum_{x\in\del C^{\star}_{n,l}}\LL^{\star}_{n,l}(x)
\hspace{-10pt}&&\hspace{-10pt}
\textstyle\sum_{j=1}^{\ell^\circ_n}A_x(j).\quad\quad
\eea
Summing up, we have established that $I_2\leq (1+o(1))(K_1+K_2)$ with $K_1$ bounded by \eqv(4.prop6.23)
and $K_2$ bounded by \eqv(4.prop6.26). But by \eqv(4.prop6.7) and \eqv(4.prop6.14),
\be
\Eq(4.prop6.27)
I_2
=
\textstyle\sum_{k=1}^{\ell^\circ_n-1}A_z(k+2).
\ee
Hence
\be
\Eq(4.prop6.28)
\textstyle
\sum_{k=1}^{\ell^\circ_n-1}A_z(k+2)
\leq
%\frac{c'}{n}
(c'/n)
\a_n
+
\sum_{1\leq l\leq L^{\star}}
\a_{n,l}
\sum_{x\in\del C^{\star}_{n,l}}\LL^{\star}_{n,l}(x)
\sum_{j=1}^{\ell^\circ_n}A_x(j),
\quad
\ee
where $ \a_{n,l}\geq 0$, $\sum_{1\leq l\leq L^{\star}} \a_{n,l}=\a_n$, and $\a_n\leq c''/(\log n)$ for some constant 
$0<c''<\infty$.

The idea  now simply is to  iterate this relation. One easily checks, going through the preceeding 
arguments that, indeed, $\sum_{j=1}^{\ell^\circ_n}A_x(j)$ can be bounded in exactly the same way as $I_2$.
Doing so and iterating  $\kappa$ times gives
\be
\Eq(4.prop6.29)
\textstyle
I_2\leq(c'/n)\sum_{m=1}^{\kappa}\a_n^{m}
+\a_n^{\kappa-1}
\sum_{1\leq l\leq L^{\star}}\a_{n,l}
\sum_{x\in\del C^{\star}_{n,l}}\LL^{\star}_{n,l}(x)
\sum_{j=1}^{\ell^\circ_n}A_x(j).
\ee
Since $A_x(j)\leq 1$, choosing $\kappa=\lfloor\log n\rfloor$ yields
% $\a_n^{\kappa}\ell^\circ_n\ll n^{-2}$ and
\be
\Eq(4.prop6.29')
%\textstyle
I_2\leq\frac{c'}{n}\frac{\a_n}{1-\a_n}+\a_n^{\kappa}\ell^\circ_n
\leq \frac{c'}{n}\frac{\a_n}{1-\a_n}+o\Bigl(\frac{1}{n^2}\Bigr).
\ee
Inserting this bound and 
\eqv(4.prop6.9) in \eqv(4.prop6.5) we finally obtain
\be
\Eq(4.prop6.30)
{\textstyle\sum_{k=1}^{\ell^\circ_n-1}}
P^\circ_z\left(J^\circ_n(k+2)=z\right)
\leq
cn^{-2}(1+o(1))+c'n^{-1}\frac{\a_n}{1-\a_n}
%\frac{c}{n^2}+\frac{c'}{n}\frac{\a_n}{1-\a_n}+o\Bigl(\frac{1}{n^2}\Bigr)
\leq \frac{C_\circ}{n\log n},
\ee
for some constant $0<C_\circ<\infty$, and this is valid on $\O^{\scriptscriptstyle{\textsf{SRW}}}$
 for all but a finite number of indices $n$.
The proof of assertion (i) of Proposition \thv(4.prop6) is complete.
\end{proof}

 \begin{proof}[Proof of  Proposition \thv(4.prop6), (ii)] The  proof is a rerun of the
 proof of assertion (i) with only three minor changes that we now indicate.
Let $1\leq l'\leq L^{\star}$ and $z,z'\in \del C^{\star}_{n,l'}$ be given. As in \eqv(4.prop6.5) write
 \be
\Eq(4.prop6.31)
\textstyle
\sum_{k=1}^{\ell^\circ_n-1}
P^\circ_{z'}\left(J^\circ_n(k)=z\right)
=
\sum_{k=1}^{\ell^\circ_n-1}
P_{z'}\left(J_n(T_{n,k}^\circ)=z \right)
=I_1+I_2
\ee
where $I_1$ and $I_2$ are the analogues of \eqv(4.prop6.6) and \eqv(4.prop6.7), respectively.
Arguing as in \eqv(4.prop6.8)-\eqv(4.prop6.9) to bound $I_1$,  but using  \eqv(4.lem2.1) of Lemma \thv(4.lem2) if $z=z'$
and  \eqv(4.lem3.1) of Lemma \thv(4.lem3) if $z\neq z'$,  we get that
\be
\Eq(4.prop6.32)
\textstyle
I_1\leq 
\sum_{k=1}^{\ell^\circ_n-1}
P^{\scriptscriptstyle{\textsf{SRW}}}_{z'}\left(J^{\scriptscriptstyle{\textsf{SRW}}}_n(k)=z \right)
\leq c''n^{-1}
\ee
 for some constant $0< c''<\infty$. This is the most impacting change since, as we shall see, this term now becomes the leading one. The second change is in the treatment of \eqv(4.prop6.16) where
% since,  whenever $l=l'$,
% %, assuming e.g.~that $y\in C^{\star}_{n,l}$, 
 the term 
$
\LL^{\star}_{n,l}(z)\1_{\{z\in\del C^{\star}_{n,l}\}}
%\sum_{1\leq l\leq L^{\star}}\1_{\{z\in\del C^{\star}_{n,l}, y\in C^{\star}_{n,l}\}}
\1_{\{k+2-i=0\}}
$
is non zero whenever $l=l'$. In that case we deduce from \eqv(4.prop4.9) and \eqv(4.prop4.0) that
$
\LL^{\star}_{n,l}(z)\leq (1+o(1)){|C^{\star}_{n,l}|}/{|\del C^{\star}_{n,l}|},
$
so that by \eqv(10.lem1.6) of Lemma \thv(10.lem1), 
\be
\LL^{\star}_{n,l}(z)\leq n^{-1}(1+o(1))
\Eq(4.prop6.33)
\ee
on $\O^{\star}$,  for all but a finite number of indices $n$.
One then easily checks that the contribution to $I_2$ coming from this term can be bounded in the same way as 
the contribution coming from the terms $I_3(x)$, that is to say, as $K_1$.
The third and last minor change is in the treatment of $K_1$. Notice that the sum over $x$ in \eqv(4.prop6.18) 
contains the term $x=z$. For this term we bound $I_3(x)$ (see \eqv(4.prop6.17))
 using  \eqv(4.lem2.1).
 % of Lemma \thv(4.lem2). 
 Clearly, this does not affect \eqv(4.prop6.21) and thus, the bound 
 \eqv(4.prop6.23) (valid  on $\O^{\scriptscriptstyle{\textsf{SRW}}}$, for all but a finite number of indices $n$) is unchanged.
 
Poceeding from there on as in the proof of  assertion (i) we finally get that
on $\O^{\scriptscriptstyle{\textsf{SRW}}}\cap \O^{\star}$,  for all but a finite number of indices $n$,
 \be
\Eq(4.prop6.34)
{\textstyle\sum_{k=1}^{\ell^\circ_n-1}}
P^\circ_{z'}\left(J^\circ_n(k)=z\right)
\leq
c''n^{-1}+c'n^{-1}\frac{\a_n}{1-\a_n}+o(n^{-2})
%\frac{c''}{n}+\frac{c'}{n}\frac{\a_n}{1-\a_n}+o\Bigl(\frac{1}{n^2}\Bigr)
\leq 
C'_\circ n^{-1}
%\frac{C'_\circ}{n},
\ee
for some constant $0<C'_\circ<\infty$.
 The proof of  assertion (ii) of Proposition \thv(4.prop6) is done.
\end{proof}

%%%%%%%%%%%%%%%%%%%%%%%%%%%%%%%%%%%%%%%%%%%%%%%%%
% 							debut Subsection \thv(4.2) 
%%%%%%%%%%%%%%%%%%%%%%%%%%%%%%%%%%%%%%%%%%%%%%%%%

\subsection{Hitting time at stationarity: proof of Proposition \thv(4.prop5).}
 \label{4.2}
 
Let us first prove, that under the assumptions of Proposition \thv(4.prop5), \eqv(4.prop5.1) holds for the continous time Markov chain $J^{*}_{n}$ introduced in Subsection \thv(4.5).
%the proof of Proposition \thv(4.prop1)
For this we use results  from \cite{AB1}.
 %bounds obtained in  on the hitting time at stationarity for  continuous time reversible Markov chains.
  Set $B\equiv A\cap I^{\star}_n$ and write
$
T^{*}(B)\equiv\inf\{t>0 \mid J^{*}_n(t)\in B\}
$.
Then, by Theorem 3 and
% the lower bound of
 Lemma 2 of \cite{AB1} we have,
%\t^\circ_n\leq  \sfrac{1}{2}n^2(1+o(1)).
\be\textstyle
P^{*}_{\pi^\circ_n}\left(T^{*}(B)>t\right)
\geq
\bigl(
1-\t^\circ_n\frac{q(B,B^c)}{1-\pi^\circ_n(B)}
\bigr)
\exp\bigl({-t\frac{q(B,B^c)}{1-\pi^\circ_n(B)}}\bigr), \quad t>0,
\Eq(4.prop5.2)
\ee
where $\t^\circ_n$ is as in \eqv(4.prop3.1) and where
$
q(B,B^c)=\sum_{x\in B}\sum_{y\notin B}\pi^\circ_n(x)p^\circ_n(x,y)=\pi^\circ_n(B)
$
as follows from \eqv(4.prop4.2) and the fact that $B\subseteq I^{\star}_n$. By Proposition \thv(4.prop2),
\be
\pi^\circ_n(B)={|B|}/{|\VV^\circ_{n}|}\leq {|I^{\star}_n|}/{|\VV^\circ_{n}|}\leq n^{-c_{\star}}(1+o(1)),
\Eq(4.prop5.3)
\ee
where we used  \eqv(4.prop2.0) and \eqv(10.lem1.2)
% of Lemma \thv(10.lem1) 
in the last inequality. From this and
\eqv(4.prop3.1), we get that
\be\textstyle
P^{*}_{\pi^\circ_n}\left(T^{*}(B)>t\right)
\geq
\bigl(
1-n^{-(c_{\star}-2)}(1+o(1))
\bigr)
\exp\bigl(-t\frac{|B|}{|\VV^\circ_{n}|}(1+o(n^{-c_{\star}}))\bigr), \quad t>0.
\Eq(4.prop5.4)
\ee

%As in Subsection \thv(4.5) 
The idea then is that for $s_n$ as in \eqv(4.4.2), $T^{*}(B)-T^\circ(B)=s_n(T^\circ(B))-T^\circ(B)$, which by 
\eqv(4.lem5.6+) should be small for $T^\circ(B)$ large.
We thus need an a priori lower bound on $T^\circ(B)$. 
Now note that $B\subset V_n(\rho^{\star}_n)$ so that by 
%\eqv(4.lem1.1)
Lemma \thv(4.lem1) and \eqv(4.prop5.3), for all $l_n\leq n^{c_{\star}}/\log n$, 
\be
\nonumber
\textstyle
P^\circ_{\pi^\circ_n}(T^\circ(B)>l_n)
%\geq
%\sum_{x\notin B}\pi^\circ_n(x)P^\circ_{x}t(T^\circ(B)>l_n)
%\geq 
%\sum_{x\notin B}\pi^\circ_n(x)P^{\scriptscriptstyle{\textsf{SRW}}}_x
%\left(T^{\scriptscriptstyle{\textsf{SRW}}}(V_n(\rho^{\star}_n)\setminus x)\geq l_n \right) 
\geq (1-\pi^\circ_n(B))\inf_{x\notin B}P^{\scriptscriptstyle{\textsf{SRW}}}_x
\left(T^{\scriptscriptstyle{\textsf{SRW}}}(V_n(\rho^{\star}_n)\setminus x)\geq l_n \right) 
%\geq (1-\pi^\circ_n(B))(1-\OO(\frac{1}{\log n}))
\geq 1-\OO(\sfrac{1}{\log n}).
\Eq(4.prop5.5)
\ee
From this bound, \eqv(4.lem5.6+), and \eqv(4.prop5.4) we easily get that for 
$\II(\zeta)$ as in \eqv(4.lem5.0) and any $\zeta>0$,
\be
P^{*}_{\pi^\circ_n}\left(T^{*}(B)>t\right)
\leq P^\circ_{\pi^\circ_n}\left(T^\circ(B)>t/(1+\zeta)\right)+2e^{-l_n\II(\zeta)}+\OO(\sfrac{1}{\log n}).
\Eq(4.prop5.6)
\ee
%%%%%%%%%%%%%%%%%%%%%%%%%%%%%%%%%%%%%%%%%%%%%%%%%
%Because
%\bea
%P^{*}_{\pi^\circ_n}\left(T^{*}(B)>t\right)
%&\leq &
%P^{*}_{\pi^\circ_n}\left(T^{*}(B)>t, |T^{*}(B)-T^\circ(B)|\leq\zeta T^\circ(B) \right)
%\\
%&+&
%P^{*}_{\pi^\circ_n}\left(T^{*}(B)>t, |T^{*}(B)-T^\circ(B)|>\zeta T^\circ(B) \right)
%\eea
%On  one hand
%\be
%P^{*}_{\pi^\circ_n}\left(T^{*}(B)>t, |T^{*}(B)-T^\circ(B)|\leq\zeta T^\circ(B) \right)
%\leq 
%P^\circ_{\pi^\circ_n}\left(T^\circ(B)>t/(1+\zeta)\right)
%\ee
%On the other hand
%\bea
%&&P^{*}_{\pi^\circ_n}\left(T^{*}(B)>t, |T^{*}(B)-T^\circ(B)|>\zeta T^\circ(B) \right)
%\\
%&\leq &
%P^{*}_{\pi^\circ_n}\left(|T^{*}(B)-T^\circ(B)|>\zeta T^\circ(B) \right)
%\\\nonumber
%&\leq &
%P^{*}_{\pi^\circ_n}\left(|T^{*}(B)-T^\circ(B)|>\zeta T^\circ(B),  T^\circ(B)>l_n\right)
%\\
%&+&
%P^{*}_{\pi^\circ_n}\left(|T^{*}(B)-T^\circ(B)|>\zeta T^\circ(B),  T^\circ(B)\leq l_n\right)
%\eea
%where
%\bea
%&&
%P^{*}_{\pi^\circ_n}\left(|T^{*}(B)-T^\circ(B)|>\zeta T^\circ(B),  T^\circ(B)\leq l_n\right)
%\\
%&\leq &
%P^{*}_{\pi^\circ_n}\left(T^\circ(B)\leq l_n\right)
%\\
%&=&1-P^{*}_{\pi^\circ_n}\left(T^\circ(B)>l_n\right)
%\leq \OO(\sfrac{1}{\log n}).
%\eea
%Finally, 
%\bea
%&&P^{*}_{\pi^\circ_n}\left(|T^{*}(B)-T^\circ(B)|>\zeta T^\circ(B),  T^\circ(B)> l_n\right)
%\\
%&=&
%\sum_{k> l_n}P^{*}_{\pi^\circ_n}\left(|T^{*}(B)-T^\circ(B)|>\zeta T^\circ(B),  T^\circ(B)=k\right)
%\\
%&\leq & 2 e^{-l_n\II(\zeta)},
%\eea
%where we used  \eqv(4.lem5.6+).
%%%%%%%%%%%%%%%%%%%%%%%%%%%%%%%%%%%%%%%%%%%%%%%%%
%
%
Taking e.g.~$\zeta=1/2$ and $l_n=n^2$ yields \eqv(4.prop5.1) and finishes the proof of Proposition  \thv(4.prop5).

%%%%%%%%%%%%%%%%%%%%%%%%%%%%%%%%%%%%%%%%%%%%%%%%%
% 							end Subsection \thv(4.2) 
%%%%%%%%%%%%%%%%%%%%%%%%%%%%%%%%%%%%%%%%%%%%%%%%%

%%%%%%%%%%%%%%%%%%%%%%%%%%%%%%%%%%%%%%%%%%%%%%%%%%%%%%%%%%%%

%%%%%%%%%%%%%%%%%%%%%%%%%%%%%%%%%%%%%%%%%%%%%%%%%%%%%%%%%%%%

\section{Proof of Theorem \thv(2.theo3) and of Theorem \thv(2.theo4)}
 \label{8}
 
% To simplify the notation we write $\rho^{\star}_n\equiv\rho^{\star}_n$ and $\varepsilon\equiv\varepsilon_n$.
 
 %We begin this section with two preparatory Lemmata.
 The proofs of Theorem \thv(2.theo3) and of Theorem \thv(2.theo4) hinge upon the next two lemmata.
 
\subsection{Preparatory Lemmata.}
 \label{8.3}
%Recall from Theorem \thv(2.theo0.MainClock) that $a_n=2^{\varepsilon n}$, $0<\varepsilon<1$. 
Let $0<\rho<1$ and, for $V_n(\rho)$ defined in \eqv(3.0.1),
%that $V_n(\rho)=\left\{x\in\VV_n\mid w_n(x)\geq r_n(\rho)\right\}$ for $\rho>0$ 
set
\be
C_{n,l}^\star(\rho)=
\begin{cases} 
 C_{n,l}^\star &\mbox{if}\, C_{n,l}^\star\cap V_n(\rho)\neq \emptyset,\\
\emptyset & \mbox{else}.
\end{cases} 
\Eq(8.3.0)
\ee
\begin{lemma}
  \TH(8.3lem1)
Assume that $c_{\star}>2$.
There exists a subset $\O_3\subset\O$ with  $\P\left(\O_3\right)=1$ such that 
on $\O_3$, for all but a finite number of indices $n$, for all $\rho^{\star}_n\leq \rho\leq1-3\rho^{\star}_n$,
\be
\left|\cup_{1\leq l\leq L^{\star}}C^{\star}_{n,l}(\rho)\right|/|\VV^\circ_{n}|
\leq  n^{-c_{\star}+1}2^{-n\rho}(1+o(1)),
\Eq(8.3lem1.2bis)
\ee
%and
%\be
%\pi^\circ_n\left(\del\left(\cup_{1\leq l\leq L^{\star}}C^{\star}_{n,l}(\rho)\right)\right)
%\leq  n^{-c_{\star}+2}2^{-n\rho}(1+o(1)).
%\Eq(8.3lem1.2)
%\ee
and, for $m^{\star}_{n,l}(x)$ as in \eqv(4.prop4.0),
\be
\sum_{1\leq l\leq L^{\star}}\sum_{x\in\del C^{\star}_{n,l}(\rho)}\pi^\circ_n(x)m^{\star}_{n,l}(x)
%\pi^\circ_n\left(\del\left(\cup_{1\leq l\leq L^{\star}}C^{\star}_{n,l}(\rho)\right)\right)
\leq  n^{-c_{\star}+1}2^{-n\rho}(1+o(1)).
\Eq(8.3lem1.2)
\ee
\end{lemma}

\begin{lemma}
  \TH(8.3lem2)
Assume that $c_{\star}>2$. On $\O^{\star}$, for all but a finite number of indices $n$,
\be
\pi^\circ_n\left(\del\left(\cup_{1\leq l\leq L^{\star}}C^{\star}_{n,l}\right)\right)
\leq 
 n^{-2(c_{\star}-1)}(1+\OO(n^{-(c_{\star}-1)})).
 \Eq(8.3lem2.2)
\ee
\end{lemma}

\begin{proof}[Proof of  Lemma  \thv(8.3lem2)] 
By \eqv(4.prop2.1),
% of Proposition \thv(4.prop2)
$
\pi^\circ_n\left(\del(\cup_{1\leq l\leq L^{\star}}C^{\star}_{n,l})\right)
\leq 
n|\cup_{1\leq l\leq L^{\star}}C^{\star}_{n,l}|/{|\VV^\circ_{n}|}
$. 
By \eqv(10.lem1.3) of Lemma \thv(10.lem1) and \eqv(4.prop2.0) of Proposition \thv(4.prop2), 
on $\O^{\star}$, for all but a finite number of indices $n$,
\be
n|\cup_{1\leq l\leq L^{\star}}C^{\star}_{n,l}|/{|\VV^\circ_{n}|}
=
n\left|V_n(\rho^{\star}_n)\setminus I^{\star}_n\right|/{|\VV^\circ_{n}|}
\leq 
nn^{-2c_{\star}+1}(1+\OO(n^{-(c_{\star}-1)})),
 \Eq(8.3lem2.3)
\ee
proving \eqv(8.3lem2.2). 
\end{proof}

\begin{proof}[Proof of  Lemma  \thv(8.3lem1)] Set $k_n^{\star}\equiv\max_{2\leq l\leq L^{\star}}|C^{\star}_{n,l}(\rho)|$ and let
\be
\textstyle
S_{n}(k)\equiv
\sum_{l=2}^{L^{\star}}|C^{\star}_{n,l}(\rho)|
\1_{\{|C^{\star}_{n,l}(\rho)|=k\}}
\Eq(8.3lem1.4)
\ee
be the total number of vertices that belong to sets 
$C^{\star}_{n,l}(\rho)$ that have cardinality $k$.
Note that by  \eqv(10.1.1) and \eqv(10.lem1.4)  of Lemma \thv(10.lem1),
on $\O^{\star}$, for large enough $n$,
\be
k_n^{\star}
%\equiv\max_{1\leq l\leq L^{\star}}|C^{\star}_{n,l}(\rho)|
%\leq (1+o(1))\frac{n{\log 2}}{(c_{\star}-1)\log n}.
\leq{n}/({(c_{\star}-2)\log n}).
\Eq(8.3lem1.5)
\ee
Now, on the one hand,
\be
\textstyle
|\cup_{1\leq l\leq L^{\star}}C^{\star}_{n,l}(\rho)|/|\VV^\circ_{n}|
=\frac{1}{|\VV^\circ_{n}|}\sum_{l=2}^{L^{\star}}|C^{\star}_{n,l}(\rho)|
=\frac{1}{|\VV^\circ_{n}|}\sum_{k=2}^{k_n^{\star}} S_{n}(k).
\Eq(8.3lem1.17)
\ee
On the other hand, by Proposition \thv(4.prop2),
\bea
\nonumber
&&
\textstyle\sum_{1\leq l\leq L^{\star}}\sum_{x\in\del C^{\star}_{n,l}(\rho)}\pi^\circ_n(x)m^{\star}_{n,l}(x)
\\
\Eq(8.3lem1.3)
&=&
\textstyle
\sum_{k=2}^{k_n^{\star}} \sum_{l=2}^{L^{\star}}
\1_{\{|C^{\star}_{n,l}(\rho)|=k\}}\frac{1}{|\VV^\circ_{n}|}
\sum_{x\in\del C^{\star}_{n,l}(\rho)}m^{\star}_{n,l}(x)
\leq 
\textstyle
%\sum_{k=2}^{k_n^{\star}} \sum_{l=2}^{L^{\star}}\1_{\{|C^{\star}_{n,l}(\rho)|=k\}}\frac{k}{|\VV^\circ_{n}|}
%=
\frac{1}{|\VV^\circ_{n}|}\sum_{k=2}^{k_n^{\star}} S_{n}(k),
\eea
where we used in the final inequality that by \eqv(4.prop4.0),
\bea
\textstyle
\sum_{x\in\del C^{\star}_{n,l}(\rho)}m^{\star}_{n,l}(x)
&=&
\textstyle
n^{-1}\sum_{y\in C^{\star}_{n,l}(\rho)}\sum_{x\in\del C^{\star}_{n,l}(\rho)}\1_{\{\dist(y,x)=1\}}
\\
&=&
\textstyle
n^{-1}\sum_{y\in C^{\star}_{n,l}(\rho)}|\del C^{\star}_{n,l}(\rho)\cap \del y|\leq |C^{\star}_{n,l}(\rho)|,
%\\
%&\leq &
%\textstyle
%n^{-1}\sum_{y\in C^{\star}_{n,l}(\rho)}n= |C^{\star}_{n,l}(\rho)|.
\eea
since $|\del C^{\star}_{n,l}(\rho)\cap \del y|\leq n$.
Let us now focus on the quantities $S_{n}(k)$, $2\leq k\leq {k_n^{\star}}$.
We claim that if $c_{\star}>2$ there exists a subset $\O^{\star\star}\subset\O$ with  
$\P\left(\O^{\star\star}\right)=1$ such that, 
on $\O^{\star\star}$, for all but a finite number of indices $n$,
for all $\rho^{\star}_n\leq \rho\leq1-3\rho^{\star}_n$,
%the following holds
\be
S_{n}(2)\leq  n^{-c_{\star}+1}2^{n(1-\rho)}(1+\OO(n^{-(c_{\star}-1)})),
\Eq(8.3lem1.6)
\ee
\be
S_{n}(3)\leq  n^{-2(c_{\star}-1)}2^{n(1-\rho)}(1+\OO(n^{-(c_{\star}-1)})),
\Eq(8.3lem1.6')
\ee
and, for all $4\leq k\leq k_n^{\star}$,
\be
S_{n}(k)\leq n^{-1} n^{-c_{\star}+1}2^{n(1-\rho)}(1+\OO(n^{-(c_{\star}-1)})).
 \Eq(8.3lem1.7)
 \ee
%Note that this shows that the sum in \eqv(8.3lem1.3) is dominated by $S_{n}(2)$. 

%In order to prove this claim
We first prove  \eqv(8.3lem1.6).  For this let us introduce the variables
$\chi^{\rho}(x)\equiv\1_{\left\{w_n(x)\geq r_n(\rho)\right\}}$,
$\chi_n^{\star,\rho}(x)\equiv\1_{\left\{r_n(\rho^{\star}_n)\leq w_n(x)<r_n(\rho) \right\}}$,
and
$\chi_n(x)\equiv\1_{\left\{w_n(x)\geq r_n(\rho^{\star}_n)\right\}}$.
They are Bernoulli r.v.'s~with 
$
\P\left(\chi^{\rho}(x)=1\right)=2^{-\rho n}
$,
$
\P\left(\chi_n(x)=1\right)
%= q_n^{-1}(\rho_n)=2^{-\rho^{\star}_n n}
=n^{-c_{\star}}
$,
and
$
\P\left(\chi_n^{\star,\rho}(x)=1\right)
%= q_n^{-1}(\rho_n)=2^{-\rho^{\star}_n n}
=n^{-c_{\star}}-2^{-\rho n}
$
respectively, that inherit the independence of the variables $(w_n(x), x\in\VV_n)$.
We then may write
$
S_{n}(2)=S^0_{n}(2)+S^1_{n}(2)
$
where, for $\GG_2$ as in \eqv(7.lem4.2) (see also the paragraph above \eqv(7.lem4.5)),
\bea
\Eq(8.3lem1.8)
S^0_{n}(2)&\equiv&\textstyle
\sum_{\CC=\{x,y\}\in \GG_2}(Y_n(x,y)+Y_n(y,x)),
%\sum_{x\in\VV_n}\sum_{y:(x,y)\in\EE_n}Y_n(x,y),
\\
\Eq(8.3lem1.9)
S^1_{n}(2)&\equiv&\textstyle\sum_{\CC=\{x,y\}\in \GG_2}Z_n(x,y),
%\sum_{x\in\VV_n}\sum_{y:(x,y)\in\EE_n}Z_n(x,y),
\eea
where
\bea
Y_n(x,y)&\equiv&\textstyle\chi^{\rho}_n(x)\chi^{\star,\rho}_n(y)\prod_{z\in(\del x\cup\del y)\setminus \{x,y\}}(1-\chi_n(z)),
\Eq(8.3lem1.10)
\\
Z_n(x,y)&\equiv&\textstyle\chi^{\rho}_n(x)\chi^{\rho}_n(y)
\prod_{z\in(\del x\cup\del y)\setminus \{x,y\}}(1-\chi_n(z)).
\Eq(8.3lem1.11)
\eea
To bound the sums \eqv(8.3lem1.8) and \eqv(8.3lem1.9)
we proceed as in the proof of \eqv(7.lem5.11) (see \eqv(7.lem5.8)-\eqv(7.lem5.11)).
Namely, we decompose  $\GG_2$ into $\GG_2=\cup_{1\leq j\leq n}\cup_{1\leq i\leq 4}\GG_2^{j,i}$, where the  $\GG_2^{j,i}$'s are defined in \eqv(7.lem5.10'), and use Bennett's bound  \eqv(7.lem5.7) to estimate the sum over each $\GG_2^{j,i}$.
Doing this we  readily get that
$
\E S^0_{n}(2)=n(n^{-c_{\star}}-2^{-\rho n})2^{n(1-\rho)}(1-n^{-c_{\star}})^{2(n-1)}
$
and
%Observing that 
%$
%\E(Y_n(x,y))^2=\E Y_n(x,y)
%$ 
%we readily get
\be
\P\left(|S^0_{n}(2)-\E S^0_{n}(2)|\geq 2n\sqrt{\E S^0_{n}(2)}\right)\leq 4ne^{-n}.
\Eq(8.3lem1.12)
\ee
Similarly, 
$
\E S^1_{n}(2)=n2^{n(1-2\rho)}(1-n^{-c_{\star}})^{2(n-1)}
$
and for all $\rho^{\star}_n\leq \rho\leq (1-4\rho^{\star}_n)/2$,
\be
\P\left(|S^1_{n}(2)-\E S^1_{n}(2)|\geq 2n\sqrt{\E S^1_{n}(2)}\right)\leq 4ne^{-n}.
\Eq(8.3lem1.13)
\ee
%As expected this quickly decays when $\rho>1/2$. 
%simply observe that $\E(Z_n(x,y)Z_n(x,z))=0$ whenever $y\neq z$, $\E(Z_n(x,y))^2=\E Z_n(x,y)$ 
For $\rho> (1-4\rho^{\star}_n)/2$ we simply use that by    Tchebychev's first order order inequality,
\be
\P\left(S^1_{n}(2)\geq 2^{-n\rho/2}\E S^0_{n}(2)\right)
\leq
 2^{-n\rho/2}.
\Eq(8.3lem1.13')
\ee
From the assumptions that 
 $\rho^{\star}_n\leq \rho\leq1-3\rho^{\star}_n$ and $c_{\star}>1$
 it then  immediately follows that Eq.~\eqv(8.3lem1.6) holds true, and this with a probability larger than $1-c_0ne^{-c_1n}$ for some constants $0<c_0,c_1<\infty$. Thus, by Borel-Cantelli Lemma, it holds on a subset of $\O$ of full measure, for all but a finite number of indices $n$.

One proves \eqv(8.3lem1.6') in a similar way. We skip the (simple but lenghty) details.
% of this long but simple proof. 
 
When $4\leq k\leq k_n^{\star}$ we do not need such a refined control on $S_n(k)$: we simply write
\be
\textstyle
S_n(k)\leq k\sum_{B\subset\VV_n}
\1_{\{\exists x\in B : \chi^{\rho}_n(x)\prod_{y\in B\setminus x}\chi^{\star,\rho}_n(y)=1\}}
\prod_{z\in\del B}(1-\chi_n(z)),
\Eq(8.3lem1.14)
\ee
where the sum is over all subsets $B\subset\VV_n$ such that $|B|=k$, and such that the graph $G(B)$ is connected.
Since the number of such sets is bounded above by $(k-1)!n^{k-1}2^n$, 
$
\E S_n(k)\leq k!n^{k-1} n^{-c_{\star}(k-1)}2^{n(1-\rho)}
$,
and a first order order Tchebychev inequality yields
\be
%\P\left(S^1_{n}(k)\geq (k_n^{\star})^{-1}\E S^0_{n}(2)\right)\leq  k! k_n^{\star}n^{-(c_{\star}-1)(k-2)}
\P\left(S_{n}(k)\geq n^{-1}\E S^0_{n}(2)\right)\leq  k! n n^{-(c_{\star}-1)(k-2)}.
\Eq(8.3lem1.15)
\ee
One easily checks that if $c_{\star}> 2$ then, for all $m\geq 3$ and all $K\leq n$,
\be
\textstyle
\sum_{k=m}^{K}  k! n^{-(c_{\star}-1)(k-2)}\leq (m!+1)n^{-(m-2)(c_{\star}-1)}.
\Eq(8.3lem1.sum)
\ee
%$
%\sum_{k=4}^{k_n^{\star}}  k! n n^{-(c_{\star}-1)(k-2)}
%\leq 8 n^{-2(c_{\star}-1)+1}
%$.
Therefore 
%\be
%%%%\P\left(S^1_{n}(k)\geq (k_n^{\star})^{-1}\E S^0_{n}(2)\right)\leq  k k_n^{\star}n^{-(c_{\star}-1)(k-2)}
$
\P\left(\cup_{4\leq k\leq k_n^{\star}} \left\{S_{n}(k)\geq n^{-1}\E S^0_{n}(2)\}\right\}\right)\leq  25 n^{-2(c_{\star}-1)+1}
$,
%\Eq(8.3lem1.16)
%\ee
which is summable when $c_{\star}> 2$. By Borel-Cantelli Lemma we conclude that
on a subset of $\O$ of full measure, for all but a finite number of indices $n$, 
\eqv(8.3lem1.7) holds true for all $4\leq k\leq k_n^{\star}$.
 This concludes the proof of the claim \eqv(8.3lem1.6)-\eqv(8.3lem1.7).
 %  $S_{n}(k)$, $2\leq k\leq {k_n^{\star}}$.

Now, by \eqv(8.3lem1.6)-\eqv(8.3lem1.7) and \eqv(4.prop2.0),
on $\O_3\equiv\O^{\star}\cap\O^{\star\star}$, for all large enough $n$, 
\be\textstyle
\frac{1}{|\VV^\circ_{n}|}\sum_{k=2}^{k_n^{\star}} S_{n}(k)
\leq (1+o(1))(1+(k_n^{\star}/n)) n^{-c_{\star}+1}2^{-n\rho}
=n^{-c_{\star}+1}2^{-n\rho}(1+o(1)),
\Eq(8.3lem1.16)
\ee
where the last equality follows from \eqv(8.3lem1.5). 
Inserting  \eqv(8.3lem1.16) in \eqv(8.3lem1.17) and in \eqv(8.3lem1.3) yields  \eqv(8.3lem1.2bis) and \eqv(8.3lem1.2), respectively. The proof of Lemma \thv(8.3lem1) is done.
\end{proof}

\subsection{Elementary properties of the chains $J_n^\dagger$ and $J_n^\circ$.}
 \label{8.0}

For easy reference we gather here a few elementary properties of
the  chains $J_n^\dagger$ and $J_n^\circ$.
%that are needed repeatedly.  
We state them without proof: recalling that
 $J_n^\circ(i)\equiv J_n(T_{n,i}^\circ)$ and $J_n^\dagger(i)\equiv J_n(T_{n,i}^\dagger)$ they 
  are immediate consequences from 
%their definition and     those of  
the definitions of the sequences $\bigl(T^\dagger_{n,j}\bigr)$ and $\bigl(T^\circ_{n,j}\bigr)$
 (see \eqv(2.1.3)-\eqv(2.2.6)).
 % and  \eqv(2.2.5). 

\begin{lemma}
\TH(8.0lem1)
To each $j\geq 0$ there corresponds a unique $i\leq j$ such that:
\item(i) 
$J_n^\dagger(j)\notin\VV^\circ_n 
\Leftrightarrow
T^\dagger_{n,j-1}=T_{n,i-1}^\circ<T^\dagger_{n,j}=T_{n,i-1}^\circ+1<T^\dagger_{n,j+1}=T_{n,i}^\circ
%\Eq(8.2.1)
$,
\item(ii) 
$J_n^\dagger(j)\in\VV^\circ_n \Leftrightarrow
T^\dagger_{n,j}=T_{n,i}^\circ.
%\Eq(8.2.1)
$
\end{lemma}
%In  case (i) there exists $1\leq l\leq L^{\star}$ such that $J_n^\dagger\in C^{\star}_{n,l}$ 
%and $J_n^\dagger(j-1)\in \del C^{\star}_{n,l}$,  $J_n^\dagger(j+1)\in\del  C^{\star}_{n,l}$.

%provide two possible descriptions of $\{J_n^\dagger(j)\in C^{\star}_{n,l}\}$

From Lemma \thv(8.0lem1), (i),  we derive two descriptions of the event 
$\{J_n^\dagger(j)\in C^{\star}_{n,l}\}$,  $j>0$, $1\leq l\leq L^{\star}$.
%From Lemma \thv(8.0lem1), (i),   we deduce two 
%%possible 
%descriptions 
%%alternative representations 
%of the event 
%$\{J_n^\dagger(j)\in C^{\star}_{n,l}\}$,  $j>0$, $1\leq l\leq L^{\star}$.
The first consists in saying
% states 
that a visit of $J_n^\dagger$ to $C^{\star}_{n,l}$ must be immediately preceded and followed by a visit to  $\del C^{\star}_{n,l}$. 
%As a first corollary to Lemma \thv(8.0lem1), (i), a visit of $J_n^\dagger$ to $C^{\star}_{n,l}$ must be immediately preceded and followed by a visit to  $\del C^{\star}_{n,l}$:
%%%%%%%%%%%%%%%%%%%
%
% preuve : ajouter et et retrancher des evenements complentaires tels que les "sommes fassent 1"
%
%%%%%%%%%%%%%%%%%%%%
\begin{corollary} 
\TH(8.0cor1)
$
\{J_n^\dagger(j)\in C^{\star}_{n,l}\}=
\{J_n^\dagger(j-1)\in \del C^{\star}_{n,l}, J_n^\dagger(j)\in C^{\star}_{n,l}, J_n^\dagger(j+1)\in \del C^{\star}_{n,l}\}
$.
\end{corollary}
%The next corollary provides  provide two possible descriptions of $\{J_n^\dagger(j)\in C^{\star}_{n,l}\}$
The second expresses the fact that  when $J_n^\dagger(j)$ enters $C^{\star}_{n,l}$, $J_n^\circ(i)$ straddles over it.
\begin{corollary}
\TH(8.0cor2)
To each $j\geq 0$ there corresponds a unique $i\leq j$ such that
\be
\{J_n(T^\dagger_{n,j})\in C^{\star}_{n,l}\}
=
\{J_n(T_{n,i-1}^\circ)\in \del C^{\star}_{n,l}, J_n(T_{n,i}^\circ)\in \del C^{\star}_{n,l}\}.
%\Eq(8.0cor2.1)
\ee
\end{corollary}
Note finally that by Lemma \thv(8.0lem1), (ii),
%%%%%%as a corollary to Lemma \thv(8.0lem1), (ii), 
the chain $J_n^\dagger$ observed only when it visits  $\VV_n^{\circ}$
is nothing but the chain $J_n^\circ$ itself:
\begin{corollary} 
\TH(8.0cor3)
$
\{J_n^\dagger(j) : \exists i >0\,\,\text{s.t.}\,\, T^\dagger_{n,j}=T_{n,i}^\circ, j\geq0\}
\stackrel{d}{=}\{J_n^\circ(i), i\geq 0\}
$.
\end{corollary}

\subsection{Proof of Theorem \thv(2.theo4).}
 \label{8.2}
 Theorem \thv(2.theo4) is a rough estimate.
 By \eqv(2.3.0),
 \be
0\leq k^\dagger_n(t)-k^\circ_n(t)=
% \lfloor a_n t\rfloor =
\textstyle\sum_{j=0}^{k^\dagger_n(t)-1}\1_{\{J_n^\dagger(j)\notin \VV^\circ_n\}}.
\Eq(8.2.0)
 \ee
We now want to replace the chain $J_n^\dagger$ and the quantity $k^\dagger_n(t)$  in the right hand side of \eqv(8.2.0) by, respectively, $J_n^\circ$
and $k^\circ_n(t)$. Note that by Corollary \thv(8.0cor1),
%each  visit $J_n^\dagger(j)$ to $\VV_n\setminus \VV^\circ_n$  must be preceded by a visit of $J_n^\dagger(j)$ to $\VV^\circ_n$.
%More precisely, 
for each $j\geq 1$,
\be
\{J_n^\dagger(j)\notin \VV^\circ_n\}
=
\cup_{1\leq l\leq L^{\star}}\{J_n^\dagger(j)\in C^{\star}_{n,l}\}
\subseteq
\{J_n^\dagger(j-1)\in \del(\cup_{1\leq l\leq L^{\star}}C^{\star}_{n,l})\}.
%, J_n^\dagger(j+1)\in \del C^{\star}_{n,l}\}
\Eq(8.2.2)
\ee
From this and the fact that 
$J_n^\dagger(0)=J_n^\circ(0)\in \VV^\circ_n$ (indeed $J_n^\dagger$ starts in $\pi^\circ_n$),
we deduce that, 
\bea
\Eq(8.2.1)
\textstyle\sum_{j=0}^{k^\dagger_n(t)-1}\1_{\{J_n^\dagger(j)\notin \VV^\circ_n\}}
&\leq &
\textstyle\sum_{j=1}^{k^\dagger_n(t)-1}\1_{\{J_n^\dagger(j-1)\in \del(\cup_{1\leq l\leq L^{\star}}C^{\star}_{n,l})\}}
\\
&\stackrel{d}{=}&
\textstyle\sum_{i=1}^{k^\circ_n(t)-1}\1_{\{J_n^\circ(i-1)\in \del(\cup_{1\leq l\leq L^{\star}}C^{\star}_{n,l})\}},
\Eq(8.2.3)
\eea
where the last equality follows from Corollary \thv(8.0cor3) and the definition of  $k^\circ_n(t)$ (see \eqv(2.3.3)).
%  and $J_n^\circ$  (see \eqv(2.3.3) and Subsection \thv(2.1) respectively).
It remains to bound the last sum in \eqv(8.2.3). Since 
$k^\circ_n(t)=\lfloor a_n t\rfloor$ is deterministic, a first order Tchebychev inequality entails that for all
$c_\circ>0$,
% $\e_n>0$, 
\be
\nonumber
%\hspace{-9pt}
P_{\pi^\circ_n}\left(\textstyle\sum_{i=1}^{\lfloor a_n t\rfloor-1}\1_{\{J_n^\circ(i-1)\in \del(\cup_{1\leq l\leq L^{\star}}C^{\star}_{n,l})\}}\geq
n^{-c_\circ}
 %\e_n
 \lfloor a_n t\rfloor\right)
\leq 
%\e_n^{-1}
n^{c_\circ}
\pi^\circ_n\left(\del(\cup_{1\leq l\leq L^{\star}}C^{\star}_{n,l})\right).
%\Eq(8.2.4)
\ee
Inserting \eqv(8.3lem2.2) of Lemma \thv(8.3lem2) in the right hand side above,
% of \eqv(8.2.4), 
and combining the resulting bound with 
\eqv(8.2.0) and \eqv(8.2.3), we get  that  on $\O^{\star}$,  for all but a finite number of indices $n$,
%By \eqv(8.3lem2.2) of Lemma \thv(8.3lem2), on $\O^{\star}$, for all but a finite number of indices $n$,
%\be
%\pi^\circ_n\left(\del(\cup_{1\leq l\leq L^{\star}}C^{\star}_{n,l})\right)
%\leq 
%n^{-2(c_{\star}-1)}(1+\OO(n^{-(c_{\star}-1)})).
%\Eq(8.2.5)
%\ee
%Choosing 
%$
%\e_n=n^{-c_\circ}
%$
%for some constant $c_\circ>0$, it follows from \eqv(8.2.0) and \eqv(8.2.3)-\eqv(8.2.5) that
%on $\O^{\star}$, for all $n$ large enough,
\be
P_{\pi^\circ_n}\left(
k^\dagger_n(t)\geq k^\circ_n(t)\left(1+ n^{-c_\circ}\right)
\right) 
\leq 
n^{-2(c_{\star}-1)+c_\circ}(1+\OO(n^{-(c_{\star}-1)})).
\Eq(8.2.6)
\ee
%This readily implies \eqv(2.theo4.1) and concludes the proof of Theorem \thv(2.theo4).\endproof
This readily implies the claim of Theorem \thv(2.theo4).\endproof

%%%%%%%%%%%%%%%%%%%%%%%%%%%%%%%%%%%%%%%%%%%%%%
%
% Remark refined versions of these very rough arguments and bounds will be given in Section ........
%
%%%%%%%%%%%%%%%%%%%%%%%%%%%%%%%%%%%%%%%%%%%%%%

%%%%%%%%%%%%%%%%%%%%%%%%%%%%%%%%%%%%%%%%%%%%%%
%
%\begin{remark}
%Using a refinement of this argument we can improve Theorem \thv(2.theo4), namely we can prove that
%for any constant $c_\circ>0$, under the assumption that $c_{\star}>2$,
%on $\O_3$, for all $n$ large enough,
%$
%P_{\pi^\circ_n}\left(
%k^\dagger_n(t)\geq k^\circ_n(t)\left(1+ n^{-c_\circ}\right)
%\right) 
%\leq 
%n^{-2c_{\star}+1+c_\circ}(1+\OO(n^{-(c_{\star}-1)}))
%$.
%%%%% see file Theorem_thv(2.theo4)_section8.text
%\end{remark}
%
%%%%%%%%%%%%%%%%%%%%%%%%%%%%%%%%%%%%%%%%%%%%%%

\subsection{Proof of Theorem \thv(2.theo3).}
 \label{8.1}

%Recall  that, for  $0<\varepsilon<1$ given,  $a_n\equiv\left[\pi_n^\circ(\TT_n^\circ(\varepsilon))\right]^{-1}$ and $c_n=r_n(\varepsilon)$.

 By definition of the Skorohod topology on $D[0,\infty)$,  
it is enough to show this result with 
$\rho_\infty$ replaced by $\rho_r$, the Skorohod metric on $D[0,r]$, for $r >0$ arbitrary. 
Choosing $r=1$ for convenience we get
\be
\PP_{\pi_n^\circ}\left(\rho_1\bigl(S_n(\cdot),S_n^\circ(\cdot)\bigr)>\e\right)
\leq
\textstyle
\PP_{\pi_n^\circ}\bigl(\sup_{0\leq t\leq 1}\wh S_n(t)>\e\bigr).
\Eq(8.1.1)
\ee
%where we used that $\wh S_n$  is nondecreasing. 
Theorem \thv(2.theo3) then is an immediate consequence of the lemma below.

\begin{lemma} 
  \TH(8.3lem4)
Assume that $c_{\star}>2$ and that $\b> \b_c(\varepsilon)$. Then
$\P$-almost surely, for all $\e>0$,
\be
\textstyle
\limsup_{n\rightarrow\infty}\PP_{\pi_n^\circ}\bigl(\sup_{0\leq t\leq 1}\wh S_n(t)>\e\bigr)=0.
\Eq(8.3lem4.1)
\ee
\end{lemma}

\begin{proof}[Proof of  Lemma  \thv(8.3lem4)] Since $\wh S_n$  is nondecreasing
\be
\Eq(8.3lem4.2)
\textstyle
\PP_{\pi_n^\circ}\bigl(\sup_{0\leq t\leq 1}\wh S_n(t)>\e\bigr)
\leq 
\PP_{\pi_n^\circ}\bigl(\wh S_n(1)>\e\bigr).
\ee
Introducing the event
\be
\AA\equiv\left\{\forall_{0\leq j\leq k^\dagger_n(1)-1}\forall_{1\leq l\leq L^{\star}}
\wh\L_n^\dagger(J_n^\dagger(j))\1_{\{J_n^\dagger(j)\in C^{\star}_{n,l}\}}\leq n
\bar\varrho_{n,l}(0)
%, J^\dagger_n(j)\in C^{\star}_{n,l}
\right\}
\Eq(8.1.2)
\ee
we have, by Corollary \thv(9.cor1), that on $\O_0\cap\O^{\star}$, for all but a finite number of indices $n$,
\be
\PP_{\pi_n^\circ}\bigl(\wh S_n(1)>\e\bigr)
\leq   e^{-n}+n^{-2(c_{\star}-1)+c_\circ}+
\PP_{\pi_n^\circ}\bigl(\wh S_n(1)>\e, \AA\bigr).
\Eq(8.1.3)
\ee
where $c_\circ>0$ is arbitrary.
From the definitions \eqv (2.3.12) and \eqv(2.2.8) of $\wh S_n$ and  $\L_n^\dagger(J_n^\dagger(i))$, and since 
$\L_n^\dagger(J_n^\dagger(i))$ is non zero if and only if $J_n^\dagger(i)\in\cup_{1\leq l\leq L^{\star}}C^{\star}_{n,l}$,
%we have that on  $\AA$,
we see that on  $\AA$,
\bea
\Eq(8.1.4)
\wh S_n(1)
%&=&c_n^{-1}\textstyle\sum_{j=0}^{\lfloor a_n \rfloor-1}\wh\L_n^\dagger(J_n^\dagger(j))
%\1_{\{
%J_n^\dagger(j-1)\in \del(\cup_{1\leq l\leq L^{\star}}C^{\star}_{n,l})
%\}}
%\1_{\{
%J_n^\dagger(j)\in \cup_{1\leq l\leq L^{\star}}C^{\star}_{n,l}
%\}}
%\\
&=&c_n^{-1}\textstyle
\sum_{l=1}^{L^{\star}}
\sum_{j=1}^{k^\dagger_n(1)-1}
\wh\L_n^\dagger(J_n^\dagger(j))
\1_{\{
J_n^\dagger(j)\in C^{\star}_{n,l}
\}}
\\
\Eq(8.1.5)
&\leq& c_n^{-1}\textstyle
\sum_{l=1}^{L^{\star}}
\sum_{j=1}^{k^\dagger_n(1)-1}
n\bar\varrho_{n,l}(0)
\1_{\{
J_n^\dagger(j)\in C^{\star}_{n,l}
\}}.
%\\
%\Eq(8.1.6)
%&\leq&
%nc_n^{-1}\textstyle
%\sum_{l=1}^{L^{\star}}\bar\varrho_{n,l}(0)
%\sum_{i=0}^{k^\circ_n(1)-1}
%\1_{\{
%J_n^\circ(i)\in \del C^{\star}_{n,l}
%\}},
\eea
%where \eqv(8.1.6) follows from \eqv(8.1.5) just as \eqv(8.2.3) followed from \eqv(8.2.1). 
Therefore,
\be
\PP_{\pi_n^\circ}\bigl(\wh S_n(1)>\e, \AA\bigr)
\leq
\PP_{\pi_n^\circ}\bigl(nc_n^{-1}\textstyle
\sum_{l=1}^{L^{\star}}\bar\varrho_{n,l}(0)
\sum_{j=1}^{k^\dagger_n(1)-1}
\1_{\{
J_n^\dagger(j)\in C^{\star}_{n,l}
\}}
>\e
\bigr).
\Eq(8.1.7)
\ee
%This bound is not so bad since the probability to find $C^{\star}_{n,l}$ starting from a point of its boundary is...
%
%Clearly, for each $1\leq l\leq L^{\star}$ the $\wh\L^{\star}_{n,l}(j)$'s form a family of i.i.d. random variables, which are independent  of one another for different values of $l$.  

The problem we still face is that the quantity $\bar\varrho_{n,l}(0)$ appearing in \eqv(8.1.7) can be very large 
%and in particular, can be much larger than $c_n$. 
compared to $c_n$. 
However, sets $C^{\star}_{n,l}$
such that this happens will typically not be visited. 
More precisely, for $C^{\star}_{n,l}(\rho)$ as in \eqv(8.3.0), one may choose the  
parameter $0<\rho<1$ in a such a way that the event
\be
\wt\AA\equiv\left\{\forall_{1\leq j\leq k^\dagger_n(1)-1}
%\forall_{1\leq l\leq L^{\star}}
J_n^\dagger(j)\notin\left(\cup_{1\leq l\leq L^{\star}}C^{\star}_{n,l}(\rho)\right)
\right\},
\Eq(8.1.8)
\ee
has probability close to one. 
%has probability going to one, as we now establish.
%where  $8\rho^{\star}_n<\rho<1-8\rho^{\star}_n$ is a  parameter to be chosen.
Indeed
\bea
\Eq(8.1.9')
\PP_{\pi_n^\circ}\bigl(\wt\AA^c\,\bigr)
=P_{\pi_n^\circ}\bigl(\wt\AA^c\,\bigr)
&=&
\textstyle
\sum_{1\leq l\leq L^{\star}}E_{\pi_n^\circ}
\sum_{j=0}^{k^\dagger_n(1)-1}\1_{\{J_n^\dagger(j)\in C^{\star}_{n,l}\}}
\\
\Eq(8.1.10')
&=&
\textstyle
\sum_{1\leq l\leq L^{\star}}E_{\pi_n^\circ}
\sum_{i=1}^{k^\circ_n(1)-1}\1_{\{J_n^\circ(i-1)\in \del C^{\star}_{n,l}, J_n^\circ(i)\in \del C^{\star}_{n,l}\}}
\eea
where \eqv(8.1.10') follows from Corollary \thv(8.0cor2).
%and $k^\circ_n(1)=\lfloor a_n \rfloor $ is deterministic. 
Note that for all  $x\in \del C^{\star}_{n,l}$, 
\be
\Eq(8.1.10'')
\hspace{-1pt}
P_{\pi^\circ_n}\left(\textstyle J_n^\circ(i-1)=x,J_n^\circ(i)\in \del C^{\star}_{n,l}\right)
=\pi^\circ_n(x)P_{x}(\textstyle J_n^\circ(1)\in \del C^{\star}_{n,l})
=\pi^\circ_n(x)m^{\star}_{n,l}(x),
\ee
where  $m^{\star}_{n,l}(x)$ is defined in \eqv(4.prop4.0).
Inserting this in \eqv(8.1.10'), it follows from \eqv(8.3lem1.2) of Lemma \thv(8.3lem1)
% with $\rho=\rho^{\star}_n$, 
that on $\O_3$, for all but a finite number of indices $n$,
%\eqv(8.2.4) is bounded above by
%%%%%%%%%%%%%%%%%%%%%%%%%%%%%%%%%%%%%%%%%%%%%%%%%%%%%%%%%%%
%                                                                            debut
%%%%%%%%%%%%%%%%%%%%%%%%%%%%%%%%%%%%%%%%%%%%%%%%%%%%%%%%%%%
\bea
\PP_{\pi_n^\circ}\bigl(\wt\AA^c\,\bigr)
&\leq&
\textstyle k^\circ_n(t)\sum_{1\leq l\leq L^{\star}}
\sum_{x\in\del C^{\star}_{n,l}(\rho)}\pi^\circ_n(x)m^{\star}_{n,l}(x)
\\
&\leq& k^\circ_n(t)n^{-c_{\star}+1}2^{-n\rho}(1+o(1))
= n2^{-n\rho}2^{n\varepsilon_n-n\rho^{\star}_n}(1+o(1)).
\Eq(8.1.9)
\eea
where we wrote
$
\varepsilon_n\equiv\frac{\log a_n}{n\log 2}
$; 
thus by \thv(1.theo1.M1), $\lim_{n\rightarrow\infty}\varepsilon_n=\varepsilon$, $0<\varepsilon<1$.
%where we inserted the definition of $a_n$.
Assume from now on that $\o\in\O_3$ 
and take $\rho\equiv\varepsilon_n-\rho^{\star}_n/2$. Then
%$
%\PP_{\pi_n^\circ}^\circ\bigl(\wt\AA^c\,\bigr)\leq n2^{-n\rho^{\star}_n/2}(1+o(1))
%$
\bea
\nonumber
\hspace{-6pt}\PP_{\pi_n^\circ}\bigl(\wh S_n(1)>\e, \AA\bigr)
&\leq&
\PP_{\pi_n^\circ}\bigl(\wt\AA^c\,\bigr)+\PP_{\pi_n^\circ}\bigl(\wh S_n(1)>\e, \AA, \wt \AA\,\bigr)
\\
&\leq&
n2^{-n\rho^{\star}_n/2}(1+o(1))
+ \PP_{\pi_n^\circ}\bigl(\wh\AA\,\bigr),
\Eq(8.1.10)
\eea
where, recalling from \eqv(3.0.1) that
$
V_n(\varepsilon_n-\rho^{\star}_n/2)=\left\{x\in\VV_n\mid w_n(x)\geq r_n(\varepsilon_n-\rho^{\star}_n/2)\right\}
$,
\be
\wh\AA\equiv\left\{
nc_n^{-1}
\textstyle
\sum_{{1\leq l\leq L^{\star} \,:\, }{C_{n,l}^\star\cap V_n(\varepsilon_n-\rho^{\star}_n/2)=\emptyset}}
%\sum_{{1\leq l\leq L^{\star}:}\atop{C_{n,l}^\star\cap V_n(\varepsilon_n-\rho^{\star}_n/2)=\emptyset}}
%%%%%%\1_{\{C_{n,l}^\star\cap\TT_n(\rho)=\emptyset\}}
\bar\varrho_{n,l}(0)
\sum_{j=1}^{k^\dagger_n(1)-1}
\1_{\{
J_n^\dagger(j)\in C^{\star}_{n,l}
\}}
>\e
\right\}.
\Eq(8.1.11)
\ee
Again, we wish to express this event in terms of the chain $J_n^\circ$
and the quantity $k^\circ_n(t)$ rather than $J_n^\dagger$ and $k^\dagger_n(t)$. For this
note that by Corollary \thv(8.0cor2), Corollary \thv(8.0cor3),  the definition 
\eqv(2.3.0) of $k^\dagger_n(t)$ and the definition \eqv(2.3.3) of $k^\circ_n(t)$, for each $1\leq l\leq L^{\star}$,
\be
\Eq(8.1.11')
\textstyle
\sum_{j=1}^{k^\dagger_n(1)-1}\1_{\{J^\dagger_n(j)\in C^{\star}_{n,l}\}}
\stackrel{d}{=}
\sum_{i=1}^{k^\circ_n(1)-1}\1_{\{J_n^\circ(i-1)\in \del C^{\star}_{n,l}, J_n^\circ(i)\in \del C^{\star}_{n,l}\}}.
\ee
Then, by Tchebychev inequality,  \eqv(8.1.11'), and \eqv(8.1.10''),
\bea
%\Eq(8.1.12)
\PP_{\pi_n^\circ}\bigl(\wh\AA\,\bigr)
&\leq&
%\textstyle
\frac{n\lfloor a_n \rfloor }{\e c_n}
\sum_{{1\leq l\leq L^{\star}:}\atop{C_{n,l}^\star\cap V_n(\varepsilon_n-\rho^{\star}_n/2)=\emptyset}}
\max_{x\in C_{n,l}^\star}w_n(x)
\sum_{x\in \del C_{n,l}^\star}\pi^\circ_n(x)m^{\star}_{n,l}(x).\quad\quad
\Eq(8.1.13)
\eea
%%%%%%%%%%%%%%%%%%%%%%%%%%%%%%%%%%%%%%%%%
We next decompose the sum 
in \eqv(8.1.13) according to the size of $\max_{x\in C_{n,l}^\star}w_n(x)$: given $K>0$ to be chosen later
define, for $0\leq k\leq K$,
\be
\II_k\equiv
\textstyle
\Bigl\{
1\leq l\leq L^{\star} \mid
r_n\bigl(\varepsilon_n-\sfrac{k+2}{2}\rho^{\star}_n\bigr)
\leq 
\max_{x\in C_{n,l}^\star}w_n(x) 
\leq 
r_n\bigl(\varepsilon_n-\sfrac{k+1}{2}\rho^{\star}_n\bigr)
\Bigr\}.
\ee
By this and the choices of $a_n$ and $c_n$ from Theoreom \thv(1.theo1.Main), 
\eqv(8.1.13) becomes
\be
\textstyle
\PP_{\pi_n^\circ}\bigl(\wh\AA\,\bigr)\leq 
\e^{-1}n 
\left(
\sum_{0\leq k\leq K}Q_{n,k}+R_n
\right),
\Eq(8.1.15)
\ee
where
\bea
\Eq(8.1.16)
Q_{n,k}&=&
\textstyle
2^{n\varepsilon_n}r_n^{-1}(\varepsilon_n)
r_n\bigl(\varepsilon_n-\sfrac{k+1}{2}\rho^{\star}_n\bigr)
\sum_{l\in\II_k}\sum_{x\in \del C_{n,l}^\star}\pi^\circ_n(x)m^{\star}_{n,l}(x),
\\
\Eq(8.1.17)
R_n&=&
\textstyle
2^{n\varepsilon_n}r_n^{-1}(\varepsilon_n)
r_n\bigl(\varepsilon_n-\sfrac{K+2}{2}\rho^{\star}_n\bigr)
\sum_{1\leq l\leq L^{\star} }\sum_{x\in \del C_{n,l}^\star}\pi^\circ_n(x)m^{\star}_{n,l}(x).
\eea
Now, 
\bea
\nonumber
\textstyle
\sum_{l\in\II_k}\sum_{x\in \del C_{n,l}^\star}\pi^\circ_n(x)m^{\star}_{n,l}(x)
&\leq &
\textstyle
\sum_{1\leq l\leq L^{\star}}\sum_{x\in\del C^{\star}_{n,l}\bigl(\varepsilon_n-\sfrac{k+2}{2}\rho^{\star}_n\bigr)}
\pi^\circ_n(x)m^{\star}_{n,l}(x)
\\
&\leq&
\textstyle
 n^{-c_{\star}+1}2^{-n\left(\varepsilon_n-\frac{k+2}{2}\rho^{\star}_n\right)}(1+o(1))
\Eq(8.1.18)
\eea
where the last inequality is  \eqv(8.3lem1.2) of Lemma \thv(8.3lem1). Inserting \eqv(8.1.18) in \eqv(8.1.16),
\be
Q_{n,k}\leq
\textstyle
n2^{\frac{kn}{2}\rho^{\star}_n}
r_n^{-1}(\varepsilon_n)r_n\bigl(\varepsilon_n-\frac{k+1}{2}\rho^{\star}_n\bigr).
\Eq(8.1.19)
\ee
Using \eqv(A1.lem1.1), the bound 
$
\sqrt{1-x}-1\leq-\frac{1}{2}x(1+\frac{1}{4}x)
$, $0<x<1$,
and the assumption that  $\b>\b_c(\varepsilon)$, so that 
$\a(\varepsilon_n)\equiv\b_c(\varepsilon_n)/\b<1$ for large enough $n$, 
it follows from \eqv(8.1.19) that
\be
Q_{n,k}\leq
\textstyle
c_0 n2^{-n\rho^{\star}_n/\a(\varepsilon_n)}2^{-n\rho^{\star}_n(1/\a(\varepsilon_n)-1)\frac{k}{2}}
\Eq(8.1.20)
\ee
for some constant $0<c_0\equiv c_0(\varepsilon_n,\b)<\infty$.
Similarly, by \eqv(8.3lem1.2) with $\rho=\rho^{\star}_n$,
\be
\textstyle
\sum_{1\leq l\leq L^{\star} }\sum_{x\in \del C_{n,l}^\star}\pi^\circ_n(x)m^{\star}_{n,l}(x)
\leq  n^{-2c_{\star}+1}(1+o(1))
\Eq(8.1.21)
\ee
and
\be
R_n\leq
\textstyle
%2nn^{-2(c_{\star}-1)}2^{n\varepsilon_n}
%r_n^{-1}(\varepsilon_n)r_n\bigl(\varepsilon_n-\sfrac{K+2}{2}\rho^{\star}_n\bigr)
n2^{n\varepsilon_n-2n\rho^{\star}_n}
r_n^{-1}(\varepsilon_n)r_n\bigl(\varepsilon_n-\sfrac{K+2}{2}\rho^{\star}_n\bigr).
\Eq(8.1.22)
\ee
Now choose $K=\bigl\lceil 2\varepsilon_n\bigl(1-\sfrac{1}{16}\bigr)/\rho^{\star}_n\bigr\rceil$. Then 
$\sfrac{K+2}{2}\rho^{\star}_n\geq \varepsilon_n\bigl(1-\sfrac{1}{16}\bigr)$ and, using \eqv(A1.lem1.1),
\be
R_n\leq
\textstyle
n2^{n\varepsilon_n-2n\rho^{\star}_n}
r_n^{-1}(\varepsilon_n)r_n\bigl(\varepsilon_n/16\bigr)
\leq
n2^{-n\varepsilon_n/4-2n\rho^{\star}_n}
\Eq(8.1.23)
\ee
for all $\b>\b_c(\varepsilon_n)$. Inserting \eqv(8.1.20) and \eqv(8.1.23) in \eqv(8.1.15),
\be
\PP_{\pi_n^\circ}\bigl(\wh\AA\,\bigr)\leq 
%2\e^{-1}
%\left(
%c_0 \frac{n^{4-c_{\star}/\a(\varepsilon_n)}}{
%1-n^{-c_{\star}(1-\a(\varepsilon_n))/2\a(\varepsilon_n)}
%}
%+nn^{-4(c_{\star}-1)}
%\right)
%\leq
2\e^{-1}\left(c_0n^{2-c_{\star}/\a(\varepsilon_n)}+n^{-2(c_{\star}-1)}2^{-n\varepsilon_n/4}\right)
\Eq(8.1.24)
\ee
for all $n$ large enough. Finally, combining \eqv(8.1.3), \eqv(8.1.10), and  \eqv(8.1.24), we obtain that 
for all $\b>\b_c(\varepsilon)$, on $\O_0\cap\O^{\star}\cap\O_3$, for all but a finite number of indices $n$,  
\be
\nonumber
\PP_{\pi_n^\circ}\bigl(\wh S_n(1)>\e\bigr)
\leq  e^{-n}+n^{-2(c_{\star}-1)+c_\circ}+2n^{-(c_{\star}-2)/2}+ 
2\e^{-1}\left(c_0n^{2-c_{\star}/\a(\varepsilon_n)}+n^{-2(c_{\star}-1)}2^{-n\varepsilon_n/4}\right)
\Eq(8.1.25)
\ee
for all $\e>0$, where $c_\circ>0$ is arbitrarily small, and where 
$\lim_{n\rightarrow\infty}\varepsilon_n=\varepsilon$, $0<\varepsilon<1$.
This yields the claim of  Lemma  \thv(8.3lem4)
since by assumption $c_{\star}>2$. \end{proof}
%%%%%%%%%%%%%%%%%%%%%%%%%%%%%%%%%%%%%%%%%%%%%%%%%%%%%%%%%%%
%                                                                            fin
%%%%%%%%%%%%%%%%%%%%%%%%%%%%%%%%%%%%%%%%%%%%%%%%%%%%%%%%%%%

The proof of Theorem \thv(2.theo3) is now complete.%%%%%%%%%%%%%%%%%%%%%%%%%%%%%%%%%%%%%%%%%%%%%%%%%%%%%%%%%%%%

%%%%%%%%%%%%%%%%%%%%%%%%%%%%%%%%%%%%%%%%%%%%%%%%%%%%%%%%%%%%

\section{Convergence of the front end clock process: Proof of Theorem \thv(2.theo1)}
 \label{6}
 
 The proofs of Theorem \thv(2.theo1) and Theorem \thv(2.theo2) rely on a method developped
 by Durrett and  Resnick  \cite{DR78} that provides sufficient conditions for partial sum processes 
 to converge to L\'evy processes. We use their results in a  specialized form  suitable for our applications
 which is taken from \cite{G12}, where this method was first applied to
  %put to use for /implemented
  the study of clock processes in random environment; see also  \cite{BG13} where it was implemented in more generality.
 %
 %We use their results in a  specialized form  suitable for our applications
% which is taken from \cite{G12}, namely, Theorem 2.1 in Subsection 2.1. (See also \cite{G10b} and
 %\cite{BG13} for later applications of this method for the convergence of clock processes of the REM 
 %and $p$-spin SK model.)
 
 \subsection{A convergence theorem for \textsc{\textbf{fecp}}}
 \label{6.1}

%\textsc{fecp}, 

 Consider the rescaled front end clock process \eqv(2.3.11),
 \be
 \Eq(2.3.11bis)
S_n^\circ(t)=c_n^{-1}\wt S_n^\circ (\lfloor a_n t\rfloor),\quad t\geq 0.
\ee
Theorem \thv(6.theo1) below is the corner stone of the proof of Theorem \thv(2.theo1). 
It deduces convergence of $S_n^\circ$ to a subordinator from a set of four conditions which we now formulate.
%It states that $S_n^\circ$ convergences to some subordinator
% provided that the four conditions formulated below are verified. 
 Note that these conditions refer to given sequences of numbers $a_n$ and  $c_n$, 
% a given initial distribution,  $\pi^\circ_n$, 
 as well as a given realization of the random environment.
%We now come to the key step in our argument. This consists in reducing
%Conditions (A1) and (A2) of Theorem \thv(2.theo1) to (i) a\emph{ mixing condition} for the
%chain $J^\circ_n$, and (ii) a \emph{law of large numbers} for the random variables
%$Q_n$.
%As everywhere in this paper the initial distribution of $J^\circ_n$ is its invariant measure $\pi^\circ_n$.
 For $t>0$ and  $u>0$ define
\be
h^{u}_n(y)=\sum_{x\in\VV^\circ_n}p^\circ_n(y,x)\exp\{-uc_n\l_n(x)\}\,,\,\,y\in\VV^\circ_n,
\Eq(6.1.1)
\ee
and, recalling the notation $k^\circ_n(t)=\lfloor a_n t\rfloor$,
\bea
\Eq(6.1.2)
\nu_n^{{{\scriptscriptstyle{J^\circ_n}},t}}(u,\infty)
&= &
%\textstyle
\sum_{j=0}^{k^\circ_n(t)-1}h^{u}_n(J^\circ_n(j)),
\\
\Eq(6.1.3)
\s_n^{{{\scriptscriptstyle{J^\circ_n}},t}}(u,\infty)
&= &
%\textstyle
\sum_{j=0}^{k^\circ_n(t)-1}\left[h^{u}_n(J^\circ_n(j))\right]^2.
\eea

%\vskip3pt
\noindent{\bf Condition (C0).}  For all $v>0$,
\be
\textstyle
\sum_{x\in\VV^\circ_n}\pi^\circ_n(x)e^{-vc_n\l_n(x)}=o(1)\,.
\Eq(2.A0')
\ee

%\noindent\textbf{Condition (A1).} 
\noindent {\bf Condition  (C1).}
There exists a $\s$-finite measure $\nu^\circ$ on $(0,\infty)$ satisfying
$
\int_0^\infty
(1\wedge u)\nu^\circ(du)<\infty
$ 
such that, for all $t>0$ and all $u>0$,
\be
\Eq(6.1.5)
P^\circ_{\pi^\circ_n}\left(
\left|
\nu_n^{{{\scriptscriptstyle{J^\circ_n}},t}}(u,\infty)
-t\nu^\circ(u,\infty)
\right|
<\e
\right)=1-o(1)\,,\quad\forall\e>0\,.
\ee

%\noindent\textbf{Condition (A2).} 
\noindent {\bf Condition  (C2).}
 For all $u>0$ and all $t>0$,
\be
\Eq(6.1.6)
P^\circ_{\pi^\circ_n}\left(
\s_n^{{{\scriptscriptstyle{J^\circ_n}},t}}(u,\infty)
<\e
\right)=1-o(1)\,,\quad\forall\e>0\,.
\ee

\noindent {\bf Condition  (C3).}
 For all $t>0$,
\be
\Eq(6.1.7)
\lim_{\e\downarrow 0}\limsup_{n\uparrow \infty} k^\circ_n(t)
\EE^\circ_{\pi^\circ_n} \1_{\{\l_n^{-1}(J^\circ_n(0))e^\circ_0\leq c_n\e\}} c_n^{-1}\l_n^{-1}(J^\circ_n(0))e^\circ_0=0.
\ee

\begin{theorem}
\TH(6.theo1) 
Let the initial  distribution of $J^\circ_n$ be its invariant measure $\pi^\circ_n$.
 For all sequences $a_n$ and $c_n$ for which
Conditions  (C0), (C1), (C2), and (C3) are verified $\P$-almost surely,
\be
 S^\circ_n \Rightarrow_{J_1}  S^\circ_{\infty}
\Eq(6.theo1.1)
\ee
$\P$-almost surely,
where $S^\circ_{\infty}$ is the L\'evy subordinator with zero drift and L\'evy measure
$\nu^\circ$.
%Convergence holds weakly in the space
%$D([0,\infty))$ equipped with the Skorokhod $J_1$-topology.
\end{theorem}

\begin{proof} This is a restatement of Theorem 1.2 of \cite{BG13} specialized to the case where 
$\theta_n$, the ``bloc length'', is equal to one.  (Theorem 1.2 of \cite{BG13} is itself a generalization 
of Theorem 1.1 of \cite{G12} with a more workable Condition (C3).) 
%convenient, clever, practical
%however condition ... in that latter reference was unnecessarily complicated) 
\end{proof}

To verify the conditions of Theorem \thv(6.theo1) we  follow a by now well established two-step strategy that was first proposed  in \cite{G10b},
%,  \cite{G12}.
and was used later in \cite{BG13}. 
The first step consists in using  the mixing property and mean local time estimates of
Proposition \thv(4.prop1) and Proposition \thv(4.prop6), respectively, to prove an almost sure ergodic theorem for 
the quantities 
 \eqv(6.1.2) and \eqv(6.1.3).
This is done in Subsection \thv(6.2) (see Theorem \thv(6.theo2)).
It then enables us  
%to replace  \eqv(6.1.2) and \eqv(6.1.3)
% by their average over the  chain variables $J^\circ_n$, thereby reducing
to reduce 
 Conditions (C1) and (C2) of Theorem \thv(6.theo1) 
to laws of large numbers in the random environment. This second step is carried out in  Subsection \thv(6.3)
(see Proposition \thv(6.prop1)). The proof is completed in Subsection \thv(6.4).

%The proof of Theorem \thv(main.1) comes in two steps. In the first we
%use the ergodic properties of the chain $J_n$ to pass from sums along
%a chain $J_n$ to averages with respect to the invariant measure of $J_n$.

%and we next prove strong laws of large numbers for the quantities

%%%%%%%%%%%%%%%%%%%%%%%%%%%%%%%%%%%%%%%%%%%%%%%%

\subsection{An ergodic theorem for \textsc{\textbf{fecp}}}
 \label{6.2}
 
 Let $\pi^{{{\scriptscriptstyle{J^\circ_n}},t}}_n(x)$ denote the average number
of visits of $J^\circ_n$ to $x$ during the first $k^\circ_n(t)$ steps,
\be
%\textstyle
\pi^{{{\scriptscriptstyle{J^\circ_n}},t}}_n(x)
=
(k^\circ_n(t))^{-1}
%\frac{1}{k^\circ_n(t)}
\sum_{j=0}^{k^\circ_n(t)-1}\1_{\{J^\circ_n(j)=x\}}\,,\quad x\in\VV^\circ_n\,.
\Eq(6.2.1)
\ee
Then \eqv(6.1.2) and \eqv(6.1.3) can be rewritten as
\bea
\Eq(6.2.2)
\nu_n^{{{\scriptscriptstyle{J^\circ_n}},t}}(u,\infty)
%&= &\sum_{j=1}^{k^\circ_n(t)}h^{u}_n(J^\circ_n(j-1))
&= &
%\textstyle
k^\circ_n(t)\sum_{y\in\VV^\circ_n}\pi^{{{\scriptscriptstyle{J^\circ_n}},t}}_n(y)h^{u}_n(y),
\\
\Eq(6.2.3)
\s_n^{{{\scriptscriptstyle{J^\circ_n}},t}}(u,\infty)
%&= &\sum_{j=1}^{k^\circ_n(t)}\left[h^{u}_n(J^\circ_n(j-1))\right]^2
&= &
%\textstyle
k^\circ_n(t)\sum_{y\in\VV^\circ_n}\pi^{{{\scriptscriptstyle{J^\circ_n}},t}}_n(y)\left[h^{u}_n(y)\right]^2.
\eea
One readily sees, using reversibility, that
\bea
\Eq(6.2.4)
E^\circ_{\pi^\circ_n}\bigl[\nu_n^{{{\scriptscriptstyle{J^\circ_n}},t}}(u,\infty)\bigr]
\hspace{-6pt}&=&\hspace{-6pt}
k^\circ_n(t)\sum_{x\in\VV^\circ_n}\pi^\circ_n(x)h^{u}_n(x)
=
({k^\circ_n(t)}/{a_n})\nu^\circ_n(u,\infty),
\\
\Eq(6.2.5)
E^\circ_{\pi^\circ_n}\bigl[\s_n^{{{\scriptscriptstyle{J^\circ_n}},t}}(u,\infty)\bigr]
\hspace{-6pt}&=&\hspace{-6pt}
k^\circ_n(t)\sum_{x\in\VV^\circ_n}\pi^\circ_n(x)\left[h^{u}_n(x)\right]^2
=({k^\circ_n(t)}/{a_n})\s^\circ_n(u,\infty),
\eea
where 
\be
\nu^\circ_n(u,\infty)
=
\Eq(6.2.6)
%a_n\sum_{x\in\VV^\circ_n}\pi^\circ_n(x)e^{-uc_n\l_n(x)}
\frac{a_n}{\left|\VV^\circ_n\right|}\sum_{x\in\VV^\circ_n}e^{-uc_n\l_n(x)},
\ee
\be
\Eq(6.2.7)
\s^\circ_n(u,\infty)
=
\frac{a_n}{\left|\VV^\circ_n\right|}\sum_{x\in\VV^\circ_n}\sum_{x'\in\VV^\circ_n}
p_n^{\circ,2}(x,x')e^{-uc_n(\l_n(x)+\l_n(x'))}.
\ee
%\smallskip
Here $p_n^{\circ,2}(\cdot,\cdot)$ denotes the $2$-steps transition probabilities of $J^\circ_n$. Note that since
$H(x)=0$ for all $x\in \VV^\circ_n\setminus I^{\star}_n$, 
where $I^{\star}_n$ is the set of isolated vertices in the partition \eqv(10.1.4),
we have
\be
\l_n(x)=
\begin{cases}
e^{\b H_n(x)},
&\hbox{\rm if}\,\,\, x\in I^{\star}_n,\\
1,&\hbox{\rm if}\,\,\, x\in\VV^\circ_n\setminus I^{\star}_n.
\end{cases}
\Eq(6.2.8)
\ee
%recall that $H_n(x)\leq 0$
%where $I^{\star}_n$ is the set of isolated vertices in the partition \eqv(10.1.4) and $H_n(x)\leq 0$ is defined in 
%\eqv(10.1.2).

\begin{theorem}
\TH(6.theo2) 
Assume that $c_{\star}>3$. 
Let $\rho^\circ_n>0$ be a decreasing sequence satisfying $\rho^\circ_n\downarrow 0$ as $n\uparrow\infty$.
There exists a sequence of subsets $\O^{\scriptscriptstyle{\textsf{EG}}}_{n}\subset\O$ with
$
\P\left[\left(\O^{\scriptscriptstyle{\textsf{EG}}}_{n}\right)^c\right]
<
{\ell^\circ_n}/({\rho^\circ_n a_n})
%\frac{\ell^\circ_nk^\circ_n(t)}{\rho^\circ_na^2_n}\,,
$,
and such that on $\O^{\scriptscriptstyle{\textsf{EG}}}_{n}$ the following holds for all large enough $n$: 
for all $t>0$, all $u>0$, and all $\e>0$,
\smallskip
\be
P^\circ_{\pi^\circ_n}\left(\left|
\nu_n^{{{\scriptscriptstyle{J^\circ_n}},t}}(u,\infty)
-
({k^\circ_n(t)}/{a_n})\nu^\circ_n(u,\infty)
%E^\circ_{\pi^\circ_n}\bigl[\nu_n^{{{\scriptscriptstyle{J^\circ_n}},t}}(u,\infty)\bigr]
\right|\geq\e\right)
\leq
\e^{-2}[C_1 t\Theta_{n,1}(u)+ t^2\Theta_{n,2}(u)]
%\,,\quad\forall\e>0\,,
\Eq(6.theo2.1) 
\ee
\smallskip
for some constant $0<C_1<\infty$, where
\be
\Theta_{n,1}(u)\equiv
\ell^\circ_ne^{-uc_n}[1+\nu^\circ_n(u,\infty)]
+\s^\circ_n(u,\infty)
+\frac{\nu^\circ_n(2u,\infty)}{n\log n}
+\rho^\circ_n\left[\E\nu^\circ_n(u,\infty)\right]^2,
\Eq(6.theo2.2) 
\ee
\be
\Theta_{n,2}(u)\equiv
2^{-n}\left[\nu^\circ_n(u,\infty)\right]^2.
\Eq(6.theo2.2') 
\ee
Moreover, for all $t>0$, all $u>0$, and all $\e'>0$,
\be
P^\circ_{\pi^\circ_n}\left(
\s_n^{{{\scriptscriptstyle{J^\circ_n}},t}}(u,\infty)
\geq\e'\right)
\leq
%\frac{k^\circ_n(t)}{\e'\, a_n}\
\frac{t}{\e'}(1+o(1))\s^\circ_n(u,\infty).
%\,,\quad\forall\e'>0\,.
\Eq(6.theo2.3) 
\ee
\end{theorem}

%%%%%%%%%%%%%%%%%%%%%%%%%%%%%%%%%%%%%%%%%%%%%

\begin{proof}[Proof of Theorem \thv(6.theo2) ] 
The upper bound \eqv(6.theo2.3) simply results from a first order Tchebychev inequality and \eqv(6.2.5).
The proof of \eqv(6.theo2.1) is more involved. It relies on a second order Tchebychev inequality, that is,
using \eqv(6.2.4),  we bound  the left hand side of \eqv(6.theo2.1)  from above by
%\be
%%\leq
%\e^{-2}
%E\Bigl[
%k^\circ_n(t)\sum_{y\in\VV^\circ_n}\left(\pi^{{{\scriptscriptstyle{J^\circ_n}},t}}_n(y)-\pi^\circ_n(y)\right)h^{u}_n(y)
%\Bigr]^2\,.
%\Eq(6.theo2.6)
%\ee
%Expanding the r.h.s. of \eqv(6.theo2.6) yields
\be
\e^{-2}
%k^2_n
(k^\circ_n(t))^2\sum_{x\in\VV^\circ_n}\sum_{y\in\VV^\circ_n}h^{u}_n(x)h^{u}_n(y)
E^\circ_{\pi^\circ_n}\left(\pi^{{{\scriptscriptstyle{J^\circ_n}},t}}_n(x)-\pi^\circ_n(x)\right)\left(\pi^{{{\scriptscriptstyle{J^\circ_n}},t}}_n(y)-\pi^\circ_n(y)\right)\,.
\Eq(6.theo2.7)
\ee
In view of \eqv(6.2.1), setting
$
\Delta_{ij}(x,y)=P^\circ_{\pi^\circ_n}\left(J^\circ_n(i)=x, J^\circ_n(j)=y\right)-\pi^\circ_n(x)\pi^\circ_n(y)
%\,,
%\Eq(6.theo2.8)
$,
the expectation in \eqv(6.theo2.7) may be rewritten as
\be
\textstyle
%(k^\circ_n(t))^2(k^\circ_n(t))^2
E^\circ_{\pi^\circ_n}\bigl(\pi^{{{\scriptscriptstyle{J^\circ_n}},t}}_n(x)-\pi^\circ_n(x)\bigr)\bigl(\pi^{{{\scriptscriptstyle{J^\circ_n}},t}}_n(y)-\pi^\circ_n(y)\bigr)
=\sum_{i=0}^{k^\circ_n(t)-1}\sum_{j=0}^{k^\circ_n(t)-1}\Delta_{ij}(x,y).
\Eq(6.theo2.9)
\ee
For $\ell^\circ_n$ defined in \eqv(4.prop1.0) we now break the sum in the r.h.s. of \eqv(6.theo2.9) into three terms:
\bea
\nonumber
I_1^{(1)}
\hspace{-6pt}&=&\hspace{-6pt}
2
\textstyle
\sum_{0\leq i\leq k^\circ_n(t)-1}\sum_{i+\ell^\circ_n\leq j\leq k^\circ_n(t)-1}\Delta_{ij}(x,y)\,,
\\
\Eq(6.theo2.10)
I_2^{(1)}
\hspace{-6pt}&=&\hspace{-6pt}
\textstyle
\sum_{0\leq i\leq k^\circ_n(t)-1}\1_{\{i=j\}}\Delta_{ij}(x,y)\,,
\\
\nonumber
I_3^{(1)}
\hspace{-6pt}&=&\hspace{-6pt}
\textstyle
2\sum_{0\leq i\leq k^\circ_n(t)-1}\sum_{i<j<i+\ell^\circ_n}\Delta_{ij}(x,y)\,.
\eea
Consider first $I_1^{(1)}$. By Proposition \thv(4.prop1),
%$\d_n=2^{-n}$,
\be
\Eq(6.theo2.10')
I_1^{(1)}
\leq
\d_n(k^\circ_n(t))^2\pi^\circ_n(x)\pi^\circ_n(y)
\leq
2^{-n}(k^\circ_n(t))^2\pi^\circ_n(x)\pi^\circ_n(y).
\ee
Turning to the term $I_2^{(1)}$, we have,
\be
\Eq(6.theo2.11)
I_2^{(1)}
=\textstyle\sum_{1\leq i\leq k^\circ_n(t)}\Delta_{ii}(x,x)\1_{\{x=y\}}
%k^\circ_n(t)\left[P^\circ_{\pi^\circ_n}\left(J^\circ_n(i)=x\right)-\pi^2_n(x)\right]\1_{\{x=y\}}
=k^\circ_n(t)\pi^\circ_n(x)(1-\pi^\circ_n(x))\1_{\{x=y\}},
\ee
where we used that $P^\circ_{\pi^\circ_n}(J^\circ_n(i)=x)=\pi^\circ_n(x)$.
Finally,
\bea
\nonumber
I_3^{(1)}
\hspace{-6pt}&\leq&\hspace{-6pt}
\textstyle
2\sum_{i=0}^{k^\circ_n(t)-1}\sum_{l=1}^{\ell^\circ_n-1}P^\circ_{\pi^\circ_n}\left(J^\circ_n(i)=x, J^\circ_n(i+l)=y\right)
%\\
%\nonumber
%&\leq& 
%\textstyle
%2\sum_{i=0}^{k^\circ_n(t)-1}\sum_{l=1}^{\ell^\circ_n-1}
%P^\circ_{\pi^\circ_n}\left(J^\circ_n(i)=x\right)P^\circ_{\pi^\circ_n}\left(J^\circ_n(i+l)=y\mid J^\circ_n(i)=x\right)
%\quad\quad
\\
\Eq(6.theo2.12)
\hspace{-6pt}&=&\hspace{-6pt} 
\textstyle
2k^\circ_n(t)\pi^\circ_n(x)\sum_{l=1}^{\ell^\circ_n-1}p_n^{\circ,l}(x,y)
\eea
where $p_n^{\circ,l}(\cdot,\cdot)$ denote the $l$-steps transition probabilities of $J^\circ_n$.
Combining our bounds on $(\overline{I}),I_2^{(1)}$, and $I_3^{(1)}$ with \eqv(6.theo2.7) we get that, for all $\e>0$,
\be
P^\circ_{\pi^\circ_n}\left(\left|
\nu_n^{{{\scriptscriptstyle{J^\circ_n}},t}}(u,\infty)
-
E^\circ_{\pi^\circ_n}\bigl[\nu_n^{{{\scriptscriptstyle{J^\circ_n}},t}}(u,\infty)\bigr]
\right|\geq\e\right)
\leq
\e^{-2}[I_1^{(2)}+I_2^{(2)}+I_3^{(2)}]\,,
\Eq(6.theo2.13)
\ee
where
\bea\Eq(6.theo2.14)
I_1^{(2)}
\hspace{-6pt}&=&\hspace{-6pt}
\textstyle
2^{-n}(k^\circ_n(t))^2\sum_{x\in\VV^\circ_n}\sum_{y\in\VV^\circ_n}h^{u}_n(x)h^{u}_n(y)\pi^\circ_n(x)\pi^\circ_n(y)\,,\nonumber
\\
I_2^{(2)}
\hspace{-6pt}&=&\hspace{-6pt}
\textstyle
k^\circ_n(t)\sum_{x\in\VV^\circ_n}\sum_{y\in\VV^\circ_n}h^{u}_n(x)h^{u}_n(y)\pi^\circ_n(x)(1-\pi^\circ_n(x))\1_{\{x=y\}}\,,
\\
I_3^{(2)}
\hspace{-6pt}&=&\hspace{-6pt}
\textstyle
2k^\circ_n(t)\sum_{x\in\VV^\circ_n}\sum_{y\in\VV^\circ_n}h^{u}_n(x)h^{u}_n(y)\pi^\circ_n(x)\sum_{l=1}^{\ell^\circ_n-1}p_n^{\circ,l}(x,y).
\nonumber
\eea
In view of \eqv(6.2.4)-\eqv(6.2.5),
\be
\Eq(6.theo2.15)
I_1^{(2)}\leq  
2^{-n}\left({k^\circ_n(t)}/{a_n}\right)^2
\left[\nu^\circ_n(u,\infty)\right]^2,
\ee
\be
\Eq(6.theo2.15')
I_2^{(2)}\leq  
({k^\circ_n(t)}/{a_n})\s^\circ_n(u,\infty).
\ee
To deal with the third term in \eqv(6.theo2.14) note first that by \eqv(6.1.1),
\be
\sum_{y\in\VV^\circ_n}p_n^{\circ,l}(x,y)h^{u}_n(y)
%=\sum_{y\in\VV^\circ_n}p_n^{\circ,l}(x,y)\sum_{z\in\VV^\circ_n}p^\circ_n(y,z)e^{-uc_n\l_n(z)}
=\sum_{z\in\VV^\circ_n}p_n^{\circ,l+1}(x,z)e^{-uc_n\l_n(z)},
\Eq(6.theo2.16)
\ee
so that
\bea\Eq(6.theo2.17)
\sum_{x\in\VV^\circ_n}\pi^\circ_n(x)h^{u}_n(x)p_n^{\circ,l+1}(x,z)
\hspace{-6pt}&=&\hspace{-6pt}
\sum_{y\in\VV^\circ_n}e^{-uc_n\l_n(y)}\sum_{x\in\VV^\circ_n}\pi^\circ_n(x)p^\circ_n(x,y)p_n^{\circ,l+1}(x,z)
\nonumber\\
\hspace{-6pt}&=&\hspace{-6pt}
\sum_{y\in\VV^\circ_n}e^{-uc_n\l_n(y)}\pi^\circ_n(y)p_n^{\circ,l+2}(y,z)\,,
\eea
where the last equality follows by reversibility.
%where we used that by reversibility, $\pi^\circ_n(x)p_n(x,y)=\pi^\circ_n(y)p_n(y,x)$.
Hence,
\bea
\nonumber
I_3^{(2)}
\hspace{-6pt}&=&\hspace{-6pt}
2k^\circ_n(t)\sum_{l=1}^{\ell^\circ_n-1}
\sum_{z\in\VV^\circ_n}\Big[\sum_{x\in\VV^\circ_n}\pi^\circ_n(x)h^{u}_n(x)p_n^{\circ,l+1}(x,z)\Big]e^{-uc_n\l_n(z)}\,,
\\
\nonumber
\hspace{-6pt}&=&\hspace{-6pt}
2\sum_{l=1}^{\ell^\circ_n-1}k^\circ_n(t)\sum_{z\in\VV^\circ_n}\sum_{y\in\VV^\circ_n}\pi^\circ_n(y)
e^{-uc_n(\l_n(y)+\l_n(z))}p_n^{\circ,l+2}(y,z)
\\
\Eq(6.theo2.18)
\hspace{-6pt}&\equiv&\hspace{-6pt}
2 ({k^\circ_n(t)}/{a_n})\sum_{z\in\VV^\circ_n}\sum_{y\in\VV^\circ_n}f_n(y,z)
\eea
where the last line defines $f_n(y,z)$. In view of \eqv(6.2.8), we have
\be
\Eq(6.theo2.23)
\sum_{z\in\VV^\circ_n\setminus I^{\star}_n}\sum_{y\in\VV^\circ_n\setminus I^{\star}_n}f_n(y,z)
\leq
\ell^\circ_ne^{-2uc_n},
\ee
\be
\Eq(6.theo2.24)
\sum_{z\in\VV^\circ_n\setminus I^{\star}_n}\sum_{y\in  I^{\star}_n}f_n(y,z)
=
\sum_{z\in  I^{\star}_n}\sum_{y\in\VV^\circ_n\setminus I^{\star}_n}f_n(y,z)
\leq
\ell^\circ_ne^{-uc_n}\nu^\circ_n(u,\infty),
\ee
where the equality above is reversibility. It thus remains to bound the term
\be
I^{(3)}
\equiv
2 \frac{k^\circ_n(t)}{a_n}\sum_{z\in I^{\star}_n}\sum_{y\in I^{\star}_n}f_n(y,z)
=
2 \frac{k^\circ_n(t)}{a_n}\sum_{l=1}^{\ell^\circ_n-1}[I_{1,l}^{(3)}+I_{2,l}^{(3)}],
\ee
where, distinguishing the cases $z=y$ and $z\neq y$,
\bea
\Eq(6.theo2.19)
I_{1,l}^{(3)}
\hspace{-6pt}&\equiv&\hspace{-6pt}
\sum_{z\in I^{\star}_n}a_n\pi^\circ_n(z)e^{-2uc_n\l_n(z)}p_n^{\circ,l+2}(z,z),
\\
\Eq(6.theo2.22)
I_{2,l}^{(3)}
\hspace{-6pt}&\equiv&\hspace{-6pt}
\sum_{z\in I^{\star}_n}\sum_{y\in I^{\star}_n : y\neq z}a_n\pi^\circ_n(y)
e^{-uc_n(\l_n(y)+\l_n(z))}p_n^{\circ,l+2}(y,z).
\eea
By Proposition \thv(4.prop6) we readily have that, on $\O^{\scriptscriptstyle{\textsf{SRW}}}$, 
for all but a finite number of indices $n$,
\be
\sum_{l=1}^{\ell^\circ_n-1}I_{1,l}^{(3)}
%=\sum_{z\in I^{\star}_n}a_n\pi^\circ_n(z)e^{-2uc_n\l_n(z)}\sum_{l=1}^{\ell^\circ_n-1}p_n^{\circ,l+2}(z,z)
\leq \frac{C_\circ}{n\log n}\nu^\circ_n(2u,\infty).
\Eq(6.theo2.20)
\ee
for some constant $0<C_\circ<\infty$.
The next lemma is designed to deal with \eqv(6.theo2.22).

\begin{lemma}
  \TH(6.lem1)
Let $\rho^\circ_n>0$ be a decreasing sequence satisfying $\rho^\circ_n\downarrow 0$ as $n\uparrow\infty$.
There exists a sequence of subsets $\O_{2,n}^{(3)}\subset\O$ with
$
\P\bigl(\O_{2,n}^{(3)}\bigr)\geq 1-{\ell^\circ_n}/({\rho^\circ_n a_n})
%\Eq(6.lem1.1)
$
such that on $\O_{2,n}^{(3)}$,
\be
\sum_{l=1}^{\ell^\circ_n-1}I_{2,l}^{(3)}<\rho^\circ_n\left[\E\nu^\circ_n(u,\infty)\right]^2\,.
\Eq(6.lem1.2)
\ee
\end{lemma}

\begin{proof}
By definition of $I^{\star}_n$,  $\dist(y,z)\geq 2$  for all  $y\in I^{\star}_n$ and
$z\in I^{\star}_n$ such that $y\neq z$. Thus
\bea
\Eq(6.lem1.3)
\E\left(e^{-uc_n(\l_n(y)+\l_n(z))}\1_{\{y\in I^{\star}_n, z\in I^{\star}_n\}}\right) \1_{\{x\neq z\}}
\hspace{-6pt}&\leq&\hspace{-6pt}
\E\left(e^{-uc_n(\l_n(y)+\l_n(z))}\right) \1_{\{\dist(y,z)\geq 2\}}
\quad\quad
\\
\Eq(6.lem1.4)
\hspace{-6pt}&\leq&\hspace{-6pt}
\left(\E e^{-uc_n\l_n(y)}\right)\left(\E e^{-uc_n\l_n(z)}\right)
\\
\Eq(6.lem1.5)
\hspace{-6pt}&=&\hspace{-6pt}
\left[a^{-1}_n\E\nu^\circ_n(u,\infty)\right]^2
\eea
where we used independence in the second line.
Therefore, by a first order Tchebychev inequality, for all $\eta>0$,
\bea
\Eq(6.lem1.6)
%&&\hspace{-12pt}
\textstyle
\P\bigl(\sum_{l=1}^{\ell^\circ_n-1}I_{2,l}^{(3)}\geq\eta\bigr)
%\\
%\leq &&\hspace{-12pt}
%\eta^{-1}\sum_{l=1}^{\ell^\circ_n-1}
%\sum_{z,y\in\VV^\circ_n\times \VV^\circ_n}a_n\pi^\circ_n(y)p_n^{\circ,l+2}(y,z)
%\E\left(e^{-uc_n(\l_n(y)+\l_n(z))}\1_{\{y\in I^{\star}_n, z\in I^{\star}_n\}}\right)
%\\
\leq &&\hspace{-12pt}
\frac{1}{\eta a_n}\left[\E\nu_n(u,\infty)\right]^2\sum_{l=1}^{\ell^\circ_n-1}
\sum_{y\in\VV^\circ_n}\pi^\circ_n(y)
\sum_{z\in\VV^\circ_n}p_n^{\circ,l+2}(y,z)
\\
\leq &&\hspace{-12pt}
\frac{\ell^\circ_n}{\eta a_n}\left[\E\nu^\circ_n(u,\infty)\right]^2.
\eea
The lemma now easily follows. 
\end{proof}

Gathering our bounds we conclude that under the assumptions
and with the notations of Proposition \thv(4.prop6) and Lemma \thv(6.lem1),
on
$
%\O^{\scriptscriptstyle{\textsf{EG}}}_{n}\equiv
\O^{\scriptscriptstyle{\textsf{SRW}}}\cap\O_{2,n}^{(3)}
$, 
for all but a finite number of indices $n$,
\be
\Eq(6.theo2.25)
I_3^{(2)}\leq 2 \frac{k^\circ_n(t)}{a_n}\left[
\ell^\circ_ne^{-uc_n}[1+\nu^\circ_n(u,\infty)]
+
C_\circ\frac{\nu^\circ_n(2u,\infty)}{n\log n}
+
\rho^\circ_n\left[\E\nu^\circ_n(u,\infty)\right]^2
\right]
\ee
for some constant $0<C_\circ<\infty$. 
Inserting the bounds \eqv(6.theo2.15), \eqv(6.theo2.15'), and \eqv(6.theo2.25) in \eqv(6.theo2.13)
now yields \eqv(6.theo2.1)-\eqv(6.theo2.2').
The proof of Theorem \thv(6.theo2) is done. 
\end{proof}
%%%%%%%%%%%%%%%%%%%%%%%%%%%%%%%%%%%%%%%%%%%%%%%%

\subsection{Almost sure convergence of $\nu^\circ_n$ and $\s^\circ_n$}
 \label{6.3}
 
 %natural next step 
%(of the argument/stage in the proof of...) 

%We now come to the key step in our argument. This consists in reducing
%Conditions (A1) and (A2) of Theorem \thv(2.theo1) to (i) a\emph{ mixing condition} for the
%chain $J^\circ_n$, and (ii) a \emph{law of large numbers} for the random variables
%$Q_n$.

 Theorem \thv(6.theo2) enables us to replace the chain dependant quantities 
$\nu_n^{{{\scriptscriptstyle{J^\circ_n}},t}}$ and $\s_n^{{{\scriptscriptstyle{J^\circ_n}},t}}$
by quantities, $\nu^\circ_n$ and $\s^\circ_n$, that now   only depend on the randomness of the environment.
 %Notice that the functions $\nu^\circ_n(u,\infty)$ and $\s^\circ_n(u,\infty)$ defined in \eqv(6.2.4) and \eqv(6.2.5) 
%no longer depend on the randomness of the chain $J^\circ_n$.
Our next step consists in 
proving laws of large numbers for
$\nu^\circ_n$ and $\s^\circ_n$.
%defined in \eqv(6.2.4) and \eqv(6.2.5), respectively, around their mean value.

%We will in fact see that $\nu^\circ_n(u,\infty)$ is "self averaging" and that $\s^\circ_n(u,\infty)$ converges to zero.

\begin{proposition}
{\TH(6.prop1)}
Under the assumptions and with the notation of Theorem \thv(1.theo1.Main)
%Given  $0<\varepsilon<1$ let $a_n$ and $c_n$ be as in Theorem \thv(1.theo1.Main).
%Let $\nu$ be as in \eqv(1.theo1.M3) and assume that $\b>\b_c(\varepsilon)$.
%Assume further that $c_{\star}>3$. 
%
% this condition comes from using Lemma \eqv(8.3lem1)
%
there exists a subset  $\O^{\scriptscriptstyle{\textsf{LLN}}}\subset\O$
with $\P(\O^{\scriptscriptstyle{\textsf{LLN}}})=1$ such that, on $\O^{\scriptscriptstyle{\textsf{LLN}}}$, the following holds: 
for all $u>0$,
\bea
\Eq(6.prop1.1)
\lim_{n\rightarrow\infty}\nu^\circ_n(u,\infty)&=&\nu(u,\infty),
\Eq(6.prop1.2)
\\
\lim_{n\rightarrow\infty}n\s^\circ_n(u,\infty)&=&\nu(2u,\infty).
\eea
\end{proposition}

 We prove the proposition by comparing $\nu^\circ_n$ and $\s^\circ_n$ to
% to the same quantities but in the dynamics
 their counterpart, $\nu^{\scriptscriptstyle{\textsf{REM}}}_n$ and $\s^{\scriptscriptstyle{\textsf{REM}}}_n$,
 in the random hopping dynamics of the non truncated REM.
%where the jump chain is the symmetric random walk, $J^{\scriptscriptstyle{\textsf{SRW}}}_n$,
%%(see \eqv(4.prop6.3)))
%and the mean holding time at $x$ is the Boltzman weights $w_n(x)$.
To define $\nu^{\scriptscriptstyle{\textsf{REM}}}_n$ and $\s^{\scriptscriptstyle{\textsf{REM}}}_n$ 
recall the definition of $w_n(x)$ from \eqv(1.1.2) and set
\be
\g_n(x)=w_n(x)/c_n.
\ee
Then, for all $u>0$,
\be
\nu^{\scriptscriptstyle{\textsf{REM}}}_n(u,\infty)
=
\Eq(6.prop1.3)
\frac{a_n}{|\VV_n|}\sum_{x\in\VV_n}e^{-u/\g_n(x)},
\ee
\be
\Eq(6.prop1.4)
\s^{\scriptscriptstyle{\textsf{REM}}}_n(u,\infty)
=
\frac{a_n}{|\VV_n|}\sum_{x\in\VV_n}\sum_{x'\in\VV_n}
p_n^{\scriptscriptstyle{\textsf{SRW}},(2)}(x,x')e^{-u(1/\g_n(x)+1/\g_n(x'))},
\ee
where $p_n^{\scriptscriptstyle{\textsf{SRW}},(2)}(\cdot,\cdot)$ denotes the $2$-steps transition probabilities of 
$J^{\scriptscriptstyle{\textsf{SRW}}}_n$ 
(see \eqv(4.prop6.3)). For later use (namely, for the treatment of Condition (C3)) we also define, for all $\e>0$,
\be
\Eq(6.prop1.4')
\eta^{\scriptscriptstyle{\textsf{REM}}}_n(\e)
=
 \frac{a_n}{|\VV_n|} \sum_{x\in \VV_n}\g_n(x)\left(1-e^{-\e /\g_n(x)}\right).
\ee
The functions $\nu^{\scriptscriptstyle{\textsf{REM}}}_n$, $\s^{\scriptscriptstyle{\textsf{REM}}}_n$,
and $\eta^{\scriptscriptstyle{\textsf{REM}}}_n$
%were studied in detail in \cite{G10b}  and 
are well understood. We know in particular that:
%(see Proposition 6.4 Section 6.2  of \cite{LV} and Proposition 5.1 Section 5.1 de \cite {G})

\begin{proposition}
%[Proposition 6.4 and  lemma 7.1 of \cite{LV}] 
{\TH(6.prop2)}
Given  $0<\varepsilon<1$ let $a_n$ and $c_n$ be as in Theorem \thv(1.theo1.Main).
Let $\nu$ be as in \eqv(1.theo1.M3) and assume that $\b>\b_c(\varepsilon)$.
Then, there exists a subset  $\O^{\scriptscriptstyle{\textsf{REM}}}\subset\O$
with $\P(\O^{\scriptscriptstyle{\textsf{REM}}})=1$ such that, on $\O^{\scriptscriptstyle{\textsf{REM}}}$, 
the following holds: 
%for all $u>0$,
\bea
\Eq(6.prop2.1)
\lim_{n\rightarrow\infty}\nu^{\scriptscriptstyle{\textsf{REM}}}_n(u,\infty)&=&\nu(u,\infty),\quad \forall u>0,
\\
\Eq(6.prop2.2)
\lim_{n\rightarrow\infty}n\s^{\scriptscriptstyle{\textsf{REM}}}_n(u,\infty)&=&\nu(2u,\infty), \quad \forall u>0,
\eea
\vspace{-5pt}
and
\vspace{-5pt}
\be
\Eq(6.prop2.2')
 \lim_{\e\rightarrow 0}\lim_{n\rightarrow\infty} \eta^{\scriptscriptstyle{\textsf{REM}}}_n(\e)=0.
\ee
\end{proposition}

Throughout this section we set 
$
\varepsilon_n\equiv\frac{\log a_n}{n\log 2}
$; 
thus by \thv(1.theo1.M1), $\lim_{n\rightarrow\infty}\varepsilon_n=\varepsilon$, $0<\varepsilon<1$.

\begin{proof} Eq.~\eqv(6.prop2.1) and \eqv(6.prop2.2) are proved in Proposition 5.1, (i), in Section 5.1 of \cite{G10b}.
The proof of \eqv(6.prop2.2') is elementary: by simple Gaussian calculations,
$
\E\eta^{\scriptscriptstyle{\textsf{REM}}}_n(\e)\leq c\e^{1-\a(\varepsilon_n)}\downarrow 0
$
as $n\uparrow\infty$ and $\e\downarrow 0$, where $0<c<\infty$ is a constant,
%so that, 
%$
% \lim_{\e\rightarrow 0}\lim_{n\rightarrow\infty} \E\eta^{\scriptscriptstyle{\textsf{REM}}}_n(\e)=0
%$,
and
$
\P\left(\left|
\eta^{\scriptscriptstyle{\textsf{REM}}}_n(\e)-\E\eta^{\scriptscriptstyle{\textsf{REM}}}_n(\e)
\right|>n^{-1}
\right)
\leq n^3 {a_n}/{|\VV_n|}
$,
which is summable under our assumptions on $a_n$. Since
$\eta^{\scriptscriptstyle{\textsf{REM}}}_n(\e)$
is a monotonic function of $\e>0$,  arguing  e.g.~as in \eqv(7.prop1.4) yields the claim \eqv(6.prop2.2').
\end{proof} 

Our next lemma establishes that $\nu^\circ_n$ and $\s^\circ_n$ are very close
to $\nu^{\scriptscriptstyle{\textsf{REM}}}_n$ and $\s^{\scriptscriptstyle{\textsf{REM}}}_n$.
%
% $\O_3$ is introduces in Lemma \eqv(8.3lem1)
%

%%%%%%%%%%%%%%%%%%%%%%%%%%%%%%%%%%%%%%%%%%%%%%%%%%%%%%
% 								debut
%%%%%%%%%%%%%%%%%%%%%%%%%%%%%%%%%%%%%%%%%%%%%%%%%%%%%%

\begin{lemma}
  \TH(6.lem2)
On $\O_3$, for all but a finite number of indices $n$, for all $u>0$,
%$0<\varepsilon<1$ and $u>0$,
  \bea
\Eq(6.lem2.1)
\left|\nu^\circ_n(u,\infty)-\nu^{\scriptscriptstyle{\textsf{REM}}}_n(u,\infty)\right|
&\leq&
2n^{-2c_{\star}+1}\nu^{\scriptscriptstyle{\textsf{REM}}}_n(u,\infty)
+2a_ne^{-un^2}
+2n^{-c_{\star}+1+2\a(\varepsilon_n)},\quad\quad
\\
\Eq(6.lem2.2)
\left|\s^\circ_n(u,\infty)-\s^{\scriptscriptstyle{\textsf{REM}}}_n(u,\infty)\right|
&\leq&
2n^{-2c_{\star}+1}\s^{\scriptscriptstyle{\textsf{REM}}}_n(u,\infty)
+4a_ne^{-un^2}
+2n^{-c_{\star}+1+2\a(\varepsilon_n)}.\quad\quad
\eea
\end{lemma}

\begin{proof}[Proof of Lemma \thv(6.lem2) ] 
The proof hinges on the observation that $c_n\l_n(x)=1/\g_n(x)$ for all $x$ in the subset 
$I^{\star}_n$ of the decomposition \eqv(10.1.4). This enables us to rewrite $\nu^\circ_n(u,\infty)$ as
\be
\nu^\circ_n(u,\infty)=({\left|\VV_n\right|}/{\left|\VV^\circ_n\right|})\nu^{\scriptscriptstyle{\textsf{REM}}}_n(u,\infty)
\Eq(6.lem2.3')
+I_1-I_2-I_3
\ee
where
\be
%0\leq 
\textstyle
I_3
\equiv({a_n}/{\left|\VV^\circ_n\right|})
\sum_{x\in\cup_{l=1}^{L^{\star}} C^{\star}_{n,l}
}e^{-u/\g_n(x)},
\Eq(6.lem2.5)
\ee
\be
%0\leq 
\textstyle
I_1
\equiv({a_n}/{\left|\VV^\circ_n\right|})\sum_{x\in\VV^\circ_n\setminus I^{\star}_n}e^{-uc_n\l_n(x)}
\leq a_ne^{-uc_n},
\Eq(6.lem2.3)
\ee
\be
%0\leq 
\textstyle
I_2
\equiv({a_n}/{\left|\VV^\circ_n\right|})\sum_{x\in\VV^\circ_n\setminus I^{\star}_n}e^{-u/\g_n(x)}
%\equiv\frac{a_n}{\left|\VV^\circ_n\right|}\sum_{x\in\VV_n:w_n(x)< r_n(\rho^{\star}_n)}e^{-u/\g_n(x)}
%\1_{\{w_n(x)< r_n(\rho^{\star}_n)\}}
\leq a_ne^{-uc_n/r_n(\rho^{\star}_n)}.
\Eq(6.lem2.4)
\ee
The bounds on $I_1$ and $I_2$ follow from the fact that on  
$\VV^\circ_n\setminus I^{\star}_n\equiv N^{\star}_n$, 
$\l_n(x)=1$ and $w_n(x)< r_n(\rho^{\star}_n)$.
In order to bound $I_3$  recall \thv(3.0.1) and set
$
\WW_n(\rho)\equiv(\cup_{l=1}^{L^{\star}} C^{\star}_{n,l})\cap V(\rho)
$
and
$
\WW^c_n(\rho)\equiv(\cup_{l=1}^{L^{\star}} C^{\star}_{n,l})\cap V^c(\rho)
$
for some $\rho>0$.
Then, on $\WW^c_n(\rho)$,  by \eqv(A1.lem1.1) of Lemma \thv(9.lem4'), 
\be
\textstyle
\frac{w_n(x)}{c_n}
\leq 
\frac{r_n(\rho)}{r_n(\varepsilon_n)}
=
\exp\{n\b\b_c(1)(\sqrt\varepsilon_n-\sqrt\rho)
-\frac{\b\log n}{2\b_c(1)}(\frac{1}{\sqrt\varepsilon_n}-\frac{1}{\sqrt\rho})+o(1)\},
\Eq(6.lem2.6)
\ee
so that choosing $\sqrt\rho=\sqrt\varepsilon_n-\frac{2\log n}{n\b\b_c(1)}$, we get
\be
\textstyle
\frac{r_n(\rho)}{r_n(\varepsilon_n)}=n^{2}\exp \left\{
\frac{\log n}{n\b\varepsilon_n2\log 2}(1+o(1))
\right\}=n^2(1+o(1)).
\Eq(6.lem2.7)
\ee
One also sees that for this choice of $\rho$, $4\rho^{\star}_n<\rho<1-4\rho^{\star}_n$ for all $0<\varepsilon<1$ and large enough $n$. Therefore 
%Since by \eqv(8.3lem1.2bis) of 
Lemma \thv(8.3lem1) applies, yielding 
%Assume that $c_{\star}>2$.
\be
\left|\WW_n(\rho)\right|/|\VV^\circ_{n}|
\leq  n^{-c_{\star}+1}2^{-n\rho}(1+o(1)),
\Eq(6.lem2.8)
\ee
on $\O_3$, for all $n$ large enough. Assume from now on that $\o\in\O_3$.
By \eqv(6.lem2.7) and \eqv(6.lem2.8), 
\be
I_3
\leq
a_ne^{-un^2}+2n^{-c_{\star}+1}2^{n(\varepsilon_n-\rho)}
\leq
a_ne^{-un^2}+2n^{-c_{\star}+1}n^{2\b_c(\varepsilon_n)/\b},
\Eq(6.lem2.9)
\ee
where we used that 
% $a_n=2^{\varepsilon_n n}$ and that
$
\varepsilon_n-\rho=(\sqrt{\varepsilon_n}-\sqrt\rho)(\sqrt{\varepsilon_n}+\sqrt\rho)
\leq 2\sqrt{\varepsilon_n}\frac{2\log n}{n\b\b_c(1)}
$.
%so that
%$
%2^{n(\varepsilon_n-\rho)}\leq n^{2\b_c(\varepsilon_n)/\b}
%$.
Eq.~\eqv(6.lem2.1) now easily follows
%\be
%I_1+I_2+I_3
%\leq
%a_n(e^{-un^2}+e^{-uc_n/r_n(\rho^{\star}_n)}+e^{-uc_n})
%+2n^{-c_{\star}+1+2\b_c(\varepsilon_n)/\b},
%\ee
%where
observing that, by \eqv(A1.lem1.1)  and \eqv(9.lem4'.2) of Lemma \thv(9.lem4'), 
$
c_n\gg c_n/r_n\bigl(\rho^{\star}_n\bigr)\gg n^2
%\log(c_n/r_n\bigl(\rho^{\star}_n\bigr))\geq
%n\b\b(\varepsilon_n)(1-o(1))
%\sqrt{2c_{\star}n\log n}+\OO(\log n)
$,
and using that
$
%\be
%\textstyle
{\left|\VV_n\right|}/{\left|\VV^\circ_n\right|}=1+n^{-2c_{\star}+1}(1+\OO(n^{-(c_{\star}-1)}))
%\Eq(4.prop2.0)
%\ee
$,
as follows from \eqv(4.prop2.0).
%%%%%%%%%%%%%%%%%%%%%%%%%%%%%%%%%%%%%%%%%%%%%%%%%%%%%%
% 								debut
%%%%%%%%%%%%%%%%%%%%%%%%%%%%%%%%%%%%%%%%%%%%%%%%%%%%%%

The proof of \eqv(6.lem2.2) follows the same pattern, using the additionnal observation that 
$
p_n^{\circ,2}(x,x')=p_n^{\scriptscriptstyle{\textsf{SRW}},2}(x,x')
$
for all $x,x'$ in $I^{\star}_n\times I^{\star}_n$. This follows from
Proposition \thv(4.prop4) and the fact that, by construction, $I^{\star}_n\cap \del C^{\star}_{n,l}=\emptyset$
for all $1\leq l\leq L^{\star}$. We skip the details.
\end{proof}

\begin{proof}[Proof of Proposition \thv(6.prop1) ] The proposition is now an immediate consequence of
Lemma \thv(6.lem2) and Proposition \thv(6.prop2). 
\end{proof}

%%%%%%%%%%%%%%%%%%%%%%%%%%%%%%%%%%%%%%%%%%%%%%%%

\subsection{Conclusion of the proof of Theorem \thv(2.theo1).}
 \label{6.4}
 
% To make the proof of verification of Theorem \thv(2.theo1) complete we are left to verify the conditions 
 
We are now ready to show that under the assumptions of Theorem \thv(1.theo1.Main), taking
for initial distribution the invariant measure $\pi_n^\circ$ of  $J^\circ_n$, the conditions of 
Theorem \thv(6.theo1) are satisfied $\P$-almost surely.
%the conditions of Theorem \thv(6.theo1) are verified with the following choices:
%take for initial distribution the invariant measure $\pi_n^\circ$ of  $J^\circ_n$;
%given  $0<\varepsilon<1$ let $a_n$ and $c_n$ be as in Theorem \thv(1.theo1.Main);
%finally, let $\nu^\circ$ in Condition (C1) be the function $\nu$ defined in \eqv(1.theo1.M3).
%%%%%see Proposition  \thv(4.prop6)
%Assume furthermore that $\b>\b_c(\varepsilon)$ and that $c_{\star}>3$. 
Firstly, by  Theorem \thv(6.theo2) and Proposition \thv(6.prop1),  Conditions (C1) and (C2) are
satisfied $\P$-almost surely. That is, $\P$-almost surely the following holds:
for all $u>0$ and all $t>0$,
\bea
\Eq(6.4.1)
&&\lim_{n\rightarrow\infty}\nu_n^{{{\scriptscriptstyle{J^\circ_n}},t}}(u,\infty)
= t\nu^\circ(u,\infty)\quad \text{in $\PP^\circ$-probability},
\\
\Eq(6.4.2)
&&\lim_{n\rightarrow\infty}\s_n^{{{\scriptscriptstyle{J^\circ_n}},t}}(u,\infty)
= 0 \quad \text{in $\PP^\circ$-probability}.
\eea
% Conditions (C1) and (C2) are now easily disposed of: choose 
Next,  in view of  \eqv(6.2.6), \eqv(2.A0') reads  $\nu^\circ_n(v,\infty)/a_n=o(1)$,
% and so, the validity of Condition (C0) is a straightforward consequence of \eqv(6.prop1.1)  of Proposition \thv(6.prop1).
and so, by \eqv(6.prop1.1)  of Proposition \thv(6.prop1),  Condition (C0) is satisfied.
It remains to check Condition (C3). As in the proof of Proposition \thv(6.prop1), we
do this by comparing the quantity 
\bea
\Eq(6.4.3)
\eta^\circ_n(\e)
&\equiv&
\lfloor a_n \rfloor
\EE^\circ_{\pi_n^\circ} \1_{\{\l_n^{-1}(J^\circ_n(0))e^\circ_0\leq c_n\e\}} c_n^{-1}\l_n^{-1}(J^\circ_n(0))e^\circ_0
\\
&=&
%\textstyle
\frac{\lfloor a_n \rfloor}{|\VV^\circ_n|} \sum_{x\in \VV^\circ_n}
c_n^{-1}\l_n^{-1}(x)\left(1-e^{-\e c_n\l_n(x)}\right)
\eea
arising in \eqv(6.1.7), to its counterpart
in the random hopping dynamics of the non truncated REM,
 $\eta^{\scriptscriptstyle{\textsf{REM}}}_n(\e)$, defined in \eqv(6.prop1.4').
%Namely 
For this we simply write that
since $\l_n(x)=1$  on $\VV^\circ_n\setminus I^{\star}_n$ and $c_n\l_n(x)=1/\g_n(x)$ on $I^{\star}_n$,
\be
\Eq(6.4.4)
%\textstyle
\eta^\circ_n(\e)
\leq
\frac{\lfloor a_n \rfloor}{c_n}
+\frac{\lfloor a_n \rfloor}{|\VV^\circ_n|} \sum_{x\in I^{\star}_n}
\g_n(x)\left(1-e^{-\e /\g_n(x)}\right)
\\
\leq
\frac{\lfloor a_n \rfloor}{c_n}
+
\frac{|\VV_n|}{|\VV^\circ_n|}\eta^{\scriptscriptstyle{\textsf{REM}}}_n(\e).
\ee
From this, \eqv(4.prop2.0), and \eqv(6.prop2.2') it follows that,
% for all $\b>\b_c(\varepsilon)$,
under the assumptions of Proposition \thv(6.prop2),
\be
\Eq(6.4.5)
 \lim_{\e\rightarrow 0}\lim_{n\rightarrow\infty} \eta^\circ_n(\e)=0, \quad
\text{ $\P$-almost surely}.
\ee
Therefore Condition (C3) is satisfied $\P$-almost surely.

Since all four conditions (C0), (C1), (C2), and (C3) are satisfied $\P$-almost surely,
 it follows from Theorem \thv(6.theo1) 
that, for our choices of $a_n$, $c_n$, $\b$, and $c_{\star}$, 
 $\P$-almost surely, 
\be
S^\circ_n\Rightarrow_{J_1}  S^\circ_{\infty}
\Eq(6.4.6)
\ee
where $S^\circ_{\infty}$ is a subordinator with zero drift and L\'evy measure
$\nu^\circ=\nu$ defined in \eqv(1.theo1.M3). The proof of Theorem \thv(2.theo1) is complete.

%%%%%%%%%%%%%%%%%%%%%%%%%%%%%%%%%%%%%%%%%%%%%%%%%%%%%%%%%%%%

%%%%%%%%%%%%%%%%%%%%%%%%%%%%%%%%%%%%%%%%%%%%%%%%%%%%%%%%%%%%

\section{Convergence of the back end clock process below the critical temperature: proof of Theorem \thv(2.theo2)}
 \label{7}
 
 \subsection{A convergence theorem for \textsc{\textbf{becp}}}
 \label{7.1}
  
%  \textsc{ \textbf{This is bold small capitals} }

Consider the rescaled  process \eqv(2.3.1),
\be
S_n^\dagger(t) = b_n^{-1}\wt S_n^\dagger (k^\dagger_n(t)),\quad t\geq 0.
\Eq(2.3.1')
\ee
%Our next theorem parallels Theorem \thv(6.theo1), namely,
%%in that
%the three conditions formulated below imply 
%convergence of  the sequence $S_n^\dagger$ to a subordinator when the initial distribution of $J^\dagger_n$ is 
%the  invariant measure $\pi^\circ_n$ of $J^\circ_n$. 
%
%The strategy of this section parallels that of Section \thv(6), namely, Theorem \thv(7.theo1) below provides
%Our next theorem 
Theorem \thv(7.theo1) below parallels Theorem \thv(6.theo1) for
\textsc{fecp}, namely,
it gives three sufficient conditions for the sequence $S_n^\dagger$ to converge to a subordinator when the initial distribution of $J^\dagger_n$ is  the  invariant measure $\pi^\circ_n$ of $J^\circ_n$. 
As before these conditions refer to given sequences of  numbers $a_n$ and $b_n$,
% a given initial distribution,   $\pi^\circ_n$, 
and a given realization of the random environment.
For $u>0$ define
\be
\bar h^{u}_n(y)=
\sum_{1\leq l\leq L^{\star}}\sum_{x\in C^{\star}_{n,l}} p_n(y,x)P_x(T^{\star}_{n,l}>b_nu)
\,,\,\,
%y\in\del C^{\star}_{n,l},
y\in\VV^\circ_n,
\Eq(7.1.1)
\ee
%\vspace{-9pt}
where $T^{\star}_{n,l}$  is the exit time \eqv(5.2.1).
% (see also \eqv(2.2.10)). 
(Note that $\bar h^{u}_n(y)=0$ unless $y\in\cup_{1\leq l\leq L^{\star}}\del C^{\star}_{n,l}$.)
%$1\leq l\leq L^{\star}$.) 
%Recall that $k^\circ_n(t)=\lfloor a_n t\rfloor$ and, for $t>0$ and $u>0$,  set
For $k^\circ_n(t)$  as in \eqv(2.3.3) define, for $t>0$ and $u>0$, 
%\vspace{-9pt}
\bea
\Eq(7.1.2)
\bar\nu_n^{{{\scriptscriptstyle{J^\circ_n}},t}}(u,\infty)
&= &
\sum_{j=0}^{k^\circ_n(t)-1}\bar h^{u}_n(J^\circ_n(j)),
\\
\Eq(7.1.3)
\bar\s_n^{{{\scriptscriptstyle{J^\circ_n}},t}}(u,\infty)
&= &
\sum_{j=0}^{k^\circ_n(t)-1}\left[\bar h^{u}_n(J^\circ_n(j))\right]^2.
\eea

\noindent {\bf Condition  (A1).}
There exists a $\s$-finite measure $\nu^\dagger$ on $(0,\infty)$ satisfying
$
\int_0^\infty
(1\wedge u)\nu^\dagger(du)<\infty
$ 
such that, for all $t>0$ and all $u>0$,
\be
\Eq(7.1.5)
P^\circ_{\pi^\circ_n}\left(
\left|
\bar\nu_n^{{{\scriptscriptstyle{J^\circ_n}},t}}(u,\infty)
-t\nu^\dagger(u,\infty)
\right|
<\e
\right)=1-o(1)\,,\quad\forall\e>0\,.
\ee

\noindent {\bf Condition  (A2).}
 For all $u>0$ and all $t>0$,
\be
\Eq(7.1.6)
P^\circ_{\pi^\circ_n}\left(
\bar\s_n^{{{\scriptscriptstyle{J^\circ_n}},t}}(u,\infty)
<\e
\right)=1-o(1)\,,\quad\forall\e>0\,.
\ee

\noindent {\bf Condition  (A3).}
 For all $t>0$,
\be
\Eq(7.1.7)
\lim_{\e\downarrow 0}\limsup_{n\uparrow \infty}
\frac{k^\circ_n(t)}{|\VV^\circ_{n}|}\sum_{1\leq l\leq L^{\star}}\sum_{x\in C^{\star}_{n,l}}
E_{x}\bigl( \1_{\{b_n^{-1}T^{\star}_{n,l}\leq \e\} }b_n^{-1}T^{\star}_{n,l}\bigr)
=0.
\ee

\begin{theorem}
\TH(7.theo1) 
%Let the initial  distribution of $J^\dagger_n$ be the invariant measure $\pi^\circ_n$ of $J^\circ_n$.
Choose for initial  distribution the invariant measure $\pi^\circ_n$ of $J^\circ_n$.~For all sequences $a_n$ and $b_n$ for which
Conditions  (A1), (A2), and (A3) are verified $\P$-almost surely,
\be
S_n^\dagger  \Rightarrow_{J_1}  S^\dagger_{\infty}
\Eq(7.theo1.1)
\ee
$\P$-almost surely, where  $S^\dagger_{\infty}$  is the L\'evy subordinator with zero drift and L\'evy measure
$\nu^\dagger$.
\end{theorem}

%Conditions  (A1) and (A2) are similar to Conditions  (C1) and (C2) of Theorem \thv(6.theo1).
%Comparing \eqv(7.1.2)  to \eqv(6.1.2),  we see that $\bar h^{u}_n$
%plays the role of $h^{u}_n$ and that both fonctions depend only on the randomness of the environment, and
%no longer on 
%%the randomness coming from 
%the chain $J_n$. 
%However we also see that $k^\circ_n(t)$, which does depend on the chain, 
%%replaces
%takes the place of the deterministic function $k^\circ_n(t)$. 
%The same remark applies to \eqv(7.1.3).
%The next theorem shows that  we can substitute $k^\circ_n(t)$ for $k^\circ_n(t)$ in
%Conditions  (A1) and (A2). 

%Conditions  (A1) and (A2) look similar to Conditions  (C1) and (C2) of Theorem \thv(6.theo1). Just as
%$h^{u}_n$ and $\s^{u}_n$, the functions $\bar h^{u}_n$ and $\bar \s^{u}_n$ are random variables that
%do not depend on the chain $J_n$ but only on the environment.

\begin{proof}[Proof of Theorem \thv(7.theo1)] 

The proof of Theorem \thv(2.theo2) relies on Theorem 2.1 of \cite{G12}, which is itself a specialization of
%specializes
Theorem 4.1 of \cite{DR78} to processes with non-negative increments.
%of a result by Durrett and  Resnick  \cite{DR78} in form that is suitable for the present application.
Throughout we fix a realisation $\o\in\O^{\star}$ of the random environment but do not make this explicit in the notation. 
With the notations
% and definitions 
of Subsection \thv(2.2) define, for $i\geq 0$,
\be
Z_{n,i} \equiv b_n^{-1}\L_n^\dagger(J_n^\dagger(i))\,.
\Eq(7.theo1.2)
\ee
Thus, by \eqv(2.2.7) and \eqv(2.3.1'),
\be
\textstyle
%%%%%%\textstyle
S_n^\dagger(t)= \sum_{i=0}^{ k^\dagger_n(t)-1}Z_{n,i}.
\Eq(7.theo1.3)
\ee
In view of \eqv(2.3.0), $k^\dagger_n(t)$ is a stopping time for each $t>0$. Furthermore,
because $J^\dagger_n$ starts in $\pi^\circ_n$, and because $\pi^\circ_n\left(\VV_n\setminus \VV^\circ_n\right)=0$,
it follows from  \eqv(2.2.8)  that  $Z_{n,0}=0$.
% This follows from  \eqv(2.2.8) and the fact that
%the initial distribution of $J^\dagger_n$, $\pi^\circ_n$, has support in $\VV^\circ_n$. 
We may thus apply  Theorem 2.1 of \cite{G12} to the sum \eqv(7.theo1.3).

%To do this
To this end let  $\{\FF^\dagger_{n,i}, n\geq 1, i\geq 0\}$ 
be the array of sub-sigma fields 
%of $\FF^X$ 
defined (with obvious notation)  by
$\FF^\dagger_{n,i}
=\s\left(J_n^\dagger(0),\dots,J_n^\dagger(i)\right)$,
for $i\geq 0$.
Clearly, for each $n$ and $i\geq 1$, $Z_{n,i}$ is $\FF^\dagger_{n,i}$
 measurable and $\FF^\dagger_{n,i-1}\subset\FF^\dagger_{n,i}$.
Next, observe that
\bea
\nonumber
\PP^\dagger_{\pi^\circ_n}\bigl(Z_{n,i}>u\,\big|\,\FF^\dagger_{n,i-1}\bigr)
%=\sum_{x\in\VV_n}\PP^\dagger_{\pi^\circ_n}\bigl(J_n^\dagger(i)=x, Z_{n,i}>u\,\big|\,\FF^\dagger_{n,i-1}\bigr),
\hspace{-6pt}&=&\hspace{-6pt}
\textstyle
\sum_{x\in\VV_n}\PP^\dagger_{\pi^\circ_n}\bigl(J_n^\dagger(i)=x, Z_{n,i}>u\,\big|\,J_n^\dagger(i-1)\bigr)
\\
\Eq(7.theo1.4)
\hspace{-6pt}&=&\hspace{-6pt}
\textstyle
\sum_{x\in\VV_n}\PP^\dagger_{\pi^\circ_n}\bigl(J_n^\dagger(i)=x, \L_n^\dagger(x)>b_nu\,\big|\,J_n^\dagger(i-1)\bigr).\quad
\eea
By \eqv(2.2.8) and \eqv(5.2.1), $\L_n^\dagger(x)=0$ if
%the definition of $T^{\star}_{n,l}$ below \eqv(2.2.10) (see also \eqv(5.2.1)),
$
x\notin
%\VV_n\setminus \VV^\circ_n=
\cup_{1\leq l\leq L^{\star}} C^{\star}_{n,l}
$,
and
$
\L_n^\dagger(x)=T^{\star}_{n,l}
$
if
$x\in C^{\star}_{n,l}$. Thus, 
\be
\Eq(7.theo1.5)
%\textstyle
\PP^\dagger_{\pi^\circ_n}\bigl(Z_{n,i}>u\,\big|\,\FF^\dagger_{n,i-1}\bigr)
=\sum_{1\leq l\leq L^{\star}}\sum_{x\in C^{\star}_{n,l}}
\PP^\dagger_{\pi^\circ_n}\bigl(J_n^\dagger(i)=x, \L_n^\dagger(x)>b_nu\,\big|\,J_n^\dagger(i-1)\bigr).
\ee
Now for all $1\leq l\leq L^{\star}$ and all $x\in C^{\star}_{n,l}$, 
\be
\Eq(7.theo1.6)
\PP^\dagger_{\pi^\circ_n}\bigl(J_n^\dagger(i)=x, \L_n^\dagger(x)>b_nu\,\big|\,J_n^\dagger(i-1)\bigr)
%\hspace{-6pt}&=&\hspace{-6pt} 
%=&&p_n^\dagger\bigl(J_n^\dagger(i-1),x\bigr)
%\PP^\dagger_{\pi^\circ_n}\bigl(\L_n^\dagger(x)>b_nu\,\big|\,J_n^\dagger(i)=x\bigr).
=
p_n^\dagger\bigl(J_n^\dagger(i-1),x\bigr)P_{x}\bigl(T^{\star}_{n,l}>b_nu\bigr)
\ee
 where, by \eqv(2.2.9), 
\be
\Eq(7.theo1.7)
p_n^\dagger\bigl(J_n^\dagger(i-1),x\bigr)=p_n\bigl(J_n^\dagger(i-1),x\bigr)\1_{\{J_n^\dagger(i-1)\in \VV^\circ_n\}}
\ee
%This because
(indeed, by definition of $J_n^\dagger$,
$J_n^\dagger(i)\in C^{\star}_{n,l}$
if and only if $J_n^\dagger(i-1)\in \del C^{\star}_{n,l}\subset \VV^\circ_n$).
In view of \eqv(7.1.1)  it follows
from  \eqv(7.theo1.4),   \eqv(7.theo1.5), \eqv(7.theo1.6) and \eqv(7.theo1.7)  that
\be
\textstyle
%%%%%%%\textstyle
 \sum_{i=1}^{k^\dagger_n(t)}
\PP^\dagger_{\pi^\circ_n}\bigl(Z_{n,i}>u\,\big|\,\FF^\dagger_{n,i-1}\bigr)
%&=&
% \sum_{i=1}^{\lfloor a_n t\rfloor}
%\sum_{1\leq l\leq L^{\star}}\sum_{x\in C^{\star}_{n,l}} 
%p_n\bigl(J_n^\dagger(i-1),x\bigr)\PP_x(T^{\star}_{n,l}>b_nu)\1_{\{J_n^\dagger(i-1)\in \VV^\circ_n\}}
%\cr
=\sum_{i=1}^{k^\dagger_n(t)}\bar h^{u}_n\bigl(J^\dagger_n(i-1)\bigr)\1_{\{J_n^\dagger(i-1)\in \VV^\circ_n\}}.
\Eq(7.theo1.8)
\ee
It remains to notice that  the chain $J^\dagger_n$ observed only when it 
%visits 
takes values in $\VV^\circ_n$ is nothing but the chain $J^\circ_n$, and that $J^\circ_n$ takes $k^\circ_n(t)$ 
steps  when $J^\dagger_n$ takes $k^\dagger_n(t)$ steps (see \eqv(2.3.3)). Thus, 
%It thus follows from \eqv(7.theo1.8) and  \eqv(7.1.2) that
\be
\textstyle
%%%%%%%\textstyle
\sum_{i=1}^{k^\dagger_n(t)}\bar h^{u}_n\bigl(J^\dagger_n(i-1)\bigr)\1_{\{J_n^\dagger(i-1)\in \VV^\circ_n\}}
\stackrel{d}{=}
\sum_{i=1}^{k^\circ_n(t)}\bar h^{u}_n\bigl(J^\circ_n(i-1)\bigr)
=
\bar\nu_n^{{{\scriptscriptstyle{J^\circ_n}},t}}(u,\infty),
\Eq(7.theo1.10)
\ee
where the first equality holds in distribution and the last is  \eqv(7.1.2).
Combining \eqv(7.theo1.8) and \eqv(7.theo1.10) now yields
\be
\textstyle
%%%%%%%\textstyle
\sum_{i=1}^{k^\dagger_n(t)} 
\PP^\dagger_{\pi^\circ_n}\bigl(Z_{n,i}>u\,\big|\,\FF^\dagger_{n,i-1}\bigr)
\stackrel{d}{=}\bar\nu_n^{{{\scriptscriptstyle{J^\circ_n}},t}}(u,\infty).
\Eq(7.theo1.11)
\ee
Similarly, we get
\be
\textstyle
%\hspace{-.1pt}
\sum_{i=1}^{k_n(t)-1}
\left[\PP^\dagger_{\pi^\circ_n}\bigl(Z_{n,i}>u\,\big|\,\FF^\dagger_{n,i-1}\bigr)\right]^2
\stackrel{d}{=}
\sum_{i=1}^{k^\dagger_n(t)}\left[
\bar h^{u}_n\bigl(J^\circ_n(i-1)\bigr)
\right]^2
=\bar\s_n^{{{\scriptscriptstyle{J^\circ_n}},t}}(u,\infty).
\Eq(7.theo1.12)
\ee
 From \eqv(7.theo1.11) and \eqv(7.theo1.12) it follows that  Conditions (A2) and (A1) of
Theorem \thv(7.theo1)  are exactly the Conditions (D1) and (D2) of  Theorem 2.1 of \cite{G12}.
To see that Condition (A3)  implies Condition (D3) we have to establish that
%for all $t>0$ and all $\e>0$,
\eqv(7.1.7) implies
\be
\lim_{\e\rightarrow 0}\limsup_{n\rightarrow\infty}
\PP^\dagger_{\pi^\circ_n}\left(S_n^{\dagger,\e}(t)>\e\right)=0\,,
\ee
where for $\e\geq 0$,
$
\textstyle
S_n^{\dagger,\e}(t)= \sum_{i=0}^{ k^\dagger_n(t)-1}Z_{n,i}\1_{\{Z_{n,i}\leq\e\}}.
$
By Theorem \thv(2.theo4) with $c_\circ=1$, on $\O^{\star}$, for all but a finite number of indices $n$, all $0< t<\infty$,
and all $\e\geq 0$,
\be
\PP^\dagger_{\pi^\circ_n}\left(S_n^{\dagger,\e}(t)>\e\right)
\leq 
\PP^\dagger_{\pi^\circ_n}\left(\textstyle\sum_{i=0}^{k_n(t)-1}Z_{n,i}\1_{\{Z_{n,i}\leq\e\}}>\e\right)
+n^{-2(c_{\star}-1)+1}(1+o(1)),
\Eq(7.theo1.17)
\ee
where $k_n(t)\equiv\lfloor k^\circ_n(t)(1+n^{-1})\rfloor=\lfloor \lfloor a_nt\rfloor(1+n^{-1})\rfloor$.
By Tchebychev inequality,
\be
\Eq(7.theo1.13)
\PP^\dagger_{\pi^\circ_n}\left(\textstyle\sum_{i=0}^{k_n(t)-1}Z_{n,i}\1_{\{Z_{n,i}\leq\e\}}>\e\right)
\leq 
\e^{-1}\textstyle\sum_{i=1}^{k_n(t)-1} \EE^\dagger_{\pi^\circ_n} \1_{\{Z_i^n\leq \e\}}Z^n_i,
\ee
and the right hand side of \eqv(7.theo1.13) is equal to
\be
\Eq(7.theo1.16)
\textstyle
%\sum_{i=1}^{k_n(t)-1} \EE^\dagger_{\pi^\circ_n} \1_{\{Z_i^n\leq \e\}}Z^n_i 
%=
\sum_{1\leq l\leq L^{\star}}\sum_{x\in C^{\star}_{n,l}}
E_{x}\bigl( \1_{\{b_n^{-1}T^{\star}_{n,l}\leq \e\} }b_n^{-1}T^{\star}_{n,l}\bigr)
\sum_{i=1}^{k_n(t)-1} \EE^\dagger_{\pi^\circ_n}\1_{\{J_n^\dagger(i)=x\} }.
\ee
By \eqv(8.2.2)
%, the remark that immediately follows,  
and \eqv(2.2.9), for all $x\in C^{\star}_{n,l}$,
\bea
\nonumber
\textstyle
\sum_{i=1}^{k_n(t)-1} \EE^\dagger_{\pi^\circ_n}\1_{\{J_n^\dagger(i)=x\} }
%=
%\sum_{i=1}^{\lf a_n t\rf}\sum_{y\in \del C^{\star}_{n,l}}
%\EE^\dagger_{\pi^\circ_n}\1_{\{J_n^\dagger(i-1)=y, J_n^\dagger(i)=x\} }
&=&
\nonumber
\textstyle
\sum_{y\in \del C^{\star}_{n,l}\cap\del x}p_n(y,x)
\sum_{i=1}^{k_n(t)-1}\EE^\dagger_{\pi^\circ_n}\1_{\{J_n^\dagger(i-1)=y\} }
\\
&\leq&
\nonumber
\textstyle
\sum_{y\in \del C^{\star}_{n,l}\cap\del x}p_n(y,x)
\sum_{i=1}^{k_n(t)-1}\pi^\circ_n(y)
\\
&\leq&
\textstyle
%\lf a_n t\rf\sum_{y\in \del C^{\star}_{n,l}}p_n(y,x)\pi^\circ_n(y)
%\leq 
{(k_n(t)-1)}/{|\VV^\circ_{n}|}
%\leq (1+n^{-1})k^\circ_n(t)/|\VV^\circ_{n}|
\Eq(7.theo1.14)
\eea
where the last inequality follows from  \eqv(4.prop2.1) and the fact that, for 
$y\in \del C^{\star}_{n,l}\cap\del x$ and $x\in C^{\star}_{n,l}$, $p_n(y,x)=n^{-1}$.
% we used \eqv(4.prop2.1) and the fact that for $x\in C^{\star}_{n,l}$ and $y\in \del C^{\star}_{n,l}$,
%$p_n(y,x)=n^{-1}\1_{\dist(x,y)=1}$.
Combining  \eqv(7.theo1.13), \eqv(7.theo1.16), and \eqv(7.theo1.14),  the probability in the left
 hand side of \eqv(7.theo1.13) is bounded above by
\be
\Eq(7.theo1.15)
%\PP^\dagger_{\pi^\circ_n}\left(\textstyle\sum_{i=0}^{k_n(t)-1}Z_{n,i}\1_{\{Z_{n,i}\leq\e\}}>\e\right)
%\leq
\textstyle
\e^{-1}(1+n^{-1})({k^\circ_n(t)}/{|\VV^\circ_{n}|})
\sum_{1\leq l\leq L^{\star}}\sum_{x\in C^{\star}_{n,l}}
E_{x}\bigl( \1_{\{b_n^{-1}T^{\star}_{n,l}\leq \e\} }b_n^{-1}T^{\star}_{n,l}\bigr).
\ee
Inserting this bound in  \eqv(7.theo1.17) yields the claim that
Condition (A3)  implies Condition (D3) of  Theorem 2.1 of \cite{G12}

Having established that, on $\O^{\star}$, the conditions of  Theorem 2.1 of \cite{G12} are verified
whenever those of Theorem \thv(7.theo1)
%Conditions (A1), (A2), and (A3) 
are verified, the proof of  Theorem \thv(7.theo1) is complete.
%  and so,
%$
%S_n^\dagger \Rightarrow S^\dagger_{\infty}
%$
% on $D([0,\infty))$,
% % equipped with the $J_1$ topology, 
% where $S^\dagger_{\infty}$ is a subordinator 
%with L\'evy measure $\nu^\dagger$ and zero drift.
\end{proof}

%%%%%%%%%%%%%%%%%%%%%%%%

%To verify these condition we closely follow the strategy proposed in  \cite{G2} (and used in \cite{AV}). 
%Namely, using the mixing property ... and ... we prove almost sure ergodic theorem for 
%the quantities 

%and we next prove strong laws of large numbers for the quantities

%To verify
%% these conditions
%Conditions (B1), (B2), and (A3)
% we closely follow the strategy of Section \thv(6), namely, we first prove an ergodic theorem
%for the quantities $\bar\nu_n^{{{\scriptscriptstyle{J^\circ_n}},t}}(u,\infty)$ and
%$\bar\s_n^{{{\scriptscriptstyle{J^\circ_n}},t}}(u,\infty)$.

%%%%%%%%%%%%%%%%%%%%%%%%%%%%%%%%%%%%%%%%%%%%%%%%

\subsection{
%An ergodic theorem
An ergodic theorem for \textsc{\textbf{becp}}
}
 \label{7.2}

To prove that Conditions (A1) and (A2) of Theorem \thv(7.theo1) are satisfied
 we closely follow the strategy of Subsection \thv(6.2) and first prove an ergodic theorem
for the quantities $\bar\nu_n^{{{\scriptscriptstyle{J^\circ_n}},t}}(u,\infty)$ and
$\bar\s_n^{{{\scriptscriptstyle{J^\circ_n}},t}}(u,\infty)$ defined in \eqv(7.1.2) and \eqv(7.1.3).
Clearly, for $\pi^{{{\scriptscriptstyle{J^\circ_n}},t}}_n(x)$ as in \eqv(6.2.1), \eqv(7.1.2) and \eqv(7.1.3) can be rewritten as
\bea
\Eq(7.2.2)
\bar\nu_n^{{{\scriptscriptstyle{J^\circ_n}},t}}(u,\infty)
&= &
%\textstyle
k^\circ_n(t)\sum_{y\in\VV^\circ_n}\pi^{{{\scriptscriptstyle{J^\circ_n}},t}}_n(y)\bar h^{u}_n(y),
\\
\Eq(7.2.3)
\bar\s_n^{{{\scriptscriptstyle{J^\circ_n}},t}}(u,\infty)
%&= &\sum_{j=1}^{\kappa^\circ_n(t)}\left[h^{u}_n(J^\circ_n(j-1))\right]^2
&= &
%\textstyle
k^\circ_n(t)\sum_{y\in\VV^\circ_n}\pi^{{{\scriptscriptstyle{J^\circ_n}},t}}_n(y)\left[\bar h^{u}_n(y)\right]^2.
\eea
%Before stating the theorem we need a little notation. 
%In order to express the mean values of \eqv(7.2.2) and \eqv(7.2.3), we introduce the following objects.
%we need a little notation.
%Let us express the mean values of \eqv(7.2.2) and \eqv(7.2.3).
Before stating our main theorem, let us express the mean values of \eqv(7.2.2) and \eqv(7.2.3)
with respect to the law $P^\circ_{\pi^\circ_n}$.
%Before stating our main theorem, let us express the mean values of \eqv(7.1.2) and \eqv(7.1.3)
%with respect to the law $P^\circ_{\pi^\circ_n}$.
Given $x\in C^{\star}_{n,l}$, $1\leq l\leq L^{\star}$, denote by
\be
\Eq(7.2.9)
Q_{n,l}^u(x)\equiv P_x(T^{\star}_{n,l}>b_nu),\,\,\, u>0,
\ee
the tail distribution of the exit time $T^{\star}_{n,l}$  given that the set $C^{\star}_{n,l}$ is entered in $x$, 
and define
\be
\bar\nu^\circ_n(u,\infty)
=
\Eq(7.2.6)
\frac{a_n}{2^n}
%\frac{a_n}{\left|\VV^\circ_n\right|}
\sum_{1\leq l\leq L^{\star}}\sum_{x\in C^{\star}_{n,l}}Q_{n,l}^u(x),
\ee
\be
\Eq(7.2.8)
\bar\s^{=}_n(u,\infty)
=
\frac{a_n}{n 2^n}
%\frac{a_n}{n\left|\VV^\circ_n\right|}
\sum_{1\leq l\leq L^{\star}}
\Bigl[\sum_{x\in C^{\star}_{n,l}}Q_{n,l}^u(x)\Bigr]^2,
\ee
\be
\Eq(7.2.7)
\bar\s^\circ_n(u,\infty)
=
\frac{a_n}{2^n}
%\frac{a_n}{\left|\VV^\circ_n\right|}
\sum_{1\leq l\leq L^{\star}}\sum_{1\leq l'\leq L^{\star}}
\sum_{x\in C^{\star}_{n,l}}\sum_{x'\in C^{\star}_{n,l'}}
Q_{n,l}^u(x)Q_{n,l'}^u(x')
%\frac{|\del x\cap \del x'|}{n^2}
%\sum_{y\in\VV^\circ_n}p_n(x,y)p_n(y,x')
\left(n^{-2}|\del x\cap \del x'|\right).
\ee
%\smallskip
%\frac{\left|\VV^\circ_n\right|}{\left|\VV_n\right|}

% (recall that $k^\circ_n(t)=\lfloor a_n t\rfloor$):

%We this we have:
\begin{lemma}
\TH(7.lem3) 
Assume that $c_{\star}>2$. Then on $\O^{\star}$, for all but a finite number of indices $n$,
\bea
\Eq(7.lem3.1) 
E^\circ_{\pi^\circ_n}\bigl[\bar\nu_n^{{{\scriptscriptstyle{J^\circ_n}},t}}(u,\infty)\bigr]
\hspace{-6pt}&=&\hspace{-6pt}
(1+o(1))({k^\circ_n(t)}/{a_n})
\bar\nu^\circ_n(u,\infty),
\\
\Eq(7.lem3.2) 
E^\circ_{\pi^\circ_n}\bigl[\bar\s_n^{{{\scriptscriptstyle{J^\circ_n}},t}}(u,\infty)\bigr]
\hspace{-6pt}&=&\hspace{-6pt}
(1+o(1))({k^\circ_n(t)}/{a_n})\bar\s^\circ_n(u,\infty).
\eea
\end{lemma}

%The notation $\bar\s^{=}_n(u,\infty)$ enters in the ..........
%\eqv(7.2.8)

% The next theorem 
%We now state our main theorem:
The main theorem of this section controls
%The content of the theorem below is twofold:
%Our main theorem, that we now state,
%it controls 
 the fluctuations of  $\bar\nu_n^{{{\scriptscriptstyle{J^\circ_n}},t}}$
around its mean value and provides an upper bound on $\bar\s_n^{{{\scriptscriptstyle{J^\circ_n}},t}}$ in terms of the random (in the environment) quantities $\bar\nu^\circ_n$, $\bar\s^\circ_n$, and $\bar\s^{=}_n$.

\begin{theorem}
\TH(7.theo2) 
Assume that $c_{\star}>3$.
Let $\bar\rho^\circ_n>0$ be a decreasing sequence satisfying $\bar\rho^\circ_n\downarrow 0$ as $n\uparrow\infty$.
There exists a sequence of subsets $\overline\O^{\scriptscriptstyle{\textsf{EG}}}_{n}\subset\O$ with
$
\P\bigl[\bigl(\overline\O^{\scriptscriptstyle{\textsf{EG}}}_{n}\bigr)^c\bigr]
<
{2^6\ell^\circ_n}/({\bar\rho^\circ_n na_n})
$,
and such that on $\overline\O^{\scriptscriptstyle{\textsf{EG}}}_{n}$ the following holds for all large enough $n$: 
for all $t>0$, all $u>0$, and all $\e>0$,
\smallskip
\be
P^\circ_{\pi^\circ_n}\left(\left|
\bar\nu_n^{{{\scriptscriptstyle{J^\circ_n}},t}}(u,\infty)
-
E^\circ_{\pi^\circ_n}\bigl[\bar\nu_n^{{{\scriptscriptstyle{J^\circ_n}},t}}(u,\infty)\bigr]
\right|\geq\e\right)
\leq
\e^{-2}[C_3 t\Theta_{n,3}(u)+ t^2\Theta_{n,4}(u)]
%\,,\quad\forall\e>0\,,
\Eq(7.theo2.1) 
\ee
\smallskip
for some constant $0<C_3<\infty$ and where, for $\varsigma^{\neq}_{n}(u)$ as in Lemma \thv(7.lem9),
\be
\Theta_{n,3}(u)\equiv 
\bar\s^\circ_n(u,\infty)
+
\bar\s^{=}_n(u,\infty)
+
\bar\rho^\circ_n\left[\varsigma^{\neq}_{n}(u)\right]^2,
\Eq(7.theo2.2) 
\ee
\be
\Theta_{n,4}(u)\equiv 
2^{-n}\left[\bar\nu^\circ_n(u,\infty)\right]^2.
\Eq(7.theo2.3) 
\ee
Moreover, for all $t>0$, all $u>0$, and all $\e'>0$,
\be
P^\circ_{\pi^\circ_n}\left(
\bar\s_n^{{{\scriptscriptstyle{J^\circ_n}},t}}(u,\infty)
\geq\e'\right)
\leq
%\frac{k^\circ_n(t)}{\e'\, a_n}\
\frac{t}{\e'}(1+o(1))\bar\s^\circ_n(u,\infty).
%\,,\quad\forall\e'>0\,.
\Eq(7.theo2.4) 
\ee
\end{theorem}

%Before proving the theorem let us express the mean values of \eqv(7.2.2) and \eqv(7.2.3).
%$E^\circ_{\pi^\circ_n}\bigl[\bar\nu_n^{{{\scriptscriptstyle{J^\circ_n}},t}}(u,\infty)\bigr]$
%and
%$E^\circ_{\pi^\circ_n}\bigl[\bar\s_n^{{{\scriptscriptstyle{J^\circ_n}},t}}(u,\infty)\bigr]$.

We now prove, in this order, Lemma \thv(7.lem3) and Theorem \thv(7.theo2).

\begin{proof}[Proof of Lemma \thv(7.lem3)] By \eqv(7.1.1), \eqv(7.2.2), and \eqv(7.2.9), 
\bea
\Eq(7.lem3.3) 
E^\circ_{\pi^\circ_n}\bigl[\bar\nu_n^{{{\scriptscriptstyle{J^\circ_n}},t}}(u,\infty)\bigr]
&=&
\textstyle
k^{\circ}_n(t)\sum_{y\in\VV^\circ_n}\pi^{\circ}_n(y)\bar h^{u}_n(y)
\\
\Eq(7.lem3.4) 
&=&
\textstyle
k^{\circ}_n(t)\sum_{y\in\VV^\circ_n}\pi^{\circ}_n(y)
\sum_{1\leq l\leq L^{\star}}\sum_{x\in C^{\star}_{n,l}} p_n(y,x)Q_{n,l}^u(x),\quad\quad
\eea
and since both $x$ and $y$ belong to $\VV^\circ_n$, $p_n(y,x)=n^{-1}$ if $\dist(x,y)=1$ and  is zero else. Thus
$\sum_{y\in\VV^\circ_n}p_n(y,x)=n^{-1}|\del x\cap\del C^{\star}_{n,l}|$ and
\be
\textstyle
\Eq(7.lem3.5) 
E^\circ_{\pi^\circ_n}\bigl[\bar\nu_n^{{{\scriptscriptstyle{J^\circ_n}},t}}(u,\infty)\bigr]
=({k^{\circ}_n(t)}/{|\VV^\circ_n|})
\sum_{1\leq l\leq L^{\star}}\sum_{x\in C^{\star}_{n,l}}
n^{-1}|\del x\cap\del C^{\star}_{n,l}|Q_{n,l}^u(x).
\ee
The claim of \eqv(7.lem3.1) now follows from \eqv(10.lem1.7) and \eqv(4.prop2.0).
%The claim of \eqv(7.lem3.1) now follows from \eqv(7.2.0) and  \eqv(7.2.1).
%
%pour info
%|\del C^{\star}_{n,l}\cap\del x| \geq n (1-\OO(\frac{1}{\log n}))
%|\VV^\circ_{n}|=2^n\left[1- n^{-2c_{\star}+1}(1+\OO(n^{-(c_{\star}-1)}))\right].
Eq.~\eqv(7.lem3.2)  is proved in a similar way. We skip the details.
\end{proof}

%%%%%%%%%%%%%%%%%%%%%%%%%%%%%%%%%%%%%%%%%%%%%

\begin{proof}[Proof of Theorem \thv(7.theo2) ] A first order Tchebychev inequality and \eqv(7.lem3.2) readily yield
\eqv(7.theo2.4). As in Theorem \thv(6.theo2), proving concentration of 
$\bar\nu_n^{{{\scriptscriptstyle{J^\circ_n}},t}}(u,\infty)$ is more involved. 
%a concentration result for
%\eqv(7.theo2.1) is more involved. 
%\Eq(7.1.2)\bar\nu_n^{{{\scriptscriptstyle{J^\circ_n}},t}}(u,\infty)
Since \eqv(7.2.2) is nothing but \eqv(6.2.2) with $h^{u}_n$ replaced by $\bar h^{u}_n$,
the proof naturally
%then is a exact copy of the proof of \eqv(6.theo2.1) of Theorem \thv(6.theo2) up to (and including)
%
%begins  like 
%
starts in the same way as the proof of \eqv(6.theo2.1) of Theorem \thv(6.theo2). More precisely, 
% up to \eqv(6.theo2.14): specifically,
%more precisely, 
substituting  $\bar h^{u}_n$ for $h^{u}_n$ in the definition \eqv(6.theo2.14) of the quantities $I_i^{(2)}$, $1\leq i\leq 3$,  we get that for all $\e>0$,
\be
P^\circ_{\pi^\circ_n}\left(\left|
\bar\nu_n^{{{\scriptscriptstyle{J^\circ_n}},t}}(u,\infty)
-
E^\circ_{\pi^\circ_n}\bigl[\bar\nu_n^{{{\scriptscriptstyle{J^\circ_n}},t}}(u,\infty)\bigr]
\right|\geq\e\right)
\leq
\e^{-2}[I_1^{(2)}+I_2^{(2)}+I_3^{(2)}].
%\Eq(6.theo2.13new)
\Eq(7.theo2.5) 
\ee
We are thus left to bound $I_i^{(2)}$, $1\leq i\leq 3$.
By \eqv(7.2.2) and \eqv(7.lem3.1),  
\be
\Eq(7.theo2.6)
I_1^{(2)}
=
2^{-n}
\left[E^\circ_{\pi^\circ_n}\bigl[\bar\nu_n^{{{\scriptscriptstyle{J^\circ_n}},t}}(u,\infty)\bigr]\right]^2
\leq 
2^{-n}\left({k^\circ_n(t)}/{a_n}\right)^2
\left[\bar\nu^\circ_n(u,\infty)\right]^2,
\ee
and by \eqv(7.2.3) and \eqv(7.lem3.2),
\be
\Eq(7.theo2.7)
I_2^{(2)}\leq 
E^\circ_{\pi^\circ_n}\bigl[\bar\s_n^{{{\scriptscriptstyle{J^\circ_n}},t}}(u,\infty)\bigr] 
\leq ({k^\circ_n(t)}/{a_n})\bar\s^\circ_n(u,\infty).
\ee
The term $I_3^{(2)}$ is a little more involved. We may write it 
%as the sum
in the form
\be
\Eq(7.theo2.8)
%\textstyle
I_3^{(2)}
\equiv
%2k^\circ_n(t)\sum_{y\in\VV^\circ_n}\sum_{y'\in\VV^\circ_n}\bar h^{u}_n(y)\bar h^{u}_n(y')
%\pi^\circ_n(y)\sum_{m=1}^{\ell^\circ_n-1}p_n^{\circ,m}(y,y')
%\\
%&=&
%2 \frac{k^\circ_n(t)}{|\VV^\circ_n|}
2 ({k^\circ_n(t)}/{a_n})({a_n}/{|\VV^\circ_n|})
\sum_{m=1}^{\ell^\circ_n-1}[I_{1,m}^{(3)}+I_{2,m}^{(3)}],
\ee
where, setting
$
%\be
f_n^{\circ,m}(x,x';y,y')\equiv p_n(y,x)p_n(y',x')p_n^{\circ,m}(y,y')
%\Eq(7.theo2.9)
%\ee
$, 
\bea
\Eq(7.theo2.10)
I_{1,m}^{(3)}
\hspace{-6pt}&\equiv&\hspace{-6pt}
\sum_{1\leq l\leq L^{\star}}\sum_{x\in C^{\star}_{n,l}} \sum_{x'\in C^{\star}_{n,l}}
%\sum_{y\in\VV^\circ_n}\sum_{y'\in\VV^\circ_n}
\sum_{y\in\del C^{\star}_{n,l}}\sum_{y'\in\del C^{\star}_{n,l}}
 Q_{n,l}^u(x)Q_{n,l}^u(x')f_n^{\circ,m}(x,x';y,y'),
\\
\Eq(7.theo2.11)
I_{2,m}^{(3)}
\hspace{-6pt}&\equiv&\hspace{-6pt}
\sum_{1\leq l, l'\leq L^{\star}:l\neq l'}\sum_{x\in C^{\star}_{n,l}} \sum_{x'\in C^{\star}_{n,l'}} 
\sum_{y\in\del C^{\star}_{n,l}}\sum_{y'\in\del C^{\star}_{n,l}}
%\sum_{y\in\VV^\circ_n}\sum_{y'\in\VV^\circ_n}
Q_{n,l}^u(x)Q_{n,l'}^u(x')f_n^{\circ,m}(x,x';y,y').\quad
\eea
%
% we used the following equality to replace sums over $y\in\VV^\circ_n$ by sums over $y\in\del C^{\star}_{n,l}$
%
%\be
%\sum_{y\in\VV^\circ_n}\sum_{y'\in\VV^\circ_n}f_n^{\circ,m}(x,x';y,y')
%=
%\sum_{y\in\del C^{\star}_{n,l}}\sum_{y'\in\del C^{\star}_{n,l}}f_n^{\circ,m}(x,x';y,y'),
%\Eq(7.theo2.12)
%\ee

Consider $I_{1,m}^{(3)}$ first. It follows from
 Proposition \thv(4.prop6), (ii) that
on $\O^{\scriptscriptstyle{\textsf{SRW}}}$, for all but a finite number of indices $n$,
for all $x,x'\in C^{\star}_{n,l}$,
\bea
\sum_{m=1}^{\ell^\circ_n-1}\sum_{y\in\del C^{\star}_{n,l}}\sum_{y'\in\del C^{\star}_{n,l}}f_n^{\circ,m}(x,x';y,y')
\leq
\frac{C'_\circ}{n}\sum_{y\in\del C^{\star}_{n,l}}\sum_{y'\in\del C^{\star}_{n,l}}p_n(y,x)p_n(y',x')
\leq
\frac{C'_\circ}{n}.\quad
\Eq(7.theo2.13)
\eea
%where the last inequality follows from
%$\sum_{y\in\VV^\circ_n}p_n(y,x)=n^{-1}|\del x\cap\del C^{\star}_{n,l}|\leq 1$ 
%(see the line above \eqv(7.lem3.5)).
(Here we used that $p_n(y,x)=p_n(y,x)$ if both $x$ and $y$ belong to $\VV^\circ_n$.)
From this and Proposition \thv(4.prop2) we readily get that
 if  $c_{\star}>2$ then on $\O^{\scriptscriptstyle{\textsf{SRW}}}\cap \O^{\star}$, 
 for large enough $n$,
\be
\Eq(7.theo2.14) 
\textstyle
%2 ({k^\circ_n(t)}/{|\VV^\circ_n|})
({a_n}/{|\VV^\circ_n|})
\sum_{m=1}^{\ell^\circ_n-1}I_{1,m}^{(3)}
\leq 
%2({k^\circ_n(t)}/{a_n})
%(2^n/|\VV^\circ_n|)
C'_\circ (1+o(1))\bar\s^{=}_n(u,\infty).
\ee
%where we used \eqv(4.prop2.0) to bound the ratio $2^n/|\VV^\circ_n|$.
%that by \eqv(4.prop2.0), $2^n/|\VV^\circ_n|\leq 1+o(1)$.
%for all but a finite number of indices $n$,

%%%%%%%%%%%%%%%%%%

The next lemma bounds the contribution to \eqv(7.theo2.8)
coming from $I_{2,m}^{(3)}$. Its proof is given in Subsection \thv(7.3.3).
%to deal with

\begin{lemma}
\TH(7.lem9) 
Assume that $c_{\star}>2$.
Let $\bar\rho^\circ_n>0$ be a decreasing sequence satisfying $\bar\rho^\circ_n\downarrow 0$ as $n\uparrow\infty$.
There exists a sequence of subsets $\overline\O_{2,n}^{(3)}\subset\O$ with
$
\P\bigl(\overline\O_{2,n}^{(3)}\bigr)\geq 1-{2^6\ell^\circ_n}/({\bar\rho^\circ_n na_n})
%\Eq(7.lem9.1) 
$
such that on $\overline\O_{2,n}^{(3)}\cap  \O^{\star}$, for all $n$ large enough,
\be
%2 \frac{k^\circ_n(t)}{|\VV^\circ_n|}
\textstyle 
%2 ({k^\circ_n(t)}/{|\VV^\circ_n|})
({a_n}/{|\VV^\circ_n|})
\sum_{m=1}^{\ell^\circ_n-1}I_{2,m}^{(3)}<
\bar\rho^\circ_n\left[\varsigma^{\neq}_{n}(u)\right]^2.
\Eq(7.lem9.2) 
\ee
where $\varsigma^{\neq}_{n}(u)$ is a positive decreasing function of $u>0$ that satisfies
\be
\lim_{n\rightarrow\infty}\varsigma^{\neq}_{n}(u)
%=\lim_{n\rightarrow\infty}\E[\bar\nu^\circ_n(u,\infty)]
=\nu^\dagger(u,\infty),\quad\forall u>0.
%\Eq(7.Lem7.7')
\Eq(7.lem9.13') 
\ee
\end{lemma}

%UNFINISHED
%Assuming 
Equipped with \thv(7.theo2.14) and Lemma \thv(7.lem9) 
%Gathering our bounds 
we conclude that under the assumptions
and with the notations of Proposition \thv(4.prop6), Lemma \thv(7.lem9), and Proposition \thv(4.prop2),
on
$
%\O^{\scriptscriptstyle{\textsf{EG}}}_{n}\equiv
\O^{\scriptscriptstyle{\textsf{SRW}}}\cap\overline\O_{2,n}^{(3)}\cap  \O^{\star}
$, 
for all but a finite number of indices $n$,
\be
\Eq(7.theo2.15)
I_3^{(2)}\leq 2 \frac{k^\circ_n(t)}{a_n}\left(
C''_\circ \bar\s^{=}_n(u,\infty)+
\bar\rho^\circ_n\left[\varsigma^{\neq}_{n}(u)\right]^2
\right)
\ee
for some constant $0<C''_\circ<\infty$. 
Inserting the bounds \eqv(7.theo2.6), \eqv(7.theo2.7), and \eqv(7.theo2.15) in \eqv(7.theo2.5)
then yields \eqv(7.theo2.1)-\eqv(7.theo2.3).
The proof of Theorem \thv(7.theo2) is done. \end{proof}

\subsection{Almost sure convergence of $\bar\nu^\circ_n$, $\bar\s^\circ_n$, and $\bar\s^{=}_n$.}
 \label{7.3}

As in Subsection \thv(6.3) our next step consists in 
proving strong laws of large numbers for the random (but now chain independant) quantities 
$\bar\nu^\circ_n$, $\bar\s^\circ_n$, and $\bar\s^{=}_n$ defined in \eqv(7.2.6), \eqv(7.2.7), and \eqv(7.2.8), respectively.
However, the complexity of these objects
%fact that these objects are quite involved 
(note in particular that they are sums of correlated random variables) makes this task much more arduous than in \textsc{fecp}.

%This however is a harder task as these objects are more complicated (in particular, they are sums of correlated random variables).

%This task is hovewer more involved as these objects are much more complicated 
%a much harder task as these objects are now sums of correlated random variables.

%\be
%(\hat a_n)^2 = n2^{\varepsilon n}, \quad \hat b_n=n b_n, \quad\text{and}\quad
%\hat a_n\P(w_n(x)\geq v \hat b_n)\sim 1.
%\ee

\begin{proposition}
{\TH(7.prop1)}
Given  $0<\varepsilon<1$ let the sequences $a_n$ and $b_n$ be defined through
\be
\lim_{n\rightarrow\infty}\frac{\log a_n}{n\log 2}=\varepsilon,\quad \sqrt{na_n}\P(w_n(x)\geq (n-1) b_n)\sim 1,
\Eq(7.prop1.0new)
\ee
Assume that $c_{\star}>2$ and let $\nu^\dagger$ be as in \eqv(1.theo2.M2). Then, there exists a subset 
$\overline\O^{\scriptscriptstyle{\textsf{LLN}}}\subset\O$ with 
$\P(\overline\O^{\scriptscriptstyle{\textsf{LLN}}})=1$ such that, 
on $\overline\O^{\scriptscriptstyle{\textsf{LLN}}}$, 
the following holds: 
%for $\nu^\dagger$ defined in \eqv(1.theo2.M2), 
for all $u>0$,
\bea
\Eq(7.prop1.1)
\lim_{n\rightarrow\infty}\bar\nu^\circ_n(u,\infty)\hspace{-6pt}&=&\hspace{-6pt}\nu^\dagger(u,\infty),
\Eq(7.prop1.2)
\\
\lim_{n\rightarrow\infty}n\bar\s^\circ_n(u,\infty)\hspace{-6pt}&=& \hspace{-6pt}
%2\nu^\dagger(2u,\infty)=
%\\
%\Eq(7.prop1.=)
\lim_{n\rightarrow\infty}n\bar\s^{=}_n(u,\infty)= 2\nu^\dagger(2u,\infty).
\eea
\end{proposition}

% we first prove convergence of the mean values of  $\bar\nu^\circ_n(u,\infty)$  and 

To prove Proposition \thv(7.prop1)  we first establish control over the mean values of $\bar\nu^\circ_n(u,\infty)$,  
$\bar\s^\circ_n(u,\infty)$, and $\bar\s^{=}_n(u,\infty)$ (see Lemmata  \thv(7.lem4) and \thv(7.lem7)), and  then prove  that these quantities concentrate around their means (in Lemmata  \thv(7.lem5), \thv(7.lem8)). 
Both these steps rely
% in a crucial way 
on the following key lemma.
%If the former task turns out to be the hardest, both of them rely in a crucial way on the following key lemma:
%The following lemma provides the corner stone of the proof s
%We first deal with $\bar\nu^\circ_n(u,\infty)$. The former task turns out to be much more arduous than the latter.
%Proving concentration around the mean is rather simple. Proving convergence of the mean is a much more arduous 
 %task. We begin with the latter.
%The next lemma is crucial to the proofs of both Lemmata \thv(7.lem4)  and \thv(7.lem5). 
Given sequences $\bar a_n$, $\bar b_n$, and two distinct vertices $x,y\in\VV_n$, set
%\be
%v_n(u; \bar a_n, \bar b_n)
%=
%{\bar  a_n}^2
%\E e^{-u\bar b_n/\min_{z\in\{x,y\}}w_n(z)}\1_{\{\min_{z\in\{x,y\}}w_n(z)\geq r_n(\rho^{\star}_n)\}}.
%\Eq(7.lem4.10') 
%\ee
\be
v_n(u; \bar a_n, \bar b_n)
=
\textstyle
{\bar  a_n}^2
\E\left[ \exp\left(-\frac{u\bar b_n}{\min\{w_n(x),w_n(y)\}}\right)\1_{\min\{w_n(x),w_n(y)\}\geq r_n(\rho^{\star}_n)\}}\right].
\Eq(7.lem5.0) 
\ee

\begin{lemma}
\TH(7.lem6) 
If the sequences $\bar a_n$ and $\bar b_n$ satisfy
%\be
%%\textstyle
%\bar a_n\P(w_n(x)\geq  \bar b_n)\sim 1,
%\quad
%\lim_{n\rightarrow\infty}\frac{\bar a_n}{2^n}=0,
%\quad\text{\rm and}\quad 
%\lim_{n\rightarrow\infty}\frac{\log \bar a_n}{n\log2}= \bar\varepsilon
%\Eq(7.lem6.1) 
%\ee
$\bar a_n\P(w_n(x)\geq  \bar b_n)\sim 1$,
$\lim_{n\rightarrow\infty}\frac{\bar a_n}{2^n}=0$,
and
$
\lim_{n\rightarrow\infty}\frac{\log \bar a_n}{n\log2}= \bar\varepsilon
$
for some $\bar\varepsilon>0$, then
%and if moreover $0<2\b_c(\bar\varepsilon)/\b\leq 1$, then
\be
\Eq(7.lem6.2) 
\lim_{n\rightarrow\infty}v_n(u; \bar a_n, \bar b_n)
=u^{-2\a( \bar\varepsilon)}2\a( \bar\varepsilon)\G(2\a( \bar\varepsilon)).
\ee
\end{lemma}

%The rest of this subsection is organized as follows. Subsection \thv(7.3.1) we prove Lemma \thv(7.lem6) ; equipped with this we establish the convergences properties of $\bar\nu^\circ_n(u,\infty)$ and $\s\nu^\circ_n(u,\infty)$ in 
%Subsection \thv(7.3.2) and Subsection \thv(7.3.3), respectively.

The rest of Subsection \thv(7.3) is organized as follows. We prove Lemma \thv(7.lem6) in Subsection \thv(7.3.1). 
Equipped with this result we establish the convergence properties of  $\bar\nu^\circ_n(u,\infty)$ in Subsection \thv(7.3.2), and those of $\bar\s^\circ_n(u,\infty)$ and $\bar\s^{=}_n(u,\infty)$ in Subsection \thv(7.3.3). Subsection \thv(7.3.3) also contains the proof of of Lemma \thv(7.lem9). 
The proof of Proposition \thv(7.prop1) is then completed in Subsection \thv(7.3.4).
%We finally conclude the proof of Proposition \thv(7.prop1) in Subsection \thv(7.3.4).
%The proof of Proposition \thv(7.prop1) is finally completed/concluded in Subsection \thv(7.3.4).

%%%%%%%%%%%%%%%%%%%%%%%%%%%%%%%%%%%%%%%%%%%%%%%%%%%%%%
\subsubsection{Proof of Lemma \thv(7.lem6)}
 \label{7.3.1}

The proof uses the following lemma, taken from \cite{G10b}.
%To prove Lemma \thv(7.lem6) we use a lemma, proved in \cite{G10b}, that we now recall.
%The lemma below, taken from \cite{G10b}, is needed in the proof of Lemma \thv(7.lem6).
%Before proving Lemma \thv(7.lem6) we state a lemma, proved in \cite{G10b}.
%for the convenience of the reader.
%given positive sequences $\bar a_n$ and $\bar b_n$ 
Set
\be
F_n(v)=\bar a_n\P(w_n(x)\geq v \bar b_n),\quad v\geq 0.
\Eq(7.lem2.0) 
\ee

\begin{lemma}[Lemma 2.1 of \cite{G10b}] 
\TH(7.lem2) 
Let the sequences $\bar a_n$ and $\bar b_n$ be as in Lemma \thv(7.lem6).

\item{(i)} For each fixed $\zeta>0$ and all $n$ sufficiently large so that $\zeta>\bar b_n^{-1}$ we have,
for all $v$ such that $\zeta\leq v<\infty$,
\be
F_n(v)= v^{-\a_n}(1+o(1)),
\Eq(7.lem2.3) 
\ee
where $0\leq \a_n=\a(\bar\varepsilon)+o(1)$.

\item{(ii)} Let $0<\d<1$.
Then, for all $v$ such that $\bar b_n^{-\d}\leq v\leq 1$ and all large enough $n$,
\be
v^{-\a_n}(1+o(1))
\leq F_n(v)\leq
\sfrac{1}{1-\d}v^{-\a_n(1-\frac{\d}{2})}(1+o(1)),
\Eq(7.lem2.4) 
\ee
where $\a_n$ is as before.
\end{lemma}

%%%%%%%%%%%%%%%%%%%%%%%%%%%%DEBUT ICI %%%%%%%%%%%%%%%%%%%%%%%%%

\begin{proof}[Proof of Lemma \thv(7.lem6)] 
For fixed $u>0$ set $f(y)=e^{-u/y}$.
%Thus $f(0)=0$, $f'(y)=(u/y^2)e^{-u/y}$.
Integrating by part in \eqv(7.lem5.0), and using  the fact that
$
{\bar  a_n}^2\P\left(\min\{w_n(x),w_n(y)\}\geq y\bar b_n\right)
%=\bar  a_n\P(w_n(x)\geq y \bar b_n)\P(w_n(y)\geq y \bar b_n)
=F^2_n(y)
$,
%where we used \eqv(7.lem2.0), 
we get
\be
\textstyle
\E[v_n(u; \bar a_n, \bar b_n)]
%&=&
%{\bar  a_n}^2\int_{0}^{\infty}f'(y)\P\left(\bar\g_n(\CC)\geq y\vee r_n(\rho^{\star}_n)/ \bar b_n\right)dy
%\\
=
F^2_n(r_n(\rho^{\star}_n)/ \bar b_n)
e^{-u\bar b_n/r_n(\rho^{\star}_n)}
+\int_{r_n(\rho^{\star}_n)/ \bar b_n}^{\infty}f'(y)F^2_n(y)dy.
\Eq(7.lem4'.3)
\ee
Note that
$
%{\bar  a_n}^2\P\left(\bar\g_n(\CC)\geq r_n(\rho^{\star}_n)/ \bar b_n\right)
F^2_n(r_n(\rho^{\star}_n)/ \bar b_n)
e^{-u\bar b_n/r_n(\rho^{\star}_n)}
\leq
{\bar  a_n}^2e^{-u\bar b_n/r_n(\rho^{\star}_n)}
\downarrow 0
$
as $n\uparrow\infty$ for all $u>0$.
To deal with the second term in the r.h.s.~of  \eqv(7.lem4'.3)  set  $I_n(a,b)=\int_{a}^{b}f'(y)F^2_n(y)dy$, $a\leq b$, and, given 
 $0<\varsigma<1$ and $\zeta>1$, write
\be
I_n(r_n(\rho^{\star}_n)/ \bar b_n,\infty)
=
I_n\bigl(r_n(\rho^{\star}_n)/ \bar b_n,\varsigma\bigr)+I_n\bigl(\varsigma,\zeta\bigr)+I_n(\zeta,\infty).
\Eq(7.lem4'.3')
\ee
Let us now show that,  as $n\uparrow\infty$, for small enough $\varsigma$ and large enough $\zeta$,
the leading contribution to \eqv(7.lem4'.3') comes from $I_n\bigl(\varsigma,\zeta\bigr)$.
%To evaluate $I_n\bigl(0,r_n^{-1/2}\bigr)$ we use that the rough upper bound
To do so we first use that by Lemma \thv(7.lem2), (ii) with $\d=1/2$,
$
I_n\bigl(r_n(\rho^{\star}_n)/ \bar b_n,\varsigma\bigr)
\leq 2(1+o(1))\int_{0}^{\varsigma}f'(y)y^{-(3/2)\a_n}dy
$
%valid
for all $0<\varsigma<1$, where $0\leq \a_n=\a(\bar\varepsilon)+o(1)$.
Now, there exists $\varsigma^*\equiv\varsigma^*(u)>0$ such that, for all $\varsigma<\varsigma^*$,
%(namely for
%$
%\varsigma\leq \frac{u^{2+(3/4)\a_n}}{2+(3/4)\a_n}
%$
%)
%the integrand in the last integral
$f'(y)y^{-(3/2)\a_n}$ is strictly increasing on $[0,\varsigma]$. Hence,
for all $\varsigma<\min(1,\varsigma^*)$,
$
I_n\bigl(r_n(\rho^{\star}_n)/ \bar b_n,\varsigma\bigr)
%\leq 2(1+o(1))\varsigma f'(\varsigma)\varsigma^{-(3/4)\a_n}
\leq 2(1+o(1))u\varsigma^{-1+(3/2)[\a(\bar\varepsilon)+o(1)]}e^{-{u}/{\varsigma}}
$,
implying that
\be
\lim_{n\rightarrow\infty}I_n\bigl(r_n(\rho^{\star}_n)/ \bar b_n,\varsigma\bigr)
\leq 2u\varsigma^{-1+(3/2)\a(\bar\varepsilon)}e^{-{u}/{\varsigma}}\,,\quad \varsigma<\min(1,\varsigma^*)\,.
\Eq(7.lem4'.3''')
\ee
To deal with $I_n\bigl(\varsigma,\zeta\bigr)$ note that
by Lemma \thv(7.lem2), (i), $F_n(y)\rightarrow y^{-\a(\bar\varepsilon)}$, $n\rightarrow\infty$,
where the convergence is uniform in  $\varsigma\leq y\leq \zeta$ since, for each $n$,
$F_n(y)$ is a monotone function, and since the limit, $y^{-\a(\bar\varepsilon)}$,
is continuous. Hence,
\be
\textstyle
\lim_{n\rightarrow\infty}I_n\bigl(\varsigma,\zeta\bigr)
=\lim_{n\rightarrow\infty}\int_{\varsigma}^{\zeta}f'(y)F^2_n(y)dy
=\int_{\varsigma}^{\zeta}f'(y)y^{-2\a(\bar\varepsilon)}dy\,.
\Eq(7.lem4'.4)
\ee
It remains to bound $I_n(\zeta,\infty)$. By
\eqv(7.lem2.3) of Lemma \thv(7.lem2),
$
I_n(\zeta,\infty)
=\int_{\zeta}^{\infty}f'(y)F^2_n(y)dy
=(1+o(1))\int_{\zeta}^{\infty}f'(y)y^{-2\a_n}dy\,,
$
% where $\a_n=\a(\varepsilon)(1+o(1))$ and where $\a(\varepsilon)$ is given by \eqv(5.theo1.0').
where again $0\leq \a_n=\a(\bar\varepsilon)+o(1)$.
%Recall that $\zeta>1$ and that $\a_n=\a(\varepsilon)+o(1)$.
Thus, for $0<\d<1$ arbitrary we have, taking $n$ large enough, that for all $y\geq\zeta>1$,
$
f'(y)y^{-2\a_n}
\leq f'(y)y^{-2\a(\bar\varepsilon)+2\d}
\leq u/y^{2(1-\d)}
$.
Therefore
$
I_n(\zeta,\infty)
\leq(1+o(1))\frac{1}{1-\d}\zeta^{-(1-2\d)}
$
and, choosing e.g.~$\d=1/4$,
\be
\lim_{n\rightarrow\infty}I_n(\zeta,\infty)\leq 4u\zeta^{-1/2}\,.
\Eq(7.lem4'.5)
\ee
Collecting \eqv(7.lem4'.3)-\eqv(7.lem4'.5),
%\eqv(7.lem4'.3), \eqv(7.lem4'.3'), \eqv(7.lem4'.3'''), \eqv(7.lem4'.4) and \eqv(7.lem4'.5),
we obtain that for all $\zeta>1$ and $\varsigma<\min(1,\varsigma^*)$,
\be
\textstyle
\E[v_n(u; \bar a_n, \bar b_n)]
=\int_{\varsigma}^{\zeta}f'(y)y^{-2\a(\bar\varepsilon)}dy
+\RR(\varsigma,\zeta)\,,
\Eq(7.lem4'.6)
\ee
where
$
0\leq \RR(\varsigma,\zeta)\leq
2u\varsigma^{-1+(3/2)\a(\bar\varepsilon)}e^{-{u}/{\varsigma}}
+4u\zeta^{-1/2}
$.
Finally, passing to the limit $\varsigma\rightarrow 0$ and $\zeta\rightarrow\infty$ in \eqv(7.lem4'.6) yields
\be
\textstyle
%\hspace{-8pt}
\E[v_n(u; \bar a_n, \bar b_n)]=\int_{0}^{\infty}f'(y)y^{-2\a(\bar\varepsilon)}dy
=u^{-2\a( \bar\varepsilon)}2\a( \bar\varepsilon)\G(2\a( \bar\varepsilon)),
\Eq(7.lem4'.8)
\ee
where  we used in the last equality that $2\a(\bar\varepsilon)>0$ since $\bar\varepsilon>0$. Lemma \thv(7.lem6) is proven.
%where the last equality holds only if $0< 2\a(\bar\varepsilon)\leq 1$.   Lemma \thv(7.lem6) is proven.
%where we used the assumption that $0< 2\a(\bar\varepsilon)\leq 1$ in the last equality. The proof of 
\end{proof}
%%%%%%%%%%%%%%%%%%%%%%%%%%%%%FIN ICI %%%%%%%%%%%%%%%%%%%%%%%%%

%We are now ready to prove Lemma \thv(7.lem4).

%%%%%%%%%%%%%%%%%%%%%%%%%%%%%%%%%%%%%%%%%%%%%%%%%%%%%%

\smallskip
\subsubsection{Convergence properties of $\bar\nu^\circ_n$}
 \label{7.3.2}
 
 As stated in the next two lemmata $\bar\nu^\circ_n$ concentrates around its mean, and the mean as a limit.
 
%In this subsection we prove concentration of $\bar\nu^\circ_n$ around its mean value
%and find the limit of the this mean, as stated in the next two lemmata.
%and control the limit of the latter.

%In this subsection we establish the following two lemmata.
%We are now ready to prove Lemma \thv(7.lem4)

\begin{lemma}
\TH(7.lem4) 
%Under the assumption and with the notation of Proposition \thv(7.prop1)
%Let $a_n$ and $b_n$ be as in Proposition \thv(7.prop1).
Assume that $c_{\star}>2$.
If $a_n$ and $b_n$ satisfy \eqv(7.prop1.0new) for some  $0<\varepsilon<1$, then
\be
\Eq(7.lem4.1) 
\lim_{n\rightarrow\infty}\E[\bar\nu^\circ_n(u,\infty)]=\nu^\dagger(u,\infty),\quad\forall u>0.
\ee
%where $\nu^\dagger$ is defined in \eqv(1.theo2.M2). 
\end{lemma}

\begin{lemma}
\TH(7.lem5) 
%Let $a_n$ and $b_n$ be as in Proposition \thv(7.prop1)
For all $L_1>0$ and $L_2\geq 0$ such that $na_nL_2/2^n=o(1)$, for all $u>0$,
\be
\textstyle
\P\bigl(\left|\bar\nu^\circ_n(u,\infty)-\E[\bar\nu^\circ_n(u,\infty)]\right|
\geq 
\phi_n(u,L_1,L_2)
\bigr)
\leq 2n e^{-L_2}+2L_1,
\Eq(7.lem5.1) 
\ee
where
$
\phi_n(u,L_1,L_2)\equiv
\sqrt{8na_nL_2V_{n,2}(2u)/2^n}
+
4n^{-(c_{\star}-1)} V_{n,1}(u)/L_1
$,
and where for each $n$, $V_{n,1}(u)$ and $V_{n,2}(u)$ are positive decreasing functions, while for each $u>0$, under the assumptions of Lemma \thv(7.lem4),
\be
\lim_{n\rightarrow\infty}V_{n,1}(u)
=\lim_{n\rightarrow\infty}V_{n,2}(u)
%=\lim_{n\rightarrow\infty}\E[\bar\nu^\circ_n(u,\infty)]
=\nu^\dagger(u,\infty).
\Eq(7.lem5.2) 
\ee
\end{lemma}

\begin{proof}[Proof of Lemma \thv(7.lem4)] Write 
$
\bar\nu^\circ_n(u,\infty)
=
\sum_{k\geq 2}
\bar\nu^{\circ,(k)}_n(u,\infty)
$
where
\be
\bar\nu^{\circ,(k)}_n(u,\infty)
\equiv
\Eq(7.lem4.4) 
\frac{a_n}{2^n}\,
{\textstyle 
\sum_{1\leq l\leq L^{\star}}\1_{\{|C^{\star}_{n,l}|=k\}}\sum_{x\in C^{\star}_{n,l}}Q_{n,l}^u(x)
}.
\ee
We saw in Subsection \thv(8.3) (see \eqv(8.3lem1.5)) that on $\O^{\star}$, 
$
%|C^{\star}_{n,l}|\leq k_n^{\star}
%\equiv 
%\frac{1}{\rho^{\star}_n[1-c_{\star}^{-1}(1+\OO(\log n/n))]}\sim
%\frac{n{\log 2}}{\log n}[c_{\star}-(1+o(1))]^{-1}
k_n^{\star}\equiv\max_{2\leq l\leq L^{\star}}|C^{\star}_{n,l}(\rho)|
 \leq\frac{n}{(c_{\star}-2)\log n}
$
for all large enough $n$.
We may thus restrict the range of $k$ to $2\leq k\leq k_n^{\star}$. 
Let now $\GG_k$ be the collection of all vertex sets $\CC\subset \VV_n$ of size $k$ such that
$G(\CC)$ forms a connected subgraph of $G(\VV_n)$,
% (see Section \thv(3) for the notation), 
\be
\GG_k=\{
\CC=\{x_1,\dots,x_k\} : \forall_{1\leq i\leq k}\exists_{1\leq j\leq k} \,\,\,\text{such that} \,\,\, \dist(x_i,x_j)=1
\}.
\Eq(7.lem4.2) 
\ee
Then \eqv(7.lem4.4) can be written as
%$\P\left(\chi_n(x)=1\right)=n^{-c_{\star}}$,
%\be
%\textstyle
%\1_{\{C^{\star}_{n,l}=\CC\}}
%=\prod_{x\in\CC}\chi_n(x)\prod_{x'\in\del\CC}\overline\chi_n(x'),
%\ee
%for all $\CC\in\GG_k$, and so 
\be
\Eq(7.lem4.3) 
\bar\nu^{\circ,(k)}_n(u,\infty)
\equiv
\frac{a_n}{2^n}\,
{\textstyle 
\sum_{\CC\in \GG_k}\prod_{x\in\CC}\chi_n(x)\prod_{x'\in\del\CC}\overline\chi_n(x')\sum_{x\in\CC}Q_{n,\CC}^u(x)
},
\ee
where $Q_{n,\CC}^u(x)$ stands for $Q_{n,l}^u(x)$ with $C^{\star}_{n,l}\equiv\CC$, and
%Where $Q_{n,\CC}^u(x)\equiv Q_{n,l}^u(x)$ with 
%%$C^{\star}_{n,l}\equiv\CC$, and
%$C^{\star}_{n,l}$ replaced by $\CC$, and
where $\chi_n(x)\equiv\1_{\left\{w_n(x)\geq r_n(\rho^{\star}_n)\right\}}$,  
$\overline\chi_n(x)\equiv1-\chi_n(x)$, are  Bernoulli variables  r.v.'s~with $\P\left(\chi_n(x)=1\right)=n^{-c_{\star}}$.
To further express  $\bar\nu^{\circ,(k)}_n$ we distinguish the case $k=2$ from the case 
$3\leq k\leq  k_n^{\star}$.

\smallskip
\noindent{\textbf{\emph{ $\bullet$ The case $k=2$.}}} Here $\GG_2$ is the set of undirected edges of $G(\VV_n)$
and $Q_{n,\CC}^u(x)$ is given by Proposition \thv(5.prop1), (i) :
observing  that $Q_{n,\CC}^u(x)=Q_{n,\CC}^u(y)$
on $\CC=\{x,y\}$,  and that, by \eqv(5.lem1.0),
$
\varrho_{n,l}(0)
%=e^{-\b\max_{x\in\CC}H_n(x)}
=\min_{x\in\CC}w_n(x)
$
when $C^{\star}_{n,l}=\CC$,
we obtain
\be
\Eq(7.lem4.5) 
\bar\nu^{\circ,(2)}_n(u,\infty)
\equiv
2\frac{a_n}{2^n}\,
{\textstyle 
\sum_{\CC\in \GG_2}\prod_{x\in\CC}\chi_n(x)\prod_{x'\in\del\CC}\overline\chi_n(x')
\left(1-
\sfrac{1}{1+
\frac{\min_{x\in\CC}w_n(x)}{(n-1)}
%\min_{x\in\CC}w_n(x)/(n-1)
}
\right)^{\lceil b_nu\rceil}
}.
\ee
From this we easily derive the bounds
\be
\textstyle
\bar\nu^{\circ,(2),-}_n(u,\infty)
(1-s_n)
\leq \bar\nu^{\circ,(2)}_n(u,\infty)\leq
\bar\nu^{\circ,(2),+}_n(u,\infty)
\Eq(7.lem4.6) 
\ee
where
$
s_n=\frac{n-1}{r_n(\rho^{\star}_n)}
$
and where, setting 
$
b_n^\pm= b_n(n-1)(1-s_n)^{\pm1}
$
and
\be
\g^{\pm}_n(\CC)=\min_{x\in\CC}w_n(x)/b^{\pm}_n,
\Eq(7.lem4.12) 
\ee
\be
\Eq(7.lem4.7) 
\bar\nu^{\circ,(2),\pm}_n(u,\infty)\equiv
2\frac{a_n}{2^n}\,
{\textstyle 
\sum_{\CC\in \GG_2}
%\prod_{x\in\CC}\chi_n(x)
\prod_{x'\in\del\CC}\overline\chi_n(x')
e^{-u/\g^\pm_n(\CC)}\1_{\{\g^\pm_n(\CC)\geq r_n(\rho^{\star}_n)/ b^{\pm}_n\}}.
%\exp\left(-\frac{u b_n^\pm}{\min_{x\in\CC}w_n(x)}\right)
}
\ee
By Lemma \thv(7.lem6),
\be
\lim_{n\rightarrow\infty}\E\bar\nu^{\circ,(2),-}_n(u,\infty)
=
\lim_{n\rightarrow\infty}\E\bar\nu^{\circ,(2),+}_n(u,\infty)
=\nu^\dagger(u,\infty).
\Eq(7.lem4.8) 
\ee
To see this note that, setting $a_n^+=a_n^-= \sqrt{n a_n}(1-n^{-c_{\star}})^{n-1}$,
%by \eqv(7.lem4.7),
\be
\E[\bar\nu^{\circ,(2),\pm}_n(u,\infty)]
=
v_n(u;  a_n^\pm, b_n^\pm).
%{\bar  a_n}^2
%\E e^{-u/\bar\g_n(\CC)}\1_{\{\bar\g_n(\CC)\geq r_n(\rho^{\star}_n)/ \bar b_n\}},\quad \CC\in\GG_2,
\Eq(7.lem4.10) 
\ee
% and $\bar b_n=b_n^+$ when dealing with $\bar\nu^{\circ,(2),+}_n$, 
% respectively, $\bar b_n=b_n^-$ when dealing with $\bar\nu^{\circ,(2),-}_n$. 
One then readily checks that for 
$a_n$, $b_n$ as in \eqv(7.prop1.0new), 
$
\lim_{n\rightarrow\infty}\frac{\log a_n^\pm}{n\log2}= \varepsilon/2
$,
$
\lim_{n\rightarrow\infty}\frac{a_n^\pm}{2^n}=0
$,
and
$
a_n^\pm\P(w_n(x)\geq  b_n^\pm)\sim 1
$.
The conditions of Lemma \thv(7.lem6)  are thus
satisfied with $\bar \varepsilon=\varepsilon/2$,
% for both $\bar b_n=b_n^-$ and $\bar b_n=b_n^+$.
yielding \eqv(7.lem4.8).
%Finally, by Lemma \thv(9.lem4'), 
Since clearly $\lim_{n\rightarrow\infty}s_n=0$, it follows from \eqv(7.lem4.6) that
\be
\lim_{n\rightarrow\infty}\E\bar\nu^{\circ,(2)}_n(u,\infty)
=\nu^\dagger(u,\infty).
\Eq(7.lem4.11)
\ee

%%%%%%%%%%%%%%%%%%%%%%  FIN %%%%%%%%%%%%%%%%%%%%%%%%%%%%%%%%%

\smallskip
\noindent{\textbf{\emph{$\bullet$ The case $3\leq k\leq k_n^{\star}$.}}} Recall that $Q_{n,\CC}^u(x)$ in \eqv(7.lem4.3) stands for $Q_{n,l}^u(x)$ with $C^{\star}_{n,l}\equiv\CC$. Similarly, denote by
%for $C^{\star}_{n,l}=\CC$, 
$\varrho_{n,\CC}(0)$, $\varrho_{n,\CC}(1)$, and $\theta^{\star}_{n,\CC}$   the quantities
$\varrho_{n,l}(0)$, $\varrho_{n,l}(1)$, and $\theta^{\star}_{n,l}$ from \eqv(5.lem1.0)-\eqv(5.prop1.1)
with $C^{\star}_{n,l}\equiv\CC$.
Let us now split the sum over $\CC\in \GG_k$ in \eqv(7.lem4.3) according to whether $b_n< n\theta^{\star}_{n,\CC}$ or
%$b_n\geq n\theta^{\star}_{n,\CC}$. Namely, set
not, and write
$
\bar\nu^{\circ,(k)}_n(u,\infty)
=\bar\nu^{\circ,(k),-}_n(u,\infty)
+\bar\nu^{\circ,(k),+}_n(u,\infty)
$
where 
\be
\Eq(7.lem4.13) 
\hspace{-0.8pt}
\bar\nu^{\circ,(k),+}_n(u,\infty)
\equiv
\frac{a_n}{2^n}\,
{\textstyle 
\sum_{\CC\in \GG_k}\prod_{x\in\CC}\chi_n(x)\prod_{x'\in\del\CC}\overline\chi_n(x')\sum_{x\in\CC}Q_{n,\CC}^u(x)
\1_{\{b_n\geq n\theta^{\star}_{n,\CC}\}}
},\hspace{-7pt}
\ee
and  $\bar\nu^{\circ,(k),-}_n(u,\infty)$ is the sum over the complement, i.e.~over terms such that $b_n< n\theta^{\star}_{n,\CC}$. The point of doing this is that if $b_n\geq n\theta^{\star}_{n,\CC}$ then, for each fixed $u>0$,  ${\lceil b_nu\rceil}\gg \theta^{\star}_{n,\CC}$ for all $x$ in $\CC$ and all $n$ large enough, so that $\bar\nu^{\circ,(k),+}_n(u,\infty)$ can be bounded using Proposition \thv(5.prop1), (ii). More precisely, on $\O_0\cap \O^{\star}$, for all but a finite number of indices $n$, by \eqv(5.prop1.0'),
\be
\Eq(7.lem4.14) 
Q_{n,\CC}^u(x)
\1_{\{b_n\geq n\theta^{\star}_{n,\CC}\}}
%\leq  e^{-(unb_n/\varrho_{n,\CC}(0)k)(1-o(1))}(1+o(1))
\leq  e^{-\{u(n-1)b_n/\varrho_{n,\CC}(0)\}}(1+o(1)), \quad \forall x\in\CC,
\ee
(since for $k\geq 3$ and large enough $n$, $k(n-1)/n(1-o(1))>1$).
% we could keep $k$ since it helps but it doesn't help so much either
Note that by \eqv(5.lem1.0),
\bea
\Eq(7.lem4.15) 
e^{-\{u(n-1)b_n/\varrho_{n,\CC}(0)\}}
&=&
\textstyle
\max_{\{x,y\}\in G(\CC)}e^{-\{u(n-1)b_n/\min\left(w_n(y),w_n(x)\right)\}}
\\
&\leq&
\Eq(7.lem4.16) 
\textstyle
%\sum_{\{x,y\}\in G(\CC)}e^{-\{u(n-1)b_n/\min\left(w_n(y),w_n(x)\right)\}}
\sum_{\{x,y\}\in G(\CC)}e^{-u/\bar\g_n(\{x,y\})},
\eea
where we now set 
%now (as in \eqv(7.lem4.12)),
\be
\bar\g_n(\CC')=\min_{x\in\CC'}w_n(x)/{(n-1)b_n}, \quad \CC'\in\GG_2.
\Eq(7.lem4.17) 
\ee
Combining these observations yields the bound
\be
\Eq(7.lem4.18) 
\bar\nu^{\circ,(k),+}_n(u,\infty)
\leq
k\frac{a_n}{2^n}\,
{\textstyle 
\sum_{\CC'\in \GG_2}
\sum_{\CC\in \GG_k:\CC'\subset\CC}\prod_{x\in\CC}\chi_n(x)
%\prod_{x'\in\del\CC}\overline\chi_n(x')
e^{-u/\bar\g_n(\CC')}
},
\ee
valid on $\O_0\cap \O^{\star}$, for all but a finite number of indices $n$, and this in turns implies that
\be
\Eq(7.lem4.19) 
\E[\bar\nu^{\circ,(k),+}_n(u,\infty)]
\leq
k (k-2)! n^{-(c_{\star}-1)(k-2)}v_n(u;  \bar  a_n, \bar b_n),
%\E\left[ {\bar  a_n}^2 e^{-u/\bar\g_n(\CC')}\1_{\{\bar\g_n(\CC')\geq r_n(\rho^{\star}_n)/ \bar b_n\}}\right],
\ee
where $\bar  a_n\equiv\sqrt{na_n}$, $\bar b_n\equiv(n-1)b_n$.
Again one sees that these sequences (that differ but slightly from 
the choices made in \eqv(7.lem4.10))  satisfy the conditions 
 of Lemma \thv(7.lem6)   with $\bar \varepsilon=\varepsilon/2$.
 Thus
\be
\lim_{n\rightarrow\infty}
\E\left[v_n(u;  \bar  a_n, \bar b_n)\right]
=\nu^\dagger(u,\infty).
\Eq(7.lem4.21)
\ee
Since by assumption $c_{\star}>2$ we may use \eqv(8.3lem1.sum) to sum \eqv(7.lem4.19) over $k$, which gives
\be
\textstyle
\sum_{3\leq k\leq k_n^{\star}}\E[\bar\nu^{\circ,(k),+}_n(u,\infty)]
\leq 4n^{-(c_{\star}-1)}\E\left[v_n(u;  \bar  a_n, \bar b_n)\right].
\Eq(7.lem4.22)
\ee

It remains to bound the term $\bar\nu^{\circ,(k),-}_n(u,\infty)$ (see the line above \eqv(7.lem4.13)). For this we use the trite bound $Q_{n,\CC}^u(x)\leq 1$ and observe that, by  \eqv(5.lem1.0)-\eqv(5.prop1.1) and the 
rightmost inequality in Lemma \thv(9.lem4),
%\eqv(9.lem4.1),
$
% that is, the bound $\varrho_{n,l}(1)\geq r_n\bigl(\rho^{\star}_n\bigr)$ from Lemma \thv(9.lem4)
%\be
%\Eq(7.lem4.24) 
%\textstyle
\1_{\{b_n< n\theta^{\star}_{n,\CC}\}}\prod_{x\in\CC}\chi_n(x)
\leq 
\1_{\{
\varrho_{n,\CC}(0)> r_n(\rho^{\star}_n)b_n/(2\b n^2k^5)
\}}
\prod_{x\in\CC}\chi_n(x).
%\ee
$
Next, by \eqv(5.lem1.0) and \eqv(7.lem4.17),
\be
%\Eq(7.lem4.25) 
\{
\varrho_{n,\CC}(0)> r_n(\rho^{\star}_n)b_n/(2\b n^2k^5)
\}
\subseteq
\textstyle
\cup_{\{x,y\}\in G(\CC)}\{
\bar\g_n(\{x,y\})> r_n(\rho^{\star}_n)/\b n^8
\}
%\textstyle
%\1_{\{
%\varrho_{n,\CC}(0)> r_n(\rho^{\star}_n)b_n/(2\b n^2k^5)
%\}}
%\leq 
%\sum_{\{x,y\}\in G(\CC)}\1_{\{
%\bar\g_n(\{x,y\})> r_n(\rho^{\star}_n)/\b n^8
%\}}
\ee
for large enough $n$. Thus
\be
\Eq(7.lem4.23) 
\bar\nu^{\circ,(k),-}_n(u,\infty)
\leq 
k\frac{a_n}{2^n}\,
{\textstyle 
\sum_{\CC'\in \GG_2}
\sum_{\CC\in \GG_k:\CC'\subset\CC}\prod_{x\in\CC}\chi_n(x)
%\prod_{x'\in\del\CC}\overline\chi_n(x')
\1_{\{
\bar\g_n(\CC')> r_n(\rho^{\star}_n)/\b n^8
\}}
},
\ee
and averaging out,
\be
\Eq(7.lem4.26) 
\E[\bar\nu^{\circ,(k),-}_n(u,\infty)]
%&\leq&
%k (k-2)! n^{-(c_{\star}-1)(k-2)}
%\left\{ {\bar  a_n}^2\P\left[\bar\g_n(\{x,y\})> r_n(\rho^{\star}_n)/\b n^8\right] \right\}
%\\
\leq
k (k-2)! n^{-(c_{\star}-1)(k-2)}
\left[ {\bar  a_n}\P(w_n(0)> \bar b_nr_n(\rho^{\star}_n)/\b n^8) \right]^2,
\ee
where $\bar  a_n$, $\bar b_n$ are as before (see the line below  \eqv(7.lem4.19)). A simple Gaussian tail estimate gives
$
\left[ {\bar  a_n}\P(w_n(0)> \bar b_nr_n(\rho^{\star}_n)/\b n^8) \right]^2
\leq
\left(
\b n^8/r_n(\rho^{\star}_n)
\right)^{2\a(\varepsilon/2)(1+o(1))}
$.
Again, the assumption that $c_{\star}>2$ enables us to use \eqv(8.3lem1.sum), yielding
\be
\Eq(7.lem4.27) 
\textstyle
%\lim_{n\rightarrow\infty}
0\leq \sum_{3\leq k\leq k_n^{\star}}\E[\bar\nu^{\circ,(k),-}_n(u,\infty)]
\leq
4n^{-(c_{\star}-1)} \left(
\b n^8/r_n(\rho^{\star}_n)
\right)^{2\a(\varepsilon/2)(1+o(1))}.
\ee
%(where $0<\varepsilon<1$)

Now set
\be
\Delta_n(u)\equiv
\textstyle\sum_{k\geq 3}\bar\nu^{\circ,(k)}_n(u,\infty)
=
\hat\nu^{\circ}_n(u,\infty)-\bar\nu^{\circ,(2)}_n(u,\infty)>0.
\Eq(7.lem4.30)
\ee
Combining  \eqv(7.lem4.22), \eqv(7.lem4.27),
and using \eqv(9.lem4'.2) to bound $r_n(\rho^{\star}_n)$,
we obtain that under the assumptions of the lemma,  for all $u>0$,
\be
\Eq(7.lem4.28) 
\textstyle
0\leq 
\E\Delta_n(u)
\leq 
4n^{-(c_{\star}-1)}\left(\E\left[v_n(u;  \bar  a_n, \bar b_n)\right]
+e^{-\b\sqrt{8 \varepsilon n\log n}}
\right),
%\nu^\dagger(u,\infty), 
\ee
and so
$
%\be
%\textstyle
\lim_{n\rightarrow\infty}
%\sum_{3\leq k\leq k_n^{\star}}\E[\bar\nu^{\circ,(k),}_n(u,\infty)]=0.
\E\Delta_n(u)=0
%\Eq(7.lem4.29)
%\ee
$.
%This and \eqv(7.lem4.11) yield  \eqv(7.lem4.1), proving Lemma \thv(7.lem4)
But this and \eqv(7.lem4.11) yield  \eqv(7.lem4.1). The proof of Lemma \thv(7.lem4) is complete. 
\end{proof}

\begin{proof}[Proof of Lemma \thv(7.lem5)] 
%The notations 
%and definitions 
%introduced in  the proof of Lemma \thv(7.lem4) are used throughout without systematic
%reminder.
%For $\bar\nu^{\circ,(2)}_n(u,\infty)$ defined in \eqv(7.lem4.5) 
As in the proof of Lemma \thv(7.lem4) we separate the contribution of $\bar\nu^{\circ,(2)}_n$ from those of 
$\bar\nu^{\circ,(k)}_n$, $k\geq 3$ (see \eqv(7.lem4.3) and \eqv(7.lem4.5) for their definitions).
%Now, for $\bar\nu^{\circ,(2)}_n(u,\infty)$ defined in \eqv(7.lem4.5),  write
Namely, we write
\be
\bar\nu^\circ_n(u,\infty)-\E[\bar\nu^\circ_n(u,\infty)]
=
\bar\nu^{\circ,(2)}_n(u,\infty)-\E[\bar\nu^{\circ,(2)}_n(u,\infty)]
+\{\Delta_n(u)-\E[\Delta_n(u)]\}
\ee
where $\Delta_n(u)$ is defined in \eqv(7.lem4.30), and we take
\bea
V_{n,1}(u)&\equiv&\E\left[v_n(u;  \bar  a_n, \bar b_n)\right]+e^{-\b\sqrt{8 \varepsilon n\log n}},
\\
V_{n,2}(u)&\equiv&\E\left[v_n(u;  a_n^+, b_n^+)\right],
\eea
where $v_n(u;  \bar  a_n, \bar b_n)$ and $v_n(u;  a_n^+, b_n^+)$
are as in \eqv(7.lem4.28) and   \eqv(7.lem4.10), respectively. Eq.~\eqv(7.lem5.2) then
follows from \eqv(7.lem4.21) and \eqv(7.lem4.11). 
%By \eqv(7.lem4.28), since $\Delta_n(u)>0$,
Note that $\Delta_n(u)>0$. Thus, by \eqv(7.lem4.28),
\be
\P\left(\left|\Delta_n(u)-\E[\Delta_n(u)]\right|\geq 4n^{-(c_{\star}-1)} V_{n,1}(u)/L_1
\right)
\leq 2L_1,\quad \forall L_1>0.
\Eq(7.lem5.5) 
\ee

Let us now establish that
%, under the assumptions of the lemma, 
for all $L_2\geq 0$ such that $na_nL_2/2^n=o(1)$,
\be
\textstyle
\P\left(\left|\bar\nu^{\circ,(2)}_n(u,\infty)-\E[\bar\nu^{\circ,(2)}_n(u,\infty)]\right|
\geq 
\sqrt{8na_nL_2V_{n,2}(2u)/2^n}\right)
\leq e^{-L_2}.
\Eq(7.lem5.4) 
\ee
Eq.~\eqv(7.lem4.5)  prompts us to set $\SS_n\equiv\sum_{\CC\in\GG_2}X(\CC)$,
where $X(\CC)\equiv Y(\CC)-\E Y(\CC)$ and
\be
\textstyle
Y(\CC)\equiv\prod_{x\in\CC}\chi_n(x)\prod_{x'\in\del\CC}\overline\chi_n(x')
\Bigl(1-
\left[1+
\frac{\min_{x\in\CC}w_n(x)}{(n-1)}
\right]^{-1}
\Bigr)^{\lceil b_nu\rceil},\quad \CC\in\GG_2.
\Eq(7.lem5.12) 
\ee
Then
\be
\P\left(\left|\bar\nu^{\circ,(2)}_n(u,\infty)-\E[\bar\nu^{\circ,(2)}_n(u,\infty)]\right|\geq \theta\right)
=\textstyle{\P\left(\left|\SS_n\right|\geq 2^{n-1} a_n^{-1}\theta\right)}.
%\Eq(5.lem7.9)
\Eq(7.lem5.8) 
\ee
Observe that $\SS_n$ is a sum of dependent random variables.
% However, this sum can be split into $2n$ smaller sums, each of them being made of independent random variables. 
%%%%%%%%%%%%%%%%%%%%%%%%%%%%%%%%%%%%%%%%%%%%%
To get round this difficulty we split it into $2n$ disjoint sums as follows. Write $\GG_2=\cup_{1\leq j\leq n}\GG_2^j$
where, for each $1\leq j\leq n$, 
$
\GG_2^j\equiv\left\{\{x,y\}\in\GG_2\mid x_j=-y_j\right\}
$
is the set of neighboring vertices that differ in exactly the $j$-th coordinate. Subdivide 
each $\GG_2^j$ into $4$ disjoint sets, 
$
\GG_2^j=\cup_{1\leq i\leq 4}\GG_2^{j,i}
$,
%defined by
\be
\GG_2^{j,i}\equiv\left\{\{x,y\}\in\GG_2^j\mid \dist({\bf 1},\{x,y\})=4m+i-1, m\geq 0\right\},
%1\leq i\leq 4, 1\leq j\leq n,
\Eq(7.lem5.10')
\ee
where ${\bf 1}$ denotes the vertex of $\VV_n$ all of whose coordinates are 1.
%${\bf 1}=(1,\dots,1)$
%Then
%$
%\GG_2=\cup_{1\leq j\leq n,1\leq i\leq 4}\GG_2^{j,i}
%$
%Using these sets, $\LL^1_n(2)$ can be decomposed into $Q= 4n$ disjoint sets 
%obtained by taking $\LL^1_n(2)\cap\GG_2^{j,i}$, $1\leq j\leq n$ and $1\leq i\leq 4$. 
Then, 
\be
\textstyle
\SS_n=\sum_{i=1}^4\sum_{j=1}^n \SS_n^{j,i},\quad
\SS_n^{j,i}
\equiv\sum_{\CC\in\GG_2^{j,i}} X(\CC).
\Eq(7.lem5.10) 
\ee
Each $\SS_n^{j,i}$ now is a sum of independent random variables, and can be controlled using
Bennett's bound \cite{Ben} for the tail behavior of sums of random variables, which we specialize as follows: if
$(X(i),\, i\in\II)$ is a family of i.i.d.~centered random variables
that satisfies $\max_i|X(i)|\leq  A$ then, setting
$\tilde B^2=\sum_{i\in\II}\E X^2(i)$,  for all $ B^2\geq \tilde B^2$,
%\be
%\P\Bigl(\Bigl|\sum_{x\in\VV_n}X(x)\Bigr|>t\Bigr)
%\leq\exp\left\{
%\frac{t}{\bar a}-\left(\frac{t}{\bar a}+\frac{\bar b^2}{\bar a^2}\right)\log\left(1+\frac{\bar at}{\bar b^2}\right)
%\right\}\,,\quad t\geq 0\,.
%\Eq(5.lem7.10)
%\ee
%The behavior of the r.h.s. of \eqv(5.lem7.10) varies depending on the relative
%size of $t$ and of the ratio $\bar b^2/\bar a$. Note in particular that
for all $t< B^2/(2 A)$, 
%\eqv(5.lem7.10) simplifies to
\be
\textstyle{\P\left(\left|\sum_{i\in\II}X(i)\right|\geq t\right)}
\leq\exp\left\{-{t^2}/{4 B^2}\right\}.
%\Eq(5.lem7.11)
\Eq(7.lem5.7) 
\ee
Since $|X(\CC)|\leq 1$ and since
$
\sum_{\CC\in\GG_2^{j,i}}\E X^2(\CC)
\leq(4na_n)^{-1}2^{n-1}V_{n,2}(2u)
$,
as follows from \eqv(7.lem4.6)-\eqv(7.lem4.10),
we may choose $ A=1$ and $ B^2=(4na_n)^{-1}2^{n-1}V_{n,2}(2u)$
in  \eqv(7.lem5.7), and we get that  for all $L_2>0$,
\be
\P\left(\left|\SS_n^{j\pm}\right|\geq  2^{n-1} (4na_n)^{-1}\theta\right)
=
\textstyle
\P\bigl(\left|\SS_n^{j\pm}\right|\geq  \sqrt{2^{n-1}L_2V_{n,2}(2u)/(4na_n)}\bigr)
\leq e^{-L_2},
%%\Eq(5.lem7.15)
\Eq(7.lem5.9) 
\ee
where we chose $\theta^2= 4na_n L_22^{-n+1}V_{n,2}(2u)$.
This choice is permissible provided that
$\theta\leq V_{n,2}(2u)/2$. In view of 
\eqv(7.lem4.8)
this will be verified for all $n$ large enough whenever $\theta\downarrow 0$
as $n\uparrow\infty$, i.e.~whenever $na_nL_2/2^n=o(1)$. 
Eq.~\eqv(7.lem5.9)  holds true for each $1\leq j\leq n$ and $1\leq i\leq 4$,
and combined with \eqv(7.lem5.10) yields
\be
\textstyle \P\left(\left|\SS_n\right|\geq 4n\sqrt{2^{n-1}L_2V_{n,2}(2u)/(4na_n)}\right)
\leq 4n e^{-L_2},
\Eq(7.lem5.11) 
\ee
which, by \eqv(7.lem5.8), is tantamount to \eqv(7.lem5.4). Combining \eqv(7.lem5.5)
and \eqv(7.lem5.4) then yields \eqv(7.lem5.1) and concludes the proof of Lemma \thv(7.lem5).
\end{proof}

\smallskip
\subsubsection{Convergence properties of $\bar\s^\circ_n$ and related functions}
 \label{7.3.3}
%Next, turning to $\bar\s^\circ_n(u,\infty)$, we establish that
%In this subsection we establish the following two lemmata.

We have:

\begin{lemma}
\TH(7.lem7) 
Under the assumption and with the notation of Lemma \thv(7.lem4),
\be
\lim_{n\rightarrow\infty}n\E[\bar\s^\circ_n(u,\infty)]=
\lim_{n\rightarrow\infty}n\E[\bar\s^{=}_n(u,\infty)]= 
2\nu^\dagger(2u,\infty),\quad\forall u>0.
\Eq(7.lem7.1) 
\ee
\end{lemma}

\begin{lemma}
\TH(7.lem8) 
%%%%%%%%%%%%%%%%%%
For all $L_1,L_3>0$ and $L_2\geq 0$ such that $na_nL_2/2^n=o(1)$, for all $u>0$,
\bea
\Eq(7.lem8.1') 
&
\textstyle\P\bigl(\left|\bar\s^{=}_n(u,\infty)-\E[\bar\s^{=}_n(u,\infty)]\right|
\geq
\tilde\phi_n(u,L_1,L_2)
\bigr)
\leq 2n e^{-L_2}+4L_1,&
\\
\Eq(7.lem8.1) 
\textstyle
&
\P\bigl(\left|\bar\s^\circ_n(u,\infty)-\E[\bar\s^\circ_n(u,\infty)]\right|
\geq \psi_n(u,L_1,L_2,L_3)
\bigr)
 \leq 2n e^{-L_2}+2L_1+2L_3,\quad
 &
\eea
where
$
\tilde\phi_n(u,L_1,L_2)\equiv
\sqrt{a_nL_2W_{n,2}(2u)/2^n}
+
4n^{-c_{\star}} W_{n,1}(u)/L_1
$,
$
\psi_n(u,L_1,L_2,L_3)\equiv
\tilde\phi_n(u,L_1,L_2)+ 2^8W^2_{n,3}(u)/(a_nL_3)
$,
and where for each $n$ and $1\leq i\leq 3$, $W_{n,i}(u)$   are positive decreasing functions, while for each $u>0$, under the assumptions of Lemma \thv(7.lem4),
\bea
\Eq(7.lem8.2) 
&\lim_{n\rightarrow\infty}W_{n,1}(u)
=\lim_{n\rightarrow\infty}W_{n,2}(u)
%=\lim_{n\rightarrow\infty}2\E[\bar\nu^\circ_n(2u,\infty)]
=2\nu^\dagger(2u,\infty),&
\\
\Eq(7.lem8.3) 
&
\lim_{n\rightarrow\infty}W^2_{n,3}(u)
%=\lim_{n\rightarrow\infty}\E[\bar\nu^\circ_n(u,\infty)]
=[\nu^\dagger(u,\infty)]^2.&
\eea
\end{lemma}

We prove  Lemmata \thv(7.lem7) and \thv(7.lem8) simultaneously.

\begin{proof}[Proof of Lemma \thv(7.lem7) and Lemma \thv(7.lem8)] 
Write 
$
\bar\s^\circ_n(u,\infty)=\bar\s^{=}_n(u,\infty)+\bar\s^{\neq}_n(u,\infty)
$
where
\bea
\Eq(7.Lem7.2)
\bar\s^{=}_n(u,\infty)
\hspace{-6pt}&=&\hspace{-6pt}
\frac{a_n}{n 2^n}
\textstyle
\sum_{1\leq l\leq L^{\star}}
\left[\sum_{x\in C^{\star}_{n,l}}Q_{n,l}^u(x)\right]^2,
\\
\Eq(7.Lem7.3)
\bar\s^{\neq}_n(u,\infty)
\hspace{-6pt}&=&\hspace{-6pt}
\frac{a_n}{n^22^n}
\textstyle
\sum_{1\leq l\neq l'\leq L^{\star}}
\sum_{x\in C^{\star}_{n,l}}\sum_{x'\in C^{\star}_{n,l'}}
Q_{n,l}^u(x)Q_{n,l'}^u(x')|\del x\cap \del x'|.\quad
%\left(n^{-2}|\del x\cap \del x'|\right).
\eea
Comparing \eqv(7.Lem7.2) to \eqv(7.2.6), we see that $n\bar\s^{=}_n(u,\infty)$ differs from $\bar\nu^\circ_n(u,\infty)$ 
%through an extra prefactor $n^{-1}$ and the fact  that $Q_{n,l}^u(x)$ is squared. 
in that the term in square brackets is squared. 
However, examining the proof of \thv(7.lem4) (see \eqv(7.lem4.5)-\eqv(7.lem4.7) and \eqv(7.lem4.13)-\eqv(7.lem4.18))
%In view of  \eqv(7.lem4.5)-\eqv(7.lem4.7) and \eqv(7.lem4.13)-\eqv(7.lem4.18), we also see that 
we also see that
% up to minor modifications
 %for all practical purposes, 
%one sees that in practice, 
 $n\bar\s^{=}_n(u,\infty)$ can be controlled in exactly the same way as $\bar\nu^\circ_n(u,\infty)$,
substituting $|C^{\star}_{n,l}|^2 Q_{n,l}^{2u}(x)$ for  $[\sum_{x\in C^{\star}_{n,l}}Q_{n,l}^u(x)]^2$.
This yields
\be
\Eq(7.Lem7.4)
\lim_{n\rightarrow\infty}n\E[\bar\s^{=}_n(u,\infty)]=2\nu^\dagger(2u,\infty),\quad\forall u>0.
\ee
Similarly, a rerun of the proof of 
%\eqv(7.lem5.1) 
Lemma \thv(7.lem5) yields, under the same assumptions  on $L_1,L_2$,
%under the assumption on $L_1,L_2$ made therein,
%for all $L'_1,L'_2\geq 0$ such that $na_nL'_2/2^n=o(1)$, 
\be
\textstyle
\P\bigl(\left|\bar\s^{=}_n(u,\infty)-\E[\bar\s^{=}_n(u,\infty)]\right|
\geq
\phi_n(u,L_1,L_2)
\bigr)
\leq 2n e^{-L_2}+4L_1,
\Eq(7.Lem7.5)
\ee
where
$
\phi_n(u,L_1,L_2)\equiv
\sqrt{a_nL_2W_{n,2}(2u)/2^n}
+
4n^{-c_{\star}} W_{n,1}(u)/L_1
$,
%where  $L_1,L_2$ are as in  Lemma \thv(7.lem5), and
and where $W_{n,1}(u)$ and $W_{n,2}(u)$ enjoy the same properties, under the same assumptions, as
$V_{n,1}(u)$ and $V_{n,2}(u)$, but  with \eqv(7.lem5.2) replaced by \eqv(7.lem8.2).
%are positive decreasing functions and satisfy
% and for each $u>0$,
% under the assumptions of Lemma \thv(7.lem7),
%\be
%\lim_{n\rightarrow\infty}W_{n,1}(u)
%=\lim_{n\rightarrow\infty}W_{n,2}(u)
%=\lim_{n\rightarrow\infty}2\E[\bar\nu^\circ_n(2u,\infty)]
%=2\nu^\dagger(2u,\infty),\quad\forall u>0.
%\Eq(7.Lem7.6)
%\ee

Let us now establish that
\be
\Eq(7.Lem7.7)
\E[\bar\s^{\neq}_n(u,\infty)]
\leq
2^8a_n^{-1}[W_{n,3}(u)]^2
\ee
for some positive decreasing function $W_{n,3}(u)$ of $u>0$, that satisfies
\be
\lim_{n\rightarrow\infty}W_{n,3}(u)
%=\lim_{n\rightarrow\infty}\E[\bar\nu^\circ_n(u,\infty)]
=\nu^\dagger(u,\infty),\quad\forall u>0.
\Eq(7.Lem7.7')
\ee
For this write
$
\bar\s^{\neq}_n(u,\infty)
=
\sum_{2\leq k,k'\leq k_n^\star}
\hat\s^{\neq,(k,k')}_n(u,\infty)
$
where, with the notation of \eqv(7.lem4.3),
\be
\Eq(7.Lem7.8)
\hat\s^{\neq,(k,k')}_n(u,\infty)
\equiv
\frac{a_n}{n^22^n}\,
{ \textstyle
%\sum_{\CC\in \GG_k}\sum_{\CC'\in \GG_{k'}}\phi(\CC,\CC')\1_{\{\CC\neq\CC'\}}
%\sum_{x\in\CC}\sum_{x'\in\CC'}\1_{\{\dist(x,x')=2\}}Q_{n,\CC'}^u(x')Q_{n,\CC}^u(x)
%
%\sum^{(1)}_{\CC\in \GG_k, \CC'\in \GG_{k'} : \CC\neq\CC'}\phi(\CC,\CC')
%\sum^{(2)}_{x\in\CC, x'\in\CC' : \dist(x,x')=2}Q_{n,\CC'}^u(x')Q_{n,\CC}^u(x)
%
\sum^{(1)}_{\CC,\CC'}\phi(\CC,\CC')\sum^{(2)}_{x, x'}Q_{n,\CC}^u(x)Q_{n,\CC'}^u(x')
}.
\ee
Here the first sum, $\Sigma^{(1)}$, is over all $\CC\in \GG_k$ and $\CC'\in \GG_{k'}$ such that $\dist(\CC,\CC')= 2$,
the second one, $\Sigma^{(2)}$, is over all $x\in\CC$ and $x'\in\CC'$ such that $\dist(x,x')=2$, and
$
\phi(\CC,\CC')\equiv
\prod_{y\in\CC\cup\CC'}\chi_n(y)
\prod_{y'\in\del\CC\cup\del\CC'}\overline\chi_n(y')
$.
Thus $\CC\cap\CC'=\emptyset$, so that that $Q_{n,\CC}^u(x)$ and $Q_{n,\CC'}^u(x')$ are independant random variables for all $x\in\CC$, $x'\in\CC'$, and averaging out,
\be
\Eq(7.Lem7.9)
\textstyle
\E\sum_{2\leq k,k'\leq k_n^\star}\hat\s^{\neq,(k,k')}_n(u,\infty)
\leq
a_n^{-1}
%\frac{4}{na_n}
[W_{n,3}(u)
%v_n(u; \bar a_n, \bar b_n)+e^{-\b\sqrt{8 \varepsilon n\log n}}
]^2 s^{(3)}_ns^{(4)}_n,
\ee
where
\be
W_{n,3}(u)\equiv\E\left[v_n(u;  \bar  a_n, \bar b_n)\right]+e^{-\b\sqrt{8 \varepsilon n\log n}}
\Eq(7.Lem7.14)
\ee
for some  sequences $\bar a_n$, $\bar b_n$ chosen as in Lemma \thv(7.lem6), and where
\be
\Eq(7.Lem7.10)
\textstyle
s^{(i)}_n=\sum_{2\leq k\leq k_n^\star}k^i(k-1)!n^{-(c_{\star}-1)(k-2)}, \quad i\geq 1.
\ee
To see this, reason that there are at most $2^n(k-1)!n^{k-1}$ sets $\CC\in \GG_k$, that
for each $\CC\in \GG_k$ there are at most $n^2k'(k'-1)!n^{k'-1}$
sets $\CC'\in \GG_{k'}$ such that $\dist(\CC,\CC')= 2$,
%car il faut fixer au moins un point dans le bord a dist 2 de $C$ et faire grossir $\CC'$ a partir de ce point
 that $\Sigma^{(2)}$ contains at most $kk'$ terms, and that, proceeding as in \eqv(7.lem4.14)-\eqv(7.lem4.18), respectively
\eqv(7.lem4.23)-\eqv(7.lem4.27), to bound the terms $Q_{n,\CC}^u$ when $k>2$, depending on whether 
$b_n< n\theta^{\star}_{n,\CC}$ or $b_n\geq n\theta^{\star}_{n,\CC}$, and proceeding as in  \eqv(7.lem4.6)-\eqv(7.lem4.8)
when $k=2$, we have
\be
\Eq(7.Lem7.11)
\E\left[\phi(\CC,\CC')Q_{n,\CC}^u(x)Q_{n,\CC'}^u(x')\right]
\leq
(kk')^2(na_n)^{-2}n^{-c_{\star}[(k-2)+(k'-2)]}
[W_{n,3}(u)]^2,
%\E v_n(u; \bar a_n, \bar b_n)+e^{-\b\sqrt{8 \varepsilon n\log n}}
\ee
for some sequences $\bar a_n$, $\bar b_n$ for which the assumptions of Lemma \thv(7.lem6) are verified,
and all $k,k'\geq 2$. Now for $c_{\star}>2$, by \eqv(8.3lem1.sum),
$
%s^{(1)}_ns^{(2)}_n\leq (4+8n^{-2(c_{\star}-1)})(4+12n^{-c_{\star}})
s^{(i)}_n\leq 2^i(1+2^{i+1}n^{-2(c_{\star}-1)})
$.
Inserting this in \eqv(7.Lem7.9) and using Lemma \thv(7.lem6)
proves the claim \eqv(7.Lem7.7)-\eqv(7.Lem7.7').
This immediately implies that
% has two immediate consequences, namely, for all $u>0$,
\be
\textstyle
\lim_{n\rightarrow\infty}
n\E[\bar\s^{\neq}_n(u,\infty)]=0,\quad \forall u>0,
\Eq(7.Lem7.12)
\ee
and that, under the assumptions and with the notation of \eqv(7.Lem7.7)-\eqv(7.Lem7.7'), for all $u>0$,
\be
\textstyle
\P\left(\left|\bar\s^{\neq}_n(u,\infty)-\E[\bar\s^{\neq}_n(u,\infty)]\right|
\geq 
2^8[W_{n,3}(u)]^2/(a_nL_3)
\right) \leq
2L_3.
\Eq(7.Lem7.13)
\ee
Lemma \thv(7.lem7) now follows from \eqv(7.Lem7.4) and \eqv(7.Lem7.12),
and Lemma \thv(7.lem8) follows from \eqv(7.Lem7.5), \eqv(7.Lem7.7'), and \eqv(7.Lem7.13).
%|C^{\star}_{n,l}|
\end{proof}

%%%%%%%%%%%%%%%%%%%%%%%%%%%%%%%%%%%%%%%%%%%%%%%%

We now prove Lemma \thv(7.lem9).

\begin{proof}[Proof of Lemma \thv(7.lem9)] Let us establish first that for all $m\geq 1$, if $c_{\star}>2$, 
\be
\Eq(7.lem9.12) 
%2 ({k^\circ_n(t)}/{|\VV^\circ_n|})\sum_{m=1}^{\ell^\circ_n-1}
\E\bigl[I_{2,m}^{(3)}\bigr]
\leq 
n^{-1}a_n^{-2}2^n
%(1+o(1))
[
W_{n,3}(u)
]^2
2^5(1+2^{4}n^{-2(c_{\star}-1)})^2,
%(1+2n^{-c_{\star}})^2(1+2n^{-2c_{\star}+1})
%\frac{\ell^\circ_n}{a_n}
%\frac{2^n}{|\VV^\circ_n|}
%\frac{k^\circ_n(t)}{a_n}
%\sum_{2\leq k,k'\leq k_n^\star}k!k'!n^{-(c_{\star}-1)[(k-2)+(k'-2)]}
\ee
where $W_{n,3}(u)>0$ is a decreasing function satisfying
%\be
$
\lim_{n\rightarrow\infty}W_{n,3}(u)
=\nu^\dagger(u,\infty)
$
for all $u>0$.
%,\quad\forall u>0.
%\Eq(7.lem9.13) 
%\ee
%%%%%%%%%%%%%%%%%%%%
%Recall the definition of $f_n^{\circ,m}(x,x';y,y')$ and $I_{2,m}^{(3)}$ from  the line before
%\eqv(7.theo2.10) and \eqv(7.theo2.11), respectively.
%We seek an upper boud on
%$
%\textstyle 2 ({k^\circ_n(t)}/{|\VV^\circ_n|})\sum_{m=1}^{\ell^\circ_n-1}\E\bigl[I_{2,m}^{(3)}\bigr]
%%\Eq(7.lem9.3) 
%$
%where $I_{2,m}^{(3)}$ is defined in  \eqv(7.theo2.11).
%Note first that the quantities $I_{2,m}^{(3)}$ and $\bar\s^{\neq}_n(u,\infty)$, 
%defined in  \eqv(7.theo2.11) and  \eqv(7.Lem7.3)respectively, are very similar.
For this note that $I_{2,m}^{(3)}$ in \eqv(7.theo2.11) is very similar to the quantity $\bar\s^{\neq}_n(u,\infty)$ defined in \eqv(7.Lem7.3).
This prompts us to write
\be
\textstyle
I_{2,m}^{(3)}
=
\sum_{2\leq k,k'\leq k_n^\star}
I_{2,m}^{(3),(k,k')}
\Eq(7.lem9.4) 
\ee
where, %using the notation of \eqv(7.lem4.3),
for $\phi(\CC,\CC')$ as in \eqv(7.Lem7.8),
\be
\Eq(7.lem9.5) 
I_{2,m}^{(3),(k,k')}
\equiv
%\frac{a_n}{n^22^n}\,
{ \textstyle
\sum^{(1)}_{\CC,\CC'}\phi(\CC,\CC')\sum^{(2)}_{x, x'}
\sum^{(3)}_{y, y'}
%\sum_{y\in\del C^{\star}_{n,l}}\sum_{y'\in\del C^{\star}_{n,l}}
%\sum_{y\in\VV^\circ_n}\sum_{y'\in\VV^\circ_n}
Q_{n,l}^u(x)Q_{n,l'}^u(x')f_n^{\circ,m}(x,x';y,y')
}.
\ee
Here the first sum, $\Sigma^{(1)}$, is over all $\CC\in \GG_k$ and $\CC'\in \GG_{k'}$ such that 
%$\dist(\CC,\CC')\geq  2$,
 $\CC\cap\CC'=\emptyset$,
the second one, $\Sigma^{(2)}$, is over all $x\in\CC$ and $x'\in\CC'$, and the third one, $\Sigma^{(3)}$,
is over all $y\in\del\CC$ and $y\in\del\CC'$.
Since $\CC\cap\CC'=\emptyset$, $Q_{n,\CC}^u(x)$ and $Q_{n,\CC'}^u(x')$ are independant random variables for all $x\in\CC$, $x'\in\CC'$. 
%Taking the expectation and 
Thus we see, using \eqv(7.Lem7.11),  that  for all  $k,k'\geq 2$, 
%$\E\bigl[I_{2,m}^{(3),(k,k')}\bigr]$ is bounded above by
\bea
\nonumber
\hspace{-8pt}&&\hspace{-8pt}\E\bigl[I_{2,m}^{(3),(k,k')}\bigr]
\\
\Eq(7.lem9.6) 
\hspace{-8pt}&\leq&\hspace{-8pt}
%%\frac{a_n}{n^22^n}\,
%%\sum_{y\in\del C^{\star}_{n,l}}\sum_{y'\in\del C^{\star}_{n,l}}
%%\sum_{y\in\VV^\circ_n}\sum_{y'\in\VV^\circ_n}
(kk')^2(na_n)^{-2}n^{-c_{\star}[(k-2)+(k'-2)]}
[
W_{n,3}(u)
]^2
{ \textstyle
\sum^{(1)}_{\CC,\CC'}\sum^{(2)}_{x, x'}
\sum^{(3)}_{y, y'}
f_n^{\circ,m}(x,x';y,y')
},
\quad\quad\quad
\eea
where $W_{n,3}(u)$ is given by \eqv(7.Lem7.14)
for some sequences $\bar a_n$, $\bar b_n$ for which the assumptions of Lemma \thv(7.lem6) are verified --
hence it has the properties claimed in the line below \eqv(7.lem9.12).
%Hence $W_{n,3}(u)$ has the properties claimed in \eqv(7.lem9.13) and the preceeding line.
To deal with the sums in \eqv(7.lem9.6), observe that given any $\CC\in\GG_{k}$, $x\in\CC$, and $y\in\del\CC$,
\bea
\Eq(7.lem9.7) 
f_n^{\circ,m}(x;y)&\equiv&
\textstyle
\sum_{\CC'\in \GG_{k'}}\sum_{x'\in\CC'}\sum_{y'\in\del\CC'}f_n^{\circ,m}(x,x';y,y')
%p_n(y',x')p_n^{\circ,m}(y,y')
\\
&=&
\Eq(7.lem9.8) 
\textstyle
p_n(x,y)\sum_{y'\in\VV^\circ_n}\sum_{\CC'}^{(4)}\sum_{x'\in\CC'}p_n(y',x')p_n^{\circ,m}(y,y')
\eea
where the sum $\Sigma^{(4)}$ is  over all $\CC'\in \GG_{k'}$ such that $\CC'\cap \del y'\neq \emptyset$. Indeed if
$\CC'\cap \del y'= \emptyset$ then $p_n(y',x')=0$ for all $x'\in\CC'$. Now 
$
\sum_{x'\in\CC'}p_n(y',x')\leq 1
$
while the number of terms in $\Sigma^{(4)}$ is at most $k'!n^{k'}$.
Thus
\be
\Eq(7.lem9.9) 
\textstyle
f_n^{\circ,m}(x;y)
\leq 
k'!n^{k'}p_n(x,y)\sum_{y'\in\VV^\circ_n}p_n^{\circ,m}(y,y')
\leq 
k'!n^{k'}p_n(x,y),
\ee
From this we readily get
\be
\Eq(7.lem9.10) 
{ \textstyle
\sum^{(1)}_{\CC,\CC'}\sum^{(2)}_{x, x'}
\sum^{(3)}_{y, y'}
f_n^{\circ,m}(x,x';y,y')
}
\leq
(k-1)!n^{k-1}k'!n^{k'},
\ee
and inserting this bound in \eqv(7.lem9.6) and  \eqv(7.lem9.4) successively yields
\be
\Eq(7.lem9.14) 
\E\bigl[I_{2,m}^{(3)}\bigr]
\leq
n^{-1}a_n^{-2}2^n
%\frac{4}{na_n}
[W_{n,3}(u)
%v_n(u; \bar a_n, \bar b_n)+e^{-\b\sqrt{8 \varepsilon n\log n}}
]^2 s^{(2)}_ns^{(3)}_n,
\ee
where $s^{(i)}_n$ is defined in \eqv(7.Lem7.10) and obeys
$
%s^{(1)}_ns^{(2)}_n\leq (4+8n^{-2(c_{\star}-1)})(4+12n^{-c_{\star}})
s^{(i)}_n\leq 2^i(1+2^{i+1}n^{-2(c_{\star}-1)})
$
whenever $c_{\star}>2$. Eq.~\eqv(7.lem9.12)  now immediately follows.
Invoking  \eqv(4.prop2.0) of Proposition \thv(4.prop2) we get that on $\O^{\star}$,  
for all but a finite number of indices $n$, if $c_{\star}>2$,
\be
\Eq(7.lem9.11) 
({a_n}/{|\VV^\circ_n|})
\sum_{m=1}^{\ell^\circ_n-1}\E\bigl[I_{2,m}^{(3)}\bigr]
\leq 
2^5(1+o(1))\frac{\ell^\circ_n}{na_n}[W_{n,3}(u)]^2.
\ee
The lemma now  follows  by a first order Tchebychev inequality.
\end{proof}

%%%%%%%%%%%%%%%%%%%%%%%%%%%%%%%%%%%%%%%%%%%%%%%%%%%%%%%%%%

\subsubsection{Proof of Proposition \thv(7.prop1)}
 \label{7.3.4}

The proof of Proposition \thv(7.prop1) is now a mere formality.
Recall that  $c_{\star}>2$ and that $a_n$ obeys \eqv(7.prop1.0new) for some  $0<\varepsilon<1$.
Choose  $L_1=n^{-1-(c_{\star}-2)/2}$ and $L_2=4\log n$ in Lemma \thv(7.lem5). 
Then $n^2a_nL_2/2^n=o(1)$, $\lim_{n\rightarrow\infty}\phi_n(u,L_1,L_2)\rightarrow 0$, and 
$\sum_n (2n e^{-L_2}+2L_1)<\infty$, so that by 
Lemma \thv(7.lem4), Lemma \thv(7.lem5), and Borel-Cantelli Lemma,  
\be
\Eq(7.prop1.3)
\lim_{n\rightarrow\infty}\bar\nu^\circ_n(u,\infty)=\nu^\dagger(u,\infty),
\quad\P-\text{almost surely},
\ee
for all $u>0$. Because $\bar\nu^\circ_n(u,\infty)$ is a sequence of monotonic functions of $u>0$ whose limit, 
$\nu^\dagger(u,\infty)$, is continuous, \eqv(7.prop1.3) entails the existence of a subset  $\overline\O^{\scriptscriptstyle{\textsf{LLN}}}_1\subset\O$ with the property that 
$\P(\overline\O^{\scriptscriptstyle{\textsf{LLN}}}_1)=1$, and such that on $\overline\O^{\scriptscriptstyle{\textsf{LLN}}}_1$,
\be
\Eq(7.prop1.4)
\lim_{n\rightarrow\infty}\bar\nu^\circ_n(u,\infty)=\nu^\dagger(u,\infty),\quad \forall u>0.
\ee
We prove in the same way, using the monotonicity of $\bar\s^\circ_n(u,\infty)$, Lemma \thv(7.lem7), and Lemma \thv(7.lem8) (with $L_1$ and $L_2$ as above and  $L_3=n^{-2}$, so that $\lim_{n\rightarrow\infty}n\psi_n(u,L_1,L_2,L_3)=0$ and $\sum_n (2n e^{-L_2}+2L_1+2L_2)<\infty$)
that there exists a subset  $\overline\O^{\scriptscriptstyle{\textsf{LLN}}}_2\subset\O$ of full measure
%with  $\P(\overline\O^{\scriptscriptstyle{\textsf{LLN}}}_2)=1$ 
such that, on $\overline\O^{\scriptscriptstyle{\textsf{LLN}}}_2$, 
\be
\Eq(7.prop1.5)
\lim_{n\rightarrow\infty}n\bar\s^\circ_n(u,\infty)=2\nu^\dagger(2u,\infty),\quad \forall u>0.
\ee
Taking
$
\overline\O^{\scriptscriptstyle{\textsf{LLN}}}=
\overline\O^{\scriptscriptstyle{\textsf{LLN}}}_1\cap \overline\O^{\scriptscriptstyle{\textsf{LLN}}}_2
$
completes the proof of Proposition \thv(7.prop1)  .
%%%%%%%%%%%%%%%%%%%%%%%%%%%%%%%%%%%%%%%%%%%%%%%%%%%%%%%% 

 %%%%%%%%%%%%%%%%%%%%%%%%%%%%%%%%%%%%%%%%%%%%%%%%

\subsection{Conclusion of the proof of Theorem \thv(2.theo2).}
 \label{7.5}

 %Assume that $c_{\star}$ is such that $n^{c_{\star}}>n^3\log n$. 

%Let the initial distribution and the sequences $a_n$ and $b_n$ be chosen as in Theorem \thv(2.theo2),
%and let $\nu^\dagger$ in Condition (A1) be the function defined in \eqv(1.theo2.M2).
%Also assume that $\b>2\b_c(\varepsilon/2)$ and that $c_{\star}$ obeys
%$
%n^{c_{\star}}>n^3\log n
%$.

%Theorem \thv(2.theo2) will be proved if we can establish that
%We have to check that 
It suffices to prove that
under the assumptions of Theorem \thv(2.theo2), Conditions (A1), (A2), and (A3) of 
Theorem \thv(7.theo1)  are verified $\P$-almost surely when $\nu^\dagger$ in Condition (A1) is  as  in \eqv(1.theo2.M2).

%when the initial distribution and the sequences $a_n$ 
%and $b_n$ are chosen as in Theorem \thv(2.theo2),  under the assumptions that 
%$\b>2\b_c(\varepsilon/2)$, that $c_{\star}$ is such that
%%obeys the relation
%$
%n^{c_{\star}}>n^3\log n
%$,
%and when $\nu^\dagger$ in Condition (A1) is  the function defined in \eqv(1.theo2.M2).

\subsubsection{Verification of Conditions (A1) and (A2).}
 \label{7.4.2}

It immediately follows 
%This is a simple matter. It follows from
  Lemma  \thv(7.lem3), Theorem \thv(7.theo2), and Proposition \thv(7.prop1) 
 that under the assumptions therein, $\P$-almost surely,   for all $u>0$ and all $t>0$,
\bea
\Eq(7.5.1)
&&\lim_{n\rightarrow\infty}\bar\nu_n^{{{\scriptscriptstyle{J^\circ_n}},t}}(u,\infty)
= t\nu^\dagger(u,\infty)\,\,\, \text{in $P^\circ$-probability},\quad\quad
\\
\Eq(7.5.2)
&&\lim_{n\rightarrow\infty}\bar\s_n^{{{\scriptscriptstyle{J^\circ_n}},t}}(u,\infty)
= 0 \,\,\, \text{in $P^\circ$-probability}.
\eea
Conditions (A1) and (A2) are thus satisfied $\P$-almost surely.

\subsubsection{Verification of Condition (A3).}
 \label{7.4.3'}

%The rest of this subsection is devoted to checking  Condition (A3).  
This still requires a little work.
Given $\e>0$, define
\be
\eta_{n,k}(\e)
\equiv
\Eq(7.4.4)
\frac{a_n}{2^n}\,
{\textstyle 
\sum_{1\leq l\leq L^{\star}}\1_{\{|C^{\star}_{n,l}|=k\}}\sum_{x\in C^{\star}_{n,l}}A_{n,l}(x),\quad k\geq 2,
}
\ee
where, given $x\in C^{\star}_{n,l}$, $1\leq l\leq L^{\star}$, 
\be
\Eq(7.4.3)
\textstyle
A_{n,l}(x)\equiv E_{x}\bigl( \1_{\{b_n^{-1}T^{\star}_{n,l}\leq \e\} }b_n^{-1}T^{\star}_{n,l}\bigr).
\ee
One readily sees that Condition (A3) will be verified $\P$-almost surely if
\be
\textstyle
\lim_{\e\downarrow 0}\limsup_{n\uparrow \infty}\sum_{k\geq 2}\eta_{n,k}(\e)=0,\quad \P\text{-a.s.}.
\Eq(7.4.C3)
\ee
%$\P$-a.s.. 
%To reach this conclusion
%To prove this our stategy is to bound the mean values $\E\eta_{n,k}(\e)$ from above, and then use a first order Tchebychev inequality. 
%The reason for this notations is that, 
Note that $\eta_{n,k}(\e)$ is of the form \eqv(7.lem4.4) with $Q_{n,l}^u(x)$ replaced by $A_{n,l}(x)$
 and hence, as in \eqv(7.lem4.3), may be written as
% Hence we may write, as in \eqv(7.lem4.3), that
\be
\Eq(7.4.5)
\eta_{n,k}(\e)
\equiv
\frac{a_n}{2^n}\,
{\textstyle 
\sum_{\CC\in \GG_k}\prod_{x\in\CC}\chi_n(x)\prod_{x'\in\del\CC}\overline\chi_n(x')\sum_{x\in\CC}
A_{n,\CC}(x)
},
\ee
where $A_{n,\CC}(x)$ stands for $A_{n,l}(x)$ with $C^{\star}_{n,l}\equiv\CC$. 
As in the proof of Lemma \thv(7.lem4) we note that on $\O^{\star}$, 
$
k_n^{\star} \leq\frac{n}{(c_{\star}-2)\log n}
$
for all large enough $n$, and treat the terms $k=2$ and $3\leq k\leq  k_n^{\star}$ separately. Throughout
the proof we set
$\bar a_n=\sqrt{na_n}$, $\bar b_n= b_n(n-1)$, and define
\be
\g_n(\CC')=\min_{x\in\CC}w_n(x)/\bar b_n, \quad \CC'\in\GG_2.
\Eq(7.4.9)
\ee

%----------
%
%The term $k=2$ / The remainder
%
%----------

\smallskip
\noindent{\textbf{\emph{ $\bullet$ The term $k=2$.}}} 
Let us establish that for all large enough $n$ and small enough $\e$, the mean and variance of $\eta_{n,2}(\e)$ obey
\bea
\Eq(7.4.14)
\E\eta_{n,2}(\e)
&\leq&
2^{-n/6}
+2\e^{1-\left(2\a_c(\frac{\varepsilon}{2})+o(1)\right)},
%+9(1+o(1))\e^{1-\frac{1}{3}\left(2\a_c(\frac{\varepsilon}{2})+o(1)\right)}
\\
\Eq(7.4.24)
\E\left(\eta_{n,2}(\e)-\E\eta_{n,2}(\e)\right)^2
&\leq&
2\frac{a_n}{2^{n}}\left[2^{-n/6}
+4\e^{2-\left(2\a_c(\frac{\varepsilon}{2})+o(1)\right)}\right].
\eea

We first prove \eqv(7.4.14).  By Proposition \thv(5.prop1), (i), and integration by parts, for all
$x\in\CC$,
\bea
\Eq(7.4.6)
A_{n,\CC}(x)
&\leq &
\textstyle
\frac{1}{b_n}\sum_{i=0}^{\lfloor b_n\e\rfloor}
\left(1-\sfrac{1}{1+\min_{x\in\CC}w_n(x)/(n-1)}\right)^i
\\
\Eq(7.4.7)
&\leq &
\textstyle
[1+o_{n,1}(1)]\varphi_{\e}\left([1+o_{n,2}(1)]{\min_{x\in\CC}w_n(x)}/{\bar b_n}\right)
%\left(1-\sfrac{1}{1+\varrho_{n,l}(0)/(n-1)}\right)^i
\eea
% we use in the second inequality that in $\CC$, $w_n(x)\geq r_n(\rho^\star)$, and used Taylor expansion of the log
where $|o_{n,i}(1)|\leq \OO(r^{-1}_n\left(\rho^{\star}_n\right))$, $i=1,2$, and 
\be
\Eq(7.4.8)
\varphi_{\e}(y)=y(1-e^{-\e/y}), \quad y\geq 0.
\ee
Plugging \eqv(7.4.7) in \eqv(7.4.5)  yields
\be
\Eq(7.4.10)
\E\eta_{n,2}(\e)
\leq
[1+o_{n,1}(1)]
{\textstyle 
\E\left[
\bar a^2_n\varphi_{\e}\left([1+o_{n,2}(1)]\g_n(\CC)\right)
\1_{\{\g_n(\CC)\geq r_n(\rho^{\star}_n)/ \bar b_n\}}
\right].
}
\ee
%To bound the last expectation
Now for $\e>r_n(\rho^{\star}_n)/ \bar b_n$ 
%let us 
split
$
\1_{\{\g_n(\CC)\geq r_n(\rho^{\star}_n)/ \bar b_n\}}
$
into 
$
\1_{\{\g_n(\CC)\geq \e\}}
+\1_{\{r_n(\rho^{\star}_n)/ \bar b_n\leq \g_n(\CC)<\e \}}
$.
On the one hand, observing that $\varphi_{\e}(y)\leq \e$ for all $y> 0$, we have
\be
\Eq(7.4.11)
\hspace{-.8pt}
\E\left[
\bar a^2_n\varphi_{\e}\left([1+o_{n,2}(1)]\g_n(\CC)\right)
\1_{\{\g_n(\CC)\geq \e\}}
\right]
\leq 
\e \bar a^2_n\P\left(\g_n(\CC)\geq \e\right)
%\e F^2_n(\e)
=\e^{1-2\left(\a_c(\frac{\varepsilon}{2})+o(1)\right)},
\ee
where the last equality 
%is \eqv(7.lem2.3) of 
follows from Lemma \thv(7.lem2), (i). On the other hand 
$\varphi_{\e}(y)\leq y$ for all $y> 0$. From this, integration by part,  and \eqv(7.lem2.0) we obtain,
\bea
\nonumber
&&
\textstyle
\E\left[
\bar a^2_n\varphi_{\e}\left([1+o_{n,2}(1)]\g_n(\CC)\right)
\1_{\{r_n(\rho^{\star}_n)/ \bar b_n\leq \g_n(\CC)<\e \}}
\right]
\\
&&
\Eq(7.4.12)
\quad\quad\quad\quad\quad\quad\quad\quad\quad\quad\quad\quad\quad\quad
\leq 
\textstyle
[1+o_{n,2}(1)] \int_{r_n(\rho^{\star}_n)/ \bar b_n}^{\e}F^2_n(y)dy.
\quad\quad\quad\quad\quad\quad
\eea
Given $0<\d<1$, split the domain of integration in \eqv(7.4.12) into 
$[r_n(\rho^{\star}_n)/ \bar b_n, \bar b_n^{-\d}]\cup[\bar b_n^{-\d},\e]$. Using  that $F^2_n(y)\leq \bar a^2_n$
on the first domain, and using Lemma \thv(7.lem2), (ii), on the the second,
%integral over $[\bar b_n^{-\d},\e]$,
\be
\Eq(7.4.13)
\textstyle
\int_{r_n(\rho^{\star}_n)/ \bar b_n}^{\e}F^2_n(y)dy 
\leq 
\bar a^2_n\bar b_n^{-\d}
+
\frac{1+o(1)}{1-2(1-\frac{\d}{2})\a_n}\bigl(\sfrac{1}{1-\d}\bigr)^2\e^{1-2(1-\frac{\d}{2})\a_n},
\ee
where $0\leq \a_n=\a_c(\frac{\varepsilon}{2})+o(1)$.
By definition of $\bar a_n$, $\bar b_n$, \eqv(A1.lem1.1), and
the assumption that 
%$2\a_c(\frac{\varepsilon}{2})<1$, that is 
$\b>2\b_c(\varepsilon/2)$, we get 
$
\bar a^2_n\bar b_n^{-\d}
%=\exp\left\{n\left[
%\b^2_c(\varepsilon/2)-\d\b\b_c(\varepsilon/2)+o(1)
%\right]\right\}
\leq 
\exp\left\{n\left[
\b^2_c(\varepsilon/2)(1-2\d(1+o(1)))
\right]\right\}
$.
Hence, choosing $\d=2/3$,
$
%\Eq(7.4.14)
%\textstyle
\int_{r_n(\rho^{\star}_n)/ \bar b_n}^{\e}F^2_n(y)dy 
\leq 
2^{-n/6}
+\frac{9}{1-[2\a_c({\varepsilon}/{2})+o(1)](2/3)}\e^{1-\frac{2}{3}\left[2\a_c({\varepsilon}/{2})+o(1)\right]}
$.
Collecting our bounds we arrive at \eqv(7.4.14).

Turning to the variance we have
\be
\Eq(7.4.25)
\E\left(\eta_{n,2}(\e)-\E\eta_{n,2}(\e)\right)^2
=\left(\frac{a_n}{2^{n}}\right)^2
\textstyle 
\E\left(\sum_{\CC\in \GG_2}[Y_n(\CC)-\E Y_n(\CC)]\right)^2,
\ee
where
% consider the collection of variables $(Y_n(\CC), \CC\in \GG_2)$,
$
Y_n(\CC)\equiv\prod_{x\in\CC}\chi_n(x)\prod_{x'\in\del\CC}\overline\chi_n(x')\sum_{x\in\CC}
A_{n,\CC}(x)
$, 
$\CC\in \GG_2$. Observing that
$Y_n(\CC)$ and $Y_n(\CC')$ are independent whenever $\CC\neq \CC'$ and 
$\del\CC\cap \del\CC'=\emptyset$, and that $Y_n(\CC)Y_n(\CC')=0$ whenever 
$\CC\neq \CC'$ and $\CC\cap \CC'\neq \emptyset$, we readily get that
$
\E\left(\eta_{n,2}(\e)-\E\eta_{n,2}(\e)\right)^2\leq I_n^{=}+I_n^{\neq}
$,
\bea
\Eq(7.4.26)
I_n^{=}
&\equiv&
\left(\frac{a_n}{2^{n}}\right)^2
\textstyle
\sum_{\CC\in \GG_2}\E[Y^2_n(\CC)],
\\
\Eq(7.4.27)
I_n^{\neq}&\equiv& 
\left(\frac{a_n}{2^{n}}\right)^2
\textstyle
%\sum_{\CC\neq\CC':\dist(\CC,\CC')=2}
\sum^{(1)}_{\CC,\CC'}
(\E[Y_n(\CC)Y_n(\CC')]-[\E Y_n(\CC)][\E Y_n(\CC')]),
\eea
where,
as in \eqv(7.Lem7.8),
%\sum^{(1)}_{\CC,\CC'}\phi(\CC,\CC')\sum^{(2)}_{x, x'}Q_{n,\CC}^u(x)Q_{n,\CC'}^u(x')
the sum $\Sigma^{(1)}$ is over all $\CC\in \GG_2$ and $\CC'\in \GG_{2'}$ such that $\dist(\CC,\CC')= 2$.
% the sum in \eqv(7.4.27) is over all 
We bound \eqv(7.4.26) in just the same way as $\E\eta_{n,2}(\e)$, namely,
using \eqv(7.4.7) in \eqv(7.4.26) gives
\be
\Eq(7.4.10')
I_n^{=}
\leq
2\frac{a_n}{2^{n}}[1+o_{n,1}(1)]
{\textstyle 
\E\left[
\bar a^2_n\varphi^2_{\e}\left([1+o_{n,2}(1)]\g_n(\CC)\right)
\1_{\{\g_n(\CC)\geq r_n(\rho^{\star}_n)/ \bar b_n\}}
\right],
}
\ee
and proceeding as in \eqv(7.4.10)-\eqv(7.4.13) to evaluate \eqv(7.4.10), we obtain \eqv(7.4.24).
%\be
%\E\eta_{n,2}(\e)
%\leq
%\sfrac{a_n}{2^{n}}[1+o(1)]\left[
%2\e^{2-2\a_n}
%+
%\bar a^2_n\bar b_n^{-2\d}
%+
%\sfrac{1+o(1)}{1-(1-\frac{\d}{2})\a_n}\bigl(\sfrac{1}{1-\d}\bigr)^2\e^{2-2(1-\frac{\d}{2})\a_n}
%\right]
%\ee
%and choosing $\d=2/3$
%\be
%\E\eta_{n,2}(\e)
%\leq
%\frac{a_n}{2^{n}}[1+o(1)]\left[
%2^{-n/6}
%+2\e^{2-\left(2\a_c(\frac{\varepsilon}{2})+o(1)\right)}
%+\sfrac{9}{1-(2\a_c(\frac{\varepsilon}{2})+o(1))/3}\e^{2-\frac{2}{3}\left(2\a_c(\frac{\varepsilon}{2})+o(1)\right)}
%\right]
%\ee
To Bound $I_n^{\neq}$ note that
\be
\Eq(7.4.28)
I_n^{\neq}=
%\left(\frac{a_n}{2^{n}}\right)^2
%\textstyle
%\sum^{(1)}_{\CC,\CC'}
%(\E[Y_n(\CC)Y_n(\CC')]-[\E Y_n(\CC)][\E Y_n(\CC')])
%=
\left(\frac{a_n}{2^{n}}\right)^2
\textstyle
\sum^{(1)}_{\CC,\CC'}
\E[Z_n(\CC)]\E[Z_n(\CC')]\Delta_n(\CC,\CC')
\ee
where
$
Z_n(\CC)\equiv\prod_{x\in\CC}\chi_n(x)\sum_{x\in\CC}
A_{n,\CC}(x)
$,
and
\bea
\nonumber
\Delta_n(\CC,\CC')&\equiv&
\textstyle
\E[\prod_{y\in\del\CC\cup\del\CC'}\overline\chi_n(y)]-
\E[\prod_{x\in\del\CC}\overline\chi_n(x)]\E [\prod_{x'\in\del\CC'}\overline\chi_n(x')]
\\
\Eq(7.4.29)
&=&(1-n^{-c_{\star}})^3(1-(1-n^{-c_{\star}})).
\eea
Observing that the right hand side of \eqv(7.4.10) (and a fortiori the r.h.s.~of \eqv(7.4.14)) is an upper bound on 
$\bar a_n^2\E[Z_n(\CC)]$, and that the sum $\Sigma^{(1)}$ contains at most $4!n^42^{n-1}$ terms, we obtain
\be
\Eq(7.4.30)
I_n^{\neq}\leq
4!(1-n^{-c_{\star}})^3n^{2-c_{\star}}2^{-n}\left(
2^{-n/6}
+2\e^{1-\left(2\a_c(\frac{\varepsilon}{2})+o(1)\right)}
\right)^2.
\ee
Combining \eqv(7.4.10') and \eqv(7.4.30)  now yields \eqv(7.4.24).

Since $n^2\sum_{n}a_n/2^n<\infty$ it follows from \eqv(7.4.14), Borel-Cantelli Lemma (through a second order Tchebychev inequality), and \eqv(7.4.24) that 
%\be
%\P\left(
%|\eta_{n,2}(\e)-\E\eta_{n,2}(\e)|>1/n)
%\right)\leq ^2n ....
%\ee
%From \eqv(7.4.24) and
$\lim_{n\uparrow \infty}\eta_{n,2}(\e)= 2\e^{1-2\a_c({\varepsilon}/{2})}$ $\P$-almost surely, for all 
$\varepsilon>0$. 
Observing that $\eta_{n,2}(\e)$ is a monotonic function of $\e$, and  arguing as in the proof 
 of Proposition \thv(7.prop1)  (see \eqv(7.prop1.3)-\eqv(7.prop1.4)), we obtain that
\be
\textstyle
\lim_{\e\downarrow 0}\limsup_{n\uparrow \infty}
%\sum_{k\geq 2}\eta_{n,k}(\e)
\eta_{n,2}(\e)
=0, \quad\P\text{- almost surely}.
\Eq(7.4.C3.1)
\ee

\smallskip
\noindent{\textbf{\emph{$\bullet$ The terms $3\leq k\leq k_n^{\star}$.}}} Note that $\eta_{n}(\e)-\eta_{n,2}(\e)>0$.
Our strategy is  to bound $\E(\eta_{n}(\e)-\eta_{n,2}(\e))$ from above and use a first order Tchebychev inequality to infer from it 
$\P$-a.s.~convergence of $\eta_{n}(\e)-\eta_{n,2}(\e)$ to zero.
%As in the proof of ....
%the terms $\E\eta_{n,k}(\e)$, $3\leq k\leq k_n^{\star}$ are of a much smaller order than $\E\eta_{n,2}(\e)$.
%We only give the mai lines of the argument.
%sketch the proof

As in the proof of Lemma \thv(7.lem4) we denote by
$\varrho_{n,\CC}(0)$, $\varrho_{n,\CC}(1)$, and $\theta^{\star}_{n,\CC}$   the quantities
$\varrho_{n,l}(0)$, $\varrho_{n,l}(1)$, and $\theta^{\star}_{n,l}$ from \eqv(5.lem1.0)-\eqv(5.prop1.1)
with $C^{\star}_{n,l}\equiv\CC$. Similarly $T^{\star}_{n,\CC}$ stands for $T^{\star}_{n,l}$
with $C^{\star}_{n,l}\equiv\CC$.
Decomposing the event
$\{b_n^{-1}T^{\star}_{n,\CC}\leq \e\}$ into
$\AA^{(1)}\cup\AA^{(2)}\cup\AA^{(2)}$
where
$
\AA^{(1)}\equiv\{T^{\star}_{n,\CC}\leq b_n\e<\theta^{\star}_{n,\CC}\}
$,
$
\AA^{(2)}\equiv\{T^{\star}_{n,\CC}<\theta^{\star}_{n,\CC}\leq b_n\e\}
$,
and
$
\AA^{(3)}\equiv\{\theta^{\star}_{n,\CC}<T^{\star}_{n,\CC}\leq b_n\e\}
$,
define
%Introducing  the events 
%%$\{b_n^{-1}T^{\star}_{n,\CC}\leq \e\}$ into
%$
%%\{b_n^{-1}T^{\star}_{n,\CC}\leq \e\}=
%\AA^{(1)}\cup\AA^{(2)}\cup\AA^{(2)}=\{b_n^{-1}T^{\star}_{n,\CC}\leq \e\}
%$,
%$
%\AA^{(1)}\equiv\{T^{\star}_{n,\CC}\leq b_n\e<\theta^{\star}_{n,\CC}\}
%$,
%$
%\AA^{(2)}\equiv\{T^{\star}_{n,\CC}<\theta^{\star}_{n,\CC}\leq b_n\e\}
%$,
%and
%$
%\AA^{(3)}\equiv\{\theta^{\star}_{n,\CC}<T^{\star}_{n,\CC}\leq b_n\e\}
%$,
%define
\be
\Eq(7.4.15)
\textstyle
A_{n,\CC}^{(i)}(x)\equiv E_{x}\bigl( b_n^{-1}T^{\star}_{n,\CC}\1_{\AA^{(i)}}\bigr),
\ee
%and let $\eta^{(i)}_{n,k}(\e)$ be defined as $\eta_{n,k}(\e)$ 
%with $A_{n,\CC}^{(i)}(x)$ substituted for $A_{n,\CC}(x)$. Thus  $\eta_{n,k}(\e)=\sum_{i=1}^3\eta^{(i)}_{n,k}(\e)$ where 
and denote by $\eta^{(i)}_{n,k}(\e)$ the quantity $\eta_{n,k}(\e)$ where $A_{n,\CC}^{(i)}(x)$ is substituted for 
$A_{n,\CC}(x)$. Thus 
\be
\Eq(7.4.15')
\textstyle
0<\eta_{n}(\e)-\eta_{n,2}(\e)=
\sum_{i=1}^{3}\sum_{k=3}^{k_n^{\star}}\eta^{(i)}_{n,k}(\e),\quad
\eta_{n,k}(\e)=\sum_{i=1}^3\eta^{(i)}_{n,k}(\e).
\ee
To bound $\eta_{n}(\e)-\eta_{n,2}(\e)$ it now suffices to bound each $\E\eta^{(i)}_{n,k}(\e)$,
$1\leq i\leq 3$, $3\leq k\leq k_n^{\star}$.

Using the trivial bound
$
A_{n,l}^{(1)}(x)\leq \e1_{\theta^{\star}_{n,\CC}\geq b_n\e}
$,
the term $\E[\eta^{(1)}_{n,k}(\e)]$ is bounded in exactly the same way as $\bar\nu^{\circ,(k),-}_n(u,\infty)$
in \eqv(7.lem4.23)-\eqv(7.lem4.27), namely
\be
\Eq(7.4.16)
\E[\eta^{(1)}_{n,k}(\e)]
\leq  
%\e^{1-2\a(\varepsilon/2)(1+o(1))}
\e
k (k-2)! n^{-(c_{\star}-1)(k-2)}
%\left[ {\bar  a_n}\P(w_n(0)>  \e\bar b_nr_n(\rho^{\star}_n)/n^7) \right]^2
\left(
2\b n^2k^5/\e r_n(\rho^{\star}_n)
\right)^{2\a(\varepsilon/2)(1+o(1))}.
\ee
From the bound
$
A_{n,\CC}^{(2)}(x)\leq b_n^{-1}\theta^{\star}_{n,\CC}\1\{b_n^{-1}\theta^{\star}_{n,\CC}\leq \e\}
$,
\eqv(5.lem1.0)-\eqv(5.prop1.1),  and  Lemma \thv(9.lem4),
%\eqv(9.lem4.1), 
we get
\bea
\Eq(7.4.17)
A_{n,\CC}^{(2)}(x)
&\leq&
\varrho_{n,\CC}(0)\sfrac{2\b n^2k^5}{\bar b_n r_n(\rho^{\star}_n)}
\1_{\{\varrho_{n,\CC}(0)
\leq \e{\bar b_n r_n(\rho^{\star}_n)}/{2\b n^2k^5}\}}
\\
&\leq&
\textstyle
\sum_{\{x,y\}\in G(\CC)}
\g_n(\{x,y\})\sfrac{2\b n^2k^5}{ r_n(\rho^{\star}_n)}
\1_{\{\g_n(\{x,y\})
\leq \e{ r_n(\rho^{\star}_n)}/{2\b n^2k^5}\}}.
\eea
Inserting this in \eqv(7.4.5) and, again, proceeding  as in  \eqv(7.lem4.23)-\eqv(7.lem4.27), 
\be
\Eq(7.4.18)
\E[\eta^{(2)}_{n,k}(\e)]
\leq 
%k (k-2)! n^{-(c_{\star}-1)(k-2)}
\frac{\b n k^5 (k-2)!}{ n^{(c_{\star}-1)(k-2)}}
\E\left[\bar  a^2_n (\g_n(\CC')/r_n(\rho^{\star}_n))
\1_{\{1/\bar b_n\leq \g_n(\CC')/r_n(\rho^{\star}_n)
\leq \e / \b n^2
%\sfrac{r_n(\rho^{\star}_n)}{2\b n^2k^5}
\}}\right].
\ee
An expectation similar to that appearing in \eqv(7.4.18) was estimated in \eqv(7.4.12).
%(with $\g_n(\CC')\leq\e$ instead of $\g_n(\CC')\leq\e r_n(\rho^{\star}_n)/n^7$)
Observing that 
$
r^{-1}_n(\rho^{\star}_n)\int_{r_n(\rho^{\star}_n)/ \bar b_n}^{\e r_n(\rho^{\star}_n)/ \b n^2} F^2_n(y)dy
=\int_{1/ \bar b_n}^{\e / \b n^2} [\tilde a_n\P(w_n(x)\geq v \tilde b_n)]^2dy
$
where
$\tilde b_n=\bar b_nr_n(\rho^{\star}_n)$ and $\tilde a_n\P(w_n(x)\geq \tilde b_n)=1$,
and proceeding as in \eqv(7.4.12)- \eqv(7.4.14), we obtain
\be
\Eq(7.4.19)
\E[\eta^{(2)}_{n,k}(\e)]
\leq 
%k (k-2)! n^{-(c_{\star}-1)(k-2)}
\frac{2\b nk^6 (k-2)!}{ n^{(c_{\star}-1)(k-2)}}
\left[2^{-n/6}r_n(\rho^{\star}_n)^{-2/3}
+2(\e/ \b n^2)^{1-\left(2\a_c(\frac{\varepsilon}{2})+o(1)\right)}
\right].
%+9(1+o(1))\e^{1-\frac{1}{3}\left(2\a_c(\frac{\varepsilon}{2})+o(1)\right)}
\ee
It remains to bound $\E[\eta^{(3)}_{n,k}(\e)]$.
Using Proposition \thv(5.prop1), (ii) and proceeding as in
\eqv(7.4.6)-\eqv(7.4.7), we get that on $\O_0\cap \O^{\star}$, 
for all but a finite number of indices $n$, for all
$x\in\CC$,
\bea
\Eq(7.4.20')
\textstyle
A_{n,\CC}^{(3)}(x)
%&\equiv& E_{x}\bigl( b_n^{-1}T^{\star}_{n,\CC}\1_{b_n^{-1}T^{\star}_{n,\CC}\leq \e}\bigr)
&\leq&
\textstyle
[1+o(1)]\varphi_{\e}\left([1+o(1)]k\varrho_{n,\CC}(0)/\bar b_n\right)
\\
&\leq&
\Eq(7.4.20)
\textstyle
[1+o(1)]\left( \e1_{\{k\varrho_{n,\CC}(0)<\bar b_n\e\}}
+
k\varrho_{n,\CC}(0)1_{\{k\varrho_{n,\CC}(0)\geq \bar b_n\e\}}
\right)
\eea
where
%, for $\varphi_{\e}$ defined in \eqv(7.4.8), 
the last line follows from the bounds $\varphi_{\e}(y)\leq \e$ and $\varphi_{\e}(y)\leq y$, both valid for all $y> 0$.
%In view of our earlier bounds on $A_{n,l}^{(1)}$, $A_{n,l}^{(2)}$
%(recall that $A_{n,l}^{(1)}(x)\leq \e1_{\theta^{\star}_{n,\CC}\geq b_n\e}$
%and $
%A_{n,\CC}^{(2)}(x)\leq b_n^{-1}\theta^{\star}_{n,\CC}\1\{b_n^{-1}\theta^{\star}_{n,\CC}\leq \e\}
%$)
Comparing the right hand side of \eqv(7.4.20) to the bounds 
\be
\Eq(7.4.22)
A_{n,l}^{(1)}(x)\leq \e\1_{\{\theta^{\star}_{n,\CC}\geq b_n\e\}},\quad
A_{n,\CC}^{(2)}(x)\leq b_n^{-1}\theta^{\star}_{n,\CC}\1_{\{b_n^{-1}\theta^{\star}_{n,\CC}\leq \e\}}
\ee
we just used, one sees that the contributions to $\E[\eta^{(3)}_{n,k}(\e)]$ coming from the two summands in \eqv(7.4.20) can be bounded 
in the same way as $\E[\eta^{(1)}_{n,k}(\e)]$ and $\E[\eta^{(2)}_{n,k}(\e)]$, respectively, 
replacing $\theta^{\star}_{n,\CC}$ by $k\varrho_{n,\CC}(0)$ in \eqv(7.4.22). 
Doing this, it follows from \eqv(7.4.16) and \eqv(7.4.19) that
\be
\Eq(7.4.21)
\hspace{-5.5pt}
\E[\eta^{(3)}_{n,k}(\e)]
\leq  
\frac{ k (k-2)! }{n^{(c_{\star}-1)(k-2)}}\left[
5\e(nk/\e)^{\left(2\a_c(\frac{\varepsilon}{2})+o(1)\right)}
+k2^{-n/6}
+k2\e^{1-\left(2\a_c(\frac{\varepsilon}{2})+o(1)\right)}
\right].
%\leq  
% k^2 (k-2)! n^{-(c_{\star}-1)(k-2)}\left[
%2^{-n/6}
%+2\e^{1-\left(2\a_c(\frac{\varepsilon}{2})+o(1)\right)}.
%\right]
\ee
Finally, collecting \eqv(7.4.16), \eqv(7.4.19), and \eqv(7.4.21), and using \eqv(8.3lem1.sum)
%(which is possible since by assumption $c_{\star}> 2$) 
to perform the sum over $k$, we arrive at
\be
\Eq(7.4.23)
\textstyle
\E[\sum_{i=1}^{3}\sum_{k=3}^{k_n^{\star}}\eta^{(i)}_{n,k}(\e)]
\leq
\frac{ c_0 }{n^{(c_{\star}-2)}}(\e/n)^{1-\left(2\a_c(\frac{\varepsilon}{2})+o(1)\right)}
+2^{-n/6}
\ee
for some constant $0<c_0\equiv c_0(\b, \varepsilon)<\infty$. 
Since by assumption $c_{\star}> 3$ and $1-\left(2\a_c(\frac{\varepsilon}{2})+o(1)\right)>0$ for all $n$ large enough,
the r.h.s.~of \eqv(7.4.23) is summabe in $n$. Combining this and \eqv(7.4.15') yields that
 $0<\eta_{n}(\e)-\eta_{n,2}(\e)\rightarrow 0$ $\P$-a.s., for all  $\varepsilon>0$. 
Since each $\eta_{n,k}(\e)$ in \eqv(7.4.5) is a monotonic function of $\e$, so is $\eta_{n}(\e)-\eta_{n,2}(\e)$, and thus, arguing again as in the proof of Proposition \thv(7.prop1),
% (see \eqv(7.prop1.4)-\eqv(7.prop1.5)),
\be
\textstyle
\lim_{\e\downarrow 0}\limsup_{n\uparrow \infty}
(\eta_{n}(\e)-\eta_{n,2}(\e))=0,\quad \P\text{-a.s.}.
%\quad\P\text{- almost surely}.
\Eq(7.4.C3.2)
\ee
Since  \eqv(7.4.C3.1) and \eqv(7.4.C3.2) imply \eqv(7.4.C3), Condition (A3) is verified $\P$-almost surely
under the assumptions of Theorem \thv(7.theo1).

%In conclusion, under the assumption of  Theorem \thv(2.theo2),

Having established that all three conditions of  Theorem \thv(7.theo1)  are satisfied $\P$-almost surely
with $\nu^\dagger$  given by \eqv(1.theo2.M2),
the proof of Theorem \thv(2.theo2) is done.

%Since all three conditions (C1), (C2), and (C3) are verified $\P$-almost surely, and since in the present case
%Condition (C0) can be omitted, it follows from Theorem \thv(6.theo1) 
%that, for the above choices of $a_n$, $c_n$, $\b$, and $c_{\star}$, 
% $\P$-almost surely, 
%\be
%S^\circ_n\Rightarrow_{J_1}  S^\circ_{\infty}
%\ee
%where $S^\circ_{\infty}$ is a subordinator with zero drift and L\'evy measure
%$\nu^\circ=\nu$ definined in \eqv(2.theo0.2). The proof of Theorem \thv(2.theo1) is complete.

%%%%%%%%%%%%%%%%%%%%%%%%%%%%%%%%%%%%%%%%%%%%%%%%%%%%%%%%%%%%

%%%%%%%%%%%%%%%%%%%%%%%%%%%%%%%%%%%%%%%%%%%%%%%%%%%%%%%%%%%
\section{Convergence of the back end clock process 
%above the critical temperature
below the critical line: proof of Theorem \thv(2.theo5)}
 \label{12}
  %\textsc{\textbf{becp}}
  
 Since \textsc{becp} below the critical line has a deterministic limit, this case is  simpler than 
that of Theorem \thv(2.theo2). The proof proceeds in two steps.
%The proof relies on a two-step strategy. 
 By \eqv(2.2.7) and \eqv(2.3.1),
\be
 S_n^\dagger (t)
=b_n^{-1}\sum_{i=0}^{k^\dagger_n(t)-1}\L_n^\dagger(J_n^\dagger(i)),
\Eq(12.1.1)
\ee
where by \eqv(2.2.8), $\L_n^\dagger(J_n^\dagger(i))$ is non zero if and only if 
$J_n^\dagger(i)\in\cup_{1\leq l\leq L^{\star}}C^{\star}_{n,l}$.
The first step consists in counting the number of distinct sets $C^{\star}_{n,l}$ that $J_n^\dagger$ visits along suitably chosen subsequences $0\leq i\leq k^\dagger_n(t)-1$, as well as the typical and maximal number of visits to each set. This information
is then used 
%allows
%serves
to extract sums of independent random variables from the sum \eqv(12.1.1), and a remainder. In a second step, and after suitable truncation, each of these sums is controlled by a classical mean-variance calculation, and the remainder is shown to be  sub-leading.

 \subsection{Step one}
 \label{10.1}
For simplicity we assume throughout that $\lfloor {k^\circ_n(t)}/{\ell^\circ_n}\rfloor={k^\circ_n(t)}/{\ell^\circ_n}$.
We first decompose the index set 
$
I=\{0,\dots, k^\dagger_n(t)-1 : J_n^\dagger(i)\in\VV^\circ_n\}
$ 
into $\ell^\circ_n$ disjoint subsets,  
$I=\cup_{0\leq j\leq\ell^\circ_n-1}I_j$, defined through 
\be
%\nonumber
I_j=\Biggl\{
0\leq i\leq k^\dagger_n(t)-1\,\Big|\, \exists 0\leq i''\leq  \frac{k^\circ_n(t)}{\ell^\circ_n}-1\,
s.t.~\sum_{i'=0}^{i}\1_{\{J_n^\dagger(i')\in\VV^\circ_n\}}=j+i''\ell^\circ_n
\Biggr\}.
\Eq(12.1.2)
\ee
The reason behind this choice is that, by definition of $J_n^\dagger$ and $J_n^\circ$,  \eqv(2.3.0), and \eqv(2.3.3),
%$k^\dagger_n(t)$, $k^\circ_n(t)$
\be
\{J_n^\dagger(i), i\in I_j\}\stackrel{d}{=}\{ J_n^\circ(j+i\ell^\circ_n), i=0,\dots,\lfloor {k^\circ_n(t)}/{\ell^\circ_n}\rfloor\}, 
\quad0\leq j\leq\ell^\circ_n-1,
\Eq(12.1.3)
\ee
(where equality holds in distribution) and by the mixing property of Proposition \thv(4.prop1), the chain  $J_n^\circ$ observed along such subsequences is well approximated by an  i.i.d.~sequence. Now clearly, visits of $J_n^\dagger$ to $\cup_{1\leq l\leq L^{\star}}C^{\star}_{n,l}$ can only occur along subsequences of the form $\{i+1: i\in  I_j\}$.
For  $0\leq j\leq\ell^\circ_n-1$, let $N_{n,j}$, respectively $M_{n,j}$, be the number of distinct sets $C^{\star}_{n,l}$, respectively the total number  of sets $C^{\star}_{n,l}$ that the chain $J_n^\dagger$ visits along each subsequence 
$
\{i+1: i\in  I_j\}
$.
%\be
%N_{n,j}(k)=\sum_{l\in\LL_n(k)}\1_{\{\exists i\in I_j : J_n^\dagger(i)\in C^{\star}_{n,l}\}},
%\quad
%M_{n,j}(k)=\sum_{l=1}^{L^{\star}}\sum_{ i\in I_j}\1_{\{J_n^\dagger(i)\in C^{\star}_{n,l}\}}.
%\Eq(12.1.4)
%\ee
Define
\be
\LL'_n\equiv\{1,\dots, L^{\star}\}\setminus\left(\cup_{2\leq k\leq k_n^{\star}}\LL_n(k)\right)
\equiv\cup_{2\leq k\leq k_n^{\star}}\LL'_n(k)
\Eq(12.1.5)
\ee
where, for $2\leq k\leq k_n^{\star}\equiv\max_{1\leq l\leq L^{\star}}|C^{\star}_{n,l}|$, 
\bea
\Eq(12.1.6)
&\LL_n(k)\equiv\{1\leq l\leq L^{\star}\,:\, |C^{\star}_{n,l}|=k,  \del C^{\star}_{n,l}\cap \del C^{\star}_{n,l'}=\emptyset\,
\forall l'\neq l, 1\leq l'\leq L^{\star} \},&
\\
&\LL'_n(k)\equiv\{l\in \LL'_n\,:\, |C^{\star}_{n,l}|=k \}.&
\Eq(12.1.6')
\eea
With these definitions $N_{n,j}=\sum_{k=2}^{k_n^{\star}}N_{n,j}(k)+N'_{n,j}$ and
$M_{n,j}=\sum_{k=2}^{k_n^{\star}}(M_{n,j}(k)+M'_{n,j})$
%\be
%\nonumber
%\textstyle
%N_{n,j}=\sum_{k=2}^{k_n^{\star}}N_{n,j}(k)+N'_{n,j},\,\,
%M_{n,j}=\sum_{k=2}^{k_n^{\star}}M_{n,j}(k)+M'_{n,j},\,\,
%M'_{n,j}=\sum_{k=2}^{k_n^{\star}}M'_{n,j}(k),
%\ee
where
\bea
%N_{n,j}=\sum_{k=2}^{k_n^{\star}}N_{n,j}(k),
&&
%\textstyle
\hspace{-12pt}
N_{n,j}(k)=\sum_{l\in\LL_n(k)}\1_{\{\exists i\in I_j : J_n^\dagger(i+1)\in C^{\star}_{n,l}\}},\,\,\,
N'_{n,j}=\sum_{l\in\LL'_n}\1_{\{\exists i\in I_j : J_n^\dagger(i+1)\in C^{\star}_{n,l}\}},\,\,\,
\Eq(12.1.7)
\\
%M_{n,j}=\sum_{k=2}^{k_n^{\star}}M_{n,j}(k),
&&
%\textstyle
\hspace{-12pt}
M_{n,j}(k)=\sum_{l\in\LL_n(k)}\sum_{ i\in I_j}\1_{\{J_n^\dagger(i+1)\in C^{\star}_{n,l}\}},\,\,\,
M'_{n,j}(k)=\sum_{l\in\LL'_n(k)}\sum_{ i\in I_j}\1_{\{J_n^\dagger(i+1)\in C^{\star}_{n,l}\}}.\,\,\,
\Eq(12.1.8)
\eea
Note that the number of sets $C^{\star}_{n,l}$, $l\in\LL_n(k)$,
% (respectively $l\in\LL'_n$) 
that are visited at least twice along a given subsequence  $\{i+1: i\in  I_j\}$ is at most $M_{n,j}(k)-N_{n,j}(k)$.
% (respectively $M'_{n,j}-N'_{n,j}$).
Lastly, let $m_{n,j}$ be the maximum number of visits to a given $C^{\star}_{n,l}$:
\be
m_{n,j}=\max_{1\leq l\leq L^{\star}}
%\sum_{i=0}^{\lfloor a_n t\rfloor-1}
\sum_{ i\in I_j}
\1_{\{J_n^\dagger(i+1)\in C^{\star}_{n,l}\}}.
\Eq(12.1.9)
\ee
%be the maximum number of visits to a given set $C^{\star}_{n,l}$.

Proposition \thv(12.prop1)  below 
yields control over these various quantities.
%states that ....... 
%introduce
In order to state it  we need the following definitions.
Given  constants $0<c_k, c'_k<\infty$ and $c_2=1$, let
 the functions $\d_n(2)$ and $\eta_n(k)$, $2\leq k\leq k_n^{\star}$, be defined through.
 \be
\eta_n(k)
=
\frac{c_kk!\lfloor a_nt\rfloor}{\ell^\circ_n 2n^{k(c_{\star}-1)+1}}(1+o(1)),
\,\,\, 
\d_n(2)=\frac{\lfloor a_nt\rfloor}{\ell^\circ_n2^{n}}\eta_n(2),
%\d_n(k)=\frac{k\lfloor a_nt\rfloor}{\ell^\circ_n2^{n+1}}\eta_n(k),
%\frac{c_k(k-1)!(k\lfloor a_nt\rfloor)^2}{(\ell^\circ_n)^2 2^{n+2}n^{k(c_{\star}-1)+1}}(1+o(1))
%\,\,\, 
%\eta'_n=c'nk_n^{\star}\eta_n(4),
\Eq(12.prop1.0)
\ee
We know (see \eqv(8.3lem1.5)) that  on $\O^{\star}$, for large enough $n$, $k_n^{\star}\leq {n}/({(c_{\star}-2)\log n})$. 
Given an integer $2<K<\infty$  define, for $2\leq k\leq K$,
\bea
\AA_{n,j,k}^{1,-}&=&
%\textstyle\bigcap_{0\leq j\leq\ell^\circ_n-1}\bigcap_{2\leq k\leq K}
\bigl\{\left|N_{n,j}(k)-\eta_n(k)\right|< n^{2c_{\star}}
%\bigl(\eta_n(k)\big)^{1/2}
\sqrt{\eta_n(k)}
\bigr\},
\Eq(12.prop1.01)
\\
\AA_{n,j,k}^{2,-}&=&
%\textstyle\bigcap_{0\leq j\leq\ell^\circ_n-1}\bigcap_{3\leq k\leq K}
\left\{M_{n,j}(k)-N_{n,j}(k)< (k/n)\eta_n(k)\right\},
%\cap\left\{ \left|[M_{n,j}(2)-N_{n,j}(2)]-\d_n(2)\right|< n^{2c_{\star}}\bigl(\d_n(2)\big)^{1/2}\right\}
\Eq(12.prop1.02)
\\\AA_{n,j,k}^{3,-}&=&
%\textstyle\bigcap_{0\leq j\leq\ell^\circ_n-1}\bigcap_{2\leq k\leq K}
\left\{ M'_{n,j}(k)<c'_knk_n^{\star}\eta_n(k+2)\right\},
\Eq(12.prop1.04)
\eea
and for $K< k\leq k_n^{\star}$,
\bea
\AA_{n,j,k}^{1,+}&=&
%\textstyle\bigcap_{0\leq j\leq\ell^\circ_n-1}\bigcap_{K< k\leq k_n^{\star}}
\bigl\{N_{n,j}(k)-n^2\eta_n(k)< n^{2c_{\star}+1}
%\bigl(\eta_n(k)\big)^{1/2}
\sqrt{\eta_n(k)}
\bigr\},
\Eq(12.prop1.01')
\\
\AA_{n,j,k}^{2,+}&=&
%\textstyle\bigcap_{0\leq j\leq\ell^\circ_n-1}\bigcap_{K< k\leq k_n^{\star}}
\left\{M_{n,j}(k)-N_{n,j}(k)< kn\eta_n(k)\right\},
%\cap\left\{ \left|[M_{n,j}(2)-N_{n,j}(2)]-\d_n(2)\right|< n^{2c_{\star}}\bigl(\d_n(2)\big)^{1/2}\right\}
\Eq(12.prop1.02')
\\\AA_{n,j,k}^{3,+}&=&
%\textstyle\bigcap_{0\leq j\leq\ell^\circ_n-1}\bigcap_{K< k\leq k_n^{\star}}
\left\{ M'_{n,j}(k)< c'_k n^3k_n^{\star}\eta_n(k+2)\right\},
\Eq(12.prop1.04')
\eea
Finally set
$
\AA_{n,j}^{0}=
%\textstyle\bigcap_{0\leq j\leq\ell^\circ_n-1}
\bigl\{ \left|[M_{n,j}(2)-N_{n,j}(2)]-\d_n(2)\right|< n^{2c_{\star}}
%\bigl(\d_n(2)\big)^{1/2}
\sqrt{\d_n(2)}
\bigr\}
$.
%\Eq(12.prop1.03)
%\ee

%Given functions $\eta_n(k)$, $\d_n(2)$, and $\eta'_n$, 
\begin{proposition}
   \TH(12.prop1)
Assume that $c_{\star}>3$ and let $2<K<\infty$ be a fixed integer.  
For all $\b>0$,  there exists a subset  $\O^{\scriptscriptstyle{\textsf{BCL}}}\subset\O$
with $\P(\O^{\scriptscriptstyle{\textsf{BCL}}})=1$ such that on 
$\O_1\cap\O^{\star}\cap\O^{\scriptscriptstyle{\textsf{BCL}}}$ 
the following holds: for all $t\in[0,T]$, for all but a finite number of indices $n$,
\be
P^\dagger_{\pi^\circ_n}\left(
\textstyle
\bigcap_{0\leq j\leq\ell^\circ_n-1}\AA_{n,j}^{0}
\bigcap_{i=1}^3\left(
(\bigcap_{2\leq k\leq K}\AA_{n,j,k}^{i,-})\cap(\bigcap_{K< k\leq k_n^{\star}}\AA_{n,j,k}^{i,+})
\right)
\right)
\geq 1- n^{-2(c_{\star}-1)+\bar c}
\Eq(12.prop1.1)
\ee
%\be
%P^\dagger_{\pi^\circ_n}(
%\AA_n^{0}\cap\AA_n^{1,-}\cap\AA_n^{2,-}\cap\AA_n^{3,-}\cap\AA_n^{1,+}\cap\AA_n^{2,+}\cap\AA_n^{3,+}
%)\geq 1- n^{-2(c_{\star}-1)+\bar c}
%\Eq(12.prop1.1)
%\ee
where $\bar c>0$ is arbitrary and for some constants $0<c_k, c'_k<\infty$ and $c_2=1$.
Moreover,
\be
P^\dagger_{\pi^\circ_n}\left( 
%\textstyle
\max_{0\leq j\leq\ell^\circ_n-1}m_{n,j}\geq \sfrac{3}{1-\varepsilon}
\right)\leq 2^{-n}.
\Eq(12.prop1.5)
\ee
  \end{proposition}

%The rest of this subsection is devoted to the proof of Proposition \thv(12.prop1).

\begin{proof}[Proof of Proposition \thv(12.prop1)] 
%We begin with two simple but important lemmata.
%The starting point of the proof are two simple lemmata.
%We begin with two needed lemmata.
The proof of Proposition \thv(12.prop1) hinges on Lemma \thv(12.lem0) below,
that allows to substitute
% new, 
simpler variables for the variables $N_{n,j}(k)$ and $M_{n,j}(k)$. In order to define them
%a lemma which we now state.
let $\{\wt J_n^\circ(i), i\geq 0, n>0\}$  and 
$
\left\{X_{n,i}(x), x\in (\cup_{1\leq l\leq L^{\star}}\del C^{\star}_{n,l}), i\geq 1\right\}
$
be two independent arrays of i.i.d.~r.v.'s, defined on  a common probability space $(\wt\O,\wt\FF,\wt P)$, such that
$\wt P(\wt J_n^\circ(i)=x)=\pi^\circ_n(x)$, $x\in\VV^\circ_n$,
and, for each $x\in (\cup_{1\leq l\leq L^{\star}}\del C^{\star}_{n,l})$,
% the $X_{n,i}$'s are Bernoulli r.v.'s with
\be
\wt P(X_{n,i}(x)=1)=1-\wt P(X_{n,i}(x)=0)
%=m^{\star}_{n,l}(x), \quad x\in \del C^{\star}_{n,l}.
=n^{-1}|\del x\cap (\cup_{1\leq l\leq L^{\star}}C^{\star}_{n,l})|.
\Eq(12.lem0.1)
\ee
Thus the $X_{n,i}$'s are Bernoulli r.v.'s.
%where $m^{\star}_{n,l}(x)$ is as in \eqv(4.prop4.0).
%Thus the $\wt J_n^\circ$'s are uniform r.v.'s on $\VV^\circ_n$,  and the $X_{n,i}$'s are Bernoulli r.v.'s. 
Using these arrays, define the quantities
\bea
%\textstyle
&\wt N_{n,j}(k)&
\textstyle\hspace{-7pt}
=\sum_{l\in\LL_n(k)}
\1_{\{\exists 0\leq i\leq \lfloor{k^\circ_n(t)}/{\ell^\circ_n}\rfloor-1:
\wt J_n^\circ(j+i\ell^\circ_n)\in \del C^{\star}_{n,l}, X_{n,i}(\wt J_n^\circ(j+i\ell^\circ_n))=1\}},\quad\quad
\Eq(12.lem0.2)
\\
&\wt M_{n,j}(k)&
\textstyle\hspace{-7pt}
=\sum_{l\in\LL_n(k)}\sum_{i=0}^{{k^\circ_n(t)}/{\ell^\circ_n} -1}
\1_{\{\wt J_n^\circ(j+i\ell^\circ_n)\in \del C^{\star}_{n,l}, X_{n,i}(\wt J_n^\circ(j+i\ell^\circ_n))=1\}}.
\Eq(12.lem0.3)
\eea
Similarly, let $\wt N'_{n,j}(k)$ and $\wt M'_{n,j}(k)$ be defined, respectively, as $\wt N_{n,j}(k)$ and $\wt M_{n,j}(k)$ with 
$\LL_n(k)$ replaced by $\LL'_n(k)$.

\begin{lemma}
   \TH(12.lem0)
For all $0\leq j\leq\ell^\circ_n-1$, all $2\leq k\leq k_n^{\star}$, and any $\eta,\rho, \eta',\rho'>0$,
\bea
P_{\pi^\circ_n}\left(\left|N_{n,j}(k)-\eta\right|\geq \rho\right) 
&\stackrel{d}{=}&
(1+\d_n)^{\lfloor {k^\circ_n(t)}/{\ell^\circ_n}\rfloor}
\wt P\bigl(\bigl|{\wt N}_{n,j}(k)-\eta\bigr|\geq \rho\bigr),
\Eq(12.lem0.4)
\\
P_{\pi^\circ_n}\left(\left|M_{n,j}(k)-\eta'\right|\geq \rho'\right) 
&\stackrel{d}{=}&
(1+\d'_n)^{\lfloor {k^\circ_n(t)}/{\ell^\circ_n}\rfloor}
\wt P\bigl(\bigl|{\wt M}_{n,j}(k)-\eta'\bigr|\geq \rho'\bigr),
\Eq(12.lem0.5)
\eea
for some $0\leq\delta_n,\delta'_n\leq   2^{-n}$, and where equality holds in distribution. The same relations hold
%with $N_{n,j}(k)$  replaced by $N'_{n,j}$ and $\LL_n(k)$ replaced by $\LL'_n$, and 
with $M_{n,j}(k)$ and $\wt M_{n,j}(k)$ replaced by $M'_{n,j}(k)$ and $\wt M'_{n,j}(k)$, respectively.
 \end{lemma}

 \begin{proof}  By \eqv(12.1.3) and Corollary \thv(8.0cor2),
 %\eqv(8.0cor2.1)
for  $0\leq j\leq\ell^\circ_n-1$ and all $2\leq k\leq k_n^{\star}$,
\bea
N_{n,j}(k)&\stackrel{d}{=}&
\textstyle
\sum_{l\in\LL_n(k)}
\1_{\{\exists 0\leq i\leq  {k^\circ_n(t)}/{\ell^\circ_n} -1:
J_n^\circ(j+i\ell^\circ_n)\in \del C^{\star}_{n,l}, J_n^\circ(j+i\ell^\circ_n+1)\in \del C^{\star}_{n,l}\}},
\Eq(12.lem0.6)
\\
M_{n,j}(k)&\stackrel{d}{=}&
\textstyle
\sum_{l\in\LL_n(k)}\sum_{i=0}^{{k^\circ_n(t)}/{\ell^\circ_n} -1}
\1_{\{J_n^\circ(j+i\ell^\circ_n)\in \del C^{\star}_{n,l}, J_n^\circ(j+i\ell^\circ_n+1)\in \del C^{\star}_{n,l}\}}.
\Eq(12.lem0.7)
\eea
Next note that by the Markov property and the mixing property of Proposition \thv(4.prop1) we have,
for all $y_i\in \VV^\circ_n$, $y'_i\in \VV^\circ_n$,  $i=0,\dots,\lfloor {k^\circ_n(t)}/{\ell^\circ_n}\rfloor$,
\bea
&&P^\circ_{\pi^\circ_n}\left(
%\cap_{i=0,\dots,\lfloor {k^\circ_n(t)}/{\ell^\circ_n}}
\cap_{0\leq i\leq \lfloor {k^\circ_n(t)} /{\ell^\circ_n}\rfloor}
\{J_n^\circ(j+i\ell^\circ_n)=y_i,J_n^\circ(j+i\ell^\circ_n+1)=y'_i\}
\right)
\cr
=
%\hspace{-10pt}
&&
\textstyle(1+\d_n)^{\lfloor {k^\circ_n(t)}/{\ell^\circ_n}\rfloor}
\prod_{0\leq i\leq \lfloor {k^\circ_n(t)}/{\ell^\circ_n}\rfloor}\pi^\circ_n(y_i)p^\circ_n(y_i,y'_i),\quad
\Eq(12.lem0.8)
\eea
for some  $0\leq\delta_n\leq   2^{-n}$, where $p^\circ_n(\cdot,\cdot)$ are the transition probabilities \eqv(2.1.5). 
%Now, by \eqv(4.prop4.4) and \eqv(4.prop4.5), 
%$
%\sum_{y'_i\in\del C^{\star}_{n,l} }p^\circ_n(y_i,y'_i)
%=1-\sum_{y'_i\notin \del C^{\star}_{n,l} }p^\circ_n(y_i,y'_i)
%=m^{\star}_{n,l}(y_i)
%$,
%for all  $y_i\in\del C^{\star}_{n,l}$.
%
% il faut faire attention que dans "l'evenement du complement", soit 1-\sum_{y'_i\notin \del C^{\star}_{n,l} }p^\circ_n(y_i,y'_i),
% la chaire peut aller visiter un autre C^{\star}_{n,l'} si \del C^{\star}_{n,l'}\cap \del C^{\star}_{n,l} est non vide pour la paire
% l, l'. C'est pour cela qu'on ala forme des proba des petites Bernoulli \eqv(12.lem0.1)
%
From this, \eqv(12.lem0.6)-\eqv(12.lem0.7), and the relations \eqv(4.prop4.4) and \eqv(4.prop4.5), the lemma easily follows.
One proves in a similar way that \eqv(12.lem0.7)
holds (for some other $0\leq\delta'_n\leq   2^{-n}$) with $M_{n,j}(k)$ and $\wt M_{n,j}(k)$ replaced by 
$M'_{n,j}(k)$ and $\wt M'_{n,j}(k)$, respectively.
 \end{proof}

% The reason behind the choice \Eq(12.1.3) is that
% by the mixing property of Proposition \thv(4.prop1), the chain  $J_n^\circ$ observed along such subsequences is well approximated by an  i.i.d.~sequence.

%The key point is that
Now the probabilities in the right hand sides of \eqv(12.lem0.4) and  \eqv(12.lem0.5) are not hard to evaluate. Set
 \be
 \Delta_{n,j}(k)=\wt M_{n,j}(k)-\wt N_{n,j}(k).
 \Eq(12.lem3.0)
 \ee

\begin{lemma}
   \TH(12.lem3)
Let $2<K<\infty$ be a fixed integer. There exists a subset  $\O^{\scriptscriptstyle{\textsf{BCL}}}\subset\O$
with $\P(\O^{\scriptscriptstyle{\textsf{BCL}}})=1$ such that, on $\O^{\scriptscriptstyle{\textsf{BCL}}}$, for all but a finite number of indices $n$, the following holds: 
there exists $\d_n(2)$, $\eta_n(k)$,  $2\leq k\leq k_n^{\star}$, satisfying
\eqv(12.prop1.0)
%\be
%\eta_n(k)
%=
%\frac{c_kk!\lfloor a_nt\rfloor}{\ell^\circ_n 2n^{k(c_{\star}-1)+1}}(1+o(1)),
%\,\,\, 
%\d_n(k)=
%%\frac{c_k(k-1)!(k\lfloor a_nt\rfloor)^2}{(\ell^\circ_n)^2 2^{n+2}n^{k(c_{\star}-1)+1}}(1+o(1))
%\frac{k\lfloor a_nt\rfloor}{\ell^\circ_n2^{n+1}}\eta_n(k),
%\,\,\, 
%\eta'_n=c'nk_n^{\star}\eta_n(4),
%\ee
%where  $0<c_k,c',c''<\infty$ are constants and $c_2=1$,
such that, for all $0\leq j\leq\ell^\circ_n-1$ and all $m_1,m_2>0$, 
%and all $2\leq k\leq k_n^{\star}$,
%$m_1,m_2,m_3,m_4>0$,
\bea
&&\wt P\bigl(
\bigl|
\Delta_{n,j}(2)-\d_n(2)
\bigr|\geq n^{m_2}
{\textstyle\sqrt{\d_n(2)}}
%\bigl(\d_n(2)\big)^{1/2}
\bigr)
\leq 2n^{-2m_2},
\Eq(12.lem3.2)
\\
&&
\wt P\bigl(
\bigl|
\wt N_{n,j}(k)-\eta_n(k)
\bigr|\geq n^{m_1}
{\textstyle\sqrt{\eta_{n,j}(2)}}
%\bigl(\eta_n(k)\big)^{1/2}
\bigr)
\leq 4n^{-2m_1},\quad 2\leq k\leq K,\quad
\Eq(12.lem3.1)
\\
&&\wt P\bigl(
\Delta_{n,j}(k)\geq (k/n)\eta_n(k)
\bigr)
\leq  e^{-c''n},\quad 3\leq k\leq K,
\Eq(12.lem3.3)
\\
&&\wt P\bigl(\wt M'_{n,j}(k)\geq 
%\eta'_n
c'_knk_n^{\star}\eta_n(2+k)
\bigr)
 \leq 2e^{-n^2/4},\quad 2\leq k\leq K,
\Eq(12.lem3.4)
\eea
and for all $0\leq j\leq\ell^\circ_n-1$, all $2<K< k\leq k_n^{\star}$, and all $m_1,m_2>0$, 
%and all $2\leq k\leq k_n^{\star}$,
%$m_1,m_2,m_3,m_4>0$,
\bea
&&\wt P\bigl(
\wt N_{n,j}(k)-n^2\eta_n(k)
\geq n^{m_1+1}
{\textstyle\sqrt{\eta_{n,j}(2)}}
%\bigl(\eta_n(k)\big)^{1/2}
\bigr)
\leq 4n^{-2m_1},
\Eq(12.lem3.1')
\\
&&\wt P\bigl(
\Delta_{n,j}(k)\geq nk\eta_n(k)
\bigr)
\leq  e^{-c''n},
\Eq(12.lem3.3')
\\
&&\wt P\bigl(\wt M'_{n,j}(k)\geq 
c'_kn^3k_n^{\star}\eta_n(2+k)
\bigr)
 \leq 2e^{-n^2/4}.
\Eq(12.lem3.4')
\eea
\end{lemma}

 \begin{proof}   First let $k=2$.
By \eqv(12.1.6),  $|\del x\cap (\cup_{1\leq l\leq L^{\star}}C^{\star}_{n,l})|=1$
for all $x\in\cup_{\LL_n(2)}\del C^{\star}_{n,l}$ so that
% $P(X_{n,i}(x)=1)=n^{-1}$ and 
\be
P\bigl(\wt J_n^\circ(j+i\ell^\circ_n)\in \del C^{\star}_{n,l}, X_{n,i}(\wt J_n^\circ(j+i\ell^\circ_n))=1\bigr)
=2(1-n^{-1})/|\VV^\circ_{n}|.
\Eq(12.1.15)
\ee
%The problem of evaluating rephrased in terms of random allocation of balls into urns:
%
The variables $\wt N_{n,j}(2)$ and $\wt M_{n,j}(2)$ can now be expressed using
%Questions on the variables $\wt N_{n,j}(2)$ can now be rephrased in terms of 
%in terms of 
random allocation of balls into urns:
each set  $C^{\star}_{n,l}$, $l\in\LL_n(2)$, stands for an urn
%and balls are independently distributed in urns with
and a ball is allocated to $C^{\star}_{n,l}$ if the event
$\{\wt J_n^\circ(j+i\ell^\circ_n)\in \del C^{\star}_{n,l}, X_{n,i}(\wt J_n^\circ(j+i\ell^\circ_n))=1\}$ occurs.
%, with probability \eqv(12.1.8). 
The total number of  balls is $\wt M_{n,j}(2)$ and the number of urns is $|\LL_n(2)|$. Then
$
\wt N_{n,j}(2)=|\LL_n(2)|-U_n
$,
where $U_n\equiv U_{n,j}(2)$ denotes the number of empty urns obtained by distributing 
$M_n\equiv \wt M_{n,j}(2)$ balls into  $L_n\equiv |\LL_n(2)|$ urns at random, uniformly. 
%In this formulation
%$\wt N_{n,j}(2)$ is the number of urns that contain at least one ball. Equivalently, 
The next lemma is classical.

\begin{lemma}[\cite{CT}  Section 9.2, Theorem 3]
%p 313
   \TH(12.lem4) Set $A_n\equiv \frac{M_n}{L_n}$ and $B_n\equiv(e^{A_n}-A_n-1)^{1/2}$. If $M_n\rightarrow\infty$ and $A_n/L_n\rightarrow 0$ as $n\rightarrow\infty$ then, for all $\d>0$,
\be
 P\left(\left|U_n-E\left[U_n\right]\right|>\d V_n
%\{E\left(U_{n,j}(2)-E\left[U_{n,j}(2)\right]\right)^2\}^{1/2}
 \right)\leq \d^{-2},
\Eq(12.lem4.1)
\ee 
where
$
E\left[U_n\right]=L_ne^{-A_n}(1+\OO(\frac{A_n}{L_n}))
$
and 
$
V^2_n=L_ne^{-2A_n}B_n^2(1+o(1))
$.
\end{lemma}

It remains to evaluate the quantities entering in Lemma \thv(12.lem4). The following lemma anticipates our needs.
%Proceeding as in the proof of 
%Lemma \thv(7.lem5) (see \eqv(7.lem5.12) - \eqv(7.lem5.11)),  we obtain that:
%
%Methode : estimer avec precision les indices $k$ petits. Puis, pour les indices plus grands, calculer la moyenne et faire 
%un Tchebychev d'ordre 1 : on perd un  facteur n^2
%

\begin{lemma}
   \TH(12.lem5) Let $2< K<\infty$ be a fixed integer. 
   If $c_{\star}>3$, there exists a subset  
   $\O^{\scriptscriptstyle{\textsf{BCL}}}\subset\O$
with $\P(\O^{\scriptscriptstyle{\textsf{BCL}}})=1$ such that, on 
$\O^{\star}\cap\O^{\scriptscriptstyle{\textsf{BCL}}}$, for all but a finite number of indices $n$ the following holds: 
for all $2\leq k\leq K$, 
%$
%\left|\LL_n(2)\right|=2^{n-1}n^{-2c_{\star}+1}(1+o(1))
%$,
$
\left|\LL_n(k)\right|=c_k(k-1)!2^{n-1}n^{-k(c_{\star}-1)-1}(1+o(1))
$
and 
$
\left|\LL'_n(k)\right|\leq c'_k(k+2)!2^{n-1}n^{-(k+2)(c_{\star}-1)-1}
$
while  for all $K< k\leq k_n^{\star}$,
$
\left|\LL_n(k)\right|\leq c_k(k-1)!2^{n-1}n^{-k(c_{\star}-1)+1}(1+o(1))
$
and
$
\left|\LL'_n(k)\right|\leq c'_k(k+2)!2^{n-1}n^{-(k+2)(c_{\star}-1)+1}
$
where $0<c_k,c'_k<\infty$ are constants and $c_2=1$.
%Moreover
%$
%?????
%$.
\end{lemma}

\begin{proof} If $k\leq K<\infty$ 
%(thus $k$ is independant of $n$) 
proceed as in the proof of Lemma \thv(7.lem5) (see \eqv(7.lem5.12)-\eqv(7.lem5.11)) with appropriate choices of the decomposition
\eqv(7.lem5.10')-\eqv(7.lem5.10). Otherwise, use a first order Tchebychev inequality. These are by now routine calculations which we leave out.
%we leave them out.
Note that working out an estimate on  $|\LL_n(k)|$ when $k\equiv k_n$ is  diverging with $n$
is a hard task, while working out an upper bound on $\E|\LL_n(k)|$ is elementary. 
This is what leads us we distinguish the cases $2\leq k\leq K$ and $K< k\leq k_n^{\star}$ in
Proposition \thv(12.prop1).
%We leave out the details.
\end{proof}

\begin{remark}
It follows from Lemma \thv(12.lem5) that if $c_{\star}>3$, the collection $C^{\star}_{n,l}$, $1\leq l\leq L^{\star}$, is dominated by ``isolated pairs", namely, by sets of size two that are at distance larger than three from all other  
sets. This is not true for smaller values of $c_{\star}$.
\end{remark}

%\begin{remark} Working out an estimate on  $|\LL_n(k)|$ when $k\equiv k_n$ is  diverging with $n$
%is a hard task, but working out an upper bound on $\E|\LL_n(k)|$ is elementary. 
%This explains why we distinguish the cases $2\leq k\leq K$ and $K< k\leq k_n^{\star}$ in
%Proposition \thv(12.prop1).
%\end{remark}

In view of  \eqv(12.1.8)-\eqv(12.1.9), $\wt M_{n,j}(2)$ is a sum of
$\left|\LL_n(2)\right|( \lfloor{k^\circ_n(t)}/{\ell^\circ_n}\rfloor-1)$ i.i.d. Bernoulli r.v.'s with success probability
$2(1-\frac{1}{n})/|\VV^\circ_{n}|$. Thus, on $\O^{\scriptscriptstyle{\textsf{BCL}}}$,
for $n$ large enough,
% for all but a finite number of indices $n$,
$
\textstyle
E\wt M_{n,j}(2)=2\frac{\left|\LL_n(2)\right|}{|\VV^\circ_{n}|}(1-\frac{1}{n})( \lfloor\frac{k^\circ_n(t)}{\ell^\circ_n}\rfloor-1)
=\frac{\lfloor a_nt\rfloor}{\ell^\circ_n n^{2c_{\star}-1}}(1+o(1))
$, 
while by Bennett's bound \eqv(7.lem5.7),
%P\left(\left|\wt M_{n,j}(2)-E\wt M_{n,j}(2)\right|\geq n \sqrt{E\wt M_{n,j}(2)}\right)\leq 2e^{-n^2/4}.
\be
\textstyle
P\left(\left|\wt M_{n,j}(2)-E\wt M_{n,j}(2)\right|< n 
\bigl(E\wt M_{n,j}(2)\bigr)^{1/2}\right)\geq 1-2e^{-n^2/4}.
%\sqrt{E\wt M_{n,j}(2)}\right)\leq 2e^{-n^2/4}.
\Eq(12.1.16)
\ee
Conditionnal on the event appearing in the probability \eqv(12.1.16),
%$\{\left|M_{n,j}(2)-EM_{n,j}(2)\right|< n \sqrt{EM_{n,j}(2)}\}$,
$\wt M_{n,j}(2)\rightarrow\infty$ as $n\rightarrow\infty$ and
$
A_n\equiv \frac{M_{n,j}(2)}{|\LL_n(2)|}
=\frac{\lfloor a_nt\rfloor}{\ell^\circ_n 2^{n-1}}(1+o(1))
\leq n^{-3}t 2^{-n(1-\varepsilon)}
\rightarrow 0
$ 
 for all $t\leq T<\infty$ and $1<\varepsilon<1$. Thus $A_n/|\LL_n(2)|\rightarrow 0$ as well. 
Thus the conditions of Lemma \thv(12.lem4) are satisfied, and so, on 
$\O^{\scriptscriptstyle{\textsf{BCL}}}$, for all but a finite number of indices $n$, for all $m_1,m_2>0$,
setting 
$
\Delta_{n,j}(2)\equiv \wt M_{n,j}(2)-\wt N_{n,j}(2)
$,
\bea
&
\textstyle
\wt P\left(
\left|
\wt N_{n,j}(2)-\wt E\wt N_{n,j}(2)
\right|\geq n^{m_1}\bigl(\wt E\wt N_{n,j}(2)\big)^{1/2}
\right)
\leq 2n^{-2m_1}+2e^{-n^2/4},\quad\quad
&
\Eq(12.1.16')
\\
&
\textstyle
\wt P\left(
\left|
\Delta_{n,j}(2)-\wt E\Delta_{n,j}(2)
\right|
\geq 
 n^{m_2}\bigl(\wt E\Delta_{n,j}(2)\big)^{1/2}
\right)
\leq 
 2n^{-2m_2}+2e^{-n^2/4},\quad\quad
 &
\Eq(12.1.17)
\eea
where
$
\wt E\wt N_{n,j}(2)
=
\frac{\lfloor a_nt\rfloor}{\ell^\circ_n n^{2c_{\star}-1}}(1+o(1))
$
$
\wt E\Delta_{n,j}(2)
=
\frac{\lfloor a_nt\rfloor^2}{(\ell^\circ_n)^2 2^{n}n^{2c_{\star}-1}}(1+o(1))
$.
This proves \eqv(12.lem3.2) and \eqv(12.lem3.1) with $k=2$. 

If  $k>2$,  $|\del x\cap (\cup_{1\leq l\leq L^{\star}}C^{\star}_{n,l})|=|\del x\cap C^{\star}_{n,l}|$  
for all $x\in\del C^{\star}_{n,l}$, $l\in\LL_n(k)$, but $|\del x\cap C^{\star}_{n,l}|$ is not always one.  However,  at most $k$ vertices  of $\del C^{\star}_{n,l}$ can have more than one neighbor in $C^{\star}_{n,l}$, and a given vertex cannot have more than $k$ neighbors.
% in $C^{\star}_{n,l}$. 
Thus,
denoting by $\del C^{\star,+}_{n,l}$ and $\del C^{\star,-}_{n,l}$ the subsets
where $|\del x\cap C^{\star}_{n,l}|=1$ and 
$|\del x\cap C^{\star}_{n,l}|>1$, respectively,
it follows from \eqv(10.lem1.6) that  on $\O^{\star}$, 
$
|\del C^{\star,+}_{n,l}|\geq |\del C^{\star}_{n,l}|-k>k(n-1)
$
%and $P(X_{n,i}(x)=1)=n^{-1}$ for  all $x\in\cup_{\LL_n(k)}\del C^{\star,+}_{n,l}$,
while 
$
|\del C^{\star,+}_{n,l}|\leq k
$.
%and $P(X_{n,i}(x)=1)\leq kn^{-1}(1+o(1))$ for $x\in\cup_{\LL_n(k)}\del C^{\star,-}_{n,l}$. 
Then 
$
\wt N_{n,j}^{+}(k) \leq \wt N_{n,j}(k) \leq \wt N_{n,j}^{+}(k)+\wt N_{n,j}^{-}(k)
$
where $\wt N_{n,j}^{\pm}(k)$ are defined as $\wt N_{n,j}(k)$ with $\del C^{\star}_{n,l}$ replaced by
$\del C^{\star,\pm}_{n,l}$. With obvious notation we write
$\wt M_{n,j}^{\pm}(k)$ and $\wt \Delta_{n,j}^{\pm}(k)$.
%\be
%\textstyle\hspace{-7pt}
%$
%\wt N_{n,j}^{+,\pm}(k)=\sum_{l\in\LL_n(k)}
%\1_{\{\exists 0\leq i\leq \lfloor{k^\pm_n(t)}/{\ell^\circ_n}\rfloor-1:
%\wt J_n^\circ(j+i\ell^\circ_n)\in \del C^{\star,\pm}_{n,l}, X_{n,i}(\wt J_n^\circ(j+i\ell^\circ_n))=1\}},\quad\quad
%\Eq(12.1.18)
%\ee

At this point we must distinguish the cases $k\leq K<\infty$ and $K<k\leq k_n^{\star}$. Assume first that
$k\leq K<\infty$.  To bound $\wt N_{n,j}^{+}(k)$ we proceed exactly as for $\wt N_{n,j}(2)$: 
the r.h.s.~of \eqv(12.1.15) is now equal to $k(1-o(1))/|\VV^\circ_{n}|$,
%\be
%P(\wt J_n^\circ(j+i\ell^\circ_n)\in \del C^{\star,+}_{n,l}, X_{n,i}(\wt J_n^\circ(j+i\ell^\circ_n))=1)
%=k(1-o(1))/|\VV^\circ_{n}|,
%\Eq(12.1.19)
%\ee
and we get that 
$\wt N_{n,j}^{+}(k)$ obeys \eqv(12.1.16) - \eqv(12.1.17), under the same assumptions, with
$
\wt E\Delta_{n,j}^{+}(k)
=
\frac{c_k(k-1)!(k\lfloor a_nt\rfloor)^2}{(\ell^\circ_n)^2 2^{n+2}n^{k(c_{\star}-1)+1}}(1+o(1))
$
and
$
\wt E\wt N_{n,j}^{+}(k)
=
\frac{c_kk!\lfloor a_nt\rfloor}{\ell^\circ_n 2n^{k(c_{\star}-1)+1}}(1+o(1))
$.
%\be
%\textstyle
%\wt P\left(
%\left|
%\wt N_{n,j}^{+,+}(k)-\frac{c_kk!\lfloor a_nt\rfloor}{\ell^\circ_n 2n^{k(c_{\star}-1)+1}}(1+o(1))
%\right|\geq n^{m_1}\bigl(\frac{c_kk!\lfloor a_nt\rfloor}{\ell^\circ_n 2n^{k(c_{\star}-1)+1}}\big)^{1/2}
%\right)
%\leq 2n^{-2m_1}+2e^{-n^2/4}.
%\Eq(12.1.20)
%\ee 
Using that the l.h.s.~of \eqv(12.1.15) is  bounded above by
$
%P(\wt J_n^\circ(j+i\ell^\circ_n)\in \del C^{\star,-}_{n,l}, X_{n,i}(\wt J_n^\circ(j+i\ell^\circ_n))=1)\leq
 k^2(1-o(1))/n|\VV^\circ_{n}|
$
%we also get, always under the same assumptions, that
%\bea
%&
%\textstyle
%\wt P\left(
%\wt N_{n,j}^{+,-}(k)\geq\wt E\wt N_{n,j}^{+,-}(k)
%+n^{m_1}\bigl(\wt E\wt N_{n,j}^{+,-}(k)\big)^{1/2}
%\right)
%\leq 2n^{-2m_1}+2e^{-n^2/4},\quad\quad
%\Eq(12.1.22)
%\\
%&
%\textstyle
%\wt P\left(
%\Delta_{n,j}^{+,-}(k)\geq\wt E\Delta_{n,j}^{+,-}(k)
%+
% n^{m_2}\bigl(\wt E\Delta_{n,j}^{+,-}(k)\big)^{1/2}
% \right)
%\leq 
% 2n^{-2m_2}+2e^{-n^2/4},\quad\quad
% &
%\Eq(12.1.23)
%\eea
%where
%$
%\wt E\wt N_{n,j}^{+,-}(k)
%=
%\frac{k}{n}\wt E\wt N_{n,j}^{+,+}(k)
%%\frac{c_kkk!\lfloor a_nt\rfloor}{\ell^\circ_n 2n^{k(c_{\star}-1)+2}}(1+o(1))
%$
%and
%$
%\wt E\Delta_{n,j}^{+,+}(k)
%=\frac{c_k(k-1)!k^4\lfloor a_nt\rfloor^2}{(\ell^\circ_n)^2 2^{n+2}n^{k(c_{\star}-1)+2}}(1+o(1))
%$.
it also follows from  \eqv(12.1.16) that
\be
\textstyle
\wt P\left(\wt M_{n,j}^{-}(k)\geq 
\frac{k}{n}E\wt M_{n,j}(k)+n \bigl(E\wt M_{n,j}(k)\bigr)^{1/2}\right)\leq 2e^{-n^2/4}.
\Eq(12.1.24)
\ee
Combining these estimates and the bound
$
\wt \Delta_{n,j}(k) \leq \wt \Delta_{n,j}^{+}(k)+ \wt M_{n,j}^{-}(k)
$,
 yield
\eqv(12.lem3.1) and \eqv(12.lem3.3)
for all $3\leq k<K$.
%
% note that we dont't get a very sharp estimate for \Delta_{n,j}(k)\geq (k/n)\eta_{n,j}(k) 
% in \Eq(12.lem3.3) when 3\leq k\leq k_n^{\star} because all the "balls" falling in  urns of type 
% "{n,j}^{+,-}(k)" could in fact belong to urns of type "{n,j}^{+,+}(k)"
%
The same arguments apply to the case $K<k\leq k_n^{\star}$ but  since 
Lemma \thv(12.lem5) now only provides an upper bound on $\left|\LL_n(k)\right|$,
they only yield one-sided inequalities, namely, \eqv(12.lem3.1') and \eqv(12.lem3.3').

It remains to bound $M'_{n,j}(k)$. For this we simply use that by \eqv(10.lem1.6), on $\O^{\star}$, the l.h.s.~of \eqv(12.1.15) is  
bounded above by
$
P(\wt J_n^\circ(j+i\ell^\circ_n)\in \del C^{\star}_{n,l})\leq n k_n^{\star}(1-o(1))/|\VV^\circ_{n}|
$
%\be
%\wt M'_{n,j}\leq
%\textstyle
%\sum_{l\in\LL'_n}\sum_{i=0}^{{k^\circ_n(t)}/{\ell^\circ_n} -1}
%\1_{\{\wt J_n^\circ(j+i\ell^\circ_n)\in \del C^{\star}_{n,l}\}},
%% that is we throw out X_{n,i}(J_n^\circ(i))=1
%\Eq(12.lem1.25)
%\ee
and, proceeding as in the proof of \eqv(12.1.16), we obtain \eqv(12.lem3.4) and  \eqv(12.lem3.4').
The proof of Lemma \thv(12.lem3) is done.
 \end{proof}

Combining the estimates of Lemma \thv(12.lem3) together with \eqv(12.lem0.4)-\eqv(12.lem0.5)
and choosing $m_1=m_2=2c_{\star}$ now easily yields \eqv(12.prop1.1).
%proves the Proposition \thv(12.prop1) with the exception of
To prove \eqv(12.prop1.5) write
%be the maximum number of visits to a given set $C^{\star}_{n,l}$.
\be
P^\dagger_{\pi^\circ_n}\left( 
\textstyle m_{n,j}\geq \kappa
\right)
\leq 
\textstyle\sum_{1\leq l\leq L^{\star}}
P^\dagger_{\pi^\circ_n}\left( 
\sum_{ i\in I_j}
\1_{\{J_n^\dagger(i+1)\in C^{\star}_{n,l}\}}\geq \kappa
\right).
%\leq 2^{-n}.
\Eq(12.1.25)
\ee
Using  Corollary \thv(8.0cor2), the mixing property of Proposition \thv(4.prop1), and observing that
$k^\circ_n(t)= \lfloor a_n t\rfloor\leq \lfloor a_n T\rfloor$, one sees that
%By Lemmata \thv(12.lem1) and \thv(12.lem2), 
\eqv(12.1.25) is bounded above by
\be
%P^\dagger_{\pi^\circ_n}\left( 
%\textstyle\sum_{ i\in I_j}
%\1_{\{J_n^\dagger(i)\in C^{\star}_{n,l}\}}\geq K
%\right)
%\leq
%(1+\d_n)^{\lfloor {k^\pm_n(t)}/{\ell^\circ_n}\rfloor}
%P^\dagger_{\pi^\circ_n}\left( 
%\textstyle\cup_{0\leq i_1<\dots<i_K\leq  {k^\circ_n(t)}/{\ell^\circ_n} -1}
%\cap_{k=1}^K\{J_n^\circ(j+i_k\ell^\circ_n)\in \del C^{\star}_{n,l}, X_{n,i_k}(J_n^\circ(i_k))=1\}
%\right)
%
% dropping the event X_{n,i_k}(J_n^\circ(i_k))=1\}
%
%(1+\d_n)^{\lfloor \frac{k^\circ_n(t)}{\ell^\circ_n}\rfloor}
%\hspace{-5pt}(1+o(1))
%\textstyle
(1+o(1))\sum_{1\leq l\leq L^{\star}}P^\dagger_{\pi^\circ_n}\left( 
\textstyle\bigcup_{0\leq i_1<\dots<i_\kappa\leq  {\lfloor a_n T\rfloor}/{\ell^\circ_n} -1}
\bigcap_{k=1}^\kappa\{\wt  J_n^\circ(j+i_k\ell^\circ_n)\in \del C^{\star}_{n,l}, \}
\right).
\Eq(12.1.26)
\ee
Using again that on $\O^{\star}$, 
$
\wt P(\wt J_n^\circ(j+i\ell^\circ_n)\in \del C^{\star}_{n,l})\leq n k_n^{\star}(1-o(1))/|\VV^\circ_{n}|
$,
and since $L^{\star}<2^n$,
\be
P^\dagger_{\pi^\circ_n}\left( 
\textstyle m_{n,j}\geq \kappa
\right)
\leq 
L^{\star}\left(\sfrac{n k_n^{\star}\lfloor a_n T\rfloor}{\ell^\circ_n|\VV^\circ_{n}|}\right)^\kappa
\leq
(1+o(1))\left({T}/{n}\right)^{\kappa}
2^{-n\{\kappa(1-\varepsilon)-1\}}.
%\leq
%\left(\sfrac{T}{n}\right)^{3}2^{-2n}
\Eq(12.1.27)
\ee
Choosing $\kappa=\sfrac{3}{1-\varepsilon}$ in \eqv(12.1.27) implies \eqv(12.prop1.5).
The proof of Proposition \thv(12.prop1) is done.
%\begin{remark}
%we in fact proved that
%$
%P^\dagger_{\pi^\circ_n}\left( 
%\textstyle\sup_{0\leq t\leq T}\max_{0\leq j\leq\ell^\circ_n-1}m_{n,j}\geq \sfrac{3}{1-\varepsilon}
%\right)\leq 2^{-n}
%$.
%\end{remark}
 \end{proof}

\subsection{Step two}
 \label{10.2}
 
 Building on Step 1 we now aim at proving that:
% Let us now establish that:
 
  \begin{proposition}
   \TH(12.prop2)
Assume that $c_{\star}>3$ and set
$
b_n=a_n\exp(n(\b/2)^2)/(\b\sqrt{\pi n})
$.
%\be
%b_n=\frac{a_ne^{n(\b/2)^2}}{\b\sqrt{\pi n}}
%\ee.
For all 
$
0<\b\leq 2\b_c(\varepsilon/2)-12\sqrt{{c_{\star}\log n}/{n}}
$,
all  $t\in[0,T]$,  and large enough $n$, 
\be
\P\left\{
P^\dagger_{\pi^\circ_n}\left(
\textstyle\left|S_n^\dagger(t)-t\right|\geq c_1 \sqrt{t/n}+c_2 t n^{-(c_{\star}-1)}
\right)\leq n^{-2}
\right\}
> 1-n^{-2}
\Eq(12.prop2.1)
\ee
for some constants $c_1, c_2>0$. 
  \end{proposition}

 \begin{proof}
By definition of the sets $I_j$ of Step 1, assuming again that 
%With the notation of Section \thv(12.1), assuming again that 
$\lfloor {k^\circ_n(t)}/{\ell^\circ_n}\rfloor={k^\circ_n(t)}/{\ell^\circ_n}$,
%$\lfloor \frac{k^\circ_n(t)}{\ell^\circ_n}\rfloor=\frac{k^\circ_n(t)}{\ell^\circ_n}$,
% $0\leq l\leq \ell^\circ_n-1$, 
\be
S_n^\dagger(t)=\frac{1}{\ell^\circ_n}
\sum_{j=0}^{\ell^\circ_n-1}S_{n,j}^\dagger (t),\quad 
S_{n,j}^\dagger (t)\equiv\frac{\ell^\circ_n}{b_n}\sum_{i\in  I_j}
\L_n^\dagger(J_n^\dagger(i+1)).
\Eq(12.2.1)
\ee
Note that each realisation of the sets
$\{C^{\star}_{n,l},1\leq l\leq L^{\star}\}$
induces a partition of
$\{1,\dots, L^{\star}\}$ into
$\cup_{2\leq k\leq k_n^{\star}}\left(\LL_n(k)\cup \LL'_n(k)\right)$
%(see \eqv(12.1.5)-\eqv(12.1.6))
and, given this partition,
%by Proposition for $\P$-almost all realisations
each $I_j$ can be decomposed into disjoint sets,
$
I_j=\cup_{2\leq k\leq k_n^{\star}}(I_{j,0}(k)\cup I_{j,1}(k)\cup I'_j(k))
$,
with the properties that:
%, using the notation of \eqv(12.1.7)-\eqv(12.1.9):
%\item{(a)}
%$|I_{j,0}(k)|=N_{n,j}(k)$, $|I_{j,1}(k)|=M_{n,j}(k)-N_{n,j}(k)$, $|I'_j(k)|=M'_{n,j}(k)$;
\item{(a)}
$
\left\{J_n^\dagger(i+1), i\in I_{j,0}(k)\right\}
$
consists of $N_{n,j}(k)$ visits to distinct sets $C^{\star}_{n,l}, l\in\LL_n(k)$;
\item{(b)}
$
\left\{J_n^\dagger(i+1), i\in I_{j,1}(k)\right\}
$
consists of $M_{n,j}(k)-N_{n,j}(k)$ visits to sets $C^{\star}_{n,l}, l\in\LL_n(k)$, 
each $C^{\star}_{n,l}$ being visited at most $m_{n,j}$ times;
\item{(c)}
$
\left\{J_n^\dagger(i+1), i\in I'_j(k)\right\}
$
consists of $M'_{n,j}(k)$ visits to sets  $C^{\star}_{n,l}, l\in\LL'_n(k)$, 
each $C^{\star}_{n,l}$ being visited at most $m_{n,j}$ times.

Denoting by $S_{n,j,0}^{k}$, $S_{n,j,1}^{k}$, and ${S'}^k_{n,j}$, $2\leq k\leq k_n^{\star}$, the restrictions of the sum in
$S_{n,j}^\dagger$ to, respectively,   $I_{j,0}(k)$, $I_{j,1}(k)$, and $I'_j(k)$, $2\leq k\leq k_n^{\star}$,
we have
\be
\textstyle
S_{n,j}^\dagger(t)=\sum_{k=2}^{k_n^{\star}}
\bigl[S_{n,j,0}^{k}(t)+S_{n,j,1}^{k}(t)+{S'}^k_{n,j}(t)\bigr].
\Eq(12.2.2)
\ee
Observe that by (a), each $S_{n,j,0}^{k}$ is a sum of $N_{n,j}(k)$ independent rv's, 
and by (b) and (c) each $S_{n,j,1}^{k}$, respectively ${S'}^k_{n,j}$, is bounded above by 
$m_{n,j}$ times a sum of (at most) $M_{n,j}(k)-N_{n,j}(k)$, respectively $M'_{n,j}(k)$, independent rv's.
%To make these bounds explicit we need the next two lemmata.
Before making these bounds explicit we state two lemmata.
%Let us now estimate each term in the sum.
%In order to elaborate on these bounds we need the next two lemmata.
%Before elaborating on these bounds we  need the  lemmata
%The next two lemmata are designed to deal with such sums when $k=2$ and $k>2$, respectively.

Let $\P_{\CC}$ denote the conditional distribution $\P(\cdot\mid \{C^{\star}_{n,l}, 1\leq l\leq L^{\star}\})$.
%Set $\bar\b\equiv\b/\sqrt 2$, $\b_c(\bar\varepsilon)\equiv\sqrt{\bar\varepsilon2\log 2}$.
%we use the notation \eqv(1.2.4)
%We use the following notations and assumptions:  $\bar a_n>0$ is an increasing sequence such that
%$
%0<\bar\varepsilon_n\equiv\frac{\log \bar a_n}{n\log 2}<1
%$;
%$\bar\b\equiv\b/\sqrt 2$ and $\b_c(\bar\varepsilon)\equiv\sqrt{\bar\varepsilon2\log 2}$;
%finally, as in \eqv(10trunc.1.2), $u_n\equiv (\sqrt n \b)^{-1}\log r_n\left(\rho^{\star}_n\right)$.

%%%%%%%%%%%%%%%%%%%%%%%%%%%%%%%%%%%%%%%%%%%%%%%%%%%%%%%
\begin{lemma}
   \TH(12.lem6)
Let $\bar a_n>0$ be an increasing sequence such that
$
0<\bar\varepsilon_n\equiv\frac{\log \bar a_n}{n\log 2}<1
$.   
Given $\{l_i, 1\leq i\leq \lfloor \bar a_nt\rfloor\}\subseteq\LL_n(2)$ let
$(G_{n,i}, 1\leq i\leq \lfloor \bar a_nt\rfloor)$ be independent geometric r.v.'s satisfying
$
P_G(G_{n,i}=k)= p_{n,i}(1-p_{n,i})^{k-1}, k\geq 1
$,
where
$
(p_{n,i})^{-1}=1+(n-1)^{-1}\varrho_{n,l_i}(0)
$
%\vfill\eject
for $\varrho_{n,l_i}(0)$ as in \eqv(5.lem1.0).
%\be
%\frac{1}{p_{n,i}}=1+\frac{1}{n-1}
%\varrho_{n,l_i}(0)=e^{-\b\max\left\{(H_n(y),H_n(x))\,|\,\{x,y\}\in C^{\star}_{n,l_i}\right\}},
%e^{-\b\max_{\{x,y\}\in C^{\star}_{n,l_i}}(H_n(y),H_n(x))}
%\varrho_{n,l_i}(0) see \eqv(5.lem1.0)
% for \varrho_{n,l}(1) see \eqv(5.lem1.0')
% \theta^{\star}_{n,l} \eqv(5.lem1.0) and \eqv(5.prop1.1)
%
%e^{-\b\sqrt{n}\max(g_i,g'_i)}.
%\ee
%
%Let $\P_{\CC}$ be the conditional distribution $\P(\cdot\mid \{C^{\star}_{n,l}, 1\leq l\leq L^{\star}\})$ and 
Let $P_G$ be the  joint law of  $(G_{n,i}, 1\leq i\leq \lfloor \bar a_nt\rfloor)$.
Set $\bar\b=\b/\sqrt 2$ and, given a constant $\bar\kappa>0$, set 
$
s_n=n^{2\bar\kappa}e^{n\bar\b\b_c(\bar\varepsilon_n)}
$
and 
$
\bar b_n=\bar a_n\E_{\CC}E_G(G_{n,i}\1_{\{G_{n,i}<s_n\}})
$.
%where $\bar\b\equiv\b/\sqrt 2$, $\b_c(\bar\varepsilon)\equiv\sqrt{\bar\varepsilon2\log 2}$.
Then, for all 
$
\bar\b\leq \b_c(\bar\varepsilon_n)-\sqrt{{8\bar\kappa\log n}/{n}}
$,
all 
$
t\in[0,T]
$,
and any sequence $\bar\d_n>n^{-\bar\kappa}$,
\be
\nonumber
\P_{\CC}\left\{P_G\left(
\textstyle\left|\frac{1}{\bar b_n}\sum_{i=1}^{\lfloor \bar a_nt\rfloor}G_{n,i}-
\frac{{\lfloor \bar a_nt\rfloor}}{{\bar a_n}}
\right|\geq \sqrt{t\bar\d_n}
\right)
\leq 2\bar\d_n^{-1}n^{-\bar\kappa}
\right\}
\geq
1-n^{-\bar\kappa+2(c_{\star}+1)},
%\Eq(12.lem6.1)
\ee
and
$
\E_{\CC}E_G(G_{n,i}\1_{\{G_{n,i}<s_n\}})=(1+o(1))\frac{n^{2c_{\star}-1}}{\bar\b\sqrt{2\pi n}}e^{n\bar\b^2/2}
%e^{{n\bar \b^2}/{2}}
%\E_{\CC}E_G(G_{n,i}\1_{\{G_{n,i}<s_n\}})=(1+o(1))\frac{2n^{2c_{\star}-1}}{\b\sqrt{\pi n}}e^{n(\b/2)^2}
$.
\end{lemma}
%%%%%%%%%%%%%%%%%%%%%%%%%%%%%%%%%%%%%%%%%%%%%%%%%%%%%%%

\begin{lemma}
   \TH(12.lem7)
   Let $\bar a_n>0$ be an increasing sequence such that
$
0<\bar\varepsilon_n\equiv\frac{\log \bar a_n}{n\log 2}<1
$.   
Given $k> 2$ and $\{l_i, 1\leq i\leq \lfloor \bar a_nt\rfloor\}\subseteq\LL_n(k)$ let 
$(Y^k_{n,i}, 1\leq i\leq \lfloor \bar a_nt\rfloor)$ be i.i.d.~integer valued r.v.'s satisfying
$
P_Y(Y^k_{n,i}>m)
\leq e^{-m(n/k\varrho_{n,l_i}(0))(1-o(1))}(1+o(1))
$
for all $m\geq \theta^{\star}_n$, where $\varrho_{n,l_i}(0)$ and $\theta^{\star}_n$ are as in 
\eqv(5.lem1.0)-\eqv(5.prop1.1).
Let $P_Y$ be the joint law of $(Y^k_{n,i}, 1\leq i\leq \lfloor \bar a_nt\rfloor)$.
\item{(i)} Given 
%a constant 
$\bar\kappa>0$, let  $\bar\b$ and $s_n$
be as in Lemma \thv(12.lem6) and set 
$
\bar b_n=\bar a_n\E_{\CC}E_Y(Y^k_{n,i}\1_{\{Y^k_{n,i}<s_n\}})
$. 
Then, for all 
$
\bar\b\leq \b_c(\bar\varepsilon_n)-\sqrt{{8\bar\kappa\log n}/{n}}
$,
all 
$
t\in[0,T]
$,
and any sequence $\bar\d_n>n^{-\bar\kappa}$,
% we have, 
\be
\nonumber
\P_{\CC}\left\{P_Y\left(
\textstyle
\frac{1}{\bar b_n}\sum_{i=1}^{\lfloor \bar a_nt\rfloor}Y^k_{n,i}-
\frac{{\lfloor \bar a_nt\rfloor}}{{\bar a_n}}
\geq \sqrt{t\bar\d_n}
\right)
 \leq 2\bar\d_n^{-1}n^{-\bar\kappa}
\right\}
\geq
1-n^{-\bar\kappa+2(c_{\star}+2)},
%\Eq(12.lem7.1)
\ee
and
$
(1+o(1))\frac{k}{\bar\b\sqrt{\pi n}}n^{2c_{\star}-1}e^{{n\bar \b^2}/{2}}
\leq
\E_{\CC}E_Y(Y^k_{n,i}\1_{\{Y^k_{n,i}<s_n\}})
\leq(1+o(1))\frac{k^2(k-1)}{2\bar\b\sqrt{2\pi n}}n^{2c_{\star}-1}e^{{n\bar \b^2}/{2}}
%e^{{n\bar \b^2}/{2}}
%note that (\b/2)^2=\frac{1}{2}(\b/\sqrt 2)^2
$.
\item{(ii)} For all $\b>0$, all 
$
t\in[0,T]
$,
 and any  $\bar\kappa',\bar\kappa''>0$, taking 
$
\bar b_n=\bar a_n\E_{\CC}E_Y(Y^k_{n,i})
$,
\be
\nonumber
\P_{\CC}\left\{P_Y\left(
\textstyle
\frac{1}{\bar b_n}\sum_{i=1}^{\lfloor \bar a_nt\rfloor}Y^k_{n,i}\geq tn^{\bar\kappa'+\bar\kappa''}
\right)
 \leq n^{-\bar\kappa'}
\right\}
\geq
1-n^{-\bar\kappa''},
%\Eq(12.lem7.1)
\ee
and
$
(1+o(1))\frac{k}{\bar\b\sqrt{\pi n}}n^{2c_{\star}-1}e^{{n\bar \b^2}/{2}}
\leq
\E_{\CC}E_Y(Y^k_{n,i})
\leq(1+o(1))\frac{k^2(k-1)}{2\bar\b\sqrt{\pi n}}n^{2c_{\star}-1}e^{{n\bar \b^2}/{2}}
%e^{{n\bar \b^2}/{2}}
%note that (\b/2)^2=\frac{1}{2}(\b/\sqrt 2)^2
$.
\end{lemma}

\begin{proof}[Proof of Lemmata \thv(12.lem6) and \thv(12.lem7)] We only indicate the main lines of the proofs. Consider first 
Lemma \thv(12.lem6). Set $G'_{n,i}=G_{n,i}\1_{\{G_{n,i}<s_n\}}$. Then, for all $\e>0$,
\be
\textstyle
P_G\left(\left|\frac{1}{\bar b_n}\sum_{i=1}^{\lfloor \bar a_nt\rfloor}G_{n,i}-\s_n\right|\geq \e \right)
\leq 
\e^{-2}\bar\s_n+\hat\s_n,
\Eq(12.lem6.2)
\ee
where 
$
\bar b_n\s_n=E_G\sum_{i=1}^{\lfloor \bar a_nt\rfloor}G'_{n,i}
$,
$
\bar b_n^2\bar\s_n=
%(\e \bar b_n)^{-2}
E_G\sum_{i=1}^{\lfloor \bar a_nt\rfloor}{G'}_{n,i}^2
$,
and
$
\hat \s_n=\sum_{i=1}^{\lfloor \bar a_nt\rfloor}P_G(G_{n,i}>s_n)
$.\hfill\break
Now
$
E_G G'_{n,i}=\frac{1}{p_{n,i}}\left(1-(1-p_{n,i})^{s_n-1}\right)
$,
$
E_G {G'}_{n,i}^2\leq 2\frac{1-p_{n,i}}{p_{n,i}^2}\left(1-(1-p_{n,i})^{s_n+1}\right)
$,
and 
$
P_G(G_{n,i}>s_n)=(1-p_{n,i})^{s_n+1}
$.
Thus each of the $\s_n$'s in \eqv(12.lem6.2) is a sum of a function of the variables
$
\varrho_{n,l_i}(0)
%=e^{-\b\max\left\{(H_n(y),H_n(x))\,|\,\{x,y\}\in C^{\star}_{n,l_i}\right\}}
$.
%$e^{-\b\sqrt{n}\max(g_i,g'_i)}$.
%\frac{1}{\bar b_n}
Note that
$
\P_{\CC}\left(\left|
\s_n-\E_{\CC} \s_n
\right|\geq\e'
\right)\leq{\e'}^{-2}\E_{\CC}\bar\s_n
$
and, by first order Tchebychev inequalities,
$
\P_{\CC}(\bar\s_n>\e')\leq {\e'}^{-1}\E_{\CC}\bar\s_n
$
and 
$
\P_{\CC}(\hat\s_n>\e')\leq {\e'}^{-1}\E_{\CC}\hat\s_n
$.
Clearly,
$
\E_{\CC}\s_n={{\lfloor \bar a_nt\rfloor}}/{{\bar a_n}}
$,
while
$
\E_{\CC}\bar\s_n\leq n^{-2\bar\kappa+2c_{\star}+3/2}
$ 
and 
$
%\E_{\CC}\hat\s_n\leq  n^{-2(\bar\kappa-(c_{\star}+1))}
\E_{\CC}\hat\s_n\leq  n^{-2(\bar\kappa-1)\b_c(\bar\varepsilon)/\bar\b+2c_{\star}}
$
for all
$
\bar\b\leq \b_c(\bar\varepsilon_n)-\sqrt{{8\bar\kappa\log n}/{n}}
$.
We do not reproduce the lenghty but simple calculations leading to these bounds. They resemble the high temperature moments estimates of the  partition function of the REM: here
%where $\max\left\{(H_n(y),H_n(x))\,|\,\{x,y\}\in C^{\star}_{n,l_i}\right\}$ plays the role of the Hamiltonian
 the quantity
\be
p^{-1}_{n,i}=1+(n-1)^{-1}
%\varrho_{n,l_i}(0)
e^{-\b\max\{(H_n(y),H_n(x))\,|\,\{x,y\}\in C^{\star}_{n,l_i}\}}
\ee
plays the role of the Boltzmann weights,
%the key point being that 
the gist of the argument being that
 the truncation of the $G_{n,i}$'s at level $s_n$ 
effectively results in a truncation of the $p^{-1}_{n,i}$'s at level $s_n$, and that $s_n$ is, roughly speaking, of the order of the maximum of the $p^{-1}_{n,i}$'s.
% (see Lemma \thv(9.lem7)).
%    clips the largest
Based on this \eqv(12.lem6.2) is easily worked out.

To prove Lemma \thv(12.lem7), (i), replace $G_{n,i}$  by $Y^k_{n,i}$ and $G'_{n,i}$ by $Y'_{n,i}=Y^k_{n,i}\1_{\{Y^k_{n,i}<s_n\}}$  
in the above expressions of $s_n$, $\bar s_n$, and $\hat s_n$, and use
that
$
\varrho_{n,l_i}(0)\geq e^{-\b\max(H_n(y),H_n(x))}
$
and
$
\varrho_{n,l_i}(0)\leq
\sum_{ \{x,y\}\in G(C^{\star}_{n,l_i})} e^{-\b\max(H_n(y),H_n(x))}
$
for any  $\{x,y\}\in G(C^{\star}_{n,l_i})$.
Again
$
\E_{\CC}\s_n={{\lfloor \bar a_nt\rfloor}}/{{\bar a_n}}
$
while 
$
\E_{\CC}\bar\s_n\leq k^{-2}n^{-2\bar\kappa+2c_{\star}+5/2}
$ 
and 
$
%\E_{\CC}\hat\s_n\leq  n^{-2(\bar\kappa-(c_{\star}+1))}
\E_{\CC}\hat\s_n\leq  k^{2-\b_c(\bar\varepsilon)/\bar\b}n^{-2(\bar\kappa-1)\b_c(\bar\varepsilon)/\bar\b-1/2+2c_{\star}}
$.
Assertion (ii) of the lemma follows from two first order Tchebychev inequalities applied in a row.
\end{proof}

We are now equipped to estimate the sum \eqv(12.2.2).
Consider first  $S_{n,j,0}^{2}(t)$. By (a), there exists a subset 
$\{l_i, 1\leq i\leq N_{n,j}(2)\}\subseteq\LL_n(2)$ such that
\be
\textstyle
S_{n,j,0}^{2}(t)=\frac{\ell^\circ_n}{b_n}\sum_{i=1}^{N_{n,j}(2)}
\L_n^\dagger(J_n^\dagger(i))\1_{\{J_n^\dagger(i)\in C^{\star}_{n,l_i}\}}
\stackrel{d}{=}\frac{\ell^\circ_n}{b_n}\sum_{i=1}^{N_{n,j}(2)}G_{n,i},
\Eq(12.2.3)
\ee
where $G_{n,i}$ is as in Lemma \thv(12.lem6). The last equality 
follows from Proposition \thv(5.prop1), (i),  and holds in distribution.
In the same way we have, by Proposition \thv(5.prop1), (ii), that on $\O^{\star}\cap\O_0$,
 \be
\textstyle
S_{n,j,0}^{k}(t)
\stackrel{d}{=}\frac{\ell^\circ_n}{b_n}\sum_{i=1}^{N_{n,j}(k)}Y^k_{n,i},\quad 2<k\leq k_n^{\star},
\Eq(12.2.4)
\ee
where $Y^k_{n,i}$ is as in Lemma \thv(12.lem7). 
%Under the assumptions of Proposition \thv(12.prop1),  
On the one hand Lemma \thv(12.lem6)  combined with Proposition \thv(12.prop1) 
yields an accurate estimate on the last sum in \eqv(12.2.3) for all 
$
\b\leq 2\b_c(\varepsilon/2)-4\sqrt{{\bar\kappa\log n}/{n}}-5(c_{\star}+1){\log n}/{n}
%\log(\ell^\circ_n n^{2c_{\star}-1})/(n\log 2)
$.
On the other hand, the sum in the r.h.s.~of  \eqv(12.2.4) can be bounded from above
using  Lemma \thv(12.lem7) and Proposition \thv(12.prop1) . Note however that the temperature threshold
in Lemma \thv(12.lem7), (i), is a decreasing function of $\bar a_n$. Namely, for $\bar a_n=N_{n,j}(k)$, by 
\eqv(12.prop1.01)-\eqv(12.prop1.01'),
$
\b_c(\bar\varepsilon_n)-\sqrt{{8\bar\kappa\log n}/{n}}
=
2\b_c(\varepsilon/2)-4\sqrt{{\bar\kappa\log n}/{n}}-\OO({k\log n}/{n})
$.
We will thus use Lemma \thv(12.lem7), (i), when e.g.~$k\leq n^{1/4}$, and Lemma \thv(12.lem7), (ii), else.
%By Lemma \thv(12.lem6) and  Lemma \thv(12.lem7), and 
% each of the sums in the r.h.s.~of \eqv(12.2.3) and \eqv(12.2.4)
%can be either estimated with precision, if $k=2$, either bounded from above
% with different accuracies if  $2\leq k\leq K$ or  
%$K< k\leq k_n^{\star}$
%(of course we always have the trivial lower bound $S_{n,j,0}^{k}(t)\geq 0$). 
Doing so we obtain that
% of Proposition \thv(12.prop1),  
on $\O_0\cap\O_1\cap\O^{\star}\cap\O^{\scriptscriptstyle{\textsf{BCL}}}$,
choosing 
%$\bar\d_n=1/\log n$ and 
%(with the notation of  Lemma \thv(12.lem6))
\be
b_n=\ell^\circ_n\E_{\CC}N_{n,j}(2)E_G(G_{n,i}\1_{\{G_{n,i}<s_n\}})\Big|_{t=1}
=
\textstyle
\frac{a_n}{(\b/\sqrt{2})\sqrt{2\pi n}}e^{n(\b/2)^2}(1+o(1)),
\Eq(12.2.5)
\ee
for all 
$
0<\b\leq 2\b_c(\varepsilon/2)-4\sqrt{{\bar\kappa\log n}/{n}}(1+o(1))
$,
any $\bar\kappa, \bar\kappa', \bar\kappa''>0$, and any $\bar\d_n>n^{-\bar\kappa}$, 
with $\P$- probability larger than 
$1-k_n^{\star}n^{-\bar\kappa+2(c_{\star}+1)}-k_n^{\star}n^{-\bar\kappa''}$,
\be
%\P\left\{
P^\dagger_{\pi^\circ_n}\left(
\textstyle\left|\sum_{k=2}^{k_n^{\star}}S_{n,j,0}^{k}(t)-t
%\frac{{\lfloor a_nt\rfloor}}{{a_n}}
\right|\geq
 \sqrt{t\bar\d_n}(1+\e_{n,1})+t\e_{n,1}
% \sqrt{\frac{t}{\log n}}(1+\e_{n,0})+t\e_{n,0}
\right)
\leq 2\bar\d_n^{-1}k_n^{\star}n^{-\bar\kappa}+k_n^{\star}n^{-\bar\kappa'},
%\leq 2k_n^{\star}n^{-\bar\kappa}\log n 
%\right\}
%\geq 1-k_n^{\star}n^{-\bar\kappa+2(c_{\star}+1)},
\Eq(12.2.6)
\ee
for some constant $c_{j,1}>0$.

Let us now establish that the contributions of the terms $S_{n,j,1}^{k}(t)$ and ${S'}^k_{n,j}(t)$ in \eqv(12.2.2) 
are subleading. 
%We do this for $k>2$ (the case  $k=2$ is deduced by substituting $G_{n,i}$ 
%for $Y^k_{n,i}$ in  \eqv(12.2.7) and  \eqv(12.2.8) below).
Consider first $S_{n,j,1}^{k}(t)$, $k>2$.
% set $ \Delta_{n,j}(k)\equiv M_{n,j}(k)-N_{n,j}(k)$. 
Let
$\{l_i, 1\leq i\leq \Delta_{n,j}(k)\}\subseteq\LL_n(k)$, $\Delta_{n,j}(k)\leq M_{n,j}(k)-N_{n,j}(k)$, 
be the labels of the distinct  sets $C^{\star}_{n,l}, l\in\LL_n(k)$,  that  $J_n^\dagger(i+1)$ visits 
as $i$ runs through $I_{j,1}(k)$, and let
$
\{\a_{n,l_i},  1\leq i\leq \Delta_{n,j}(k)\}
$
be the number of visits to each $C^{\star}_{n,l_i}$.
Reasoning as in \eqv(12.2.4),
\be
\textstyle
S_{n,j,1}^{k}(t)
%=\frac{\ell^\circ_n}{b_n}\sum_{i=1}^{\Delta_{n,j}(k)}
%\a_{n,l_i}\L_n^\dagger(J_n^\dagger(i))\1_{\{J_n^\dagger(i)\in C^{\star}_{n,l_i}\}}
\stackrel{d}{=}\frac{\ell^\circ_n}{b_n}\sum_{i=1}^{\Delta_{n,j}(k)}\a_{n,l_i}Y^k_{n,i}
\leq m_{n,j}\frac{\ell^\circ_n}{b_n}\sum_{i=1}^{M_{n,j}(k)-N_{n,j}(k)}Y^k_{n,i},
\Eq(12.2.9)
\ee
where the last inequality is deterministic and follows from the fact that 
$0<\a_{n,l_i}\leq m_{n,j}$ (see \eqv(12.1.9)) and  $Y^k_{n,i}\geq 0$.
In the same way, for $k>2$, the sum ${S'}^k_{n,j}(t)$
% in \eqv(12.2.2) 
obeys
\be
\textstyle
{S'}^k_{n,j}(t)
\stackrel{d}{=}\frac{\ell^\circ_n}{b_n}\sum_{i=1}^{\Delta'_{n,j}(k)}\a'_{n,l_i}Y^k_{n,i}
\leq m_{n,j}\frac{\ell^\circ_n}{b_n}\sum_{i=1}^{M'_{n,j}(k)}Y^k_{n,i},
\Eq(12.2.10)
\ee
for some $0<\a'_{n,l_i}\leq m_{n,j}$ and $\Delta'_{n,j}(k)<M'_{n,j}(k)$. 
If $k=2$ substitute $G_{n,i}$  for $Y^k_{n,i}$ in \eqv(12.2.9) and \eqv(12.2.10).
%The case  $k=2$ is deduced from \eqv(12.2.7) and \eqv(12.2.8) substituting $G_{n,i}$  for $Y^k_{n,i}$ therein.
Using again Lemma \thv(12.lem6), Lemma \thv(12.lem7), 
and Proposition \thv(12.prop1), we obtain that under the assumption of Proposition \thv(12.prop1),  
on $\O_0\cap\O_1\cap\O^{\star}\cap\O^{\scriptscriptstyle{\textsf{BCL}}}$,
for $b_n$ as in \eqv(12.2.5),
for all 
$
0<\b\leq 2\b_c(\varepsilon/2)-4\sqrt{{\bar\kappa\log n}/{n}}(1+o(1))
$,
any $\bar\kappa, \bar\kappa', \bar\kappa''>0$, and any $\bar\d_n>n^{-\bar\kappa}$, 
with $\P$- probability larger than 
$1-2k_n^{\star}n^{-\bar\kappa+2(c_{\star}+1)}-2k_n^{\star}n^{-\bar\kappa''}$,
\be
P^\dagger_{\pi^\circ_n}\left(\textstyle
\sum_{k=2}^{k_n^{\star}}(S_{n,j,1}^{k}(t)+{S'}^k_{n,j}(t))
\geq
(t+ \sqrt{t\bar\d_n})\e_{n,2}
 \right)
\leq 4\bar\d_n^{-1}k_n^{\star}n^{-\bar\kappa}+2_n^{\star}n^{-\bar\kappa'},
\Eq(12.2.11)
\ee
where
% (using \eqv(8.3lem1.sum))
$
\e_{n,2}\leq c_{j,2} n^{-c_{\star}}(1+n^{-K(c_{\star}-1)+2})+ a_n2^{-n+1}
+c_{n,3}( n^{-2(c_{\star}-2)}+n^{-K(c_{\star}-1)+4})
$
for some constants $c_{j,2},c_{n,3}>0$. 

Combining \eqv(12.2.6) and \eqv(12.2.11) yields an estimate on $S_{n,j}^\dagger (t)$ for each 
$0\leq j\leq \ell^\circ_n-1$, and inserting in \eqv(12.2.1) we arrive at:

\begin{lemma}
   \TH(12.lem8)
Let $b_n$ be as in \eqv(12.2.5). Under the assumption of Proposition \thv(12.prop1),  
on $\O_0\cap\O_1\cap\O^{\star}\cap\O^{\scriptscriptstyle{\textsf{BCL}}}$,
the following holds:
for all 
$
0<\b\leq 2\b_c(\varepsilon/2)-4\sqrt{{\bar\kappa\log n}/{n}}(1+o(1))
$,
all 
$
t\in[0,T]
$,
any $\bar\kappa, \bar\kappa', \bar\kappa''>0$, and any sequence $\bar\d_n>n^{-\bar\kappa}$,
with $\P$- probability larger than 
$1-4\ell^\circ_nk_n^{\star}n^{-\bar\kappa+2(c_{\star}+1)}-4\ell^\circ_nk_n^{\star}n^{-\bar\kappa''}$,
\be
P^\dagger_{\pi^\circ_n}\left(
\textstyle\left|S_n^\dagger(t)-t
\right|\geq
 \sqrt{t\bar\d_n}(1+\e_{n,0})+t\e_{n,0}
\right)
\leq 6\bar\d_n^{-1}\ell^\circ_nk_n^{\star}n^{-\bar\kappa}+4\ell^\circ_nk_n^{\star}n^{-\bar\kappa'},
\Eq(12.lem8.1)
\ee
$
\e_{n,0}\leq c_{j,0} n^{-(c_{\star}-1)}(1+n^{-K(c_{\star}-1)+2})
$
%\Eq(12.2.7)
%\ee
for some constant $c_{j,0}>0$ and $K>2$.
\end{lemma}

Choosing $\bar\d_n=1/n$, $\bar\kappa>6+2(c_{\star}+1)$, and $\bar\kappa'=\bar\kappa''>6$ in
Lemma \thv(12.lem8), and taking $\bar c=1$ in \eqv(12.prop1.1)  
yields the statement of  Proposition \thv(12.prop2). \end{proof}

\begin{proof}[Proof of Theorem \thv(2.theo5)]
Theorem \thv(2.theo5) immediately follows from Proposition \thv(12.prop2)
since  $S_n^\dagger(t)$ is a monotonous function whose limit is itself continuous.
\end{proof}

%Together with the fact that $S_n^\dagger(t)$ is a monotonous function 
%(and its limit continuous),  \eqv(12.lem8.2) implies .......... the statement of the result ..........

%Under the assumptions of Proposition \thv(12.prop1), 
%with $P^\dagger_{\pi^\circ_n}$ larger than $1- n^{-2(c_{\star}-1)+\bar c}$, 
%each 

%%%%%%%%%%%%%%%%%%%%%%%%%%%%%%%%%%%%%%%%%%%%%%%%%%%%%%%%%%%%

%%%%%%%%%%%%%%%%%%%%%%%%%%%%%%%%%%%%%%%%%%%%%%%%%%%%%%%%%%%%

 \section{Appendix B: Speed of convergence to Perron projector.}
 \label{B}

All the matrix analysis results quoted in this section are well known and can be found in
%are quoted from
%can be found in 
\cite{HJ}  (see in particular Subsection 8.5 p.~515 for nonnegative and primitive matrices).

\subsection{The case of nonnegative and primitive matrices}
 \label{B.1}

 Let $A=(a(i,j))_{1\leq i,j\leq N}$ be an $N\times N$ nonnegative and primitive matrix, that is to say, $A\geq 0$ and $A^k>0$ 
 for some $k\geq 1$.
 %\cite{HJ}, by Theo 8.5.2 p. 516 a nonnegative matrx A is primitive iff $A^k>0$ for some $k\geq 1$
 % example: transition matrix of an irreducible and aperiodic  Markov Chain, or sub-Markov Chain
By Perron's theorem 
%(see e.g.~\cite{HJ}, Theorem 8.5.1 p.~516)
 its eigenvalues satisfy
\be
\Eq(B1.lem1.1)
\l_0>|\l_1|\geq |\l_2|\geq \dots \geq |\l_{N-1}|\,\,\,\hbox{\text{and}}\,\,\, \l_0>0,
\ee
%denoting by $\rho(A)$ the spectral radius of $A$, $\rho(A)=\l_0>0$.
and there exists a unique vector $u>0$ such that $uA=\l_0u$, and a unique vector $v>0$ such that $Av=\l_0v$, normalized such that 
\be
\Eq(B1.lem1.2)
\textstyle
\sum_{1\leq i\leq N}u(i)=1, \sum_{1\leq i\leq N}u(i)v(i)=1;
\ee
these are  the left and right Perron eigenvectors of $A$. Futhermore, denoting by 
\be
\Eq(B1.lem1.3)
L =(v(i)u(j))_{1\leq i,j\leq N}
\ee
the associated projector, we have $\lim_{m\rightarrow\infty}(\l_0^{-1}A)^m=L$.
%there exists some $C=C(r,A)$ such that $\|(\l_0^{-1}A)^m-L\|_{\infty}<Cr^m$, where $r=|\l_1|/\l_0<1$. 
The lemma below provides 
an explicit bound on the speed of this convergence under an additionnal symmetry (or ``reversibility'') assumption.
%(for its properties see  \cite{HJ}, Lemma 8.2.7, p.~498  together with the remark below Definition 8.5.0, p.~516).
Write $A^m=(a^{(m)}(i,j))_{1\leq i,j\leq N}$.

 \begin{lemma}
   \TH(B1.lem1)
 % Let $A=(a(i,j))_{1\leq i,j\leq N}$ be a nonnegative and primitive matrix and write $A^m=(a^{(m)}(i,j))_{1\leq i,j\leq N}$.
  Assume that 
 there exists a vector 
 %$\pi_A=(\pi_A(i))_{1\leq i\leq N}$
$\pi_A>0$ such that  $\sum_{1\leq i\leq N}\pi_A(i)=1$, and
 \be
\Eq(B1.lem1.0)
 \pi_A(i)a(i,j)=\pi_A(j)a(j,i)\quad\text{for all } 1\leq i,j\leq N.
\ee
Then, for all $m\in\N$ and all $1\leq i\leq N$, we have
%\be
%\Eq(B1.lem1.4)
%\left|a^{(m)}(i,j)-\l_0v(i)u(j)\right|
%\leq
%|\l_1|^mC(i,j),
%\ee
%and
\be
\Eq(B1.lem1.5)
\textstyle
\left|\sum_{1\leq j\leq N}a^{(m)}(i,j)-\l^m_0v(i)\right|
\leq 
|\l_1|^m/\sqrt{\pi_A(i)}.
%\sqrt{1/{\pi_A(i)}}.
\ee
 \end{lemma}
 
 \begin{proof}[Proof of Lemma \thv(B1.lem1)] Let $S$ be the diagonal matrix with diagonal entries 
 $s(i,i)\equiv\sqrt{\pi_A(i)}$. The matrix $\wh A\equiv SAS^{-1}$ has entries 
 $\hat a(i,j)\equiv\sqrt{\pi_A(i)/\pi_A(j)}a(i,j)$ and so, by  \eqv(B1.lem1.0), is symmetric. 
From the identity $\wh A^m=SA^mS^{-1}$ it follows 
%that $A\geq 0$ implies that $B\geq 0$ and $A^k>0$ implies that $B^k>0$. 
%that, just as $A$, $\wh A$ is a nonnegative and primitive matrix.
that if $A$ is a nonnegative and primitive matrix then so is $\wh A$.
Finally, since $A$ and $\wh A$ are similar they have the same eigenvalues, 
given by \eqv(B1.lem1.1).
For $L$ given by  \eqv(B1.lem1.3) set
$\wh L\equiv SLS^{-1}=(\hat v(i)\hat u(j))_{1\leq i,j\leq N}$
where
$
\hat u\equiv uS^{-1}
%=(u(i)/\pi_A(i))_{1\leq i\leq N}
$, 
$
\hat v\equiv Sv
%=(v(i)\pi_A(i))_{1\leq i\leq N}
$. 
One checks that
%$\sum_{1\leq i\leq N}\hat u(i)\hat v(i)=1$, 
$\hat u\wh A=\l_0\hat u$, $\wh A\hat v=\l_0\hat v$ and thus, since $\wh A$ is symmetric, $\hat v\wh A=\l_0\hat v$, $\wh A\hat u=\l_0\hat u$. But this implies that $\hat u=\hat v$ and so, $\wh L$ is symmetric.

Set $R\equiv A-\l_0L$ and $\wh R\equiv SRS^{-1}=\wh A-\l_0\wh L$. 
It results from the above that $\wh R$ is symmetric, and
since $R$ and $\wh R$ are similar they have the same eigenvalues.
The proof of Lemma \thv(B1.lem1) now hinges on the following key facts.
Given a matrix $B$ we denote by $\s(B)$ its spectrum and by $\rho(B)=\max\{|\l| : \l\in\s(B)\}$  its spectral radius. We have:
\item\hskip1truecm (i) for all $m\in\N$, $R^m=A^m-(\l_0L)^m=A^m-\l_0^mL$, and
\item\hskip1truecm (ii) $\rho(A-\l_0L)\leq |\l_1|<\rho(A)=\l_0$.\hfill\break
%Likewise,
%\item\hskip1truecm (iii) for all $m\in\N$,  $\wh R^m=\wh A^m-(\l_0\wh L)^m=\wh A^m-\l_0^m\wh L$.\hfill\break
%\item\hskip1truecm (iv) $\rho(\wh A-\l_0\wh L)\leq |\l_1|<\rho(\wh A)=\l_0$.\hfill\break
%\item\hskip1truecm (iii) for all $m\in\N$,  $\wh R^m=\wh A^m-(\l_0\wh L)^m=\wh A^m-\l_0^m\wh L$, and
%\item\hskip1truecm (iv) $\rho(\wh A-\l_0\wh L)\leq |\l_1|<\rho(\wh A)=\l_0$.\hfill\break
For a proof see (b), (e), and (h) of Lemma 8.2.7  p.~498 of  \cite{HJ}  together with the remark below Definition 8.5.0, p.~516.

To proceed observe that 
%$R=S^{-1}\wh RS$ so that 
$R^m=S^{-1}\wh R^mS$. 
Thus, denoting by $\d_i$ the unit vector with  components $\d_i(j)=1$ if $i=j$ and $\d_i(j)=0$ else, it follows from (i)
that
\be
\Eq(B1.lem1.6)
a^{(m)}(i,j)-\l_0^mv(i)u(j)=(\d_i,R^m\d_j)=(\d_i,S^{-1}\wh R^mS\d_j)=\sqrt{\frac{\pi_A(j)}{\pi_A(i)}} (\d_i,\wh R^m\d_j).
\ee
Summing \eqv(B1.lem1.6) over $j$ we get, by Cauchy-Schwarz's inequality  and \eqv(B1.lem1.2), that
\be
\Eq(B1.lem1.8)
\textstyle
\left|\sum_{1\leq j\leq N}a^{(m)}(i,j)-\l_0^mv(i)\right|
%=\sum_{1\leq j\leq N}\sqrt{\frac{\pi_A(j)}{\pi_A(i)}} (\d_i,\wh R^m\d_j)
\leq
\textstyle
\sqrt{1/{\pi_A(i)}}\left[\sum_{1\leq j\leq N}(\d_i,\wh R^m\d_j)^2\right]^{1/2}.
\ee
Since $\wh R$ is symmetric,
$
\sum_{1\leq j\leq N}(\d_i,\wh R^m\d_j)^2
%=\sum_{1\leq j\leq N}(\d_i,\wh R^m\d_j)(\d_j,\wh R^m\d_i)
=(\d_i,\wh R^{2m}\d_i)
$.
Now,
\be
\Eq(B1.lem1.7)
\textstyle
(\d_i,\wh R^{2m}\d_i)
\leq \sup_{\|x\|_2=1}(x,\wh R^{2m}x)
%= \sup_{\|x\|_2=1}(x, R^{2m}x)
\leq (\rho(\wh R))^{2m}=(\rho(R))^{2m}
\leq |\l_1|^{2m}.
\ee
The middle inequality in \eqv(B1.lem1.7) follows from the minimax characterization of the largest eigenvalue of $\wh R^{2m}$; the ensuing equality states that
%we used that  
 having the same eigenvalues, $\wh R$ and $R$ have same spectral radius, and the last inequality is (ii). 
 Combining \eqv(B1.lem1.8) and \eqv(B1.lem1.7)  proves \eqv(B1.lem1.5).
The proof of the lemma is done.
 \end{proof}

\subsection{Irreducible periodic matrices with period 2}
%, periodic with period 2
 \label{B.2}
 
Assume that $A$ is an $N\times N$ nonnegative and irreducible matrix, namely,
% the matrix $A=(a(i,j))_{1\leq i,j\leq N}$ is nonnegative and irreducible, namely
 $A\geq 0$ and $(I+A)^{N-1}>0$.  Assume further that
  there exists a permutation matrix $B$ such that
\be
\Eq(B2.0)
 B^TAB=
\left(
  \begin{array}{ c c }
     0 & A^{+-} \\
     A^{-+} & 0
  \end{array} \right)\,\,\, \text{and}\,\,\,
   B^TA^2B=
\left(
  \begin{array}{ c c }
     A_2^{++} & 0\\
     0 & A_2^{--}
  \end{array} \right),
\ee
where the square matrices $A_2^{++}\equiv A^{+-}A^{-+}$ and  $A_2^{--}\equiv A^{-+}A^{+-}$ are primitive $N^+\times N^+$, 
respectively $N^-\times N^-$ matrices for some $N^+>0$ and $N^->0$ such that $N^++N^-=N$. This implies that $A$ is periodic with period two, that $A_2^{\pm}$ are aperiodic and, more generally, that powers of $A$ can be analysed in terms of powers of primitive matrices.
%Without loss of generality we may assume that $B$ is the identity matrix. 
To simplify the presentation we assume from now on that $B$ is the identity matrix.
%In what follows 

 Part of Perron's theorem generalizes to nonnegative and irreducible matrices:
 %such an $A$.
 denoting by $\l_0$ its largest eigenvalue  we know that $\rho(A)=\l_0>0$, that there exist a unique (up to a normalization) and positive left (respectively, right)  Perron eigenvector associated to $\l_0$,
% left and right  Perron eigenvectors (i.e. positive eigenv
 %vectors $u$ and $v$ such that
and that $\l_0$ has multiplicity one. Thus the eigenspace associated to this
 eigenvector is one dimensional. However the largest eigenvalue, $\l_0$, is not the only eigenvalue 
 of maximal modulus, i.e.~satisfying $\rho(A)=\l_0$.
 Indeed  one readily derives from \eqv(B2.0) that the spectrum of $A$ is symmetric: writing
% \be
%\Eq(B2.1')
$
 \l_0\geq \l_1\geq \l_2\geq \dots \geq \l_{N-1},
%\ee
$
 we have
\be
\Eq(B2.1)
\l_i = -\l_{N-i-1},\quad 0\leq i\leq N-1,\,\,\,
\ee
and if $Aw=\l_i w$, $w=(w^+,w^-)$ where $w^\pm$ denote vectors in $\R^{N^\pm}$,
% and $w=(w^+,w^-)\in \R^N$ denotes their concatenation, 
then $A\bar w=-\l_{N-i-1} \bar w$ where $\bar w=(w^+,-w^-)$.
%(here  $w^\pm$ denote vectors in $\R^{N^\pm}$ and $w=(w^+,w^-)\in \R^N$ denotes their concatenation).
Thus 
\be
\Eq(B2.2)
\rho(A)=\l_0=|\l_{N-1}|>|\l_i|, \quad 1\leq i\leq N-2.
\ee
Using in addition that the  matrices $A_2^{\pm}$ are primitive, one deduces that
there exist unique left and right Perron eigenvectors, $u$ and $v$,
\bea
\Eq(B2.3)
&&uA=\l_0u,\,\,\,u=(u^+,u^-)>0,\\
\Eq(B2.3')
&&Av=\l_0v,\,\,\,v=(v^+,v^-)>0,
\eea
%a unique vector $u=(u^+,u^-)>0$ such that $uA=\l_0u$, and a unique vector $v=(v^+,v^-)>0$ such that $Av=\l_0v$, 
 normalized such that 
\bea
\Eq(B2.4)
&&\textstyle\sum_{1\leq i\leq N^+}u^+(i)=\sum_{1\leq i\leq N^-}u^-(i)=\frac 12,\\
\Eq(B2.4')
&&\textstyle\sum_{1\leq i\leq N^+}u^+(i)v^+(i)=\sum_{1\leq i\leq N^-}u^-(i)v^-(i)=\frac 12,
\eea
and that, setting $\bar u=(u^+,-u^-)$ and $\bar v=(v^+,-v^-)$, 
$\bar uA=-\l_0\bar u$ and $A\bar v=-\l_0\bar v$.

With obvious notation  define 
\be
\Eq(B2.5)
L =(v(i)u(j))_{1\leq i,j\leq N},\,\,\, \overline L=(\bar v(i)\bar u(j))_{1\leq i,j\leq N}.
\ee
The next lemma establishes that 
\be
\Eq(B2.6)
\lim_{m\rightarrow\infty}(\l_0^{-1}A)^{2m}=(L+\overline L),\,\,\,
\lim_{m\rightarrow\infty}(\l_0^{-1}A)^{2m+1}=(L-\overline L),
\ee
and gives bounds on the speed of convergence under the same symmetry assumption as in Lemma \thv(B1.lem1). Note that $L+\overline L$ has $(i,j)$-th entry $2v^\pm(i)u^\pm(j)$ if  $1\leq i, j\leq N^\pm$ and
zero else, and that  $L-\overline L$ has $(i,j)$-th entry $2v^\pm(i)u^\mp(j)$ if $1\leq i\leq N^\pm$ and $1\leq j\leq N^\mp$, and
zero else.

 \begin{lemma}
   \TH(B2.lem1)
 % Let $A=(a(i,j))_{1\leq i,j\leq N}$ be a nonnegative and primitive matrix and write $A^m=(a^{(m)}(i,j))_{1\leq i,j\leq N}$.
  Assume that 
 there exists a vector 
 %$\pi_A=(\pi_A(i))_{1\leq i\leq N}$
$\pi_A>0$ such that  $\sum_{1\leq i\leq N}\pi_A(i)=1$, and
 \be
\Eq(B2.lem1.0)
 \pi_A(i)a(i,j)=\pi_A(j)a(j,i)\quad\text{for all } 1\leq i,j\leq N.
\ee
Then for all $m\in\N$ we have:
 for all  $1\leq i\leq N^\pm$,
\be
\Eq(B2.lem1.2)
\textstyle
\left|\sum_{1\leq j\leq N^\pm}a^{(2m)}(i,j)-\l_0^{2m}v(i)\right|
\leq 
|\l_1|^{2m}/\sqrt{\pi_A(i)},
\ee
and for all $1\leq i\leq N^\pm$,
\be
\Eq(B2.lem1.4)
\textstyle
\left|\sum_{1\leq j\leq N\mp}a^{(2m+1)}(i,j)-\l_0^{2m+1}v(i)\right|
\leq 
|\l_1|^{2m+1}/\sqrt{\pi_A(i)}.
\ee
 \end{lemma}

 \begin{proof}[Proof of Lemma \thv(B2.lem1)] 
One easily checks that for all $m\in\N$,
$L^m=L$, $\overline L^m=\overline L$, $L\overline L=\overline L L=0$, 
$LA=AL=0$, and $\overline LA=A\overline L=0$.
%$L(A-\Pi)=(A-\Pi)L=0$, and $\overline L(A-\Pi)=(A-\Pi)\overline L=0$. , $\Pi=(L-\overline L)$
Thus, setting $R=A-\l_0(L-\overline L)$ we have, for all $m\in\N$,
\item\hskip1truecm (i)  $R^{2m}=A^{2m}-(\l_0(L-\overline L))^{2m}=A^{2m}-\l_0^{2m}(L+\overline L)$, 
\item\hskip1truecm (ii) $R^{2m+1}=A^{2m+1}-(\l_0(L-\overline L))^{2m+1}=A^{2m+1}-\l_0^{2m+1}(L-\overline L)$,\hfill\break
and, using \eqv(B2.2), 
\item\hskip1truecm (iii) $\rho(A-\l_0(L-\overline L))\leq \l_1<\rho(A)=\l_0$.\hfill\break
%The symmetrized matrix $\wh A$ with entries  $\hat a(i,j)\equiv\sqrt{\pi_A(i)/\pi_A(j)}a(i,j)$ enjoys the same properties.
%From there on the proof is a rerun of the proof of Lemma \thv(B1.lem1).
Using this the proof of Lemma \thv(B2.lem1) is a simple rerun of the proof of Lemma \thv(B1.lem1).
\end{proof}
 % Now the proof of Lemma \thv(B1.lem1) relies on such an assumption
% In Section \thv(5) we 

%the hypercube graph is bipartite as well as any (connected) subgraph of the hypercube

 \hfill\break

%%%%%%%%%%%%%%%%%%%%%%%%%%%%%%%%%%%%%%%%%%%%%%%%%%%%%%%%%%%%

\bibliographystyle {abbrv}      % {plain}

\def\cprime{$'$}

\end{document}